\documentclass[12pt,reqno, final]{amsart}

\usepackage{amssymb,amsmath,graphicx,
	amsfonts,euscript}
\usepackage{color}
\usepackage{cite}

\makeatletter
\@namedef{subjclassname@2020}{%
	\textup{2020} Mathematics Subject Classification}
\makeatother

\allowdisplaybreaks

\newcommand{\be}{\begin{eqnarray}}
\newcommand{\ee}{\end{eqnarray}}
\newcommand{\bes}{\begin{eqnarray*}}
\newcommand{\ees}{\end{eqnarray*}}

\newcommand{\D}{\displaystyle}

\newcommand{\la}{\label}

\newcommand{\Om}{\Omega}

\newcommand{\ep}{\varepsilon}

\newcommand{\na}{\nabla}

\def\R{\mathbb{R}}

\newcommand{\p}{\partial}

\newtheorem{thm}{Theorem}[section]
\newtheorem{pro}{Proposition}[section]
\newtheorem{lem}{Lemma}[section]
\newtheorem{cor}{\textbf Corollary}[section]
\newtheorem{re}{Remark}[section]

\newcommand{\beq}{\begin{equation}}
\newcommand{\eeq}{\end{equation}}
\newcommand{\ben}{\begin{eqnarray}}
\newcommand{\een}{\end{eqnarray}}
\newcommand{\beno}{\begin{eqnarray*}}
\newcommand{\eeno}{\end{eqnarray*}}

\numberwithin{equation}{section}
\subjclass[2020]{35A01, 35B35, 35B65, 76D03, 76E25}
\keywords{3D MHD equations; Background magnetic field;  Partial dissipation; Stability; Decay estimates}

\begin{document}
	
	\title[ MHD equations]
   {Global Stability and Sharp decay estimates for 3D MHD equations with only  vertical dissipation near a background magnetic field}
	
\author[\small S. Lai, J. Wu, J. Zhang  and  X. Zhao]{\small Suhua Lai$^{1}$, Jiahong Wu$^{2}$, Jianwen Zhang$^{3}$ and  Xiaokui Zhao$^{4}$}

\address{1. Suhua Lai,  School of Mathematics and Statistics \& Jiangxi Provincial Center for Applied Mathematics, Jiangxi Normal University, Nanchang, Jiangxi,  330022, P. R. China}
 \email{shlai@jxnu.edu.cn}

 \address{2. Jiahong Wu, Department of Mathematics, University of Notre Dame, Notre Dame, IN 46556, USA}
 \email{jwu29@nd.edu}

\address{3. Jianwen Zhang, School of Mathematical Sciences, Xiamen University, Xiamen, Fujian, 361005, P. R. China}
\email{jwzhang@xmu.edu.cn}

\address{4. Xiaokui Zhao, School of Mathematics and Information Science, Henan Polytechnic University, Jiaozuo, Henan 454000, P. R. China}
 \email{xkzhao@hpu.edu.cn}

\vskip .2in
\begin{abstract}	
 This paper is concerned with the stability and large-time behavior of 3D incompressible MHD equations with only  vertical dissipation near a background magnetic field.  By making full use of  the dissipation generated by the background magnetic field, we first establish the global stability of the solutions in $H^3$-norm. Then, the optimal decay rates of the solutions are obtained, which are consistent with the 2D classical heat equation. Moreover, some  enhanced decay rates of $(u_1,b_1)$ are also achieved.  In other words, the decay estimates of the second or third component of velocity/magnetic field coincide with those of 2D heat kernel, while the first component behaves like the 3D heat kernel. This is mainly due to the divergence-free condition  and the anisotropic structure. The results obtained improve the previous ones due to Lin-Wu-Zhu \cite{LinWZ1,LinWZ2}.

\end{abstract}
	\maketitle

\section{Introduction}
Magnetohydrodynamics (MHD) is one of the most important branches of continuum mechanics, which deals with the interplay between electromagnetic fields and conducting fluids. The MHD system is coupled by the Navier-Stokes equations for fluid dynamics and the Maxwell equations for electromagnetism.
Since the pioneering work of Alfv\'{e}n \cite{Alf}, MHD is widely used in many fields, such as geomagnetism and planetary magnetism, astrophysics, nuclear fusion (plasma) physics, and liquid metal technology (see, e.g., \cite{BD, DA, PF}).  In recent years, the stability of MHD equations has attracted a lot of attention. A significant nonlinear phenomena of MHD  observed by physical experiments and numerical simulations is that a background magnetic field can smooth and stabilize turbulent electrically conducting fluids (see, e.g., \cite{AMSF, AL, GBM, IGBHM, Moffatt}).

In this paper, we consider  the following 3D  MHD equations with only vertical velocity dissipation:
\begin{equation}\label{MHD1}
\begin{cases}
U_t-\mu\partial_3^2U=B\cdot\nabla B-U\cdot\nabla U-\na P, \quad x\in\R^3,\,t>0,\\
B_t-\eta\p_2^2 B-\eta\p_3^2B=B\cdot\nabla U-U\cdot\nabla B,\\
{\rm \nabla\cdot}\,U={\rm \nabla\cdot}\,B=0,\\
(U,B)(x,t)|_{t=0}=(U_0,B_0)(x),
\end{cases}
\end{equation}	
where $U=(U_1, U_2, U_3)^{\top},~ B=(B_1, B_2, B_3)^{\top}$ and $P$ are the velocity field, the
magnetic field, and the pressure, respectively. The positive constants
$\mu>0$ and  $\eta>0$  are the viscosity coefficient and the magnetic diffusivity coefficient, respectively.
\vskip .1in
Our goal is to understand the stability of  3D anisotropic MHD system (\ref{MHD1}) near a equilibrium state $(U^{(0)},B^{(0)})$,
$$
U^{(0)} \equiv 0, \quad B^{(0)} \equiv e_2 := (0,1,0).
$$
Thus, the perturbation $(u, b)$ around the steady state $(U^{(0)},B^{(0)})$ with
$$
u :=U-U^{(0)}, \qquad b:= B-B^{(0)}.
$$
is  governed by
\begin{equation}\label{PMHD}
\left\{\begin{aligned} &u_t-\mu\partial_3^2u-\partial_2b=b\cdot\nabla b-u\cdot\nabla u-\na p,\quad x\in\R^3,\,t>0,
\\
&b_t-\eta\p_2^2 b-\eta\p_3^2b-\partial_2u=b\cdot\nabla u-u\cdot\nabla b,
\\
&{\rm \nabla\cdot}\,u={\rm \nabla\cdot}\,b=0,\\
&u(x, 0) = u_0(x),\,\,\, b(x, 0) = b_0(x).
\end{aligned}
\right.
\end{equation}
\vskip .1in
In fact,  if $B=0$ in (\ref{MHD1}), then it reduces to the 3D Navier-Stokes equations with only vertical  dissipation,
\begin{equation}\label{NS}
\begin{cases}
U_t-\mu\partial_3^2U=B\cdot\nabla B-U\cdot\nabla U-\na P, \quad x\in\R^3,\,t>0,\\
{\rm \nabla\cdot}\,U=0,\\
U(x,t)|_{t=0}=U_0(x).
\end{cases}
\end{equation}
The global stability and large-time behavior of (\ref{NS}) is a challenging  problem. The major difficulty lies in the fact that the dissipation in only one direction is too weak to  control all the nonlinear terms in the whole space $\R^3$.
In a recent work \cite{PZ},  Paicu-Zhang    established the global well-posedness of smooth solutions with  small data  for (\ref{NS}) on the strip domain  $\Om=\R^2\times [0,1]$. The analysis of \cite{PZ} strongly relies on the geometry of the domain and a {Poincar\'{e}} type inequality, which no longer holds true for the case of whole space. Motivated by this, the main purpose of this paper is to study the global stability of MHD and to reveal  the influence of magnetic field on fluid stability from a mathematical point of view.
\vskip .1in
In the following, we briefly recall some noteworthy works on the stability of   MHD equations near a background magnetic field. In the pioneering work \cite{BSS},   Bardos-Sulem-Sulem  first demonstrated the stability effect of magnetic fields in ideal MHD. Cai-Lei \cite{CL} and He-Xu-Yu \cite{HXY}, via different approaches, successfully solved the stability problem of the ideal MHD system and its fully dissipative counterpart
(with identical viscosity and resistivity) near a background magnetic field.
Wei-Zhang \cite{WZ} allowed the viscosity and resistivity coefficients to be slightly different.
The  stability of 2D viscous and non-resistive MHD equations was first considered by Lin-Xu-Zhang \cite{LXZ}, and then was extended and improved by  \cite{ CR,  JFJS, RWXZ, RXZ, Wu1, ZT2}. We also refer to \cite{AZ,DengZ,PZZ,TW} for the stability theory of 3D viscous and non-resistive MHD equations. In \cite{ZZ},  Zhou-Zhu  proved the stability of  2D inviscid and resistive MHD equations on the periodic domain by using  the odevity conditions proposed in \cite{PZZ}. Wu-Wu-Xu \cite{Wu3} (also cf.  \cite{DYZ}) studied the stability of 2D  MHD system with only velocity damping.
\vskip .1in

Recently, the stability of the incompressible MHD equations with partial dissipation has  attracted more and more attention, since it mathematically reveals the stability mechanism of a background magnetic field, compared with the Navier-Stokes equations. For 2D MHD equations with partial dissipation,  one can refer to  \cite{LJWY,BLWSIAM,FengFWu,LWZ1,LWZ2,LXW,PaicuZ} and the references therein for interest. Next, we would like to introduce some works on the stability of 3D  MHD equations with partial dissipation, which is certainly more complicated and difficult. The first work is due to  Wu-Zhu \cite{Wu4}, where the authors considered the  3D  MHD equations with horizontal dissipation $\Delta_{h}u$ and vertical magnetic diffusion $\p^2_{3}b$. Then, Lin-Wu-Zhu \cite{LinWZ1} proved  the stability of 3D  MHD equations with  only one-directional dissipation $\p^2_{1}u$ and horizontal magnetic diffusion $\Delta_{h}b$ in Sobolev space $H^4(\R^3)$. Lin-Wu-Zhu \cite{LinWZ2}  improved the stability result under the assumption that initial data $(u_0,b_0)$ satisfy
$$\|(u_0,b_0)\|_{H^3}+\|\|\p_1^k\left(u_0,b_0\right)\|_{L_{x_1}^2 L_{x_2x_3}^1} \leq \varepsilon, \quad\text{with}\quad k=0,1,2.$$
Recently,  the stability of an initial-boundary value problem of 3D  MHD equations with only vertical dissipation $\p^2_{3}u$ and $\p^2_{3}b$ was studied by  Lin-Wu-Suo \cite{LWS} in the case of strip domain $\Om=\R^2\times [0,1]$ with Dirichlet boundary conditions by using the Poincar\'{e} type inequality.

The main purpose of this paper is to study the stability theory and large-time behavior of the Cauchy problem (\ref{PMHD}) on the whole domain $\R^3$. Our first result of this paper, concerning the global solvability and stability   in the $H^3$-framework, is formulated in the following theorem.
\begin{thm}\la{thm1.1}
	Assume that $(u_0,b_0)\in H^3$ with ${\rm \nabla\cdot}\,u_0={\rm \nabla\cdot}\,b_0=0$. There exists an absolutely  positive constant  $\varepsilon>0$, depending only on $\mu$ and $\eta$, such that if
	$$\|(u_0,b_0)\|_{H^3}\leq \varepsilon,$$
	then the  problem \eqref{PMHD} has a unique global solution $(u, b)$ on $\R^3\times[0,\infty)$, satisfying
	\begin{equation}
	\|(u,b)(t)\|_{H^3}^2+\int_0^T\left(\|(\p_3u,\p_2b,\p_3 b)\|_{H^3}^2+\|\p_2u\|_{H^2}^2\right)d t\leq \ep^2.	\label{int0}
	\end{equation}
\end{thm}
\begin{re}
   Theorem \ref{thm1.1} improves the resuts obtained in \cite{LinWZ1,LinWZ2}. In  \cite{LinWZ1}, Lin-Wu-Zhu studied the global stability under the smallness condition $\|(u_0,b_0)\|_{H^4}\leq \varepsilon$. This result was extended and improved by \cite{LinWZ2}, where the following conditions were technically needed for the global $H^3$ stability:
 $$\|(u_0,b_0)\|_{H^3}+ \|\p_1^k\left(u_0,b_0\right)\|_{L_{x_1}^2L_{x_2x_3}^1} \leq \varepsilon,, \quad\text{with}\quad k=0,1,2.$$
 In the present paper, the global  $H^3$ stability  of the solutions is shown by requiring only that the initial data $\|(u_0,b_0)\|_{H^3}$ is suitably small.
\end{re}

The issue of large-time behavior and  decay rates of the global solutions is more subtle.
We will derive the sharp decay rates of the solutions by the method of spectral analysis. For the notational convenience,  we denote by
$$
\begin{cases}
\|f\|_{L^p}\triangleq \|f\|_{L^{p}(\mathbb R^3)},\quad \|f\|_{W^{k,p}}\triangleq\|f\|_{W^{k,p}(\mathbb R^3)},\\
\|f\|_{L^2}\triangleq \|f\|_{L^{2}(\mathbb R^3)},\quad \|f\|_{W^{k,2}}\triangleq\|f\|_{H^k}=\|f\|_{H^k(\mathbb R^3)},\\
\|f\|_{L^rL^qL^p}\triangleq \left\|\left\|\|f\|_{L^p(\R)}\right\|_{L^q(\R)}\right\|_{L^r(\R)},\quad \|f\|_{L_{x_1}^qL_{x_2,x_3}^p}\triangleq \left\| \|f\|_{L_{x_2,x_3}^p(\R^2)} \right\|_{L_{x_1}^q(\R)}.
\end{cases}
$$
Moreover, the same letter  $ C $ will be used to denote the generic positive constants, which may be different from  line to line.

Our second and main result of this paper is concerned with the sharp decay rates of the global solutions.
\begin{thm}\label{thm1.2}
	In addition to the conditions of Theorem \ref{thm1.1}, assume that $(u_0,b_0)\in L^1$ and $\p_1^k(u_0,b_0)\in L^2_{x_1}L^1_{x_2,x_3} $ for $ k=0,1,2 $. There exists  an absolutely  positive constant  $\varepsilon>0$, such that if
	\begin{align} \label{int1}
	\|(u_0,b_0)\|_{H^3}+\|(u_0,b_0)\|_{L^1}+ \|\p_1^k(u_0, b_0)\|_{L_{x_1}^2L_{x_2,x_3}^1} \leq \varepsilon,
	\end{align}
	then   the global solutions  $(u,b)$ of the problem \eqref{PMHD} satisfy
	\begin{eqnarray*}
		&& \|\p_1^k(u,b)(t)\|_{L^2}\leq C_0\varepsilon(1+t)^{-\frac{1}{2}},  \quad\quad\,\,  \|\p_{i}(u, b)(t)\|_{L^2}\leq C_0\varepsilon(1+t)^{-1},  \\
		&& \|\p_{i}\p_3(u,b)(t)\|_{L^2}\leq C_0\varepsilon(1+t)^{-\frac32},\quad\;\,\,
		\|\p_l\p_2(u,b)(t)\|_{L^2}\leq C_0\varepsilon(1+t)^{-\frac{11}{12}},  \\
		&& \|\p_{1}\p_3(u,b)(t)\|_{L^2}\leq C_0\varepsilon(1+t)^{-1}, \quad\;\;\, \|\p_1^2(u,b)\|_{L^2}\leq C_0\varepsilon(1+t)^{-\frac{3}{8}}, \\
		&& \|\p_l\p_2\p_3(u,b)(t)\|_{L^2}\leq C_0\varepsilon(1+t)^{-\frac{11}{12}},\;\, \|\p_1^2\p_3(u, b)(t)\|_{L^2}\leq C_0\varepsilon(1+t)^{-\frac12},
		\\
		&&\|\p_1\p_3^2(u, b)(t)\|_{L^2}\leq C_0\varepsilon(1+t)^{-\frac54},\quad \;\; \|\p_2\p_3^2(u, b)(t)\|_{L^2}\leq C_0\varepsilon(1+t)^{-\frac{11}{6}},\\
		&& \| \p_3^3(u, b)(t)\|_{L^2}\leq C_0\varepsilon(1+t)^{-\frac{23}{12}},\quad\;\;\;\; \|(u_1,b_1)(t)\|_{L^2}\leq C_0\varepsilon(1+t)^{-\frac{3}{4}},\\
 && \|\p_{i}(u_1, b_1)(t)\|_{L^2}\leq C_0\varepsilon(1+t)^{-\frac{5}{4}},  \quad\;\;\,
 \|\p_{i}\p_3(u_1, b_1)(t)\|_{L^2}\leq C_0\varepsilon(1+t)^{-\frac{7}{4}},\\
 &&  \| \p_3^3(u_1, b_1)(t)\|_{L^2}\leq C_0\varepsilon(1+t)^{-\frac{9}{4}},
	\end{eqnarray*}
	where $k=0,1,~i=2,~3,~l=1,~2$.
	
\end{thm}

\begin{re} The decay rates stated in
  Theorem \ref{thm1.2} are  sharper  than those derived in \cite{LinWZ2}.  Indeed, Lin-Wu-Zhu \cite{LinWZ2} showed that for any $\delta\in(0,1)$,
	\begin{equation*}
	\begin{cases}
	\|(\p_1\p_ju,\p_1\p_jb)(t)\|_{L^2}\leq C(1+t)^{-\frac{5}{4}+\delta},\quad \quad~j=1,2,\\
	\|(\p_{2}\p_ju,\p_{2}\p_jb)(t)\|_{L^2}\leq C(1+t)^{-\frac{2}{3}+\delta}, ~~~j=2,3.	
	\end{cases}
	\end{equation*}
which are improved in the present paper to
	\begin{equation*}
	\begin{cases}
	\|(\p_1\p_ju,\p_1\p_jb)(t)\|_{L^2}\leq C(1+t)^{-\frac{3}{2}},\quad\quad~~~~j=1,2,\\
	\|(\p_{2}\p_ju,\p_{2}\p_jb)(t)\|_{L^2}\leq C(1+t)^{-\frac{11}{12}}, ~~j=2,3.	
	\end{cases}
	\end{equation*}
	Furthermore, we can prove the decay rates of the third-order derivatives $\|\p_i\p_j\p_3(u, b) \|_{L^2}$ with $i,j\in\{1,2\} $.
	It is also worth pointing out that the decay rates of $(u,b)$, $\p_i(u,b)$ and $\p_i\p_3(u,b)$ with $i\in\{2,3\}$ are the same as those for 2D heat equation, while the decay rates of $(u_1,b_1)$ are consistent with the 3D heat equation and are optimal.
\end{re}

If we assume further that $(u_0,b_0)\in{H^4}$,  then we can show that the decay rates of the partial derivatives involving the third direction $x_3$ are optimal, compared with the 2D heat equation.

\begin{thm}\label{thm1.3}
	Assume that for $ k=0,1,2,3$, $\p_1^k(u_0,b_0)\in L^2_{x_1}L^1_{x_2,x_3}$, $ (u_0,b_0)\in H^4$,
	$ \nabla\cdot u_0=0 $, and $ \nabla\cdot b_0=0 $. Then, there exists sufficiently small $\varepsilon>0 $, such that if
	\begin{equation}\label{int2}
	\|(u_0,b_0)\|_{H^4}+\|(u_0,b_0)\|_{L^1}+\|\p_1^k(u_0,b_0)\|_{L_{x_1}^2L_{x_2,x_3}^1} \leq\varepsilon,
	\end{equation}
	then  the following large-time decay estimates hold for the global solutions $(u,b)$:
	\begin{eqnarray*}
		&& \|\p_1^k(u,b)(t)\|_{L^2}\leq C_0\varepsilon(1+t)^{-\frac{1}{2}},  \;\;\;\;\:\;\;\;\;    \|\p_{i}(u,b)(t)\|_{L^2}\leq C_0\varepsilon(1+t)^{-1},  \\
		&& \|\p_1\p_{i}(u,b)(t)\|_{L^2}\leq C_0\varepsilon(1+t)^{-1},\;\;\;\;\:\;\:
		\|\p_i\p_j(u, b)(t)\|_{L^2}\leq C_0\varepsilon(1+t)^{-\frac32},  \\
		&& \|\p_1^2(u, b)\|_{L^2}\leq C_0\varepsilon(1+t)^{-\frac{1}{2}}, \;\;\;\;\;\;\;\;\:\;\;\:\; \|\p_{1}\p_{i}\p_3(u,b)(t)\|_{L^2}\leq C_0\varepsilon(1+t)^{-\frac32}, \\
		&& \|\p_l\p_2^2(u,b)(t)\|_{L^2}\leq C_0\varepsilon(1+t)^{-1},\;\;\,\;\:\;\:  \|\p_1^2\p_i(u,b)(t)\|_{L^2}\leq C_0\varepsilon(1+t)^{-1},\\
		&& \|\p_{i}\p_{j}\p_3(u,b)(t)\|_{L^2}\leq C_0\varepsilon(1+t)^{-2}, \quad
		\|\p_1^3(u,b)(t)\|_{L^2}\leq C_0\varepsilon(1+t)^{-\frac38},\\
		&&\|\p_1^3\p_3(u,b)(t)\|_{L^2}\leq C_0\varepsilon(1+t)^{-\frac12},\quad \;\; \|\p_1^2\p_3^2(u,b)\|_{L^2}\leq C_0\varepsilon(1+t)^{-\frac{11}{8}},\\
		&& \|\p_l\p_h\p_2\p_3(u,b)\|_{L^2}\leq C_0\varepsilon(1+t)^{-1},\quad\, \|\p_1\p_3^3(u,b)\|_{L^2}\leq C_0\varepsilon(1+t)^{-2},\\
		&&\|\p_l\p_2\p_3^2(u,b)\|_{L^2}\leq C_0\varepsilon(1+t)^{-2},
		\quad\quad\,
		\|\p_i\p_3^3(u,b)\|_{L^2}\leq C_0\varepsilon(1+t)^{-\frac52},\\
		&&\|(u_1,b_1)\|_{L^2}\leq C_0\varepsilon(1+t)^{-\frac34},
		\quad\quad\quad\quad
		\|\p_i(u_1,b_1)\|_{L^2}\leq C_0\varepsilon(1+t)^{-\frac54},	
		\\
		&&\|\p_i\p_j(u_1,b_1)\|_{L^2}\leq C_0\varepsilon(1+t)^{-\frac74},
		\quad\quad\,\,
		\|\p_i\p_3^2(u_1,b_1)\|_{L^2}\leq C_0\varepsilon(1+t)^{-\frac94},\\
		&&\|\p_i\p_3^3(u_1,b_1)\|_{L^2}\leq C_0\varepsilon(1+t)^{-\frac{11}{4}},	
	\end{eqnarray*}
	where $k=0,1,~i,~j=2,~3,~l,h=1,~2$.
	\end{thm}
\begin{re}
The method herein can be applied to the $H^N$-framework with $N\geq 4$. In particular, if
$$\|(u_0,b_0)\|_{H^N}+\|(u_0,b_0)\|_{L^1}+\|\p_1^k(u_0, b_0)\|_{L_{x_1}^2L_{x_2,x_3}^1} \leq\varepsilon
$$
with $ k=0,1,\cdots, N-1$, then it holds that
$$
\|(\p_3^{N}u,\p_3^{N}b)\|_{L^2}\leq C_0\varepsilon(1+t)^{-\frac{N+1}{2}}.
$$
\end{re}

\begin{re} Due to the resemblance of mathematical structure,  the same results as the ones stated in
 Theorem \ref{thm1.1}--\ref{thm1.3} can be achieved  for the 3D MHD equations with only  vertical magnetic diffusion:
$$
\begin{cases}
\partial_t u+u\cdot
\nabla u+\nabla p=
\nu\partial_2^2u+\nu\partial_3^2u+b\cdot \nabla b+\partial_{2}b,\\
\partial_t b+u\cdot
\nabla b=\eta\p_3^2b+b\cdot\nabla u+\partial_2u,\\
{\rm \nabla\cdot}\,u={\rm \nabla\cdot}\,b=0,\\
(u,b)(x,t)|_{t=0}=(u_0,b_0)(x).
\end{cases}
$$
\end{re}

Next, we briefly sketch the main ideas of the proofs in Theorems \ref{thm1.1}--\ref{thm1.3}.  Since the local-in-time existence result can be shown by the standard method (see, e.g., \cite{MaBe}), to prove Theorem \ref{thm1.1}, our main task is to derive the global-in-time a prior estimates of the solutions.  The main difficulties arise from  the lack of velocity dissipations in both $x_1$ and $x_2$ directions, which will be supplemented by  the additional dissipation generated by the background magnetic field.  To do this, we adopt the method of ``double energy", which have been successfully used in the previous works (see, for example, \cite{LWZ1,LWZ2}, \cite{LJWY}--\cite{LinWZ2}, etc). The first energy functional is the  natural $H^3$-energy $\mathcal{E}_1(t)$ induced by the partial dissipations of (\ref{PMHD}),
\begin{align}
\mathcal{E}_1(t)\triangleq\sup_{0\le\tau\le t}\|(u,b)(\tau)\|_{H^3}^2+2\int_0^t \left(\mu\|\p_3u(\tau)\|_{H^3}^2 +\eta\|\left(\p_2b, \p_3b\right)(\tau)\|_{H^3}^2  \right)d \tau, \label{E1}
\end{align}
and the second one $\mathcal{E}_2(t)$ is the additional $H^2$-energy of $\p_2u$,
\begin{align}
\mathcal{E}_2(t)\triangleq \int_0^t \|\p_2u(\tau)\|_{H^2}^2  d \tau,\label{E2}
\end{align}
which is indeed generated by the background magnetic field. The estimates $\mathcal{E}_1(t)$ and $\mathcal{E}_2(t)$ will be built up, based on the standard $L^2$-method and the applications of anisotropic Sobolev inequalities (cf. Lemmas \ref{anip1} and \ref{anip2}).
However, some of
the nonlinear terms   cannot be bounded by $\mathcal{E}_1(t)$ and $\mathcal{E}_2(t)$ directly. The most difficult terms are
\beno
{\rm D}_1 \triangleq2\int \p_2 u_2 \left(\p_1^3 u_2\right)^2   d x,
\,  {\rm D}_2 \triangleq \int  \p_1^3 u_2 \p_2 u_3\p_1^3 u_3  \ d x,
\, {\rm D}_3\triangleq3\int \p_2 u_2 \left(\p_1^3 u_3\right)^2   d x.
\eeno
To deal with these terms,  we make use of (\ref{PMHD}) to replace $\p_2u_2$ and $ \p_2u_3$ by (see (\ref{m13}))
\begin{align}
\p_2u_2&=\p_t(b_2-\eta\p_2u_{2})-\eta\p_3^2b_2
+\mu\eta\p_2\p_3^2u_2\nonumber
\\
&\quad
-\eta\p_2\left(
u\cdot\na u_2-b\cdot\na b_2+\p_2p
\right)+u\cdot\na b_2-b\cdot\na u_2,\nonumber\\
\p_2u_3&=\p_t(b_3-\eta\p_2u_{3})-\eta\p_3^2b_3
+\mu\eta\p_2\p_3^2u_3
\nonumber\\
&\quad
-\eta\p_2\left(
u\cdot\na u_3-b\cdot\na b_3+\p_3p
\right)+u\cdot\na b_3-b\cdot\na u_3.\label{m131}
\end{align}
For example, by several substitutions and integration by parts, the difficulty of the term $D_1$ is shifted to the treatment of $F_4$ (see \eqref{D1g}),
$$
F_4\triangleq4\int b\cdot\na\p_1^3 b_2\p_1^3u_2(b_2-\eta\p_2u_2)\ dx.
$$
Also, this cannot be bounded directly. To circumvent this difficulty, we  artificially retain the term $E_1$ (see \eqref{H53}),
\begin{align*}
E_1\triangleq2\int\left(\p_1^3b_2\right)^2\p_2u_2\ dx,
\end{align*}
which arises from  the treatment of the nonlinear term $b\cdot \na u $. Based on (\ref{m131}) and integration by parts, the term $E_1$ will produce a ``good" term $R_4$  (see \eqref{E11g}),
$$
R_4\triangleq4 \int b\cdot\na\p_1^3 u_2\p_1^3b_2(b_2-\eta\p_2u_2) dx,
$$
which, combined with $F_4$ and integrated by parts, leads to the desired bound of $D_1+E_1$ (see \eqref{DE1g}).
The estimates of the other two terms $D_2$ and $D_3$ can be done in a similar manner, if we add the terms $E_2$ and $E_3$,
\begin{align*}
 E_2\triangleq\int \p_1^3b_2 \p_2u_3 \p_1^3b_3\ dx,  \quad  E_3\triangleq3\int\left(\p_1^3b_3\right)^2\p_2u_2\ dx,
\end{align*}
to $D_2$ and $D_3$, respectively. With all the global a priori estimates at hand,  Theorem \ref{thm1.1} then follows from the bootstrap  arguments (see, e.g. \cite{Tao2006}). We refer to Section \ref{Sec.2} for more technical details.
\vskip .1in

The proofs of the decay rates  stated in Theorems \ref{thm1.2}--\ref{thm1.3} are built upon the spectral analysis and the integral representations in \eqref{u} and \eqref{b}. In particular, the precise  upper bounds for kernel functions play a mathematically important role. This forces us to divide  the frequency space into suitable sub-domains and to derive the sharp estimates of each kernel function in these sub-domains (see Proposition \ref{lem2.1}). By virtue of Proposition \ref{lem2.1},  we then carry out some  reasonable ansatz of the decay rates and provide the decay estimates of the nonlinear terms. With these preparations at hand, we can adopt the bootstrap argument to establish the desired decay estimates by elaborate calculations. The entire analysis strongly relies on Corollaries \ref{PI} and \ref{divfree}, which coincide respectively with the estimates of 2D and 3D heat kernel. It is worth mentioning that Corollary \ref{divfree} is a consequence of the divergence-free condition and  will be used to prove the enhanced  decay rates of $(u_1,b_1)$. Compared to the previous work \cite{LinWZ2}, the analysis of the present paper is more delicate and technical. The details can be found in Section 4 and 5.
\vskip .1in

The rest of the paper is organized as follows.  In Section 2, we present some useful mathematical tools and establish the sharp upper bounds of kernel functions in different sub-domains. The global stability result, i.e. Theorem \ref{thm1.1}, will be shown in Section 3. The decay rates stated in Theorems \ref{thm1.2} and \ref{thm1.3} will be derived in Section 4 and 5,
respectively.

\section{Preliminaries}
\subsection{Some elementary inequalities} \label{Sec}
In this subsection, we will recall some elementary inequalities
and results which will be used frequently later. We begin with  the anisotropic inequalities for triple products (see \cite{LinWZ1,LinWZ2, Wu4}).
\begin{lem}\label{anip1}
	Assume $f$, $\p_1f$, $g$, $\p_2g$ , $h$  and $\p_3 h$  are all in $L^2(\mathbb R^3)$. Then,
	\begin{align}
	&\int_{\mathbb R^3}|fgh|dx\leq C\|f\|_{L^2}^{\frac12}\|\partial_1f\|_{L^2}^{\frac12}\|g\|_{L^2}^{\frac12}\|\partial_{2}g\|_{L^2}^{\frac12}\|h\|_{L^2}^{\frac12}\|\partial_{3}h\|_{L^2}^{\frac12},\nonumber\\
	&\int_{\mathbb R^3}|fgh|dx\leq C\|f\|_{L^2}\|g\|_{L^2}^{\frac14}\|\partial_{i}g\|_{L^2}^{\frac14}
	\|\p_jg\|_{L^2}^{\frac14}\|\p_i\p_jg\|_{L^2}^{\frac14}\|h\|_{L^2}^{\frac12}\|\partial_{k}h\|_{L^2}^{\frac12}\nonumber\\
	&\quad \quad \quad \quad \quad \leq C\|f\|_{L^2}\|g\|_{H^1}^{\frac12}\|\p_ig\|_{H^1}^{\frac12}\|h\|_{L^2}^{\frac12}\|\partial_{k}h\|_{L^2}^{\frac12},\nonumber
	\end{align}
	where $i,j,k\in{1,2,3} $ and $i\neq j\neq k$.
\end{lem}

The following lemma provides the anisotropic upper bounds for quadruple product, which is  very mathematically helpful in the treatments of nonlinear terms.
\begin{lem} \label{anip2}
	The following estimates hold when the right-hand sides are all bounded.
	\begin{align}
	\int_{\mathbb R^3}|efgh|dx
	&\leq C \|e\|_{L^2}^{\frac12}\|\p_i e\|_{L^2}^{\frac12}\|f\|_{L^2}^{\frac12}\|\p_i f\|_{L^2}^{\frac12} \|g\|_{H^{1}}^{\frac12}\|\p_kg\|_{H^{1}}^{\frac12} \|h\|_{H^{1}}^{\frac12}\|\p_kh\|_{H^{1}}^{\frac12}\nonumber\\
	&\leq C(\| (e,f )\|_{L^2}^{2}+\| (g,h )\|_{H^1}^{2} ) (\| (\p_ie,\p_if )\|_{L^2}^{2}+\|(\p_kg,\p_kh)\|_{H^1}^{2}),\nonumber
	\end{align}
where $i,k\in{1,2,3} $ and $i\neq k$.
\end{lem}
\begin{proof}
The proof  relies on the basic one-dimensional Sobolev inequality:
 	\begin{align} \|g\|_{L^{\infty}(\R)}\leq\sqrt{2}\|g\|_{L^2(\R)}^{\frac12}\|g'\|_{L^2(\R)}^{\frac12},\label{1IE}
\end{align}
and the Minkowski inequality
\begin{align}
\|\|f\|_{L^{q}_y(\R^n)}\|_{L^{p}_x(\R^m)}\leq\|\|f\|_{L^{p}_x(\R^m)}\|_{L^{q}_y(\R^n)}, \quad \forall \,\, 1\leq q\leq p \leq \infty,\label{MIE}
\end{align}
where $f=f(x,y)$ is a measurable function on $(x,y)\in\R^m\times\R^n$. With the help of \eqref{1IE}--\eqref{MIE}, we have
\begin{align}
&\int_{\mathbb R^3}|efgh|dx
\leq C \|e\|_{L_{x_j,x_k}^{2}L_{x_i}^{\infty}}\|f\|_{L_{x_j,x_k}^{2}L_{x_i}^{\infty}}\|g\|_{L_{x_j,x_k}^{\infty}L_{x_i}^{2}}\|h\|_{L_{x_j,x_k}^{\infty}L_{x_i}^{2}}\nonumber \\
&\quad\leq C \left\|\|e\|_{L^2_{x_i}}^{\frac12}\|\p_i e\|_{L^2_{x_i}}^{\frac12}\right\|_{L_{x_j,x_k}^{2} }\left\|\|f\|_{L^2_{x_i}}^{\frac12}\|\p_i f\|_{L^2_{x_i}}^{\frac12}\right\|_{L_{x_j,x_k}^{2} }
\|g\|_{L_{x_i}^{2}L_{x_j,x_k}^{\infty}}\|h\|_{L_{x_i}^{2}L_{x_j,x_k}^{\infty}}\nonumber \\
&\quad\leq C \|e\|_{L^2}^{\frac12}\|\p_i e\|_{L^2}^{\frac12}\|f\|_{L^2}^{\frac12}\|\p_i f\|_{L^2}^{\frac12}\|g\|_{L_{x_i}^{2}L_{x_j,x_k}^{\infty}}\|h\|_{L_{x_i}^{2}L_{x_j,x_k}^{\infty}}. \label{4an}
\end{align}

By virtue $\eqref{1IE}$ and  H\"{o}lder's inequality, we obtain
\begin{align}
\|g\|_{L_{x_i}^{2}L_{x_j,x_k}^{\infty}}
&\leq C\left\| \|g\|_{L^2_{x_k}}^{\frac12}\|\p_k g\|_{L^2_{x_k}}^{\frac12}\right\|_{L_{x_i}^2L_{x_j}^{\infty}}\leq C\left\| \|g\|_{L_{x_j}^{\infty} }\right\|_{L_{x_i,x_k}^{2}}^{\frac12}\left\|\|\p_k g\|_{L_{x_j}^{\infty}}\right\|_{L_{x_i,x_k}^{2}} ^{\frac12}\nonumber\\
&\leq C \|g\|_{L^{2}}^{\frac14}\|\p_jg\|_{L^{2}}^{\frac14} \|\p_kg\|_{L^{2}}^{\frac14}\|\p_{jk}g\|_{L^{2}}^{\frac14}\leq C \|g\|_{H^{1}}^{\frac12}\|\p_kg\|_{H^{1}}^{\frac12}, \label{ge}
\end{align}
and analogously,
\begin{align}
\|h\|_{L_{x_i}^{2}L_{x_j}^{\infty}L_{x_k}^{\infty}}
\leq C \|h\|_{H^{1}}^{\frac12}\|\p_kh\|_{H^{1}}^{\frac12}. \label{he}
\end{align}

Now, substituting \eqref{ge} and \eqref{he} into \eqref{4an}, we infer from Cauchy-Schwar's inequality that
\begin{align}
&\int_{\mathbb R^3}|efgh|dx\leq C \|e\|_{L^2}^{\frac12}\|\p_i e\|_{L^2}^{\frac12}\|f\|_{L^2}^{\frac12}\|\p_i f\|_{L^2}^{\frac12} \|g\|_{H^{1}}^{\frac12}\|\p_kg\|_{H^{1}}^{\frac12} \|h\|_{H^{1}}^{\frac12}\|\p_kh\|_{H^{1}}^{\frac12}\nonumber\\
&\quad\leq C \|e\|_{L^2}^2\|\p_i e\|_{L^2}^2+\|f\|_{L^2}^2\|\p_i f\|_{L^2}^2+ \|g\|_{H^{1}}^2\|\p_kg\|_{H^{1}}^2+ \|h\|_{H^{1}}^2\|\p_kh\|_{H^{1}}^2\nonumber\\
&\quad\leq C\left(\|\left(e,f\right)\|_{L^2}^{2}+\|\left(g,h\right)\|_{H^1}^{2}\right)\left(\|\left(\p_ie,\p_if\right)\|_{L^2}^{2}+\|\left(\p_kg,\p_kh\right)\|_{H^1}^{2}\right). \nonumber
\end{align}
The proof of Lemma \ref{anip2}  is thus complete.
\end{proof}

\subsection{ Integral representation  and  the kernels}

In this subsection, we aim to derive
integral representation of  (\ref{PMHD}) and study the sharp bounds for the kernel functions, which will be used to prove the decay rates of the solutions. To begin, we first operate the Fourier transform of (\ref{PMHD}) to get that ($\widehat u= (\widehat{u_1},\widehat{u_2},\widehat{u_3})$ and $\widehat b= (\widehat{b_1},\widehat{b_2},\widehat{b_3})$)
\begin{equation*}
\p_t \left(\begin{array}{c}
\widehat{u_i}\\\\
\widehat{b_i}\\\end{array}\right)=\mathbb A\left(\begin{array}{c}
\widehat{u_i}\\\\
\widehat{b_i}\\\end{array}\right)+\left(\begin{array}{c}
\widehat{N_{1i}}\\\\
\widehat{N_{2i}}\\\end{array}\right), \quad\forall\ i=1,2,3,
\end{equation*}
where $\mathbb A$ comes from the linear operators, and $N_1$, $N_2$ are the nonlinear terms,
$$\mathbb A\triangleq \left(\begin{array}{cc}
-\mu\xi_3^2~&i\xi_2\\\\
i\xi_2~&-\eta\left(\xi_2^2+\xi_3^2\right)\\\end{array}\right) $$
and
$$
\begin{cases}
 N_1=( N_{11},N_{12},N_{13})\triangleq{\mathbb P}\left(b\cdot \nabla b-u\cdot \nabla u\right), \\[2mm]
   N_2=(N_{21},N_{22},N_{23})\triangleq b\cdot \nabla u-u\cdot \nabla b.
   \end{cases}
$$

The characteristic polynomial of $A$ is determined by
\begin{equation}\nonumber\label{cpA}
\lambda^2+\left[\mu\xi^2_3+\eta\left(\xi^2_2+\xi^2_3\right)\right]\lambda+\mu\eta\xi^2_3\left(\xi^2_2+\xi^2_3\right)+\xi^2_2=0,
\end{equation}
from which we obtain the eigenvalues of $\mathbb A$:
\begin{align}
\lambda_1=\frac{-\left[\mu\xi^2_3+\eta\left(\xi^2_2+\xi^2_3\right)\right]-\sqrt{\Gamma}}{2},  \quad  \lambda_2=\frac{-\left[\mu\xi^2_3+\eta\left(\xi^2_2+\xi^2_3\right)\right]+\sqrt{\Gamma}}{2}\label{la}
\end{align}
with
\begin{equation*}
\Gamma\triangleq\left(\mu\xi^2_3+\eta\xi^2_{\nu}\right)^2-4\left(\mu\eta\xi^2_3 \xi^2_{\nu}+\xi^2_2\right),\quad \xi^2_{\nu}\triangleq \xi^2_2+\xi^2_3.
\end{equation*}

By computing the corresponding eigenvectors and diagonalizing $\mathbb A$, we find
\begin{align}
\widehat{u}(\xi, t)= \widehat{K_1}(t)\widehat{u_0}+\widehat{K_2}(t)\widehat{b_0}+\int_0^t\left(\widehat{K_1}(t-\tau)\widehat{N_1}(\tau)+\widehat{K_2}(t-\tau)\widehat{N_2}(\tau)\right)d\tau \label{u}
\end{align}
and
\begin{align}
\widehat{b}(\xi, t)= \widehat{K_2}(t)\widehat{u_0}+\widehat{K_3}(t)\widehat{b_0}+\int_0^t\left(\widehat{K_2}(t-\tau)\widehat{N_1}(\tau)+\widehat{K_3}(t-\tau)\widehat{N_2}(\tau)\right)d\tau, \label{b}
\end{align}
where the kernel functions $\widehat{K_1}, \widehat{K_2} $ and $\widehat{K_3}$ are given by
\begin{align}
\widehat{K_1}=-\mu\xi_3^2G_1+G_2,\quad  \widehat{K_2}=i\xi_2G_1,\quad \widehat{K_3}=\mu\xi_3^2G_1+G_3,\label{kerf}\
\end{align}
with
\begin{align*}
&G_1=\frac{e^{\lambda_2 t}-e^{\lambda_1 t}}{\lambda_2-\lambda_1},\quad  G_2=\frac{\lambda_2e^{\lambda_1 t}-\lambda_1e^{\lambda_2 t}}{\lambda_2-\lambda_1}=e^{\lambda_1 t}-\lambda_1G_1=e^{\lambda_2 t}-\lambda_2G_1,\nonumber\\
&G_3=\frac{\lambda_2e^{\lambda_2 t}-\lambda_1e^{\lambda_1 t}}{\lambda_2-\lambda_1}=e^{\lambda_1 t}+\lambda_2G_1=e^{\lambda_2 t}+\lambda_1G_1. 
\end{align*}

It is worth pointing out that  when $\lambda_1 =\lambda_2$, the representation in \eqref{u}  and \eqref{b} remains valid, provided we
replace $G_1$ by its limiting form:
$$
G_1= \lim_{\lambda_2\rightarrow\lambda_1}\frac{e^{\lambda_2t}- e^{\lambda_1t}}{\lambda_2-\lambda_1}=te^{\lambda_1t}.
$$

Obviously, the behavior of $\widehat{K_1}(\xi, t)$, $\widehat{K_2}(\xi, t)$ and $\widehat{K_3}(\xi, t)$
depends strongly on the Fourier frequencies $\xi$. So, we need to study the upper bounds for the kernel functions in the different sub-domains of frequency space.
\begin{pro}\label{lem2.1}
	Let   $\Om_1$ and $\Om_2$ be the sub-domains  of $\R^3$ defined by
	$$\Om_1\triangleq\left\{\xi \in \R^3:\  \Gamma \leq \frac{\left(\mu\xi_3^2+\eta\xi_{\nu}^2\right)^2}{4}  \quad {\rm or} \quad \left(\mu\xi_3^2+\eta\xi_{\nu}^2\right)^2 \leq\frac{16}{3}\left(\mu\eta\xi_3^2\xi_{\nu}^2+\xi_2^2\right)\right\},$$
	
	$$\Om_2\triangleq\left\{\xi \in \R^3:\  \Gamma > \frac{\left(\mu\xi_3^2+\eta\xi_{\nu}^2\right)^2}{4}  \quad {\rm or} \quad \left(\mu\xi_3^2+\eta\xi_{\nu}^2\right)^2 >\frac{16}{3}\left(\mu\eta\xi_3^2\xi_{\nu}^2+\xi_2^2\right)  \right\}.$$
	Then $G_1$, $G_2$ and $G_3$, and $\widehat{K_1}$, $\widehat{K_2}$ and $\widehat{K_3}$ admit the following upper bounds:

	\begin{enumerate}
		\item[(I)]  There exists some $c=c(\nu,\eta)>0$ such that for any $\xi \in \Om_1$,
		\begin{align}
		\D & Re\lambda_1 \leq -\frac{1}{2}\left(\mu\xi_3^2+\eta\xi_{\nu}^2\right), \quad Re\lambda_2 \leq -\frac{1}{4}\left(\mu\xi_3^2+\eta\xi_{\nu}^2\right),\nonumber\\
		\D &|G_1|\leq te^{-\frac{1}{4}\left(\mu\xi_3^2+\eta\xi_{\nu}^2\right)t},\quad  |\widehat{K_1}|, |\widehat{K_2}|,|\widehat{K_3}| \leq C e^{-c\xi_{\nu}^2 t}.\label{KS1}
		\end{align}
		\item[(II)]   There exists some $c=c(\nu,\eta)>0$ such that for any  $\xi \in \Om_2$,
		\begin{align}
		\D &\lambda_1 \leq -\frac{3}{4}\left(\mu\xi_3^2+\eta\xi_{\nu}^2\right), \quad \lambda_2 \leq - \frac{\mu\eta\xi_3^2\xi_{\nu}^2+\xi_2^2 }{\mu\xi_3^2+\eta\xi_{\nu}^2},\nonumber\\
		\D &|G_1|\leq 2\left(\mu\xi_3^2+\eta\xi_{\nu}^2\right)^{-1}\left(e^{\lambda_1 t}+e^{\lambda_2t}\right),\nonumber\\
		\D & |\widehat{K_1}|, |\widehat{K_2}|,|\widehat{K_3}|\leq Ce^{-c\left(\mu\xi_3^2+\eta\xi_{\nu}^2\right)t}+Ce^{- \frac{\mu\eta\xi_3^2\xi_{\nu}^2+\xi_2^2}{\mu\xi_3^2+\eta\xi_{\nu}^2}t}.\label{KS2}
		\end{align}
	\item[(III)]
Let  $\Om_{2i}$ with $i=1,2,3$ be the subsets of $\Om_2$ as follows,
	\begin{align}
	\D & \Om_{21}\triangleq\left\{\xi\in \Om_2 ,\quad \mu\xi_3^2>\eta\xi_{\nu}^2\right\},\nonumber\\
	\D & \Om_{22}\triangleq\left\{\xi\in \Om_2, \quad \mu\xi_3^2\leq\eta\xi_{\nu}^2\quad {\rm and}\quad |\xi_2|\leq|\xi_3|\right\},\nonumber\\
	\D & \Om_{23}\triangleq\left\{\xi\in \Om_2, \quad \mu\xi_3^2\leq\eta\xi_{\nu}^2\quad {\rm and}\quad |\xi_2|>|\xi_3|\right\}.\nonumber
	\end{align}
	Then  the following bounds hold for $K_1$, $K_2$ and $K_3$,
		\begin{align}\label{KS21}
		|\widehat{K_1}|, |\widehat{K_2}|,|\widehat{K_3}| &\leq C e^{-c\xi_{\nu}^2 t},\quad\forall\ \xi\in \Om_{21},\\
		\label{KS22}
		|\widehat{K_1}|, |\widehat{K_2}|,|\widehat{K_3}| &\leq C e^{-c\xi_{\nu}^2 t},\quad\forall\  \xi\in \Om_{22},
		\\
		 \label{KS23}
		|\widehat{K_1}|, |\widehat{K_2}|,|\widehat{K_3}| &\leq C\left(e^{-c(1+\xi_3^2)t}+e^{-c\xi_{\nu}^2t}\right),\quad\forall\ \xi\in \Om_{23}.
		\end{align}	

\end{enumerate}
\end{pro}

\begin{proof}
	(I) For  $\xi \in \Om_1$, the eigenvalues $\lambda_1$  and $\lambda_2$ given by \eqref{la} obviously satisfy
	$$Re\lambda_1 \leq -\frac{1}{2}\left(\mu\xi_3^2+\eta\xi_{\nu}^2\right), \quad Re\lambda_2 \leq -\frac{1}{4}\left(\mu\xi_3^2+\eta\xi_{\nu}^2\right).$$

	We divide the analysis into two cases,
	\beq\nonumber\label{cas}
	\Gamma < 0\quad\mbox{and}\quad 0\leq\sqrt{\Gamma}\leq \frac{\mu\xi^2_3+\eta\left(\xi^2_2+\xi^2_3\right)}{2}.
	\eeq

	\textbf{Case 1. $\Gamma < 0$.} In this case, the eigenvalues
	$\lambda_1$ and $\lambda_2$ are a pair of complex conjugates, satisfying
	\begin{equation}\nonumber\label{lab2}
	|\lambda_1|,~|\lambda_2|=\sqrt{ \mu\eta\xi^2_3\left(\xi^2_2+\xi^2_3\right)+\xi^2_2 }.
	\end{equation}
	In addition,
	\beq \label{gb}
	G_1 =\frac{e^{\lambda_2 t}-e^{\lambda_1 t}}{\lambda_2-\lambda_1} = e^{-\frac{1}{2}\left(\mu\xi_3^2+\eta\xi_{\nu}^2\right)t} \,\frac{2\sin (\frac12Q t)}{Q},
	\eeq
	where
	$$
	Q:=\sqrt{4\left(\mu\eta\xi^2_3\left(\xi^2_2+\xi^2_3\right)+\xi^2_2\right)-\left[\mu\xi^2_3+\eta\left(\xi^2_2+\xi^2_3\right)\right]^2}.
	$$

	By the simple fact that $|\sin \rho|\le |\rho|$ for any $\rho\in \mathbb R$, we infer from (\ref{gb}) that
	$$
	|G_1| \le t \, e^{-\frac{1}{2}\left(\mu\xi_3^2+\eta\xi_{\nu}^2\right)t}.
	$$
	It follows from the definitions of $\widehat K_1$ and $\widehat K_3$  in (\ref{kerf}) that
	\begin{align}
	& |\widehat{K_1}|=|\mu\xi_3^2G_1+G_2|=|\mu\xi_3^2G_1+e^{\lambda_1t}-\lambda_1G_1|,\label{K1}	\\
	& |\widehat{K_3}|=|\mu\xi_3^2G_1+G_3|=|\mu\xi_3^2G_1+e^{\lambda_2t}+\lambda_1G_1|.\label{K3}
	\end{align}
	
Next, we first  estimate the term  $|\lambda_1G_1|$. On one hand, if $|\lambda_1|\leq\sqrt{-\Gamma}$, then
	\begin{align*}
	|\lambda_1G_1|&=|\lambda_1|\frac{\left| e^{\lambda_2 t}-e^{\lambda_1 t}\right|} {\left| \sqrt{-\Gamma}\right|}\leq |e^{\lambda_1 t}|+|e^{\lambda_2 t}|  \leq Ce^{-\frac{1}{2}\left(\mu\xi_3^2+\eta\xi_{\nu}^2\right) t}.
	\end{align*}
	On the other hand, if $|\lambda_1|>\sqrt{-\Gamma}$, then
	$$
	3\left(\mu\eta\xi^2_3\left(\xi^2_2+\xi^2_3\right)+\xi^2_2\right)<\left(\mu\xi_3^2+\eta\xi_{\nu}^2\right)^2 \quad
	{\rm or} \quad  |\lambda_1|<\frac{1}{\sqrt{3}}\left(\mu\xi_3^2+\eta\xi_{\nu}^2\right),$$
which,  together with \eqref{gb},  yields
	\begin{align*}
	|\lambda_1G_1|&=|\lambda_1|e^{-\frac{1}{2}\left(\mu\xi_3^2+\eta\xi_{\nu}^2\right)t} \left|\frac{2\sin (\frac12Q t)}{Q}\right|\\
	&\leq \frac{1}{\sqrt{3}}\left(\mu\xi_3^2+\eta\xi_{\nu}^2\right)t\, e^{-\frac{1}{2}\left(\mu\xi_3^2+\eta\xi_{\nu}^2\right)t}  \leq Ce^{-c \xi_{\nu}^2 t}.
	\end{align*}

So, in the case that $\Gamma<0$ we have $|\lambda_1G_1|\leq Ce^{-c \xi_{\nu}^2 t}$. As a result,
	\begin{align*}
	|\widehat{K_1}|=|\mu\xi_3^2G_1+e^{\lambda_1t}-\lambda_1G_1| \leq  Ce^{-c \xi_{\nu}^2 t},
	\end{align*}
	and analogously,
	\begin{align*}
	|\widehat{K_3}|=|\mu\xi_3^2G_1+e^{\lambda_2t}+\lambda_1G_1| \leq  Ce^{-c \xi_{\nu}^2 t}.
	\end{align*}

	To deal with $\widehat{K_2}$, we first observe that if   $|\xi_2|\leq\sqrt{-\Gamma}$, then
	\begin{align*}
	|\widehat{K_2}|&=|\xi_2|\frac{\left| e^{\lambda_2 t}-e^{\lambda_1 t}\right|} { \sqrt{-\Gamma}}\leq |e^{\lambda_1 t}|+|e^{\lambda_2 t}|  \leq Ce^{-\frac{1}{2}\left(\mu\xi_3^2+\eta\xi_{\nu}^2\right) t}.
	\end{align*}

Moreover, if $|\xi_2|>\sqrt{-\Gamma}$, then
	$$
	\xi^2_2> 4\left(\mu\eta\xi^2_3\left(\xi^2_2+\xi^2_3\right)+\xi^2_2\right)-\left(\mu\xi_3^2+\eta\xi_{\nu}^2\right)^2,
	$$
	which implies
	\begin{align}\label{k2x}
	3\xi^2_2<\left(\mu\xi_3^2+\eta\xi_{\nu}^2\right)^2 \quad
	{\rm i.e.} \quad |\xi_2|<\frac{1}{\sqrt{3}}\left(\mu\xi_3^2+\eta\xi_{\nu}^2\right).
	\end{align}
	so that, it follows from \eqref{kerf}, \eqref{gb}  and \eqref{k2x} that
	\begin{align*}
	|\widehat{K_2}|&=|\xi_2|e^{-\frac{1}{2}\left(\mu\xi_3^2+\eta\xi_{\nu}^2\right)t} \left|\frac{2\sin (\frac12Q t)}{Q}\right|  \leq Ce^{-c \xi_{\nu}^2 t}.
	\end{align*}
This finishes the proof of (\ref{KS1}) in the case $\Gamma<0$.

\vskip 2mm

	\textbf{Case 2. $ 0\leq\sqrt{\Gamma}\leq \frac{\mu\xi^2_3+\eta\xi_{\nu}^2}{2}$.} In this case,
	both $\lambda_1$ and $\lambda_2$ are real and satisfy
	\begin{align*}
	-\frac{3}{4}\left(\mu\xi_3^2+\eta\xi_{\nu}^2\right)\leq\lambda_1\leq -\frac{1}{2}\left(\mu\xi_3^2+\eta\xi_{\nu}^2\right),\\[1mm]
-\frac{1}{2}\left(\mu\xi_3^2+\eta\xi_{\nu}^2\right)\leq\lambda_2\leq -\frac{1}{4}\left(\mu\xi_3^2+\eta\xi_{\nu}^2\right).
	\end{align*}

	It follows from the mean-value theorem that there is a $\zeta \in (\lambda_1, \lambda_2)$ such that
	$$
	G_1 =te^{\zeta}\leq te^{-\frac{1}{4}\left(\mu\xi_3^2+\eta\xi_{\nu}^2\right)t},
	$$
	and thus,
	\beq
	|\lambda_1G_1|\leq \frac{3}{4}\left(\mu\xi_3^2+\eta\xi_{\nu}^2\right)te^{-\frac{1}{4}\left(\mu\xi_3^2+\eta\xi_{\nu}^2\right)t}\leq Ce^{-c\xi_{\nu}^2 t}.\label{gb2}
	\eeq
With (\ref{gb2}) at hand, we can derive the desired bounds of $|\widehat{K_1}|$ and $|\widehat{K_3}|$ in a similar manner as that in Case 1.

	To bound $|\widehat{K_2}|$, we first use the fact that $\sqrt{\Gamma}\geq0$ to get
	$$
	\left(\mu\xi_3^2+\eta\xi_{\nu}^2\right)^2\geq4\left(\mu\eta\xi^2_3\xi_{\nu}^2+\xi^2_2\right),  $$
	and hence,
	\begin{align*}
	4\xi^2_2\leq4\left(\mu\eta\xi^2_3\xi_{\nu}^2+\xi^2_2\right)\leq \left(\mu\xi_3^2+\eta\xi_{\nu}^2\right)^2 \quad
	{\rm i.e.} \quad |\xi_2|\leq\frac{1}{2}\left(\mu\xi_3^2+\eta\xi_{\nu}^2\right),
	\end{align*}
	from which it immediately follows that
	$$
	|\widehat{K_2}|=|\xi_2G_1|\leq |\xi_2|te^{-\frac{1}{4}\left(\mu\xi_3^2+\eta\xi_{\nu}^2\right)t}\leq Ce^{-c\xi_{\nu}^2 t}.
	$$
	This, together with Case 1, shows that (\ref{KS1}) holds for all  $\xi\in\Omega_1$.
	
	\vskip .1in
	(II) For  $\xi \in \Om_2$, it is clear that $\lambda_1$ and $\lambda_2$ are real and
	\begin{align}
	&-\left(\mu\xi_3^2+\eta\xi_{\nu}^2\right)\leq\lambda_1\leq -\frac{3}{4}\left(\mu\xi_3^2+\eta\xi_{\nu}^2\right),\nonumber\\
	&\quad\lambda_2=\frac{-\left[\mu\xi^2_3+\eta\left(\xi^2_2+\xi^2_3\right)\right]+\sqrt{\Gamma}}{2}\nonumber\\
	&\qquad=-\frac{2\left(\mu\eta\xi^2_3\xi_{\nu}^2+\xi^2_2\right)}{\sqrt{\Gamma}+\mu\xi^2_3+\eta\xi_{\nu}^2} \leq -\frac{ \mu\eta\xi^2_3\xi_{\nu}^2+\xi^2_2}{ \mu\xi^2_3+\eta\xi_{\nu}^2},\label{ll}
	\end{align}
	from which it is easily seen that
	$$
	\lambda_2-\lambda_1=\sqrt{\Gamma} \geq  \frac{1}{2}\left( \mu\xi^2_3+\eta\xi_{\nu}^2\right),
	$$
	and consequently, the term $G_1$ admits the following upper bound:
	\begin{align}
	|G_1| \le 2\left( \mu\xi^2_3+\eta\xi_{\nu}^2\right)^{-1}\left(e^{\lambda_1 t}+e^{\lambda_2t}\right).\label{G2}
	\end{align}

	As an immediate result of \eqref{ll} and \eqref{G2}, one has
	$$
	|\lambda_1G_1| \le 2\left(e^{\lambda_1 t}+e^{\lambda_2t}\right),
	$$
	so that, it is easily derived from \eqref{K1}--\eqref{K3} and
	$$
	|\widehat{K_1}|,\ |\widehat{K_3}| \le C\left(e^{\lambda_1 t}+e^{\lambda_2t}\right)\leq Ce^{-c\left(\mu\xi_3^2+\eta\xi_{\nu}^2\right)t}+Ce^{- \frac{\mu\eta\xi_3^2\xi_{\nu}^2+\xi_2^2}{\mu\xi_3^2+\eta\xi_{\nu}^2}t}.
	$$

	For $\widehat{K_2}$, since $\xi\in \Om_2$, we have
	$$
	\left(\mu\xi_3^2+\eta\xi_{\nu}^2\right)^2 >\frac{16}{3}\left(\mu\eta\xi_3^2\xi_{\nu}^2+\xi_2^2\right),
	$$
	that is,
	$$
	\frac{16}{3} \xi_2^2 <\left(\mu\xi_3^2+\eta\xi_{\nu}^2\right)^2 \quad
	{\rm or}
	\quad
	|\xi_2| < \frac{\sqrt{3}}{4}\left(\mu\xi_3^2+\eta\xi_{\nu}^2\right),
	$$
from which and \eqref{G2}, we find
	\begin{align*}
	|\widehat{K_2}|=|\xi_2 G_1|\leq C\left(e^{\lambda_1 t}+e^{\lambda_2t}\right)\leq Ce^{-c\left(\mu\xi_3^2+\eta\xi_{\nu}^2\right)t}+Ce^{- \frac{\mu\eta\xi_3^2\xi_{\nu}^2+\xi_2^2}{\mu\xi_3^2+\eta\xi_{\nu}^2}t}.
	\end{align*}
This finishes the proof of the assertions stated in (\ref{KS2}).

\vskip .1in
	(III) For $\xi\in \Om_{21}$,
	$$
	-\frac{\mu\eta\xi_3^2\xi_{\nu}^2+\xi_2^2}{\mu\xi_3^2+\eta\xi_{\nu}^2}\leq -\frac{\mu\eta\xi_3^2\xi_{\nu}^2+\xi_2^2}{2\mu\xi_3^2}\leq-\frac{\eta}{2}\xi_{\nu}^2.
	$$
	For $\xi\in \Om_{22}$,
	$$
	-\frac{\mu\eta\xi_3^2\xi_{\nu}^2+\xi_2^2}{\mu\xi_3^2+\eta\xi_{\nu}^2}\leq -\frac{\mu\eta\xi_3^2\xi_{\nu}^2+\xi_2^2}{2\eta\xi_{\nu}^2}\leq-\frac{\mu}{2}\xi_{3}^2 \leq-\frac{\mu}{4}\xi_{\nu}^2.
	$$
	For $\xi\in \Om_{23}$,
	$$
	-\frac{\mu\eta\xi_3^2\xi_{\nu}^2+\xi_2^2}{\mu\xi_3^2+\eta\xi_{\nu}^2}\leq -\frac{\mu\eta\xi_3^2\xi_{\nu}^2+\xi_2^2}{2\eta\xi_{\nu}^2}=-\frac{\mu}{2}\xi_{3}^2-\frac{\xi_{2}^2}{2\eta\left(\xi_{2}^2+\xi_{3}^2\right)}\leq-\frac{\mu}{2}\xi_{3}^2-\frac{1}{4\eta}.
	$$
	Thus, combining these bounds with (\ref{KS2})$_3$, we immediately arrive at (\ref{KS21})--(\ref{KS23}). The proof of Proposition \ref{lem2.1} is therefore complete.
\end{proof}

\subsection{Mathematical tools for the decay estimates}
In this subsection, we state some useful inequalities which play important roles in the derivations of  the decay rates. We first  recall the following lemma concerning  the exact decay rate for the solution operator
associated with a fractional Laplacian (cf. \cite{Schonbek,Wu0}).

\begin{lem}\label{OD}
	Let $\alpha \ge 0$ and $\beta>0$ be real numbers. Then for any $1\le q\le p\le\infty$,
	\begin{align*}
	\|(-\Delta)^\alpha e^{-(-\Delta)^\beta t}f\|_{L^p(\mathbb R^d)}\le \,C\, t^{-\frac{\alpha}{\beta} - \frac{d}{2\beta}(\frac1q-\frac1p)}\, \|f\|_{L^q(\mathbb R^d)},\quad \forall\ t>0.
	\end{align*}
\end{lem}

\vskip .1in
It is clear from Proposition \ref{lem2.1} that there is lack of $\xi_1$-decay in  kernel functions. To overcome this difficulty and to derive the decay rates, we have to assume that the initial data is equipped with $L^2_{x_1}L^1_{x_2,x_3}$-norm. Moreover, as an  immediate consequence of  Lemma \ref{OD}, we have
the following corollary which is  technically used in the analysis.
\begin{cor}\label{PI} Assume that  $u\in L_{x_1}^2L^1_{x_2,x_3}\cap L^2$. Then for any $\alpha\geq0$, there exists an absolutely   positive constant $C>0$, depending only on $\alpha$, such that
\begin{equation}\label{Pli}
\||\xi_\nu|^\alpha e^{- \xi_{\nu}^2t}\widehat{u}(\xi)\|_{L^2_\xi}\leq C(1+t)^{-\frac{1+\alpha}{2}} \left\|\|u\|_{L^1_{x_2,x_3}}\right\|_{L^2_{x_1}} ,\quad\forall\ t\geq 1,
\end{equation}
where $\xi=(\xi_1,\xi_\nu)$ with $\xi_\nu=(\xi_2,\xi_3)$.
\end{cor}
\begin{proof}
 Let $\Delta_\nu \triangleq\p_2^2+\p_3^2$. Then, it follows from Plancherel theorem, Fubini theorem and Lemma \ref{OD}  that for any $t\geq1$,
\begin{align*}
&\||\xi_\nu|^\alpha e^{- \xi_{\nu}^2t}\widehat{u}(\xi)\|_{L^2_\xi}\\
&\quad\leq \|(-\Delta_\nu)^{\frac{\alpha}{2}}e^{\Delta_{\nu}t}u\|_{L^2}
=\left\|\|(-\Delta_\nu)^{\frac{\alpha}{2}}e^{\Delta_{\nu}t}u \|_{L_{x_2,x_3}^2}  \right\|_{L_{x_1}^2} \nonumber\\
&\quad\leq Ct^{-\frac{1+\alpha}{2}}\left\|\|u \|_{L_{x_2,x_3}^1} \right\|_{L_{x_1}^2}\leq C(1+t)^{-\frac{1+\alpha}{2}}\left\|\|u \|_{L_{x_2,x_3}^1} \right\|_{L_{x_1}^2} ,
\end{align*}
which ends the proof of (\ref{Pli}).
\end{proof}

Furthermore, if $u$ is a  vector-valued function satisfying the divergence-free condition, then we have the following refined decay rate of $u_1$.
\begin{cor}
\label{divfree} Let $u=(u_1,u_2,u_3)$ be a smooth function, satisfying $u\in L^1\cap L^2$ and $\nabla\cdot u=0$. Then for any $\beta\geq0$, there exists an absolutely   positive constant $C>0$, depending only on $\beta$, such that
\begin{equation}
\||\xi_\nu|^\beta e^{-\xi_{\nu}^2t}\widehat{u_1}(\xi)\|_{L^2_\xi}\leq C(1+t)^{-\frac{3+2\beta}{4}}  \|u\|_{L^1} ,\quad\forall \ t\geq1.\label{udiv}
\end{equation}
\end{cor}
\begin{proof} Let $\nabla_\nu \triangleq (\p_2,\p_3)$. Then, it holds that  $\p_1u_1=-\p_\nu u_\nu$ with $u_\nu\triangleq(u_2,u_3)$, due to the divergence-free condition $\nabla\cdot u=0$. Thus, using Plancherel theorem and Fubini theorem, we deduce from Lemma \ref{OD}  that for any $t\geq1$,
\begin{align*}
&\||\xi_\nu|^\beta e^{-\xi_{\nu}^2t}\widehat{u_1}(\xi)\|_{L^2_\xi}^2
=\|(-\Delta_\nu)^{\frac\beta 2}e^{\Delta_{\nu}t}u_1\|_{L^2}^2 \nonumber\\
&\quad\leq C\left\|\|(-\Delta_\nu)^{\frac\beta 2}e^{\Delta_{\nu}t}u_1 \|_{L_{x_1}^1}^{\frac12} \|(-\Delta_\nu)^{\frac\beta 2}e^{\Delta_{\nu}t}u_1 \|_{L_{x_1}^{\infty}}^{\frac12} \right\|_{L_{x_2,x_3}^1}^2 \nonumber\\
&\quad\leq C\left\|\|(-\Delta_\nu)^{\frac\beta 2}e^{\Delta_{\nu}t}u_1 \|_{L_{x_1}^1} \right\|_{L_{x_2,x_3}^2}\left \|\|(-\Delta_\nu)^{\frac\beta 2}e^{\Delta_{\nu}t}\p_1u_1 \|_{L_{x_1}^1}\right\|_{L_{x_2,x_3}^2} \nonumber\\
&\quad\leq C\left\|(-\Delta_\nu)^{\frac\beta 2}\|e^{ \Delta_{\nu}t}u \|_{L_{x_2,x_3}^2} \right\|_{L_{x_1}^1} \left\|\|(-\Delta_\nu)^{\frac\beta 2}e^{ \Delta_{\nu}t}\na_{\nu}\cdot u_{{\nu}}\|_{L_{x_2,x_3}^2}\right\|_{L_{x_1}^1} \nonumber\\
&\quad\leq C\left(t^{-\frac{1+\beta}{2}}\left\|\|u \|_{L^1_{x_2,x_3}}\right\|_{L^1_{x_1}} \right)\left(t^{-\frac{2+\beta}{2}}\left\|\|u \|_{L^1_{x_2,x_3}}\right\|_{L^1_{x_1}}\right)\notag\\
&\quad \leq C t^{-\frac{3+2\beta}{2}}\|u \|_{L^1}^2\leq (1+t)^{-\frac{3+2\beta}{2}}\|u \|_{L^1}^2,
\end{align*}
where we have also used H\"{o}lder and Minkowski inequalities.
\end{proof}

The following two lemmas are concerned with the upper bounds with sharp decay rates for two special integrals (see, e.g.,  \cite{Wan}).

\begin{lem}\label{ID}
	For any  $0<s_1\le s_2$, there exists a positive constant $C$, depending only on $s_1$ and $s_2$, such that for any $t\geq0$,
	\begin{align*}
	&\int_0^t(1+t-\tau)^{-s_1}(1+\tau)^{-s_2}\ d\tau
	\le\left\{
	\begin{array}{ll}
	C(1+t)^{-s_1},\quad\quad&{\rm if }\quad  s_2> 1;
	\\
	C(1+t)^{-s_1}\ln(1+t), \quad\quad&{\rm if }\quad s_2=1;
	\\
	C(1+t)^{1-s_1-s_2},\quad\quad&{\rm if }\quad  s_2<1.
	\end{array}
	\right.
	\end{align*}
\end{lem}

\begin{lem} \label{ED}
	For any $c>0$ and $s>0$, it holds that
	\begin{align*}	
	\int_0^t e^{-c(t-\tau)}(1+\tau)^{-s}\ d\tau\le C(1+t)^{-s},\quad\forall\ t\geq0.
	\end{align*}
\end{lem}

\section{Global Stability and Proof of Theorem \ref{thm1.1} }\label{Sec.2}

This section concerns the global well-posedness of smooth solutions  stated in Theorem \ref{thm1.1}. Since the local existence result can be shown in a similar manner as that in \cite{MaBe}, the main task of the present paper is to derive the global a priori estimates of the solutions. This will be done by using the method of dual-layer energy and applying the bootstrapping arguments.

\subsection{ A Priori Estimates (I)}

The first step is to deal with the natural $H^3$-energy $\mathcal{E}_1(t)$, based on the anisotropic Sobolev inequalities and the divergence-free conditions.

\begin{lem}\label{lemma3.1} Let $\mathcal{E}_1(t)$ and $\mathcal{E}_2(t)$ be the same ones defined in \eqref{E1} and \eqref{E2}, respectively. Then for any $t\geq0$,
\begin{align}
\mathcal{E}_{1}(t)
\leq C\left(\mathcal{E}_1(0)+\mathcal{E}_1^{\frac32}(0)\right)+C\left(\mathcal{E}_1^{\frac32}(t)+\mathcal{E}_2^{\frac32}(t)\right)
+C\left(\mathcal{E}_1^2(t)+\mathcal{E}_2^2(t)\right). \label{EE1}
\end{align}
\end{lem}
\noindent
\begin{proof} Since $\|(u, b)\|_{H^3}\sim\|(u, b)\|_{L^2}+\|(\nabla^3u, \nabla^3b)\|_{L^2}$, it suffices to bound the $L^2$-norms and the homogeneous $\dot H^3$-norms of $(u, b)$.
\vskip .1in
\textbf{Step I. The estimate of $L^2$-norm}
\vskip .1in
Based on the divergence-free conditions  $\nabla\cdot u=\nabla\cdot b=0$, it is easily derived from (\ref{PMHD}) that for any $0\leq t\leq T$,
\begin{equation}\label{L2bound}
\|(u,b)(t)\|_{L^2}^2+2\int_0^t\left(\mu\|\p_3u\|_{L^2}^2+\eta\|\p_2b \|_{L^2}^2+\eta\|\p_3b \|_{L^2}^2\right)d\tau=\|(u_0,b_0)\|_{L^2}^2.
\end{equation}

\textbf{Step II. The preliminary bound  of $\dot{ H^3}$-norm}
\vskip .1in
Applying $\partial_i^3$ with $i=1,2,3$ to \eqref{PMHD} and dotting them with $(\p_i^3u, \p_i^3b)$  in $L^2$, we obtain after integrating by parts that
\begin{align}
&\frac12\frac{d}{d t}\sum_{i=1}^{3}\|(\partial_i^3u, \partial_i^3b)\|_{L^2}^2+\mu \sum_{i=1}^{3}\|\p_i^3\p_3u\|_{L^2}^2
+\eta\sum_{i=1}^{3}\left(\|\p_i^3\p_{2}b\|_{L^2}^2+ \|\p_i^3\p_{3}b\|_{L^2}^2\right)\nonumber\\
&\quad\triangleq H_1+H_2+H_3+H_4+H_5,\label{H3bound}
\end{align}
where
\beno
&& H_1\triangleq \sum_{i=1}^{3}\int  \left(\partial_i^3\partial_2 b \cdot \partial_i^3 u +\partial_i^3\partial_2 u \cdot \partial_i^3 b\right) d x,
\quad
H_2 \triangleq -\sum_{i=1}^{3}\int \partial_i^3(u\cdot \nabla u)\cdot \partial_i^3 u \ d x,
\\
&& H_3 \triangleq \sum_{i=1}^{3}\int \left[\partial_i^3(b\cdot \nabla b)-b\cdot \nabla\partial_i^3 b \right]\cdot \partial_i^3 u \ dx,
\quad
H_4 \triangleq -\sum_{i=1}^{3}\int \partial_i^3(u\cdot \nabla b)\cdot \partial_i^3 b \ d x,
\\
&&
H_5 \triangleq \sum_{i=1}^{3}\int \left[\partial_i^3(b\cdot \nabla u)-b\cdot \nabla\partial_i^3 u  \right]\cdot \partial_i^3 b \ dx.
\eeno

We are now in a position of estimating each term on the right-hand side of (\ref{H3bound}). First, it is easily seen that
\begin{align}
H_1=0.  \label{H1}
\end{align}

Due to $\nabla \cdot u=0$, by direct calculations we have
\begin{align*}
H_2&=-\sum_{i=1}^{3}\int \partial_i^3u\cdot \nabla u\cdot \partial_i^3 u \ d x-3\sum_{i=1}^{3}\int \partial_i^2 u\cdot \nabla\partial_i u\cdot \partial_i^3 u \ d x\\
&\quad -3\sum_{i=1}^{3}\int \partial_iu\cdot \nabla\partial_i^2 u\cdot \partial_i^3 u \ d x\triangleq\sum_{j=1}^{3}H_{2j}.
\end{align*}

It follows from  H\"{o}lder's and Sobolev's inequalities that
\begin{align}
H_{21}&=-\int \left(\partial_1^3u\cdot \nabla u\cdot \partial_1^3 u+\partial_2^3u\cdot \nabla u\cdot \partial_2^3 u+\partial_3^3u\cdot \nabla u\cdot \partial_3^3 u\right) \ d x\notag\\
&\leq C\|\nabla u\|_{L^\infty}\left(\|\partial_2^3 u\|_{L^2}^2+\|\partial_3^2u\|_{L^2}^2\right)+H_{211}\notag\\
&\leq C\|u\|_{H^3}\left(\|\partial_2  u\|_{H^2}^2+\|\partial_3 u\|_{H^2}^2\right)+H_{211}, \label{H21}
\end{align}
where
$$
H_{211}\triangleq-\int \partial_1^3u\cdot \nabla u\cdot \partial_1^3 u \ d x.
$$

In view of the divergence-free condition $\nabla\cdot u=0$, we get
\begin{equation}\label{u1-23}
\|\p_1 u_1\|_{\dot H^k}\leq \|\p_2u_2\|_{\dot H^k}+\|\p_3u_3\|_{\dot H^k}\leq \|\p_2u\|_{\dot H^k}+\|\p_3u\|_{\dot H^k},\quad k=1,2,\ldots,
\end{equation}
so that, by Lemma \ref{anip1} we have from integration by parts that
\begin{align}
H_{211}
&=-\int \partial_1^3u_1\p_1 u\cdot\partial_1^3 u \ d x-\int \partial_1^3u_2\p_2 u_1\partial_1^3 u_1\ d x -\int \partial_1^3u_2\p_2 u_2\partial_1^3 u_2  \ d x \notag\\
&\quad-\int \partial_1^3u_2\p_2 u_3\partial_1^3 u_3 \ d x+\int \p_3\partial_1^3u_3\left(u\cdot\partial_1^3 u\right) d x+\int \partial_1^3u_3\left(u\cdot\p_3\partial_1^3 u\right) d x \notag\\
&\leq C\|\p_1^3 u_1\|_{L^2}\|\p_1u\|_{L^2}^{\frac14}\|\p_1^2u\|_{L^2}^{\frac14}\|\p_1\p_2u\|_{L^2}^{\frac14}\|\p_1^2\p_2u\|_{L^2}^{\frac14}\|\p_1^3 u \|_{L^2}^{\frac12}\|\p_3\p_1^3 u \|_{L^2}^{\frac12}\notag\\
&\quad+ C\|\p_1^3 u_1\|_{L^2}\|\p_2u_1\|_{L^\infty}\|\p_1^3u_2\|_{L^2}+D_1^1+D_2\notag\\
&\quad+ C\|\p_3\p_1^3 u \|_{L^2}\|u\|_{L^2}^{\frac14}\|\p_1u\|_{L^2}^{\frac14}\|\p_2u\|_{L^2}^{\frac14}\|\p_1\p_2u\|_{L^2}^{\frac14}\|\p_1^3 u \|_{L^2}^{\frac12}\|\p_3\p_1^3 u \|_{L^2}^{\frac12}\notag\\
&\leq C\| u\|_{H^3}\left(\|\partial_2 u\|^{2}_{H^2}+\|\p_3u\|^{2}_{H^3} \right)+D_1^1+D_2,\label{H211}
\end{align}
where
$$
D_1^1\triangleq-\int  \p_2 u_2\left(\partial_1^3u_2 \right)^2 d x, \quad\quad
D_2\triangleq-\int \partial_1^3u_2\p_2 u_3\partial_1^3 u_3 \ d x.
$$

Thus, substituting \eqref{H211} to \eqref{H21}, we find
\begin{align}
H_{21}\leq C\|u\|_{H^3}\left(\|\partial_2  u\|_{H^2}^2+\|\partial_3 u\|_{H^2}^2\right)+D_1^1+D_2. \label{H21g}
\end{align}

In terms of (\ref{u1-23}) and Lemma \ref{anip1}, we can  bound $H_{22}$ by
\begin{align}
H_{22}&=-3\int \left(\partial_1^2u_1\partial_1^2u\cdot \partial_1^3 u+\partial_1^2u_2\partial_2\partial_1 u\cdot \partial_1^3 u+\partial_1^2u_3 \p_3\partial_1 u\cdot \partial_1^3 u\right) \ d x\notag\\
&\quad-3\int \left(\partial_2^2u\cdot \nabla \partial_2u\cdot \partial_2^3 u+\partial_3^2u\cdot \nabla\partial_3 u\cdot \partial_3^3 u\right)  d x\notag\\
&\leq C\|\partial_1^2u_1\|_{L^2}^{\frac12}\|\partial_1^3u_1\|_{L^2}^{\frac12}\|\p_1^2u\|_{L^2}^{\frac12}\|\p_2\p_1^2u \|_{L^2}^{\frac12}\|\partial_1^3 u\|_{L^2}^{\frac12}\|\p_3\partial_1^3 u\|_{L^2}^{\frac12}
\notag\\
&\quad+C\|\partial_1^2u_2\|_{L^2}^{\frac12}\|\p_2\partial_1^2u_2\|_{L^2}^{\frac12}\|\p_2\p_1u\|_{L^2}^{\frac12}\|\p_2\p_1^2u \|_{L^2}^{\frac12}\|\partial_1^3 u\|_{L^2}^{\frac12}\|\p_3\partial_1^3 u\|_{L^2}^{\frac12}
\notag\\
&\quad+C\|\partial_1^2u_3\|_{L^2}^{\frac12}\|\p_2\partial_1^2u_3\|_{L^2}^{\frac12}\|\p_3\p_1u\|_{L^2}^{\frac12}\|\p_3\p_1^2u \|_{L^2}^{\frac12}\|\partial_1^3 u\|_{L^2}^{\frac12}\|\p_3\partial_1^3 u\|_{L^2}^{\frac12}
\notag\\
&\quad+C\left(\|\partial_2^2u \|_{L^3}\|\na\partial_2 u\|_{L^6}\|\partial_2^3 u\|_{L^2} +\|\p_3^2 u \|_{L^3}\|\na\p_3 u \|_{L^6}\|\partial_3^3 u\|_{L^2}\right)\notag\\
&\leq C\|u\|_{H^3}\left(\|\partial_2  u\|_{H^2}^2+\|\partial_3 u\|_{H^3}^2\right),\label{H22}
\end{align}
and analogously,
\begin{align}
H_{23}&= -3\int \left(\partial_1u_1\partial_1^3 u\cdot \partial_1^3 u+\partial_1u_2\partial_2\partial_1^2 u\cdot \partial_1^3 u+\partial_1u_3 \p_3\partial_1^2 u\cdot \partial_1^3 u\right)  d x \notag\\
&\quad
-3\int \left(\partial_2u\cdot \nabla \partial_2^2u\cdot \partial_2^3 u+\partial_3u\cdot \nabla\partial_3^2 u\cdot \partial_3^3 u\right) d x\notag\\
&\leq C\|\p_2\p_1^2u \|_{L^2} \|\partial_1 u_2\|_{L^2}^{\frac14}\|\partial_1^2 u_2\|_{L^2}^{\frac14}\|\partial_1\p_2 u_2\|_{L^2}^{\frac14}\|\partial_1^2 \p_2u_2\|_{L^2}^{\frac14} \|\partial_1^3 u\|_{L^2}^{\frac12}\|\p_3\partial_1^3 u\|_{L^2}^{\frac12}
\notag\\
&\quad+C\|\p_3\p_1^2u \|_{L^2} \|\partial_1 u_3\|_{L^2}^{\frac14}\|\partial_1^2 u_3\|_{L^2}^{\frac14}\|\partial_1\p_2 u_3\|_{L^2}^{\frac14}\|\partial_1^2 \p_2u_3\|_{L^2}^{\frac14} \|\partial_1^3 u\|_{L^2}^{\frac12}\|\p_3\partial_1^3 u\|_{L^2}^{\frac12}
\notag\\
&\quad+C\left(\|\partial_2 u \|_{L^\infty}\|\na\partial_2 u\|_{L^2}\|\partial_2^3 u\|_{L^2} +\|\p_3 u \|_{L^\infty}\|\na\p_3^2 u \|_{L^2}\|\partial_3^3 u\|_{L^2}\right)+H_{231}\notag\\
&\leq C\|u\|_{H^3}\left(\|\partial_2  u\|_{H^2}^2+\|\partial_3 u\|_{H^3}^2\right)+H_{231}, \label{H23}
\end{align}
where
\begin{align*}
H_{231}\triangleq-3\int \partial_1u_1\partial_1^3u \cdot \partial_1^3 u  \ d x.
\end{align*}

Note that $\|\p_1u_1\|_{L^\infty}\leq C\|\p_1u_1\|_{H^2}\leq C\|(\p_2u_2,\p_3u_3)\|_{H^2}$, due to the divergence-free condition $\nabla \cdot u=0$ and the Sobolev embedding inequality. So, integrating by parts and using Lemma \ref{anip1}, we obtain
\begin{align}
H_{231}&=-3\int \partial_1u_1\partial_1^3u_1 \partial_1^3 u_1 \ d x-3\int\partial_1u_1\partial_1^3u_2 \partial_1^3 u_2\ d x-3\int\partial_1u_1\partial_1^3u_3 \partial_1^3 u_3  \ d x\notag\\
&=-3\int \partial_1u_1\partial_1^3u_1 \partial_1^3 u_1 \ d x+3\int\partial_2u_2\partial_1^3u_2 \partial_1^3 u_2\ d x-6\int u_3\partial_1^3u_2 \partial_3\partial_1^3 u_2  \ d x\notag\\
&\quad+3\int\partial_2u_2\partial_1^3u_3 \partial_1^3 u_3\ d x-6\int u_3\partial_1^3u_3 \partial_3\partial_1^3 u_3  \ d x\notag\\
&\leq C\|\p_1 u_1\|_{L^\infty} \|\partial_1^3 u_1\|_{L^2}^2+D_1^2+D_3\notag\\
&\quad+C\|\p_3\p_1^3u \|_{L^2} \|u_3\|_{L^2}^{\frac14}\|\partial_1 u_3\|_{L^2}^{\frac14}\|\p_2 u_3\|_{L^2}^{\frac14}\|\partial_1 \p_2u_3\|_{L^2}^{\frac14} \|\partial_1^3 u\|_{L^2}^{\frac12}\|\p_3\partial_1^3 u\|_{L^2}^{\frac12}\notag\\
&\leq C\|u\|_{H^3}\left(\|\partial_2  u\|_{H^2}^2+\|\partial_3 u\|_{H^3}^2\right)+D_1^2+D_3, \label{H231}
\end{align}
where
$$
D_1^2\triangleq3\int\partial_2u_2\left(\partial_1^3u_2 \right)^2  d x, \quad
D_3\triangleq3\int\partial_2u_2\left(\partial_1^3u_3 \right)^2  d x.
$$

Now, plugging \eqref{H231} into \eqref{H23}, we arrive at
\begin{align*}
H_{23}
\leq C\|u\|_{H^3}\left(\|\partial_2  u\|_{H^2}^2+\|\partial_3 u\|_{H^3}^2\right)+D_1^2+D_3,
\end{align*}
which, together with \eqref{H21g} and \eqref{H22}, gives
\begin{align}
H_2\leq C\|u\|_{H^3}\left(\|\partial_2  u\|_{H^2}^2+\|\partial_3 u\|_{H^3}^2\right)+D_1+D_2+D_3.\label{H2g}
\end{align}
Here, the quantities $D_2, D_3$ are the same ones as above and
$$
D_1\triangleq D_1^1+D_1^2=2\int\partial_2u_2\left(\partial_1^3u_2 \right)^2 d x.
$$

The treatments of $D_1$, $D_2$ and $D_3$ are postponed to the next step.
We proceed to estimate $H_{3}$, $H_{4}$  and $H_{5}$ term by term.
For $H_{3}$,  we have
\begin{align}
H_3&=3\sum_{i=1}^{3} \int \p_i b \cdot \nabla \p_i^2b \cdot \partial_i^3 u \ d x+3\sum_{i=1}^{3} \int\p_i^2b\cdot \nabla \p_i b\cdot \partial_i^3 u \ d x\notag\\
&\quad+\sum_{i=1}^{3}\int \p_i^3b\cdot \nabla b\cdot \partial_i^3 u \ d x
\triangleq\sum_{j=1}^{3}H_{3j}.\label{H3}
\end{align}

To estimate the terms on the right-hand side, we first observe that
that
\begin{equation}\label{b1-23}
\|\p_1 b_1\|_{\dot H^k}\leq \|\p_2b_2\|_{\dot H^k}+\|\p_3b_3\|_{\dot H^k}\leq \|\p_2b\|_{\dot H^k}+\|\p_3b\|_{\dot H^k},\quad k=1,2,\ldots,
\end{equation}
due to the divergence-free condition  $\nabla\cdot b=0$. This, together with   Lemma \ref{anip1}, yields that (keeping in mind that $\p_\nu^k=(\p_2^k,\p_3^k)$ with $k=1,2,3$)
\begin{align}
H_{31}
&
=3\int\p_1b_1\p_1^3b\cdot\p_1^3u\ d x+3\int\p_1b_2\p_2\p_1^2b\cdot\p_1^3u\ d x\notag\\
&\quad+3\int\p_1b_3\p_3\p_1^2b\cdot\p_1^3u\ d x
+3\int\p_\nu b\cdot\na\p_\nu^2b\cdot\p_\nu^3u\ d x
\notag\\
 &\leq
C\|\p_1b_1\|_{L^2}^{\frac{1}{2}}\|\p_1^2b_1\|_{L^2}^{\frac{1}{2}}\|\p_1^3b\|_{L^2}^{\frac{1}{2}}
\|\p_2\p_1^3b\|_{L^2}^{\frac{1}{2}}\|\p_1^3u\|_{L^2}^{\frac{1}{2}}\|\p_3\p_1^3u\|_{L^2}^{\frac{1}{2}}\notag\\ 	
&\quad+C\|\p_1b_2\|_{L^2}^{\frac{1}{2}}\|\p_2\p_1b_2\|_{L^2}^{\frac{1}{2}}\|\p_2\p_1^2b \|_{L^2}^{\frac{1}{2}}\|\p_2\p_1^3b\|_{L^2}^{\frac{1}{2}}\|\p_1^3u\|_{L^2}^{\frac{1}{2}}\|\|\p_3\p_1^3u\|_{L^2}^{\frac{1}{2}}
\notag\\
&\quad+C\|\p_1b_3\|_{L^2}^{\frac{1}{2}}\|\p_2\p_1b_3\|_{L^2}^{\frac{1}{2}}\|\p_3\p_1^2b \|_{L^2}^{\frac{1}{2}}\|\p_3\p_1^3b\|_{L^2}^{\frac{1}{2}}\|\p_1^3u\|_{L^2}^{\frac{1}{2}}\|\|\p_3\p_1^3u\|_{L^2}^{\frac{1}{2}}
\notag\\
&\quad
+C\|\p_2b\|_{L^\infty}\|\na\p_2^2b\|_{L^2}\|\p_2^3u\|_{L^2}
+C\|\p_3b\|_{L^\infty}\|\na\p_3^2b\|_{L^2}\|\p_3^3u\|_{L^2}
\notag\\
&\leq
C\|(u,b)\|_{H^3}\left(\|(\p_3b,\p_3u)\|_{H^3}^2+\|\p_2b\|_{H^3}^2\right).\label{H31}
\end{align}

Analogously to the derivations of (\ref{H22}) and (\ref{H23}), we deduce
\begin{align}
H_{32}
& \leq
C\|\p_1^2b_1\|_{L^2}^{\frac{1}{2}}\|\p_1^3b_1\|_{L^2}^{\frac{1}{2}}\|\p_1^2b\|_{L^2}^{\frac{1}{2}}\|\p_2\p_1^2b\|_{L^2}^{\frac{1}{2}}\|\p_1^3u\|_{L^2}^{\frac{1}{2}}\|\p_3\p_1^3u\|_{L^2}^{\frac{1}{2}}
\notag\\
&\quad
+C\|\p_1^2b_2\|_{L^2}^{\frac{1}{2}}\|\p_2\p_1^2b_2\|_{L^2}^{\frac{1}{2}}\|\p_2\p_1b\|_{L^2}^{\frac{1}{2}}\|\p_2\p_1^2b\|_{L^2}^{\frac{1}{2}}\|\p_1^3u\|_{L^2}^{\frac{1}{2}}\|\p_3\p_1^3u\|_{L^2}^{\frac{1}{2}}
\notag\\
&\quad
+C\|\p_1^2b_3\|_{L^2}^{\frac{1}{2}}\|\p_2\p_1^2b_3\|_{L^2}^{\frac{1}{2}}\|\p_3\p_1b\|_{L^2}^{\frac{1}{2}}\|\p_3\p_1^2b\|_{L^2}^{\frac{1}{2}}\|\p_1^3u\|_{L^2}^{\frac{1}{2}}\|\p_3\p_1^3u\|_{L^2}^{\frac{1}{2}}
\notag\\
&\quad
+C\|\p_2^2b\|_{L^6}\|\na\p_2b\|_{L^3}\|\p_2^3u\|_{L^2}+C\|\p_3^2b\|_{L^6}\|\na\p_3b\|_{L^3}\|\p_3^3u\|_{L^2}
\notag\\
&\leq C\|(u,b)\|_{H^3}\Big(\|(\p_3b,\p_3u)\|_{H^3}^2+\|\p_2b\|_{H^3}^2\Big),\label{H32}
\end{align}
and
\begin{align}
H_{33}
& \leq
C\|\p_1^3b_1\|_{L^2}^{\frac{1}{2}}\|\p_1^4b_1\|_{L^2}^{\frac{1}{2}}\|\p_1 b\|_{L^2}^{\frac{1}{2}}\|\p_2\p_1b\|_{L^2}^{\frac{1}{2}}\|\p_1^3u\|_{L^2}^{\frac{1}{2}}\|\p_3\p_1^3u\|_{L^2}^{\frac{1}{2}}
\notag\\
&\quad
+C\|\p_1^3b_2\|_{L^2}^{\frac{1}{2}}\|\p_2\p_1^3b_2\|_{L^2}^{\frac{1}{2}}\|\p_2b\|_{L^2}^{\frac{1}{2}}\|\p_2\p_1b\|_{L^2}^{\frac{1}{2}}\|\p_1^3u\|_{L^2}^{\frac{1}{2}}\|\p_3\p_1^3u\|_{L^2}^{\frac{1}{2}}
\notag\\
&\quad
+C\|\p_1^3b_3\|_{L^2}^{\frac{1}{2}}\|\p_2\p_1^3b_3\|_{L^2}^{\frac{1}{2}}\|\p_3b\|_{L^2}^{\frac{1}{2}}\|\p_3\p_1b\|_{L^2}^{\frac{1}{2}}\|\p_1^3u\|_{L^2}^{\frac{1}{2}}\|\p_3\p_1^3u\|_{L^2}^{\frac{1}{2}}
\notag\\
&\quad
+C\|\p_2^3b\|_{L^2}\|\na b\|_{L^\infty}\|\p_2^3u\|_{L^2}+C\|\p_3^3b\|_{L^2}\|\na b\|_{L^\infty}\|\p_3^3u\|_{L^2}
\notag\\
&\leq C\|(u,b)\|_{H^3}\left(\|(\p_3b,\p_3u)\|_{H^3}^2+\|\p_2b\|_{H^3}^2\right).\label{H33}
\end{align}

Thus, putting \eqref{H31},\eqref{H32} and \eqref{H33} into \eqref{H3}, we get
\begin{align}
H_3\leq C\|(u,b)\|_{H^3}\left(\|(\p_3b,\p_3u)\|_{H^3}^2+\|\p_2b\|_{H^3}^2\right).\label{H3g}
\end{align}

In order to  estimate $ H_4 $, we rewrite it in the form:
\begin{align}
H_4&=-3\sum_{i=1}^{3} \int \p_i u \cdot \nabla \p_i^2b \cdot \partial_i^3 b \ d x-3\sum_{i=1}^{3} \int\p_i^2u\cdot \nabla \p_i b\cdot \partial_i^3 b \ d x\notag\\
&\quad-\sum_{i=1}^{3}\int \p_i^3u\cdot \nabla b\cdot \partial_i^3 b \ d x
\triangleq\sum_{j=1}^{3}H_{4j}.\label{H4}
\end{align}

Using (\ref{u1-23})  and Lemma \ref{anip1}, we find
\begin{align}
H_{41}
&
=-3\int\p_1u_1\p_1^3b\cdot\p_1^3b\ d x-3\int\p_1u_2\p_2\p_1^2b\cdot\p_1^3b\ d x
\notag\\
&\quad-3\int\p_1u_3\p_3\p_1^2b\cdot\p_1^3b\ d x
-3\int\p_\nu u\cdot\na\p_\nu^2b\cdot\p_\nu^3b\ d x
\notag\\
&\leq
C\|\p_1u_1\|_{L^2}^{\frac{1}{2}}\|\p_1^2u_1\|_{L^2}^{\frac{1}{2}}\|\p_1^3b\|_{L^2}^{\frac{1}{2}}\|\p_2\p_1^3b\|_{L^2}^{\frac{1}{2}}\|\p_1^3b\|_{L^2}^{\frac{1}{2}}\|\p_3\p_1^3b\|_{L^2}^{\frac{1}{2}}\notag\\ 	
&\quad+C\|\p_1u_2\|_{L^2}^{\frac{1}{2}}\|\p_3\p_1u_2\|_{L^2}^{\frac{1}{2}}\|\p_2\p_1^2b \|_{L^2}^{\frac{1}{2}}\|\p_2\p_1^3b\|_{L^2}^{\frac{1}{2}}\|\p_1^3b\|_{L^2}^{\frac{1}{2}}\|\|\p_2\p_1^3b\|_{L^2}^{\frac{1}{2}}
\notag\\
&\quad+C\|\p_1u_3\|_{L^2}^{\frac{1}{2}}\|\p_3\p_1u_3\|_{L^2}^{\frac{1}{2}}\|\p_3\p_1^2b \|_{L^2}^{\frac{1}{2}}\|\p_3\p_1^3b\|_{L^2}^{\frac{1}{2}}\|\p_1^3b\|_{L^2}^{\frac{1}{2}}\|\|\p_2\p_1^3b\|_{L^2}^{\frac{1}{2}}
\notag\\
&\quad
+C\|\p_2u\|_{L^\infty}\|\na\p_2^2b\|_{L^2}\|\p_2^3b\|_{L^2}
+C\|\p_3u\|_{L^\infty}\|\na\p_3^2b\|_{L^2}\|\p_3^3b\|_{L^2}
\notag\\
&\leq
C\|(u,b)\|_{H^3}\left(\|(\p_3u,\p_2b,\p_3b )\|_{H^3}^2+\|\p_2u\|_{H^2}^2\right).\label{H41}
\end{align}

In a similar manner,
\begin{align}
H_{42}
& \leq
C\|\p_1^2u_1\|_{L^2}^{\frac{1}{2}}\|\p_1^3u_1\|_{L^2}^{\frac{1}{2}}\|\p_1^2b\|_{L^2}^{\frac{1}{2}}\|\p_2\p_1^2b\|_{L^2}^{\frac{1}{2}}\|\p_1^3b\|_{L^2}^{\frac{1}{2}}\|\p_3\p_1^3b\|_{L^2}^{\frac{1}{2}}
\notag\\
&\quad
+C\|\p_1^2u_2\|_{L^2}^{\frac{1}{2}}\|\p_3\p_1^2u_2\|_{L^2}^{\frac{1}{2}}\|\p_2\p_1b\|_{L^2}^{\frac{1}{2}}\|\p_2\p_1^2b\|_{L^2}^{\frac{1}{2}}\|\p_1^3b\|_{L^2}^{\frac{1}{2}}\|\p_2\p_1^3b\|_{L^2}^{\frac{1}{2}}
\notag\\
&\quad
+C\|\p_1^2u_3\|_{L^2}^{\frac{1}{2}}\|\p_3\p_1^2u_3\|_{L^2}^{\frac{1}{2}}\|\p_3\p_1b\|_{L^2}^{\frac{1}{2}}\|\p_3\p_1^2b\|_{L^2}^{\frac{1}{2}}\|\p_1^3b\|_{L^2}^{\frac{1}{2}}\|\p_2\p_1^3b\|_{L^2}^{\frac{1}{2}}
\notag\\
&\quad
+C\|\p_2^2u\|_{L^6}\|\na\p_2b\|_{L^3}\|\p_2^3b\|_{L^2}+C\|\p_3^2u\|_{L^6}\|\na\p_3b\|_{L^3}\|\p_3^3b\|_{L^2}
\notag\\
&\leq C\|(u,b)\|_{H^3}\left(\|(\p_3u,\p_2b,\p_3b )\|_{H^3}^2+\|\p_2u\|_{H^2}^2\right),\label{H42}
\end{align}
 and
\begin{align}
H_{43}
& \leq
C\|\p_1^3u_1\|_{L^2} \|\p_1 b \|_{L^2}^{\frac{1}{4}}\|\p_1^2 b\|_{L^2}^{\frac{1}{4}}\|\p_2\p_1b\|_{L^2}^{\frac{1}{4}}\|\p_2\p_1^2b\|_{L^2}^{\frac{1}{4}}\|\p_1^3b\|_{L^2}^{\frac{1}{2}}\|\p_3\p_1^3b\|_{L^2}^{\frac{1}{2}}
\notag\\
&\quad
+C\|\p_1^3u_2\|_{L^2}^{\frac{1}{2}}\|\p_3\p_1^3u_2\|_{L^2}^{\frac{1}{2}}\|\p_2b\|_{L^2}^{\frac{1}{2}}\|\p_2\p_1b\|_{L^2}^{\frac{1}{2}}\|\p_1^3b\|_{L^2}^{\frac{1}{2}}\|\p_2\p_1^3b\|_{L^2}^{\frac{1}{2}}
\notag\\
&\quad
+C\|\p_1^3u_3\|_{L^2}^{\frac{1}{2}}\|\p_3\p_1^3u_3\|_{L^2}^{\frac{1}{2}}\|\p_3b\|_{L^2}^{\frac{1}{2}}\|\p_3\p_1b\|_{L^2}^{\frac{1}{2}}\|\p_1^3b\|_{L^2}^{\frac{1}{2}}\|\p_2\p_1^3b\|_{L^2}^{\frac{1}{2}}
\notag\\
&\quad
+C\|\p_2^3u\|_{L^2}\|\na b\|_{L^\infty}\|\p_2^3b\|_{L^2}+C\|\p_3^3u\|_{L^2}\|\na b\|_{L^\infty}\|\p_3^3b\|_{L^2}
\notag\\
&\leq C\|(u,b)\|_{H^3}\left(\|(\p_3u,\p_2b,\p_3b )\|_{H^3}^2+\|\p_2u\|_{H^2}^2\right).\label{H43}
\end{align}

Thus, substituting \eqref{H41},\eqref{H42} and \eqref{H43} into \eqref{H4} gives
\begin{align}
H_4\leq C\|(u,b)\|_{H^3}\left(\|(\p_3u,\p_2b,\p_3b )\|_{H^3}^2+\|\p_2u\|_{H^2}^2\right).\label{H4g}
\end{align}

Next, we estimate the term $H_5$, which can be written as follows:
\begin{align}
H_5&=3\sum_{i=1}^{3} \int \p_i b \cdot \nabla \p_i^2u \cdot \partial_i^3 b \ d x+3\sum_{i=1}^{3} \int\p_i^2b\cdot \nabla \p_i u\cdot \partial_i^3 b \ d x\notag\\
&\quad+\sum_{i=1}^{3}\int \p_i^3b\cdot \nabla u\cdot \partial_i^3 b \ d x
\triangleq\sum_{j=1}^{3}H_{5j}.\label{H5}
\end{align}

Similarly to the proof of (\ref{H31}), we deduce from (\ref{b1-23}) and Lemma \ref{anip1} that
\begin{align}
H_{51}&
=3\int\p_1b_1\p_1^3u\cdot\p_1^3b\ d x+3\int\p_1b_2\p_2\p_1^2u\cdot\p_1^3b\ d x\notag\\
&\quad +3\int\p_1b_3\p_3\p_1^2u\cdot\p_1^3b\ d x
+3\int\p_\nu b\cdot\na\p_\nu^2u\cdot\p_\nu^3b\ d x
\notag\\
&\leq
C\|\p_1b_1\|_{L^2}^{\frac{1}{2}}\|\p_1^2b_1\|_{L^2}^{\frac{1}{2}}\|\p_1^3u\|_{L^2}^{\frac{1}{2}}\|\p_3\p_1^3u\|_{L^2}^{\frac{1}{2}}\|\p_1^3b\|_{L^2}^{\frac{1}{2}}\|\p_2\p_1^3b\|_{L^2}^{\frac{1}{2}}\notag\\ 	
&\quad+C\|\p_1b_2\|_{L^2}^{\frac{1}{4}}\|\p_1^2b_2\|_{L^2}^{\frac{1}{4}}\|\p_1\p_2b_2\|_{L^2}^{\frac{1}{4}}\|\p_2\p_1^2b_2\|_{L^2}^{\frac{1}{4}}\|\p_2\p_1^2u \|_{L^2}\|\p_1^3b\|_{L^2}^{\frac{1}{2}}\|\p_3\p_1^3b\|_{L^2}^{\frac{1}{2}}
\notag\\
&\quad+C\|\p_1b_3\|_{L^2}^{\frac{1}{2}}\|\p_2\p_1b_3\|_{L^2}^{\frac{1}{2}}\|\p_3\p_1^2u \|_{L^2}^{\frac{1}{2}}\|\p_3\p_1^3u\|_{L^2}^{\frac{1}{2}}\|\p_1^3b\|_{L^2}^{\frac{1}{2}}\|\|\p_3\p_1^3b\|_{L^2}^{\frac{1}{2}}
\notag\\
&\quad
+C\|\p_2b\|_{L^\infty}\|\na\p_2^2u\|_{L^2}\|\p_2^3b\|_{L^2}
+C\|\p_3b\|_{L^\infty}\|\na\p_3^2u\|_{L^2}\|\p_3^3b\|_{L^2}
\notag\\
&\leq
C\|(u,b)\|_{H^3}\left(\|(\p_3u,\p_2b,\p_3b )\|_{H^3}^2+\|\p_2u\|_{H^2}^2\right),\label{H51}
\end{align}
and
\begin{align}
H_{52}
& \leq
C\|\p_1^2b_1\|_{L^2}^{\frac{1}{2}}\|\p_1^3b_1\|_{L^2}^{\frac{1}{2}}\|\p_1^2u\|_{L^2}^{\frac{1}{2}}\|\p_2\p_1^2u\|_{L^2}^{\frac{1}{2}}\|\p_1^3b\|_{L^2}^{\frac{1}{2}}\|\p_3\p_1^3b\|_{L^2}^{\frac{1}{2}}
\notag\\
&\quad
+C\|\p_1^2b_2\|_{L^2}^{\frac{1}{2}}\|\p_2\p_1^2b_2\|_{L^2}^{\frac{1}{2}}\|\p_2\p_1u\|_{L^2}^{\frac{1}{2}}\|\p_2\p_1^2u\|_{L^2}^{\frac{1}{2}}\|\p_1^3b\|_{L^2}^{\frac{1}{2}}\|\p_3\p_1^3b\|_{L^2}^{\frac{1}{2}}
\notag\\
&\quad
+C\|\p_1^2b_3\|_{L^2}^{\frac{1}{2}}\|\p_2\p_1^2b_3\|_{L^2}^{\frac{1}{2}}\|\p_3\p_1u\|_{L^2}^{\frac{1}{2}}\|\p_3\p_1^2u\|_{L^2}^{\frac{1}{2}}\|\p_1^3b\|_{L^2}^{\frac{1}{2}}\|\p_3\p_1^3b\|_{L^2}^{\frac{1}{2}}
\notag\\
&\quad
+C\|\p_2^2b\|_{L^6}\|\na\p_2u\|_{L^3}\|\p_2^3b\|_{L^2}+C\|\p_3^2b\|_{L^6}\|\na\p_3u\|_{L^3}\|\p_3^3b\|_{L^2}
\notag\\
&\leq C\|(u,b)\|_{H^3}\left(\|(\p_3u,\p_2b,\p_3b )\|_{H^3}^2+\|\p_2u\|_{H^2}^2\right).\label{H52}
\end{align}

Let $u_h \triangleq (u_1,u_2)$ and $b_h\triangleq(b_1,b_2)$. By direct calculations, we have
\begin{align*}
H_{53}
&=\int\p_1^3b_1\p_1u\cdot\p_1^3b\ d x
+\int\p_1^3b_2\p_2u_1\p_1^3b_1\ dx
+2\int\left(\p_1^3b_2\right)^2\p_2u_2\ dx
\notag\\
&\quad-\int\left(\p_1^3b_2\right)^2\p_2u_2\ dx+2\int\p_1^3b_2\p_2u_3\p_1^3b_3\ dx-\int\p_1^3b_2\p_2u_3\p_1^3b_3\ dx\notag\\
&\quad+\int\p_1^3b_3\p_3u_h\cdot\p_1^3b_h\ d x
-\int\p_1^3b_3\p_1u_1 \p_1^3b_3\ d x
-4\int\p_1^3b_3\p_2u_2 \p_1^3b_3\ d x
\notag\\
&\quad
+3\int\p_1^3b_3\p_2u_2\p_1^3b_3\ d x
+\int\p_2^3b\cdot\na u\cdot\p_2^3b\ d x
+\int\p_3^3b\cdot\na u\cdot\p_3^3b\ d x,
\end{align*}
and hence, it follows from (\ref{u1-23}), (\ref{b1-23}) and Lemma \ref{anip1} that
\begin{align}
H_{53}
& \leq
C\|\p_1^3b_1\|_{L^2}^{\frac{1}{2}}\|\p_1^4b_1\|_{L^2}^{\frac{1}{2}}\|\p_1u\|_{L^2}^{\frac{1}{2}}\|\p_3\p_1u\|_{L^2}^{\frac{1}{2}}\|\p_1^3b\|_{L^2}^{\frac{1}{2}}\|\p_2\p_1^3b\|_{L^2}^{\frac{1}{2}}
\notag\\
&\quad
+C\|\p_1^3b_2\|_{L^2}^{\frac{1}{2}}\|\p_2\p_1^3b_2\|_{L^2}^{\frac{1}{2}}\|\p_2u\|_{L^2}^{\frac{1}{2}}\|\p_2\p_1 u\|_{L^2}^{\frac{1}{2}}\|\p_1^3b\|_{L^2}^{\frac{1}{2}}\|\p_3\p_1^3b\|_{L^2}^{\frac{1}{2}}
\notag\\
&\quad
+C\|\p_1^3b_3\|_{L^2}^{\frac{1}{2}}\|\p_2\p_1^3b_3\|_{L^2}^{\frac{1}{2}}\|\p_3u\|_{L^2}^{\frac{1}{2}}\|\p_3\p_1 u\|_{L^2}^{\frac{1}{2}}\|\p_1^3b\|_{L^2}^{\frac{1}{2}}\|\p_3\p_1^3b\|_{L^2}^{\frac{1}{2}}
\notag\\
&\quad +C\|\p_1b_3\|_{L^2} \|\p_2\p_1^3b_3\|_{L^2}^{\frac{1}{2}}\|\p_3\p_1^3b_3\|_{L^2}^{\frac{1}{2}} \|\p_1 u_1\|_{L^2}^{\frac{1}{2}}\|\p_{1}^2u_1\|_{L^2}^{\frac{1}{2}} \notag\\
&\quad +C\|\p_1b_3\|_{L^2} \|\p_2\p_1^3b_3\|_{L^2}^{\frac{1}{2}}\|\p_3\p_1^3b_3\|_{L^2}^{\frac{1}{2}} \|\p_2 u_2\|_{L^2}^{\frac{1}{2}}\|\p_{1}\p_2u_2\|_{L^2}^{\frac{1}{2}} \notag\\
&\quad
+C \|\na u\|_{L^
\infty}\|\p_2^3b\|_{L^2}^2+C \|\na u\|_{L^\infty}\|\p_3^3b\|_{L^2}^2+ E_1+E_2+E_3
\notag\\
&\leq C\|(u,b)\|_{H^3}\left(\|(\p_3u,\p_2b,\p_3b )\|_{H^3}^2+\|\p_2u\|_{H^2}^2\right)+ E_1+E_2+E_3,\label{H53}
\end{align}
where
\begin{align*}
E_1&\triangleq2\int\left(\p_1^3b_2\right)^2\p_2u_2\ dx, \quad  E_2\triangleq-\int \p_1^3b_2 \p_2u_3 \p_1^3b_3\ dx,  \\
E_3&\triangleq3\int\left(\p_1^3b_3\right)^2\p_2u_2\ dx.
\end{align*}

Combining  \eqref{H51}, \eqref{H52}, \eqref{H53} with \eqref{H5} gives
\begin{align}
H_5\leq C\|(u,b)\|_{H^3}\left(\|(\p_3u,\p_2b,\p_3b )\|_{H^3}^2+\|\p_2u\|_{H^2}^2\right)+ E_1+E_2+E_3.\label{H5g}
\end{align}

Therefore, inserting \eqref{H1}, \eqref{H2g}, \eqref{H3g}, \eqref{H4g} and   \eqref{H5g} into \eqref{H3bound}, we immediately obtain the following preliminary bound of $\dot H^3$-norms:
\begin{align}
\frac12&\frac{d}{d t}\|(\nabla^3 u, \nabla^3 b)\|^2_{L^2}+\mu\|\p_3\nabla^3 u \|^2_{L^2}+\eta\left(\|\p_2\nabla^3 b\|^2_{L^2}+\|\p_3\nabla^3 b\|^2_{L^2}\right)\notag\\
& \leq C\|(u,b)\|_{H^3}\Big(\|(\p_3u,\p_2b,\p_3b )\|_{H^3}^2+\|\p_2u\|_{H^2}^2\Big)+\sum_{i=1}^3D_i+\sum_{i=1}^3 E_i. \label{HU3}
\end{align}

\textbf{Step III. The estimates of $D_1$, $D_2$  and $D_3$ }
\vskip .1in
Clearly, to close the a priori estimates stated in (\ref{HU3}), we still need to deal with $D_i$   and $E_i$ with $i=1,2,3$, based on the special structure of (\ref{PMHD}).
\vskip .1in
\textbf{Step III-1. The estimates of $D_1$ and $D_3$}
\vskip .1in
We start with  the estimate of  $D_1$. First, it follows from \eqref{PMHD} that
\begin{align}
&\p_2u_2=\p_tb_2-\eta\p_2^2b_2-\eta\p_3^2b_2+u\cdot\na b_2-b\cdot\na u_2,\label{pu2}\\
&\p_2u_3=\p_tb_3-\eta\p_2^2b_3-\eta\p_3^2b_3+u\cdot\na b_3-b\cdot\na u_3,\label{pu3}\\
&\p_2b_2=\p_tu_2-\mu\p_3^2u_2+u\cdot\na u_2-b\cdot\na b_2+\p_2p,\label{pb2}\\
&\p_2b_3=\p_tu_3-\mu\p_3^2u_3+u\cdot\na u_3-b\cdot\na b_3+\p_3p.\label{pb3}
\end{align}
Plugging \eqref{pu2}--\eqref{pu3} into \eqref{pb2}--\eqref{pb3} gives
\begin{equation}\label{m13}
\begin{cases}
\p_2u_2=\p_t(b_2-\eta\p_2u_{2})-\eta\p_3^2b_2
+\mu\eta\p_2\p_3^2u_2
\\ \quad\qquad
-\eta\p_2\left(
u\cdot\na u_2-b\cdot\na b_2+\p_2p
\right)+u\cdot\na b_2-b\cdot\na u_2,\\[2mm]
\p_2u_3=\p_t(b_3-\eta\p_2u_{3})-\eta\p_3^2b_3
+\mu\eta\p_2\p_3^2u_3
\\ \quad\qquad
-\eta\p_2\left(
u\cdot\na u_3-b\cdot\na b_3+\p_3p
\right)+u\cdot\na b_3-b\cdot\na u_3.
\end{cases}
\end{equation}

Recalling the definition of $D_1$, by virtue of $\eqref{m13}_1 $ we obtain after integrating by parts that
\begin{align}
D_1
&=2\int\left(\p_1^3u_2\right)^2\p_t(b_2-\eta\p_2u_{2})\ d x
-2\eta\int\left(\p_1^3u_2\right)^2\p_3^2b_2\ d x
\notag\\
&\quad
+2\mu\eta\int\left(\p_1^3u_2\right)^2\p_2\p_3^2u_2\ d x+2\int\left(\p_1^3u_2\right)^2\left(
u\cdot\na b_2-b\cdot\na u_2
\right) d x
\notag\\
&\quad
-2\eta\int\left(\p_1^3u_2\right)^2\p_2\left(
u\cdot\na u_2-b\cdot\na b_2+\p_2p
\right) d x
\notag\\
&=
2\frac{d}{d t}\int\left(\p_1^3u_2\right)^2(b_2-\eta\p_2u_{2})\ d x
-4\int\p_1^3u_2\p_1^3\p_tu_2(b_2-\eta\p_2u_{2})\ d x
\notag\\
&\quad
+4\eta\int\p_3\p_1^3u_2\p_1^3u_2\p_3b_2\ d x
-4\mu\eta\int\p_3\p_1^3u_2\p_1^3u_2\p_3\p_2u_2\ d x
\notag\\
&\quad
-2\eta\int\left(\p_1^3u_2\right)^2\p_2\left(
u\cdot\na u_2-b\cdot\na b_2+\p_2p
\right) d x
\notag\\
&\quad
+2\int\left(\p_1^3u_2\right)^2
u\cdot\na b_2\ d x-2\int\left(\p_1^3u_2\right)^2b\cdot\na u_2
\ d x\notag\\
&\triangleq
2\frac{d}{d t}\int\left(\p_1^3u_2\right)^2(b_2-\eta\p_2u_{2})\ d x
+\sum_{j=1}^{6}D_{1j}.\label{D1}
\end{align}

The terms on the right-hand side of (\ref{D1}) will be treated one by one. First, thanks to \eqref{pb2}, it holds that
\begin{align}
D_{11}
&=-4\int\p_1^3u_2\p_1^3\p_2b_2(b_2-\eta\p_2u_{2})\ d x
+4\mu\int\p_3\p_1^3u_2
\p_3\p_1^3u_2(b_2-\eta\p_2u_2)\ d x
 \notag\\
 &\quad
+4\mu\int\p_1^3u_2
\p_3\p_1^3u_2\p_3b_2\ d x
-4\mu\eta\int\p_1^3u_2
\p_3\p_1^3u_2\p_3\p_2u_2\ d x
 \notag\\
&\quad+4\int\p_1^3(u\cdot\na u_2-b\cdot\na b_2+\p_2p)\p_1^3u_2(b_2-\eta\p_2u_2)\ d x
\triangleq\sum_{j=1}^{5}D_{11j},\label{D11}
\end{align}
where the first four terms can be bounded as follows, using  Lemma \ref{anip1}, H\"older's and Sobolev's inequalities.
\begin{align}
\sum_{j=1}^4D_{11j}
&\leq
C\|\p_1^3\p_2b_2\|_{L^2}\|\p_1^3u_2\|_{L^2}^{\frac{1}{2}}\|\p_3\p_1^3u_2\|_{L^2}^{\frac{1}{2}}\|b_2\|_{H^1}^{\frac{1}{2}}\|\p_2b_2\|_{H^1}^{\frac{1}{2}}
\notag\\
&\quad
+C\|\p_1^3u_2\|_{L^2}\|\p_2\p_1^3b_2\|_{L^2} \|\p_2u_2\|_{L^\infty}
\notag\\
&\quad
+C\|\p_3\p_1^3u_2\|_{L^2}\|\p_3\p_1^3u_2\|_{L^2}\left(\|b_2\|_{L^\infty}
+\|\p_2u_2\|_{L^\infty}\right)
\notag\\
&\quad
+C\|\p_1^3u_2\|_{L^2}\|\p_3\p_1^3u_2\|_{L^2}\left(\|\p_3b_2\|_{L^\infty}+\|\p_3\p_2u_2\|_{L^\infty}\right)   \notag\\
&\quad\leq
C\|(u,b)\|_{H^3}(\|\p_2u\|_{H^2}^2+\|\p_3u\|_{H^3}^2+\|\p_2b\|_{H^3}^2).\label{D111-4}
\end{align}

To deal with the term associated with the pressure,
we observe that
\begin{align}
p&=(-\Delta)^{-1}\na\cdot(u\cdot\na u- b\cdot\na b)  =(-\Delta)^{-1}(\p_iu_j\p_ju_i-\p_ib_j\p_jb_i),\label{p}
\end{align}
and hence
\begin{align}
D_{115}&=4\int(\p_1^3u\cdot\na u_2
+3\p_1^2u\cdot\na\p_1 u_2
+3\p_1u\cdot\na\p_1^2 u_2
)\p_1^3u_2(b_2-\eta\p_2u_2)\ dx
\notag\\
&\quad
-4\int(\p_1^3b\cdot\na b_2
+3\p_1^2b\cdot\na\p_1 b_2
+3\p_1b\cdot\na\p_1^2 b_2
)\p_1^3u_2(b_2-\eta\p_2u_2)\ dx \label{D115}
\\
&\quad
+2\int u\cdot\na(\p_1^3u_2)^2(b_2-\eta\p_2u_2)\ dx
-4\int b\cdot\na\p_1^3 b_2\p_1^3u_2(b_2-\eta\p_2u_2)\ dx
\notag\\
&\quad
+4\int \p_2\p_1^3(-\Delta)^{-1}(\p_iu_j\p_ju_i-\p_ib_j\p_jb_i)\p_1^3u_2(b_2-\eta\p_2u_2)\ dx
\triangleq\sum_{j=1}^5F_j.\notag
\end{align}

Using (\ref{u1-23}), Lemma \ref{anip2} and Sobolev's embedding inequality, one has
\begin{align}
F_1&=4\int\left(\p_1^3u\cdot\na u_2
+3\p_1u\cdot\na\p_1^2 u_2
)\p_1^3u_2(b_2-\eta\p_2u_2\right)\ dx\notag\\
&\quad
+12\int (\p_1^2u_1\p^2_1 u_2+ \p_1^2u_2\p_2\p_1 u_2+\p_1^2u_3\p_3\p_1 u_2)\p_1^3u_2\left(b_2-\eta\p_2u_2\right) dx \notag\\
&\leq
C\left(\|(\p_1^3 u,\p_1^3u_2)\|_{L^2}^{2}+\|(\na u_2, b_2, \p_2u_2)\|_{H^1}^{2}\right)\notag\\
&\qquad\times\left(\|(\p_3\p_1^3 u,\p_3\p_1^3u_2)\|_{L^2}^{2}+\|(\p_2\na u_2,\p_2b_2, \p_2^2u_2)\|_{H^1}^{2}\right)\notag\\
&\quad+C\left(\|(\na\p_1^2 u_2,\p_1^3u_2)\|_{L^2}^{2}+\|(\p_1 u, b_2, \p_2u_2)\|_{H^1}^{2}\right)\notag\\
&\qquad\times\left(\|(\p_3\na\p_1^2 u_2,\p_3\p_1^3u_2)\|_{L^2}^{2}+\|(\p_2\p_1 u,\p_2b_2, \p_2^2u_2)\|_{H^1}^{2}\right)\notag\\
&\quad+C\|(b_2,\p_2u_2)\|_{L^\infty}\left(\|\p_1^2u_{1}\|_{L^2}^{\frac{1}{2}}\|\p_1^3u_{1}\|_{L^2}^{\frac{1}{2}}\|\p_1^2u_{2}\|_{L^2}^{\frac{1}{2}}
\|\p_2\p_1^2u_2\|_{L^2}^{\frac{1}{2}}\|\p_1^3u_2\|_{L^2}^{\frac{1}{2}}\right.\notag\\
&\qquad+\left. \|\p_1^2u_{\nu}\|_{L^2}^{\frac{1}{2}}\|\p_2\p_1^2u_{\nu}\|_{L^2}^{\frac{1}{2}}\|\p_\nu\p_1u_{2}\|_{L^2}^{\frac{1}{2}}
\|\p_\nu\p_1^2u_2\|_{L^2}^{\frac{1}{2}}\right)\|\p_1^3u_2\|_{L^2}^{\frac{1}{2}}\|\p_3\p_1^3u_2\|_{L^2}^{\frac{1}{2}}
\notag\\
&\leq C\|(u, b)\|_{H^3}^2\left(\|\p_2u\|_{H^2}^2+\|\p_3u \|_{H^3}^2+\|\p_2b\|_{H^3}^2\right),\label{F1}
\end{align}
and analogously, by (\ref{b1-23}) one gets
\begin{align}
F_2\leq C\|(u, b)\|_{H^3}^2\left(\|\p_2u\|_{H^2}^2+\|\p_3u \|_{H^3}^2+\|\p_2b\|_{H^3}^2+\|\p_3b\|_{H^3}^2\right).\label{F2}
\end{align}

Since the Riesz operator $\p_i (-\Delta)^{-\frac12}$ with $i=1,2,3$
is bounded in $L^r$  for any $1<r<\infty$, we see that
\begin{align}
&\|\p_2\p_1^3(-\Delta)^{-1}(\p_iu_j\p_ju_i-\p_ib_j\p_jb_i)\|_{L^2}\notag\\
&\quad\leq C\|\p_2\p_{1}\left(\p_iu_j\p_ju_i-\p_ib_j\p_jb_i\right)\|_{L^2}\notag\\
&\quad\leq C\|\p_2\p_{1}\na u\|_{L^2}\|\na u\|_{L^\infty}+C\|\p_2\na u\|_{L^6}\|\p_1\na u\|_{L^3}\notag\\
&\qquad+C\|\p_2\p_{1}\na b\|_{L^2}\|\na b\|_{L^\infty}+C\|\p_2\na b\|_{L^6}\|\p_1\na b\|_{L^3}\notag\\
&\quad\leq C\|(u,b)\|_{H^3}\left(\|\p_2u\|_{H^2} +\|\p_2b\|_{H^2} \right),\label{RS}
\end{align}
from which, \eqref{he}  and Lemma \ref{anip1} it follows that
\begin{align}
F_5&=4\int \p_2\p_1^3(-\Delta)^{-1}(\p_iu_j\p_ju_i-\p_ib_j\p_jb_i)\p_1^3u_2(b_2-\eta\p_2u_2)\ dx\notag\\
&\leq
C\|\p_2\p_1^3(-\Delta)^{-1}(\p_iu_j\p_ju_i-\p_ib_j\p_jb_i)\|_{L^2}\notag\\
&\quad
\times\|\p_1^3u_2\|_{{L_{x_1}^{2}L_{x_2}^{2}L_{x_3}^{\infty}}}\left(\|b_2\|_{{L_{x_1}^{\infty}L_{x_2}^{\infty}L_{x_3}^{2}}}
+\|\p_2u_2\|_{{L_{x_1}^{\infty}L_{x_2}^{\infty}L_{x_3}^{2}}}\right)\notag\\
&\leq
C\|\p_2\p_{1}\left(\p_iu_j\p_ju_i-\p_ib_j\p_jb_i\right)\|_{L^2}\notag\\
&\quad\times\|\p_1^3u_2\|_{L^2}^{\frac12}\|\p_3\p_1^3u_2\|_{L^2}^{\frac12} \left(\|b_2\|_{H^1}^{\frac12}  \|\p_2b_2\|_{H^1}^{\frac{1}{2}}+\|\p_2u_2\|_{H^1}^{\frac12}  \|\p_2^2u_2\|_{H^1}^{\frac{1}{2}}\right)\notag\\
&\leq
C\|(u,b)\|^2_{H^3}\left(\|\p_2u\|_{H^2}^2+\|\p_3u\|_{H^3}^2+\|\p_2b\|_{H^2}^2\right).\label{F5}
\end{align}

Thus, inserting \eqref{F1}, \eqref{F2} and \eqref{F5} into \eqref{D115} shows
\begin{align*}
D_{115}\leq C\|(u,b)\|^2_{H^3}\left(\|\p_2u\|_{H^2}^2+\|\p_3u\|_{H^3}^2+\|\p_2b\|_{H^2}^2\right)+F_3+F_4,
\end{align*}
which, combined with  \eqref{D111-4} and \eqref{D11}, yields
\begin{align}\label{D11g}
D_{11}\leq C\left(\|(u,b)\|_{H^3}+\|(u,b)\|^2_{H^3}\right)\left(\|\p_2u\|_{H^2}^2+\|(\p_3u,\p_2b)\|_{H^3}^2 \right)+F_3+F_4.
\end{align}

For $D_{12}$, $D_{13}$ and $D_{16}$, we infer from  Lemmas \ref{anip1} and   \ref{anip2} that
\begin{align}
&D_{12}+D_{13}+D_{16}\notag\\
&\quad =4\eta\int\p_3\p_1^3u_2\p_1^3u_2\p_3b_2\ d x
-4\mu\eta\int\p_3\p_1^3u_2\p_1^3u_2\p_3\p_2u_2\ d x\notag\\
&\qquad-2\int \p_1^3u_2\p_1^3u_2 (b\cdot\na u_2)
\ d x\notag\\
&\quad \leq C\|\p_{3}\p_1^3u_{2}\|_{L^2}\|\p_1^3u_{2}\|_{L^2}^{\frac{1}{2}}\|\p_3\p_1^3u_{2}\|_{L^2}^{\frac{1}{2}}\notag\\
&\qquad\quad\times\left(\|\p_3b_2\|_{H^1}^{\frac{1}{2}}\|\p_2\p_3b_2\|_{H^1}^{\frac{1}{2}}+\|\p_3\p_2u_2\|_{H^1}^{\frac{1}{2}}\|\p_3\p_2^2u_2\|_{H^1}^{\frac{1}{2}}\right)\notag\\
&\qquad+C\left(\|\p_1^3u_2\|_{L^2}^{2}+\|(b, \na u_2)\|_{H^1}^{2}\right)\left(\|\p_3\p_1^3u_2\|_{L^2}^{2}+\|(\p_2b,\p_2\na u_2)\|_{H^1}^{2}\right)\notag\\
&\quad\leq C\left(\|(u,b)\|_{H^3}+\|(u,b)\|^2_{H^3}\right)\left(\|\p_2u\|_{H^2}^2+\|(\p_3u,\p_2b,\p_3b)\|_{H^3}^2\right).\label{D1236}
\end{align}

In order to  estimate $D_{14}$, we first deal with the pressure term. Indeed, similarly to the derivation of (\ref{RS}), it is easily seen from (\ref{p}) that
\begin{align}
\|\p_2^2p\|_{{L_{x_1,x_2}^{\infty} L_{x_3}^{2}}}
& \leq C\|\p_2^2p\|_{L^{2}}^{\frac14}\|\p_2^3p\|_{L^{2}}^{\frac14}\|\p_1\p_2^2p\|_{L^{2}}^{\frac14}\|\p_1\p_2^3p\|_{L^{2}}^{\frac14}\notag\\
& \leq C\|\p_2(u\cdot \na u- b\cdot\na b)\|_{L^{2}}^{\frac14}\|\p_2(\p_iu_j\p_ju_i-\p_ib_j\p_jb_i)\|_{L^{2}}^{\frac12}\notag\\
&\quad\times\|\p_2^2(\p_iu_j\p_ju_i-\p_ib_j\p_jb_i)\|_{L^{2}}^{\frac14}.\label{22pk}
\end{align}

 By \eqref{1IE}, \eqref{MIE} and \eqref{ge}, we have
\begin{align}
&\|\p_2(u\cdot \na u)\|_{L^{2}}^{\frac14}\|\p_2(\p_iu_j\p_ju_i)\|_{L^{2}}^{\frac12}\|\p_2^2(\p_iu_j\p_ju_i)\|_{L^{2}}^{\frac14}\notag\\
&\quad \leq C\left(\|\p_2u\|_{{L_{x_1,x_2}^{2}L_{x_3}^{\infty}}}\|\na u\|_{{L_{x_1,x_2}^{\infty} L_{x_3}^{2}}}+\|\p_2\na u\|_{{L_{x_1,x_2}^{2}L_{x_3}^{\infty}}}\| u\|_{{L_{x_1,x_2}^{\infty}L_{x_3}^{2}}}\right)^{\frac14}\notag\\
&\qquad\times C\left(\|\p_2\na u\|_{{L_{x_1,x_2}^{2}L_{x_3}^{\infty}}}\|\na u\|_{{L_{x_1,x_2}^{\infty}L_{x_3}^{2}}}\right)^{\frac12}\notag\\
&\qquad\times C\left(\|\p_2^2\na u\|_{{L_{x_1,x_2}^{2}L_{x_3}^{\infty}}}\|\na u\|_{{L_{x_1,x_2}^{\infty}L_{x_3}^{2}}}+\|\p_2\na u\|_{L_6}\| \p_2\na u\|_{L^3}\right)^{\frac14}\notag\\
&\quad \leq C\left(\|\p_2u\|_{L^2}^{\frac12}\|\p_3\p_2u\|_{L^2}^{\frac12}\|\na u\|_{H^1}^{\frac12}\|\p_2\na u\|_{H^1}^{\frac12}\right.\notag\\
&\qquad\qquad+\left.\|\p_2\na u\|_{L^2}^{\frac12}\|\p_3\p_2\na u\|_{L^2}^{\frac12}\| u\|_{H^1}^{\frac12}\|\p_2 u\|_{H^1}^{\frac12}\right)^{\frac14}\notag\\
&\qquad\times C\left(\|\p_2\na u\|_{L^2}^{\frac12}\|\p_3\p_2\na u\|_{L^2}^{\frac12}\|\na u\|_{H^1}^{\frac12}\|\p_2\na u\|_{H^1}^{\frac12}\right)^{\frac12}\notag\\
&\qquad\times C\left(\|\p_2^2\na u\|_{L^2}^{\frac12}\|\p_3\p_2^2\na u\|_{L^2}^{\frac12}\|\na u\|_{H^1}^{\frac12}\|\p_2\na u\|_{H^1}^{\frac12}+\|\p_2\na u\|_{H^1}^2\right)^{\frac14}\notag\\
&\quad \leq C\|u\|_{H^3}^{\frac12}\|\p_2u\|_{H^2}^{\frac{11}{8}}\|\p_3u\|_{H^3}^{\frac{1}{8}}.\label{22pu}
\end{align}
The other terms in (\ref{22pk}) admit similar bounds. Hence, we deduce
\begin{align}
\|\p_2^2p\|_{{L_{x_1,x_2}^{\infty}L_{x_3}^{2}}}
\leq C\|(u,b)\|_{H^3}^{\frac12}\|(\p_2u,\p_2b)\|_{H^2}^{\frac{11}{8}}\|(\p_3u,\p_3b)\|_{H^3}^{\frac{1}{8}}.\label{22pg}
\end{align}

In view of (\ref{22pg}), Lemmas \ref{anip1} and   \ref{anip2}, we find
\begin{align}
D_{14}
&=-2\eta\int\left(\p_1^3u_2\right)^2\left(\p_2
u\cdot\na u_2-\p_2
b\cdot\na b_2-b\cdot \na\p_2 b_2\right)\ d x\notag\\
&\quad-2\eta\int\left(\p_1^3u_2\right)^2\p_2^2p
\ d x-2\eta\int\left(\p_1^3u_2\right)^2
u\cdot \na \p_2u_2\ d x\notag\\
& \leq
C\left(\|\p_1^3u_2\|_{L^2}^{2}+\|(\p_2 u, \na u_2, \p_2b, \na b_2, b, \na\p_2 b_2)\|_{H^1}^{2}\right)\notag\\
&\qquad\times\left(\|\p_3\p_1^3u_2 \|_{L^2}^{2}+\|(\p_2^2 u, \p_2\na u_2, \p_2^2b, \p_2\na b_2, \p_2b, \p_2^2\na b_2)\|_{H^1}^{2}\right)\notag\\
&\quad+C\|\p_1^3u_2 \|_{L^2}\|\p_1^3u_2\|_{{L_{x_1,x_2}^{2}L_{x_3}^{\infty}}}\|\p_2^2p\|_{{L_{x_1,x_2}^{\infty}L_{x_3}^{2}}}+\Xi_1\notag\\
& \leq C\|(u,b)\|^2_{H^3}\left(\|\p_2u\|_{H^2}^2+\|(\p_3u,\p_2b,\p_3b)\|_{H^3}^2\right)+\Xi_1\notag\\
&\quad+C\|\p_1^3u_2\|_{L^2}^{\frac32}\|\p_3\p_1^3u_2\|_{L^2}^{\frac12}
\|(u,b)\|_{H^3}^{\frac12}\|(\p_2u,\p_2b)\|_{H^2}^{\frac{11}{8}}\|(\p_3u,\p_3b)\|_{H^3}^{\frac{1}{8}}\notag\\
&\quad\leq C\|(u,b)\|^2_{H^3}\left(\|\p_2u\|_{H^2}^2+\|(\p_3u,\p_2b,\p_3b)\|_{H^3}^2\right)+\Xi_1,\label{D14}
\end{align}
where
$$
\Xi_1\triangleq-2\eta\int\left(\p_1^3u_2\right)^2
u\cdot \na \p_2u_2\ d x.
$$

A key and amazing consequence of the above estimates is that
\begin{align}
D_{15}+F_3+\Xi_1&=2\int\left(\p_1^3u_2\right)^2
u\cdot\na b_2\ d x+2\int u\cdot\na(\p_1^3u_2)^2(b_2-\eta\p_2u_2)\ dx\notag\\
&\quad-2\eta\int\left(\p_1^3u_2\right)^2
u\cdot \na \p_2u_2\ d x
=0,\label{DF3H}
\end{align}
and consequently, it follows from (\ref{D1}), \eqref{D11g}, \eqref{D1236}, \eqref{D14}  and \eqref{DF3H} that
\begin{align}\label{D1g}
D_{1}&\leq2\frac{d}{d t}\int\left(\p_1^3u_2\right)^2(b_2-\eta\p_2u_{2})\ d x \notag\\
&\quad+C\|(u,b)\|^2_{H^3}\left(\|\p_2u\|_{H^2}^2+\|(\p_3u,\p_2b,\p_3b)\|_{H^3}^2\right)+F_4.
\end{align}
where
$$
F_4=-4\int b\cdot\na\p_1^3 b_2\p_1^3u_2(b_2-\eta\p_2u_2)\ dx.
$$

It remains to estimate $ F_4 $, which will be tackled with the term $ E_1 $ in $ H_{53} $. To do this, we use \eqref{m13}$ _{1} $  to get  that
\begin{align}
E_1
&=2\int\left(\p_1^3b_2\right)^2\p_t(b_2-\eta\p_2u_{2})\ dx-2\eta\int\left(\p_1^3b_2\right)^2\p_3^2b_2\ dx\notag\\
&\quad
+2\mu\eta\int\left(\p_1^3b_2\right)^2\p_2\p_3^2u_2\ dx+2\int\left(\p_1^3b_2\right)^2\left(u\cdot\na b_2-b\cdot\na u_2\right)\ dx\notag\\
&\quad-2\eta\int\left(\p_1^3b_2\right)^2\p_2\left(
u\cdot\na u_2-b\cdot\na b_2+\p_2p
\right)\ dx\notag\\
&=
2\frac{d}{d t}\int\left(\p_1^3b_2\right)^2(b_2-\eta\p_2u_{2})\ dx
-4\int\p_1^3\p_tb_2\p_1^3b_2\left(b_2-\eta\p_2u_{2}\right)\ dx\notag\\
&\quad
-2\eta\int\left(\p_1^3b_2\right)^2\p_3^2b_2\ dx
+2\mu\eta\int\left(\p_1^3b_2\right)^2\p_2\p_3^2u_2\ dx\notag\\
&\quad
-2\eta\int\left(\p_1^3b_2\right)^2\p_2\left(
u\cdot\na u_2-b\cdot\na b_2+\p_2p
\right) dx\notag\\
&\quad+2\int\left(\p_1^3b_2\right)^2\left(u\cdot\na b_2-b\cdot\na u_2\right) dx \notag\\
&\triangleq2\frac{d}{d t}\int\left(\p_1^3b_2\right)^2(b_2-\eta\p_2u_{2})\ dx+\sum_{i=1}^{5}E_{1i}.\label{E1-1}
\end{align}

In view of  \eqref{pu2}, we have
\begin{align}
E_{11}
&=-4\int\p_1^3\left(\p_2u_2+\eta\p_2^2b_2+\eta\p_3^2b_2\right) \p_1^3b_2\left(b_2-\eta\p_2u_{2}\right) dx\notag\\
&\quad
+4\int\p_1^3\left(u\cdot\na b_2\right) \p_1^3b_2\left(b_2-\eta\p_2u_{2}\right) dx\notag\\
&\quad
-4\int\p_1^3\left(b\cdot\na u_2\right) \p_1^3b_2\left(b_2-\eta\p_2u_{2}\right) dx
\triangleq\sum_{i=1}^{3}E_{11i}, \label{E11}
\end{align}

Based upon integration by parts, we have
\begin{align*}
E_{111}
&=4\int\p_1^3\left(u_2+\eta\p_2b_2\right) \p_2\p_1^3b_2\left(b_2-\eta\p_2u_{2}\right) dx\notag\\
&\quad+4\int\p_1^3\left(u_2+\eta\p_2b_2\right) \p_1^3b_2\left(\p_2b_2-\eta\p_2^2u_{2}\right) dx
\notag\\
&\quad+4\eta\int \left(\p_3\p_1^3b_2\right)^2\left(b_2-\eta\p_2u_{2}\right)dx\notag\\
&\quad+4\eta\int\p_1^3\p_3b_2\p_1^3b_2\left(\p_3b_2-\eta\p_3\p_2u_{2}\right) dx,
\end{align*}
so that, by Lemma \ref{anip1} we deduce
\begin{align}\label{E111}
E_{111}
&\leq C\|\p_1^3u_2\|_{L^2}^{\frac12}\|\p_3\p_1^3u_2\|_{L^2}^{\frac12}\|\p_2\p_1^3b_2\|_{L^2} \|\left(b_2, \p_2 u_2\right)\|_{H^1}^{\frac12}
\|\p_2\left(b_2, \p_2 u_2\right)\|_{H^1}^{\frac12}
\notag\\
&
\quad+C\|\p_1^3u_2\|_{L^2}^{\frac12}\|\p_3\p_1^3u_2\|_{L^2}^{\frac12}\|\p_1^3b_2\|_{L^2}^{\frac12}\|\p_2\p_1^3b_2\|_{L^2}^{\frac12}  \|\p_2b_2 \|_{L^2}^{\frac12}\|\p_1\p_2b_2\|_{L^2}^{\frac12}\notag\\
&
\quad+C\|\p_1^3u_2\|_{L^2}^{\frac12}\|\p_3\p_1^3u_2\|_{L^2}^{\frac12}\|\p_1^3b_2\|_{L^2}^{\frac12}\|\p_2\p_1^3b_2\|_{L^2}^{\frac12}  \|\p_2^2 u_2\|_{L^2}^{\frac12}\|\p_1\p_2^2 u_2\|_{L^2}^{\frac12}\notag\\
&
\quad+C \|\p_2\p_1^3b_2\|_{L^2}\| \p_1^3b_2\|_{L^2}^{\frac12}\|\p_2\p_1^3b_2\|_{L^2}^{\frac12}  \|\p_2^2 u_2\|_{H^1}^{\frac12} \|\p_3\p_2^2 u_2\|_{H^1}^{\frac12} \notag\\
&
\quad
+C \left(\|\p_3\p_1^3b_2\|_{L^2}^{2}+ \|\p_2\p_1^3b_2\|_{L^2}^{2}\right)\|\left(b_2, \p_2 u_2\right)\|_{L^\infty}\notag\\
&
\quad+C \left(\|\p_2\p_1^3b_2\|_{L^2}+\|\p_3\p_1^3b_2\|_{L^2}\right)\| \p_1^3b_2\|_{L^2} \|\left(\p_2b_2,\p_3b_2, \p_3\p_2  u_2\right)\|_{L^\infty}\notag\\
&\leq C\|(u,b)\|_{H^3}\left(\|\p_2u\|_{H^2}^2+\|(\p_3u,\p_2b,\p_3b)\|_{H^3}^2\right).
\end{align}

By straightforward calculation, we have
\begin{align}
& E_{112}+E_{113}\notag\\
&\quad=4\int(\p_1^3u\cdot\na b_2
+3\p_1^2u\cdot\na\p_1 b_2
+3\p_1u\cdot\na\p_1^2 b_2
)\p_1^3b_2(b_2-\eta\p_2u_2)\ dx
\notag\\
&\qquad
-4\int(\p_1^3b\cdot\na u_2
+3\p_1^2b\cdot\na\p_1 u_2
+3\p_1b\cdot\na\p_1^2 u_2
)\p_1^3b_2(b_2-\eta\p_2u_2)\ dx
\notag\\
&\qquad
-2\int (\p_1^3b_2)^2 u\cdot\na(b_2-\eta\p_2u_2)\ dx
-4 \int b\cdot\na\p_1^3 u_2\p_1^3b_2\left(b_2-\eta\p_2u_2\right) dx
\notag\\
&\quad\triangleq\sum_{j=1}^{4}R_{j}.\label{EF}
\end{align}

In terms of Lemmas \ref{anip1} and \ref{anip2}, we obtain in a similar manner as the derivation of (\ref{F1}) that
\begin{align}
R_1
&\leq
C\left(\|(\p_1^3 u,\p_1^3b_2,\p_1^2\na b)\|_{L^2}^{2}+\|(\na b_2, b_2,\p_2u_2,\p_1u)\|_{H^1}^{2}\right)\notag\\
&\qquad\times \left(\|\p_3(\p_1^3 u,\p_1^3b_2,\p_1^2\na b)\|_{L^2}^{2}+\|\p_2(\na b_2, b_2,\p_2 u_2,\p_1u)\|_{H^1}^{2}\right)\notag\\
&\quad+C\|\left(b_2,\p_2u_2\right)\|_{L^\infty}\left(\|\p_1^2u_{1}\|_{L^2}^{\frac{1}{2}}\|\p_1^3u_{1}\|_{L^2}^{\frac{1}{2}}\|\p_1^2b_{2}\|_{L^2}^{\frac{1}{2}}
\|\p_2\p_1^2b_2\|_{L^2}^{\frac{1}{2}}\right.\notag\\
&\qquad+\left.\|\p_1^2u_{\nu}\|_{L^2}^{\frac{1}{2}}\|\p_2\p_1^2u_{\nu}\|_{L^2}^{\frac{1}{2}}\|\p_\nu\p_1b_{2}\|_{L^2}^{\frac{1}{2}}
\|\p_\nu\p_1^2b_2\|_{L^2}^{\frac{1}{2}}\right)\|\p_1^3b_2\|_{L^2}^{\frac{1}{2}}\|\p_3\p_1^3b_2\|_{L^2}^{\frac{1}{2}}\notag\\
&\leq C\| (u, b )\|_{H^3}^2\left(\|\p_2u\|_{H^2}^2+\|\p_3u \|_{H^3}^2+\|\p_2b\|_{H^3}^2\right).\label{R1}
\end{align}

Analogously,
\begin{align}
R_2\leq C\| (u, b )\|_{H^3}^2\left(\|\p_2u\|_{H^2}^2+\|\p_3u \|_{H^3}^2+\|\p_2b\|_{H^3}^2+\|\p_3b\|_{H^3}^2\right),\label{R2}
\end{align}
and
\begin{align}
R_3
&\leq
C\left(\|(\p_1^3b_2,\p_1^3b_2) \|_{L^2}^{2}+\|(u,\na b_2,  \na\p_2u_2)\|_{H^1}^{2}\right)\notag\\
&\qquad\times\left(\|\p_2(\p_1^3b_2,\p_1^3b_2) \|_{L^2}^{2}+\|\p_3(u,\na b_2,  \na\p_2u_2)\|_{H^1}^{2}\right)\notag\\
&\leq C\|(u, b)\|_{H^3}^2\left(\|\p_2u\|_{H^2}^2+\|\p_2b\|_{H^3}^2+\|\p_3b\|_{H^3}^2\right).\label{R3}
\end{align}

It follows from (\ref{E11}), \eqref{E111}, \eqref{EF}, \eqref{R1}, \eqref{R2} and \eqref{R3},   we have
\begin{align}
E_{11}
&\leq C \left(\|(u, b)\|_{H^3}+\|(u, b)\|_{H^3}^2\right)\left(\|\p_2u\|_{H^2}^2+\|(\p_3u, \p_2b, \p_3b) \|_{H^3}^2 \right)+R_4,\label{E11g}
\end{align}
where
$$
R_4=-4 \int b\cdot\na\p_1^3 u_2\p_1^3b_2(b_2-\eta\p_2u_2) dx.
$$

For $E_{12}$, $E_{13}$  and $E_{15}$, an application of Lemmas \ref{anip1}  and  \ref{anip2} yields
\begin{align}
&E_{12}+E_{13}+E_{15}\notag\\
&\quad \leq C \|\p_1^3b_{2}\|_{L^2} \|\p_2\p_1^3b_{2}\|_{L^2}^{\frac{1}{2}}\|\p_3\p_1^3b_{2}\|_{L^2}^{\frac{1}{2}} \|\p_3^2b_{2}\|_{L^2}^{\frac{1}{2}}\|\p_1\p_3^2b_{2}\|_{L^2}^{\frac{1}{2}}\notag\\
&\qquad+C\|\p_1^3b_{2}\|_{L^2} \|\p_2\p_1^3b_{2}\|_{L^2}^{\frac{1}{2}}\|\p_3\p_1^3b_{2}\|_{L^2}^{\frac{1}{2}} \|\p_2\p_3^2u_{2}\|_{L^2}^{\frac{1}{2}}\|\p_1\p_2\p_3^2u_{2}\|_{L^2}^{\frac{1}{2}}\notag\\
&\qquad+C\left(\|\p_1^3b_2 \|_{L^2}^{2}+\|(u,\na b_2)\|_{H^1}^{2}\right)\left(\|\p_3\p_1^3b_2 \|_{L^2}^{2}+\|\p_2(u,\na b_2 )\|_{H^1}^{2}\right)\notag\\
&\qquad+C\left(\| \p_1^3b_2\|_{L^2}^{2}+\|(b,\na u_2)\|_{H^1}^{2}\right)\left(\|\p_3\p_1^3b_2 \|_{L^2}^{2}+\|\p_2(b,\na u_2 )\|_{H^1}^{2}\right)\notag\\
&\quad\leq C\left(\| (u, b )\|_{H^3}+\| (u, b )\|_{H^3}^2\right)\left(\|\p_2u\|_{H^2}^2+\| (\p_3u, \p_2b, \p_3b ) \|_{H^3}^2 \right). \label{E1235}
\end{align}

We proceed to estimate  $E_{14}$.  Similarly to the estimate of (\ref{D14}), we infer from  \eqref{22pg}  and Lemma \ref{anip2} that
\begin{align}
E_{14}
& \leq
C\left(\|\p_1^3b_2 \|_{L^2}^{2}+\|(\p_2 u, \na u_2)\|_{H^1}^{2}\right)\left(\|\p_3\p_1^3b_2\|_{L^2}^{2}+\|\p_2(\p_2u, \na u_2)\|_{H^1}^{2}\right)\notag\\
&\quad+C\|\p_1^3b_{2}\|_{L^2} \|\p_2\p_1^3b_{2}\|_{L^2}^{\frac{1}{2}}\|\p_3\p_1^3b_{2}\|_{L^2}^{\frac{1}{2}} \|\na \p_2u_{2}\|_{L^2}^{\frac{1}{2}}\|\p_1\na \p_2u_{2}\|_{L^2}^{\frac{1}{2}}\|u\|_{L^\infty}
\notag\\
&\quad+C\left(\|\p_1^3b_2\|_{L^2}^{2}+\|(\p_2b, \na b_2)\|_{H^1}^{2}\right)\left(\|\p_3\p_1^3b_2 \|_{L^2}^{2}+\|\p_2(\p_2b, \na b_2)\|_{H^1}^{2}\right)\notag\\
&\quad+C\left(\|\p_1^3b_2\|_{L^2}^{2}+\|(b, \na\p_2 b_2)\|_{H^1}^{2}\right)\left(\|\p_3\p_1^3b_2 \|_{L^2}^{2}+\|\p_2(b, \p_2\na b_2)\|_{H^1}^{2}\right)\notag\\
&\quad+C\|\p_1^3b_2 \|_{L^2}\|\p_1^3b_2\|_{{L_{x_1}^{2}L_{x_2}^{2}L_{x_3}^{\infty}}}\|\p_2^2p\|_{{L_{x_1}^{\infty}L_{x_2}^{\infty}L_{x_3}^{2}}}\notag\\
&\leq C\|(u,b)\|^2_{H^3}\left(\|\p_2u\|_{H^2}^2+\|(\p_3u,\p_2b,\p_3b)\|_{H^3}^2\right)\notag\\
&\quad+C\|\p_1^3b_2\|_{L^2}^{\frac32}\|\p_3\p_1^3b_2\|_{L^2}^{\frac12}\|(u,b)\|_{H^3}^{\frac12}\|(\p_2u,\p_2b)\|_{H^2}^{\frac{11}{8}}\|(\p_3u,\p_3b)\|_{H^3}^{\frac{1}{8}}\notag\\
&\leq C\|(u,b)\|^2_{H^3}\left(\|\p_2u\|_{H^2}^2+\|(\p_3u,\p_2b,\p_3b)\|_{H^3}^2\right).\label{E14}
\end{align}

Thus, plugging \eqref{E11g}, \eqref{E1235}, \eqref{E14} into \eqref{E1-1}, we have
\begin{align}
E_{1} &\leq 2\frac{d}{d t}\int\left(\p_1^3b_2\right)^2(b_2-\eta\p_2u_{2})\ dx \notag\\
&\quad +C\left(\| (u, b )\|_{H^3}+\| (u, b )\|_{H^3}^2\right)\left(\|\p_2u\|_{H^2}^2+\|(\p_3u,\p_2b,\p_3b)\|_{H^3}^2\right)+R_4,  \notag\
\end{align}
which, combined with \eqref{D1g}, yields
\begin{align}
D_1+E_{1}&\leq 2\frac{d}{d t}\int\left(\left(\p_1^3u_2\right)^2+\left(\p_1^3b_2\right)^2\right)(b_2-\eta\p_2u_{2})\ dx\notag\\
&\quad +C\left(\| (u, b )\|_{H^3}+\| (u, b )\|_{H^3}^2\right)\left(\|\p_2u\|_{H^2}^2+\|(\p_3u,\p_2b,\p_3b)\|_{H^3}^2\right),   \label{DE1g}
\end{align}
since it follows from Lemmas \ref{anip1} and \ref{anip2} that
\begin{align}
&R_4+F_4=-4\int b\cdot\na (\p_1^3 u_2\p_1^3b_2 )(b_2-\eta\p_2u_2) dx=4\int \p_1^3 u_2\p_1^3b_2  b\cdot\na(b_2-\eta\p_2u_2) dx \notag\\
&\quad\leq C\left(\|(\p_1^3 u_2, \p_1^3b_2)\|_{L^2}^2+\|(b,\na b_2)\|_{H^1}^2\right)
\left(\|\p_3(\p_1^3 u_2, \p_1^3b_2)\|_{L^2}^2+\|\p_2(b,\na b_2)\|_{H^1}^2\right)\notag\\
&\qquad +C\|\p_1^3 u_2\|_{L^2}^{\frac12}\|\p_3\p_1^3u_2\|_{L^2}^{\frac12}
\|\p_1^3 b_2\|_{L^2}^{\frac12}\|\p_2\p_1^3b_2\|_{L^2}^{\frac12}\|\p_2\nabla u_2\|_{L^2}^{\frac12}\|\p_1\p_2\nabla u_2\|_{L^2}^{\frac12}\|b\|_{L^\infty}\notag\\
&\quad\leq C\| (u, b)\|_{H^3}^2\left(\|\p_2u\|_{H^2}^2+\|\p_3u \|_{H^3}^2+\|\p_2b\|_{H^3}^2+\|\p_3b\|_{H^3}^2\right).\notag
\end{align}

In the exactly same way as the treatments of $D_1$ and $E_1$, we have
\begin{align}
D_3+E_{3}&\leq 3\frac{d}{d t}\int\left(\left(\p_1^3u_3\right)^2+\left(\p_1^3b_3\right)^2\right)(b_2-\eta\p_2u_{2})\ dx\notag\\
&\quad +C\left(\| (u, b )\|_{H^3}+\| (u, b )\|_{H^3}^2\right)\left(\|\p_2u\|_{H^2}^2+\|(\p_3u,\p_2b,\p_3b)\|_{H^3}^2\right).   \label{DE3g}
\end{align}

\vskip .1in
\textbf{Step III-2. The estimate of $D_2$}
\vskip .1in
Based upon $\eqref{m13}_2 $ and integration by parts, we see that
\begin{align}
D_2&=
-\frac{d}{d t}\int\partial_1^3u_2\partial_1^3 u_3(b_3-\eta\p_2u_{3})\ d x
+ \int\p_1^3u_2\p_1^3\p_tu_3(b_3-\eta\p_2u_{3})\ d x
\notag\\
&\quad+ \int\p_1^3\p_tu_2\p_1^3u_3(b_3-\eta\p_2u_{3})\ d x-\eta\int\p_3\p_1^3u_2\p_1^3u_3\p_3b_3\ d x
\notag\\
&\quad
-\eta\int\p_1^3u_2\p_3\p_1^3u_3\p_3b_3\ d x
+\mu\eta\int\p_3\p_1^3u_2\p_1^3u_3\p_3\p_2u_3\ d x\notag\\
&\quad
+\mu\eta\int\p_1^3u_2\p_3\p_1^3u_3\p_3\p_2u_3\ d x-\int\p_1^3u_2 \p_1^3u_3
u\cdot\na b_3\ d x\notag\\
&\quad+\int \p_1^3u_2 \p_1^3u_3b\cdot\na u_3
\ d x+\eta\int\p_1^3u_2 \p_1^3u_3\p_2\left(
u\cdot\na u_3-b\cdot\na b_3+\p_3p
\right) d x
\notag\\
&\triangleq-\frac{d}{d t}\int\partial_1^3u_2\partial_1^3 u_3(b_3-\eta\p_2u_{3})\ d x+\sum_{j=1}^{9}D_{2j}. \label{D2}
\end{align}

In view of \eqref{pb2} and \eqref{pb3}, one has
\begin{align}
D_{21}+D_{22}
=&\int\p_1^3u_2\p_1^3\p_tu_3(b_3-\eta\p_2u_{3})\ d x
 + \int\p_1^3\p_tu_2\p_1^3u_3(b_3-\eta\p_2u_{3})\ d x\notag\\
=&\int\p_1^3u_2\p_1^3\p_2b_3(b_3-\eta\p_2u_{3})\ d x
-\mu\int\p_3\p_1^3u_2
\p_3\p_1^3u_3(b_3-\eta\p_2u_3)\ d x
\notag\\&
-\mu\int\p_1^3u_2
\p_3\p_1^3u_3\p_3b_3\ d x
+\mu\eta\int\p_1^3u_2
\p_3\p_1^3u_3\p_3\p_2u_3\ d x
\notag\\
&-\int \p_1^3(u\cdot\na u_3-b\cdot\na b_3+\p_3p)\p_1^3u_2(b_3-\eta\p_2u_3)\ d x\notag\\
&
+\int\p_1^3\p_2b_2\p_1^3u_3(b_3-\eta\p_2u_{3})\ d x
-\mu\int
\p_3\p_1^3u_2\p_3\p_1^3u_3(b_3-\eta\p_2u_3)\ d x
\notag\\
&
-\mu\int\p_3\p_1^3u_2
 \p_1^3u_3\p_3 b_3\ d x
+\mu\eta\int\p_3\p_1^3u_2
\p_1^3u_3\p_3\p_2u_3\ d x
\notag\\
&-\int\p_1^3(u\cdot\na u_2-b\cdot\na b_2+\p_2p)\p_1^3u_3(b_3-\eta\p_2u_3)\ d x
\triangleq\sum_{j=1}^{10}L_{j}.\label{D2122}
\end{align}

Similarly to the estimate of \eqref{D111-4} and \eqref{D115}, we have
\begin{align}
\sum_{i=1}^4L_{i}+\sum_{i=6}^9L_{i}
\leq
C\|(u,b)\|_{H^3}\left(\|\p_2u\|_{H^2}^2+\|(\p_3u,\p_3b,\p_2b)\|_{H^3}^2\right).\label{L1234}
\end{align}
and
\begin{align}
L_{5}+L_{10}
\leq
C\|(u,b)\|^2_{H^3}\left(\|\p_2u\|_{H^2}^2+\|\left(\p_3u,\p_2b, \p_3b\right)\|_{H^3}^2 \right)+S_1+S_2.\label{L510}
\end{align}
where
\begin{align*}
S_1&=-\int u\cdot\na(\p_1^3u_3\p_1^3u_2) (b_3-\eta\p_2u_3) dx,\\
S_2&=\int \left(b\cdot\na\p_1^3 b_3\p_1^3 u_2+b\cdot\na\p_1^3 b_2\p_1^3 u_3\right)(b_3-\eta\p_2u_3) dx.
\end{align*}

Plugging \eqref{L1234} and \eqref{L510} into (\ref{D2122}) gives
\begin{align}
D_{21}+D_{22}&\leq
C\left(\|(u,b)\|_{H^3}+\|(u,b)\|^2_{H^3}\right)\notag\\
&\quad\times\left(\|\p_2u\|_{H^2}^2+\| (\p_3u,\p_2b, \p_3b )\|_{H^3}^2 \right)+S_1+S_2.\label{D211}
\end{align}

Analogously to the estimates of  $D_{12},D_{13},D_{16}$ in (\ref{D1236}) and $D_{14}$ in (\ref{D14}), we have
\begin{align}
&D_{23}+D_{24}+D_{25}+D_{26}+D_{28}\notag\\
&\leq C\left(\|(u,b)\|_{H^3}+\|(u,b)\|^2_{H^3}\right)\left(\|\p_2u\|_{H^2}^2+\|(\p_3u,\p_2b,\p_3b)\|_{H^3}^2\right),\label{D234569}
\end{align}
and
\begin{align}
D_{29}\leq C\|(u,b)\|^2_{H^3}\left(\|\p_2u\|_{H^2}^2+\|(\p_3u,\p_2b,\p_3b)\|_{H^3}^2\right)+\Xi_2.\label{D27}
\end{align}
where
$$
\Xi_2=\eta\int \p_1^3u_2 \p_1^3u_3
u\cdot \na \p_2u_3\ d x.
$$

Similarly to that in (\ref{DF3H}), it is easily checked that
\begin{align}
D_{27}+S_1+\Xi_2
=0.\label{DSH}
\end{align}

Thanks to \eqref{D211}, \eqref{D234569}, \eqref{D27} and \eqref{DSH}, we conclude from (\ref{D2}) that
\begin{align}\label{D2g}
D_{2}&\leq-\frac{d}{d t}\int\left(\p_1^3u_2\p_1^3u_3\right) (b_3-\eta\p_2u_{3})\ d x +S_2\notag\\
&\quad+\left(\|(u,b)\|_{H^3}+\|(u,b)\|^2_{H^3}\right)\left(\|\p_2u\|_{H^2}^2+\|(\p_3u,\p_2b,\p_3b)\|_{H^3}^2\right).
\end{align}

The treatment of $ S_2 $ strongly relies on the estimate of  $ E_2 $ in (\ref{H53}). Similarly to the derivation of  (\ref{E1-1}), we have by \eqref{m13}$ _{1} $ that
\begin{align}
E_2
&=
-\frac{d}{d t}\int\p_1^3b_2\p_1^3b_3(b_3-\eta\p_2u_{3})\ dx
+\int\p_1^3\p_tb_2\p_1^3b_3\left(b_3-\eta\p_2u_{3}\right) dx\notag\\
&\quad
+\int\p_1^3b_2\p_1^3\p_tb_3\left(b_3-\eta\p_2u_{3}\right)\ dx+\eta\int\p_1^3b_2\p_1^3b_3\p_3^2b_3\ dx\notag\\
&\quad
-\mu\eta\int\p_1^3b_2\p_1^3b_3\p_2\p_3^2u_3\ dx-\int\p_1^3b_2\p_1^3b_3\left(u\cdot\na b_3-b\cdot\na u_3\right)\ dx\notag\\
&\quad+\eta\int\p_1^3b_2\p_1^3b_3\p_2\left(
u\cdot\na u_3-b\cdot\na b_3+\p_3p
\right) dx\notag\\
&\triangleq -\frac{d}{d t}\int\p_1^3b_2\p_1^3b_3(b_3-\eta\p_2u_{3})\ dx+\sum_{i=1}^{6}E_{2i}.\label{E2-1}
\end{align}

By \eqref{pu2} and \eqref{pu3}, we have
\begin{align}
E_{21}+E_{22}
&= \int\p_1^3\p_tb_2\p_1^3b_3\left(b_3-\eta\p_2u_{3}\right)\ dx+\int\p_1^3b_2\p_1^3\p_tb_3\left(b_3-\eta\p_2u_{3}\right)  dx
\notag\\
&= \int\p_1^3\left(\p_2u_2+\eta\p_2^2b_2+\eta\p_3^2b_2\right) \p_1^3b_3\left(b_3-\eta\p_2u_{3}\right)\ dx\notag\\
&\quad+ \int\p_1^3\left(\p_2u_3+\eta\p_2^2b_3+\eta\p_3^2b_3\right)\p_1^3b_2 \left(b_3-\eta\p_2u_{3}\right)  dx
\notag\\
&\quad-\int\left[\p_1^3\left(u\cdot\na b_2\right) \p_1^3b_3+\p_1^3\left(u\cdot\na b_3\right) \p_1^3b_2\right]\left(b_3-\eta\p_2u_{3}\right)  dx\notag\\
&\quad
+\int\left[\p_1^3\left(b\cdot\na u_2\right) \p_1^3b_3+\p_1^3\left(b\cdot\na u_3\right) \p_1^3b_2\right]\left(b_3-\eta\p_2u_{3}\right)  dx\notag\\
&\triangleq \sum_{i=1}^{4}E_{21i},\label{E2122}
\end{align}

Analogously to the estimate of $E_{111}$ in (\ref{E111}), by Lemma \ref{anip1} we obtain after  integrating by parts and
\begin{align}\label{E211}
E_{211}+E_{212}
\leq& C\|(u,b)\|_{H^3}\left(\|\p_2u\|_{H^2}^2+\|(\p_3u,\p_2b,\p_3b)\|_{H^3}^2\right).
\end{align}

Note that
\begin{align}
E_{213}+E_{214}
&=-\int(\p_1^3u\cdot\na b_2
+3\p_1^2u\cdot\na\p_1 b_2
+3\p_1u\cdot\na\p_1^2 b_2
)\p_1^3b_3(b_3-\eta\p_2u_3)\ dx
\notag\\
&\quad-\int(\p_1^3u\cdot\na b_3
+3\p_1^2u\cdot\na\p_1 b_3
+3\p_1u\cdot\na\p_1^2 b_3
)\p_1^3b_2(b_3-\eta\p_2u_3)\ dx
\notag\\
&\quad
+\int(\p_1^3b\cdot\na u_2
+3\p_1^2b\cdot\na\p_1 u_2
+3\p_1b\cdot\na\p_1^2 u_2
)\p_1^3b_3(b_3-\eta\p_2u_3)\ dx
\notag\\
&\quad
+\int(\p_1^3b\cdot\na u_3
+3\p_1^2b\cdot\na\p_1 u_3
+3\p_1b\cdot\na\p_1^2 u_3
)\p_1^3b_2(b_3-\eta\p_2u_3)\ dx
\notag\\
&\quad
-\int\left( u\cdot\na\p_1^3b_2\p_1^3b_3+u\cdot\na\p_1^3b_3\p_1^3b_2\right)(b_3-\eta\p_2u_3) dx  \notag\\
&\quad+\int \left(b\cdot\na\p_1^3 u_2\p_1^3 b_3+b\cdot\na\p_1^3 u_3\p_1^3 b_2\right)(b_3-\eta\p_2u_3) dx\notag\\
&\triangleq\sum_{i=1}^{6}M_{i}.\label{EF234}
\end{align}

In a similar manner as that used in the derivation of (\ref{R1}), we find
\begin{align}
\sum_{i=1}^4M_i
&\leq C\|\left(u, b\right)\|_{H^3}^2\left(\|\p_2u\|_{H^2}^2+\|(\p_3u,\p_2b,\p_3b)\|_{H^3}^2\right).\label{A1234}
\end{align}

Integrating by parts and using  the divergence-free condition $\na \cdot u=0$, we infer from Lemma \ref{anip2} that
\begin{align}
M_5&= \int \p_1^3b_2 \p_1^3b_3 u\cdot\na(b_3-\eta\p_2u_3)\ dx \notag\\
&\leq
C\left(\|(\p_1^3b_2,\p_1^3b_3)\|_{L^2}^{2}+\|(u,\na b_3,  \na\p_2u_3)\|_{H^1}^{2}\right)\notag\\
&\qquad\times\left(\|\p_2(\p_1^3b_2,\p_1^3b_3)\|_{L^2}^{2}+\|\p_3(u,\na b_3,  \na\p_2u_3)\|_{H^1}^{2}\right)\notag\\
&\leq C\| (u, b )\|_{H^3}^2\left(\|\p_2u\|_{H^2}^2+\|(\p_3u, \p_2b,\p_3b)\|_{H^3}^2\right).\label{A5}
\end{align}

Based on \eqref{E211}, \eqref{EF234}, \eqref{A1234} and \eqref{A5},  we know from (\ref{E2122}) that
\begin{align}
E_{21}+E_{22}
&\leq C(\|(u, b)\|_{H^3}+\|(u, b)\|_{H^3}^2)\notag\\
&\quad\times(\|\p_2u\|_{H^2}^2+\|(\p_3u, \p_2b, \p_3b) \|_{H^3}^2)+M_6,\label{E2122g}
\end{align}
where
$$
M_6=\int \left(b\cdot\na\p_1^3 u_2\p_1^3 b_3+b\cdot\na\p_1^3 u_3\p_1^3 b_2\right)(b_3-\eta\p_2u_3)\ dx.
$$

Analogously to the derivation of (\ref{E1235}), we have by Lemmas \ref{anip1}  and  \ref{anip2} that
\begin{align}
E_{23}+E_{24}+E_{25}
&\leq C\left(\| (u, b )\|_{H^3}+\| (u, b )\|_{H^3}^2\right)\notag\\
&\quad\times\left(\|\p_2u\|_{H^2}^2+\| (\p_3u, \p_2b, \p_3b ) \|_{H^3}^2 \right),\label{E2346}
\end{align}

It remains to  bound $E_{26}$.  Indeed, it follows from (\ref{p}) and the $L^r$-estimates ($1<r<\infty$) of Riesz operator  that
\begin{align*}
\|\p_2\p_3p\|_{{L_{x_1,x_2}^{\infty} L_{x_3}^{2}}}
& \leq C\|\p_2\p_3p\|_{L^{2}}^{\frac14}\|\p_2^2\p_3p\|_{L^{2}}^{\frac14}\|\p_1\p_2\p_3p\|_{L^{2}}^{\frac14}\|\p_1\p_2^2\p_3p\|_{L^{2}}^{\frac14}\notag\\
& \leq C\|\p_2(u\cdot \na u- b\cdot\na b)\|_{L^{2}}^{\frac14}\|\p_2(\p_iu_j\p_ju_i-\p_ib_j\p_jb_i)\|_{L^{2}}^{\frac12}\notag\\
&\quad\times\|\p_2^2(\p_iu_j\p_ju_i-\p_ib_j\p_jb_i)\|_{L^{2}}^{\frac14},
\end{align*}
so that, by \eqref{22pg}  we see that
$$
\|\p_2\p_3p\|_{{L_{x_1,x_2}^{\infty} L_{x_3}^{2}}}\leq C\|(u,b)\|_{H^3}^{\frac12}\|(\p_2u,\p_2b)\|_{H^2}^{\frac{11}{8}}\|(\p_3u,\p_3b)\|_{H^3}^{\frac{1}{8}}.
$$
Hence, similarly to the estimate of (\ref{E14}), we find
\begin{align}
E_{26}
&\leq C\|(u,b)\|^2_{H^3}\left(\|\p_2u\|_{H^2}^2+\|(\p_3u,\p_2b,\p_3b)\|_{H^3}^2\right).\label{E25}
\end{align}

Now, putting \eqref{E2122g}, \eqref{E2346}, \eqref{E25} into \eqref{E2-1} gives
\begin{align}
E_{2}& \leq-\frac{d}{d t}\int \p_1^3b_2\p_1^3b_3 (b_3-\eta\p_2u_{3})\ dx +M_6\notag\\
&\quad + C\left(\| (u, b )\|_{H^3}+\| (u, b )\|_{H^3}^2\right)\left(\|\p_2u\|_{H^2}^2+\|(\p_3u,\p_2b,\p_3b)\|_{H^3}^2\right),  \notag\
\end{align}
which, combining with \eqref{D2g}, yields
\begin{align}
D_2+E_{2}&\leq-\frac{d}{d t}\int \left( \p_1^3u_2\p_1^3u_3 +\p_1^3b_2\p_1^3b_3\right) (b_3-\eta\p_2u_{3})\ dx\notag\\
&\quad +C\left(\| (u, b )\|_{H^3}+\| (u, b )\|_{H^3}^2\right)\left(\|\p_2u\|_{H^2}^2+\|(\p_3u,\p_2b,\p_3b)\|_{H^3}^2\right),   \label{DE2g}
\end{align}
since it is easily deduced in a similar manner as the treatment of $R_4+F_4$ that
\begin{align}
M_6+S_2&=\int b\cdot\na\left(\p_1^3 u_2\p_1^3b_3+\p_1^3u_3\p_1^3 b_2\right)(b_3-\eta\p_2u_3) dx\notag\\
&=- \int\left(\p_1^3 u_2\p_1^3b_3+\p_1^3u_3\p_1^3 b_2\right) b\cdot\na(b_3-\eta\p_2u_3) dx \notag\\
&\leq C\|(u, b)\|_{H^3}^2\left(\|\p_2u\|_{H^2}^2+\|\p_3u \|_{H^3}^2+\|\p_2b\|_{H^3}^2+\|\p_3b\|_{H^3}^2\right).\notag
\end{align}

\vskip .1in
\textbf{Step IV. The estimate of $\mathcal{E}_1(t)$}
\vskip .1in

Now, inserting  \eqref{DE1g}, \eqref{DE3g} and \eqref{DE2g}  into \eqref{HU3}, we conclude that
\begin{align}
\frac12&\frac{d}{d t}\|(\nabla^3 u, \nabla^3 b)\|^2_{L^2}+\mu\|\p_3\nabla^3 u \|^2_{L^2}+\eta\left(\|\p_2\nabla^3 b\|^2_{L^2}+\|\p_3\nabla^3 b\|^2_{L^2}\right)\notag\\
& \leq 2\frac{d}{d t}\int\left(\left(\p_1^3u_2\right)^2+\left(\p_1^3b_2\right)^2\right)(b_2-\eta\p_2u_{2})\ dx\notag\\
&\quad -\frac{d}{d t}\int \left( \p_1^3u_2\p_1^3u_3+ \p_1^3b_2\p_1^3b_3\right)  (b_3-\eta\p_2u_{3})\ dx\notag\\
&\quad +3\frac{d}{d t}\int\left(\left(\p_1^3u_3\right)^2+\left(\p_1^3b_3\right)^2\right)(b_2-\eta\p_2u_{2})\ dx\notag\\
&\quad +C\left(\|(u, b)\|_{H^3}+\|(u, b)\|_{H^3}^2\right)\left(\|\p_2u\|_{H^2}^2+\|(\p_3u,\p_2b,\p_3b)\|_{H^3}^2\right),   \nonumber
\end{align}
which, added to (\ref{L2bound}) and integrated over $(0,t)$, we obtain
\begin{align*}
\mathcal{E}_{1}(t)
&\leq \left(\mathcal{E}_1(0)+\mathcal{E}_1^{\frac32}(0)\right)+C\mathcal{E}_{1}^{\frac32}(t)
+C\left(\mathcal{E}_1^{\frac12}(t)+\mathcal{E}_1(t)\right)\left(\mathcal{E}_2(t)+\mathcal{E}_1(t)\right), 
\end{align*}
which, combined with Cauchy-Schwarz's inequality, gives rise to (\ref{EE1}).
The proof of Lemma \ref{lemma3.1} is therefore complete.
\end{proof}

\vskip .1in

\subsection{ A Priori Estimates (II)}

In the subsection,  we study the dissipation generated by the background magnetic field and prove the estimate of  $\mathcal{E}_{2}(t)$.

\begin{lem}\label{lemma3.2} For any $t\geq0$, it holds that
	\begin{align}
	\D \mathcal{E}_2(t)
	\leq C\mathcal{E}_1(0)+C\mathcal{E}_1(t)+C\mathcal{E}_1^{\frac32}(t)+C\mathcal{E}_2^{\frac32}(t). \label{EE2}
	\end{align}
\end{lem}

\begin{proof}
 It suffices to establish the estimates of the following two items:
$$\int_0^t  \|\p_2u(\tau)\|_{L^2}^2 d \tau \quad  \rm {and} \quad \int_0^t\|\na^2\p_2u(\tau)\|_{L^2}^2  d \tau,$$
which will be achieved by using  the special structure of    equation  $\eqref{PMHD}_2$,
\begin{align}
\partial_{2}u=\partial_t b+u\cdot
\nabla b-\eta\p_2^2 b-\eta\p_3^2b-b\cdot \nabla u. \label{ue}
\end{align}

\vskip .1in
\textbf{Step I. The estimate of $\|\p_2u(t)\|_{L^2}$}
\vskip .1in

Multiplying \eqref{ue} by $\partial_2 u$ in $L^2$ and  integrating by parts, we get
\begin{align}
\|\partial_{2}u\|^2_{L^2}&= \int \p_2 u\cdot \p_tb \ dx+\int  u\cdot
\nabla b \cdot \p_2 u\ dx\notag\\
&\quad -\eta\int\left(\p_2^2 b+\p_3^2b\right)\cdot\p_2u \ dx-\int b\cdot \nabla u \cdot \p_2 u\ dx\notag\\
&\triangleq J_1+J_2+J_3+J_4.\label{E21}
\end{align}

In terms of $\eqref{PMHD}_1$ and the fact that $\nabla\cdot b=0$, we deduce after integrating by parts that
\begin{align}
J_1&=\frac{d}{d t}\int  \p_2 u\cdot b \ dx+\int \p_2 b\cdot
\left(\partial_{2}b+\mu\p_3^2u+b\cdot \nabla b-u\cdot
\nabla u\right)dx \nonumber\\
&\leq \frac{d}{d t}\int  \p_2 u\cdot b \ dx+C \| \p_2 b\|^2_{H^3}+C\|\p_3u\|^2_{H^3}\notag\\
&\quad+ C\|\p_2 b\|^{\frac 12}_{L^2}\|\p_1\p_2b\|^{\frac 12}_{L^2}\|b\|^{\frac 12}_{L^2}\|\p_2 b\|^{\frac 12}_{L^2}\|\nabla b\|^{\frac 12}_{L^2}\|\p_3\nabla b\|^{\frac 12}_{L^2}\nonumber\\
&\quad+C \|\p_2 b\|^{\frac 12}_{L^2}\|\p_1\p_2b\|^{\frac 12}_{L^2}\|u\|^{\frac 12}_{L^2}\|\p_2 u\|^{\frac 12}_{L^2}\|\nabla u\|^{\frac 12}_{L^2}\|\p_3\nabla u\|^{\frac 12}_{L^2}\nonumber\\
&\leq \frac{d}{d t}\int  \p_2 u\cdot b \ dx +C\|(\p_3u,\p_2b )\|^2_{H^3}\notag\\
&\quad+ C\| (u, b )\|_{H^3}\left(\|\p_2 u\|^2_{H^2}+\|(\p_3u,\p_2b,\p_3b)\|^2_{H^3}\right),\label{J1}
\end{align}
and similarly,
\begin{align}
J_{2}+J_{4}
&\leq C \|\p_2 u\|^{\frac 12}_{L^2}\|\p_1\p_2u\|^{\frac 12}_{L^2}\|u\|^{\frac 12}_{L^2}\|\p_2 u\|^{\frac 12}_{L^2}\|\nabla b\|^{\frac 12}_{L^2}\|\p_3\nabla b\|^{\frac 12}_{L^2}\nonumber\\
&\quad+C \|\p_2 u\|^{\frac 12}_{L^2}\|\p_1\p_2u\|^{\frac 12}_{L^2}\|b\|^{\frac 12}_{L^2}\|\p_2 b\|^{\frac 12}_{L^2}\|\nabla u\|^{\frac 12}_{L^2}\|\p_3\nabla u\|^{\frac 12}_{L^2}\nonumber\\
&\leq  C\| (u, b )\|_{H^3}\left(\|\p_2 u\|^2_{H^2}+\|(\p_3u,\p_2b,\p_3b)\|^2_{H^3}\right).\label{J24}
\end{align}

Finally, it is easily seen that
\begin{align*}
J_{3} \leq \frac12\|\partial_{2} u \|^2_{L^2}+C\|\left(\p_2b,\p_3b\right)\|_{H^3}^2,
\end{align*}
from which, \eqref{J1}, \eqref{J24} and  (\ref{E21}), we know
\begin{align}
\|\p_2u\|_{L^2}^2&\leq2 \frac{d}{d t}\int  \p_2 u\cdot b \ dx+C\|(\p_3u,\p_2b,\p_3b)\|^2_{H^3}\notag\\
&\quad+C\| (u, b )\|_{H^3}\left(\|\p_2 u\|^2_{H^2}+\|(\p_3u,\p_2b,\p_3b)\|^2_{H^3}\right).\label{p2u}
\end{align}

\vskip .1in
\textbf{Step II. The estimate of $\|\nabla^2\p_2u(t)\|_{L^2}$}
\vskip .1in

In order to estimate $\|\nabla^2\p_2u\|_{L^2}$, we apply $\nabla^2$ to \eqref{ue} and multiply it by $ \nabla^2 \partial_2u$ in $L^2$ to get that
\begin{align}
\|\nabla^2 \p_2 u\|^2_{L^2}&= \int \nabla^2\p_2 u\cdot \p_t \nabla^2 b \ dx+\int \nabla^2(u\cdot
\nabla b ) \cdot \nabla^2 \p_2 u\ dx\notag\\
&\quad -\eta\int \nabla^2\left(\p_2^2 b+\p_3^2b\right) \cdot \nabla^2\p_2 u \ dx-\int\nabla^2 (b\cdot \nabla u) \cdot \nabla^2\p_2 u \ dx\notag\\
&\triangleq W_1+W_2+W_3+W_4. \label{E22}
\end{align}

Owing to $\eqref{PMHD}_1$ and $\nabla\cdot b=0$, we see that (\ref{E22})
\begin{align}
W_1&=\frac{d}{d t}\int  \nabla^2\p_2 u\cdot \nabla^2 b \ dx+\int  \nabla^2 \p_2 b\cdot
\nabla^2\left(\p_2b+\mu\p_3^2u \right) dx\notag\\
&\quad+\int  \nabla^2 \p_2 b\cdot
\nabla^2 \left( b\cdot \nabla b-u\cdot
\nabla u\right) dx\notag\\
&\triangleq\frac{d}{d t}\int  \nabla^2\p_2 u\cdot \nabla^2 b \ dx+ W_{11}+W_{12},\label{W1}
\end{align}
where the second term on the right-hand side can be easily bounded by
\begin{align}
W_{11}\leq C \| \p_2 b\|^2_{H^3}+C\|\p_3  u\|_{H^3}^2.\label{W11}
\end{align}

Integrating by parts and Lemma \ref{anip1}, we have
\begin{align*}
W_{12}
&=\int \nabla^2\p_2  b\cdot
\left(\nabla^2 b_i \cdot \p_i b+2\nabla b_i \cdot \p_i\nabla b+b_i\p_i\nabla^2 b\right) dx\notag\\
&\quad-\int \nabla^2\p_2 b\cdot
\left(\nabla^2 u_i \cdot \p_i u+2\nabla u_i \cdot \p_i\nabla u+u_i\p_i\nabla^2 u\right) dx\notag\\
&\leq C\|\p_2 \nabla^2b\|^{\frac 12}_{L^2}\|\p_1\p_2\nabla^2b\|^{\frac 12}_{L^2}\|\na b\|^{\frac 12}_{L^2}\|\p_2\na b\|^{\frac 12}_{L^2}\|\nabla^2 b\|^{\frac 12}_{L^2}\|\p_3\nabla^2 b\|^{\frac 12}_{L^2}\notag\\
&\quad+C\|\p_2\nabla^2b\|^{\frac 12}_{L^2}\|\p_1\p_2\nabla^2b\|^{\frac 12}_{L^2}\|b\|^{\frac 12}_{L^2}\|\p_2 b\|^{\frac 12}_{L^2}\|\nabla^3 b\|^{\frac 12}_{L^2}\|\p_3\nabla^3 b\|^{\frac 12}_{L^2}\\
&\quad+ C\|\p_2 \nabla^2b\|^{\frac 12}_{L^2}\|\p_1\p_2\nabla^2b\|^{\frac 12}_{L^2}\|\na u\|^{\frac 12}_{L^2}\|\p_2\na u\|^{\frac 12}_{L^2}\|\nabla^2 u\|^{\frac 12}_{L^2}\|\p_3\nabla^2 u\|^{\frac 12}_{L^2}\\
&\quad+C\|\p_2\nabla^2b\|^{\frac 12}_{L^2}\|\p_1\p_2\nabla^2b\|^{\frac 12}_{L^2}\|u\|^{\frac 12}_{L^2}\|\p_2 u\|^{\frac 12}_{L^2}\|\nabla^3 u\|^{\frac 12}_{L^2}\|\p_3\nabla^3 u\|^{\frac 12}_{L^2}\\
&\leq C\|(u, b)\|_{H^3}\left(\|\p_2 u\|^2_{H^2}+\|(\p_3u,\p_2b,\p_3b)\|^2_{H^3}\right),
\end{align*}
which, together with (\ref{W11}) and (\ref{W1}), implies that
\begin{align}
W_1&\leq \frac{d}{d t}\int  \nabla^2\p_2 u\cdot \nabla^2 b \ dx+ C \| \p_2 b\|^2_{H^3}+C\|\p_3  u\|_{H^3}^2\notag
\\
&\quad+
C\|(u, b)\|_{H^3}\left(\|\p_2 u\|^2_{H^2}+\|(\p_3u,\p_2b,\p_3b)\|^2_{H^3}\right).\label{W1-1}
\end{align}

Analogously, we deduce from  Lemma \ref{anip1} that
\begin{align}
W_2+W_4
&=\int \nabla^2\p_2 u\cdot
\left(\nabla^2 u_i \cdot \p_i b+2\nabla u_i \cdot \p_i\nabla b +u_i\p_i\nabla^2 b\right) dx\notag\\
&\quad-\int\nabla^2\p_2u\cdot
\left(\nabla^2 b_i \cdot \p_i u+2\nabla b_i \cdot \p_i\nabla u+ b_i\p_i\nabla^2 u\right) dx\notag\\
&\leq C\|\p_2 \nabla^2u\|_{L^2}\|\na b\|_{H^1}^{\frac12} \|\partial_{2}\na b\|_{H^1}^{\frac12}\|\na^2u\|_{L^2}^{\frac12}\|\p_3\na^2u\|_{L^2}^{\frac12}\notag\\
&\quad+C\|\p_2 \nabla^2u\|_{L^2}\|\na u\|_{H^1}^{\frac12}\|\p_2 \na u\|_{H^1}^{\frac12}\|\na^2 b\|_{L^2}^{\frac12}\|\partial_{3}\na^2 b\|_{L^2}^{\frac12}\notag\\
&\quad+ C\|\p_2 \nabla^2u\|_{L^2}\|u\|_{H^1}^{\frac12}\|\p_2 u\|_{H^1}^{\frac12}\|\na^3 b\|_{L^2}^{\frac12}\|\partial_{3}\na^3 b\|_{L^2}^{\frac12}\notag\\
&\quad+C\|\p_2 \nabla^2u\|_{L^2}\|b\|_{H^1}^{\frac12}\|\p_2 b\|_{H^1}^{\frac12}\|\na^3 u\|_{L^2}^{\frac12}\|\partial_{3}\na^3 u\|_{L^2}^{\frac12}\notag\\
&\leq C\|(u, b)\|_{H^3}\left(\|\p_2 u\|^2_{H^2}+\|(\p_3u,\p_2b,\p_3b)\|^2_{H^3}\right),\label{W24}
\end{align}
and finally,
\begin{align}
W_{3}\leq \frac12\|\partial_{2}\na^2 u \|^2_{L^2}+C\|\left(\p_2b,\p_3b\right)\|_{H^3}^2.\label{W3}
\end{align}

Hence, inserting  (\ref{W1-1}), (\ref{W24}) and (\ref{W3}) into (\ref{E22}), we arrive at
\begin{align}
\|\p_2\nabla^2u\|_{L^2}&\leq2\frac{d}{d t}\int  \nabla^2\p_2 u\cdot \nabla^2 b \ dx+C\|\left(\p_2b,\p_3b, \p_3u\right)\|_{H^3}^2\notag\\
&\quad+C\| (u, b )\|_{H^3}\left(\|\p_2 u\|^2_{H^2}+\|(\p_3u,\p_2b,\p_3b)\|^2_{H^3}\right).\label{p2nu}
\end{align}

\vskip .1in
\textbf{Step III. The estimate of $\mathcal{E}_2(t)$}
\vskip .1in
Now,  adding up  \eqref{p2u} and \eqref{p2nu}, we obtain after  integrating it over $(0,t)$ that
\begin{align*}
\int_0^t\| \p_2 u\|^2_{H^2}d\tau
&\leq C\mathcal{E}_1(0)+C\mathcal{E}_1(t)+C\mathcal{E}_1^{\frac12}(t)\left(\mathcal{E}_1(t)+\mathcal{E}_2(t)\right)
\end{align*}
 which, together with Cauchy-Schwarz's inequality, leads to the desired estimate  \eqref{EE2} immediately.
 \end{proof}

\subsection{ Proof of Theorem \ref{thm1.1} } With Lemmas \ref{lemma3.1} and \ref{lemma3.2} at hand, we can now
prove the global stability result stated in Theorem \ref{thm1.1} by applying the bootstrapping argument (see, e.g. \cite{Tao2006}).

\begin{proof}[Proof of Theorem \ref{thm1.1}]
	Since the local well-posedness of smooth solutions to the problem \eqref{PMHD} can be established
	by the standard approach (see, e.g. \cite{MaBe}), it suffices to  to establish the global bounds  of the solutions.

Indeed, since it holds that $\mathcal{E}(0)=\mathcal{E}_1(0)=\|(u_0,b_0)\|^2_{H^3}$ for $\mathcal{E}(t)=\mathcal{E}_1(t)+\mathcal{E}_2(t)$,  we easily deduce   from  \eqref{EE1} and   \eqref{EE2} that
	\begin{align}
	\mathcal{E}(t)\leq C_1\left(\mathcal{E}(0)
	 +\mathcal{E}(0)^{\frac32}\right) +C_2 \mathcal{E}(t)^{\frac32}
	+C_3 \mathcal{E}(t)^2.\label{bsa}
	\end{align}
Then, based upon  \eqref{bsa}, an application of the bootstrap  argument leads to the  stability result established in Theorem \ref{thm1.1}, provided the initial data $ \|(u_0,b_0)\|^2_{H^3}$ is chosen to be sufficiently small such that
	\begin{align}
	C_1\left(\mathcal{E}(0)
	+\mathcal{E}(0)^{\frac32}\right)\leq \frac14\min \left\{\frac1{16 C^2_2},\   \frac{1}{4C_3} \right\}.\label{iee}
	\end{align}
	In fact, if we make the ansatz that for $0<t\le \infty$,
	\begin{equation}\label{ee}
	\mathcal{E}(t)\le \min\left\{\frac1{16 C^2_2},\   \frac{1}{4C_3} \right\},
	\end{equation}
	then (\ref{bsa}) implies
	\begin{align}
	\mathcal{E}(t)&\leq C_1\left(\mathcal{E}(0)
	+\mathcal{E}(0)^{\frac32}\right)+\frac12\mathcal{E}(t), \notag
	\end{align}
    which, together with (\ref{iee}), yields
	\begin{align*}
	\mathcal{E}(t)&\leq 2C_1\left(\mathcal{E}(0)
	+\mathcal{E}(0)^{\frac32}\right)\leq  \frac12\min\left\{\frac1{16 C^2_2},\   \frac{1}{4C_3} \right\}.
	\end{align*}
	The bootstrapping argument then asserts that  (\ref{ee})  actually holds for all time. The proof of Theorem \ref{thm1.1} is therefore complete.
\end{proof}

\section{Decay Rates and Proof of Theorem \ref{thm1.2} }\label{Sec.3}
\subsection{Estimates of nonlinear terms.}

This section aims to  prove the decay rates of the global solutions, based on  Proposition \ref{lem2.1}. To do this, we first deal with the nonlinear term $N_1$,
\begin{align}
\widehat{N_1}=\widehat{\mathbb P\left(b\cdot\nabla b-u\cdot\nabla u\right)}=(\widehat{N_{11}},\widehat{N_{12}},\widehat{N_{13}}),\label{N1p}
\end{align}

In the following, we only consider the estimate of
$\widehat{\mathbb P\left(b\cdot\nabla b\right)}$, since the other term $\widehat{\mathbb P\left(u\cdot\nabla u\right)}$ can be treated in the same way. In fact,
\begin{align*}
\widehat{\mathbb P\left(b\cdot\nabla b\right)}&=\sum_{k=1}^{3}i\xi_k\widehat{b_kb}+\sum_{k,l=1}^{3}i\xi |\xi|^{-2}i\xi_ki\xi_l\widehat{b_kb_l}\\
&=\left(\widehat{\mathbb P_1\left(b\cdot\nabla b\right)}, \widehat{\mathbb P_2\left(b\cdot\nabla b\right)}, \widehat{\mathbb P_3\left(b\cdot\nabla b\right)}\right),
\end{align*}
where
\begin{align}
\widehat{\mathbb P_j\left(b\cdot\nabla b\right)}&\triangleq \sum_{k=1}^{3}i\xi_k\widehat{b_kb_j}+\sum_{k,l=1}^{3}i\xi_j |\xi|^{-2}i\xi_ki\xi_l\widehat{b_kb_l},\notag\\
&=\widehat{\left(b\cdot\nabla b_j\right)}-\widehat{\p_j\Delta^{-1}\na\cdot\left(b\cdot\na b\right)},
\quad j=1,2,3. \label{P}
\end{align}

Next, we present a more concise expression of $\widehat{\mathbb P_1\left(b\cdot\nabla b\right)}$.
Let $\xi_\nu=(\xi_2,\xi_3)$ and $|\xi|^2=\xi_1^2+|\xi_\nu|^2$. Then, by direct calculations we have
\begin{align}
\widehat{\mathbb P_1\left(b\cdot\nabla b\right)}&=\sum_{k=1}^{3}i\xi_k\widehat{b_kb_1}+\sum_{k,l=1}^{3}i\xi_1|\xi|^{-2}i\xi_ki\xi_l\widehat{b_kb_l} \nonumber\\
&=\sum_{k=1}^{3}i\xi_k\widehat{b_kb_1}-\sum_{k,l=1}^{3}i\xi_1|\xi|^{-2}\xi_k\xi_l\widehat{b_kb_l}\nonumber\\
&=\sum_{k=1}^{3}i\xi_k\widehat{b_kb_1}-\sum_{k=1}^{3}i\xi_1^2|\xi|^{-2}\xi_k\widehat{b_kb_1}-\sum_{k=1}^{3}\sum_{l=2}^{3}i\xi_1|\xi|^{-2}\xi_k\xi_l\widehat{b_kb_l} \nonumber\\
&=\sum_{k=1}^{3}i\xi_{\nu}^2|\xi|^{-2}\xi_k\widehat{b_kb_1}-\sum_{k=1}^{3}\sum_{l=2}^{3}i\xi_1|\xi|^{-2}\xi_k\xi_l\widehat{b_kb_l}.\label{P1}
\end{align}

The proof of decay rates is based on the  bootstrapping argument and the following ansatz with $k=0,1$, $i=2,3$ and $l=1,2$:
\begin{align}
& \|\p_1^k(u,b)(t)\|_{L^2}\leq C_0\varepsilon(1+t)^{-\frac{1}{2}},  \quad\quad\,  \|\p_{i}(u, b)(t)\|_{L^2}\leq C_0\varepsilon(1+t)^{-1},  \nonumber\\
& \|(u_1,b_1)(t)\|_{L^2}\leq C_0\varepsilon(1+t)^{-\frac{3}{4}},  \quad\quad\;\,   \|\p_{i}(u_1, b_1)(t)\|_{L^2}\leq C_0\varepsilon(1+t)^{-\frac{5}{4}},  \nonumber\\
& \|\p_{i}\p_3(u,b)(t)\|_{L^2}\leq C_0\varepsilon(1+t)^{-\frac32},\quad\;\,\,
\|\p_l\p_2(u,b)(t)\|_{L^2}\leq C_0\varepsilon(1+t)^{-\frac{11}{12}},  \nonumber\\
& \|\p_{1}\p_3(u,b)(t)\|_{L^2}\leq C_0\varepsilon(1+t)^{-1}, \quad\;\;\, \|\p_1^2(u,b)\|_{L^2}\leq C_0\varepsilon(1+t)^{-\frac{3}{8}}, \label{u10} \\
& \|\p_{i}\p_3(u_1, b_1)(t)\|_{L^2}\leq C_0\varepsilon(1+t)^{-\frac{7}{4}},\quad  \| \p_3^3(u_1, b_1)(t)\|_{L^2}\leq C_0\varepsilon(1+t)^{-\frac{9}{4}},\nonumber\\
& \|\p_l\p_2\p_3(u,b)(t)\|_{L^2}\leq C_0\varepsilon(1+t)^{-\frac{11}{12}},\;\,\, \|\p_1^2\p_3(u, b)(t)\|_{L^2}\leq C_0\varepsilon(1+t)^{-\frac12}, \nonumber
\\
&\|\p_1\p_3^2(u, b)(t)\|_{L^2}\leq C_0\varepsilon(1+t)^{-\frac54},\quad \;\; \|\p_2\p_3^2(u, b)(t)\|_{L^2}\leq C_0\varepsilon(1+t)^{-\frac{11}{6}},\nonumber\\
& \| \p_3^3(u, b)(t)\|_{L^2}\leq C_0\varepsilon(1+t)^{-\frac{23}{12}},\nonumber
\end{align}
where $C_0$ is an absolutely positive constant to be determined later.

By exploiting the integral expression of the solutions and the upper bound of the kernel function, we can show that the half upper bounds of  the solutions in \eqref{u10}  can be halved, provided $C_0$ is chosen to be large enough and $\varepsilon$ is chosen to be suitably small. Then the half upper bounds in \eqref{u10} are indeed justifiable by applying the  bootstrapping argument. The verification is a long and tedious process involving repeated applications of various anisotropic inequalities and the upper bounds on the kernel functions. To begin, we first derive the decay rates of the nonlinear terms, which plays a key role in the proof of Theorem \ref{thm1.2}.
\vskip .1in
\begin{lem}\label{lemma5.1}
	Assume that $(u,b)$	is a smooth solution of the problem \eqref{PMHD}, satisfying \eqref{u10}. Then, there exists  an absolute positive constant $C>0$, such that
	\begin{align}
	&\left\|\| b\cdot \na b\|_{L_{x_2,x_3}^1}\right\|_{L_{x_1}^2}\leq CC_0^2\varepsilon^2 \left(1+t\right)^{-\frac{11}{8}},\label{pbl1}\\
	&\left\|\|\p_1\left( b\cdot \na b\right)\|_{L_{x_2,x_3}^1}\right\|_{L_{x_1}^2}\leq CC_0^2\varepsilon^2 \left(1+t\right)^{-\frac{5}{4}},\label{1bb11}\\
	&\|\p_1\left(b\cdot \na b\right)\|_{L^2}\leq  C C_0^{2} \varepsilon^2 \left(1+t\right)^{-\frac{37}{24}}, \label{11bb2}\\
	&\|\p_2\left( b\cdot \na b\right)\|_{L^2}\leq CC_0^2\varepsilon^2 \left(1+t\right)^{-\frac{11}{6}},   \label{22bb2}\\
	&\|\p_3\left( b\cdot \na b\right)\|_{L^2}\leq CC_0^2\varepsilon^2 \left(1+t\right)^{-\frac{13}{6}},   \label{23bb2}\\
	&\left\|\|\p_1^2\left( b\cdot \na b\right)\|_{L_{x_2,x_3}^1}\right\|_{L_{x_1}^2}\leq  CC_0\varepsilon\left(1+t\right)^{-\frac12}\|\p_{\nu}b\|_{H^2}+ CC_0\varepsilon^2\left(1+t\right)^{-\frac78},\label{11bb11} \\
	&\|\p_l\p_h(b\cdot \na b)\|_{L^2}\leq CC_0\varepsilon^2\left(1+t\right)^{-\frac{11}{12}},\quad\forall\ l,h\in\{1,2\},   \label{lhbb2}\\
	&\|\p_1\p_3(b\cdot \na b)\|_{L^2}\leq CC_0^2\varepsilon^2 \left(1+t\right)^{-\frac{5}{3}}, \label{13bbl2} \\
	&\|\p_2\p_3(b\cdot \na b)\|_{L^2}\leq CC_0^2\varepsilon^2 \left(1+t\right)^{-\frac{11}{6}},   \label{23bbl2} \\
	&\|\p_3^2(b\cdot \na b)\|_{L^2}\leq CC_0^2\varepsilon^2 \left(1+t\right)^{-\frac{29}{12}}.   \label{33bbl2}
	\end{align}
\end{lem}

\begin{proof}
Using the  Minkowski inequality (\ref{MIE}) and  the ansatz in \eqref{u10}, we have
\begin{align*}
\left\|\| b\cdot \na b \|_{L_{x_2,x_3}^1}\right\|_{L_{x_1}^2}&
\leq \left\|\| b_1\p_1b\|_{L_{x_2,x_3}^1}\right\|_{L_{x_1}^2}+\left\|\| b_{\nu}\p_{\nu}b\|_{L_{x_2,x_3}^1}\right\|_{L_{x_1}^2} \\
&\leq \left\|\| b_1\p_1b\|_{L_{x_1}^2}\right\|_{L_{x_2,x_3}^1}+\left\|\| b_{\nu}\p_{\nu}b\|_{L_{x_1}^2}\right\|_{L_{x_2,x_3}^1} \\
&\leq \left\|\| b_1\|_{L_{x_1}^\infty}\|\p_1b\|_{L_{x_1}^2}\right\|_{L_{x_2,x_3}^1}+\left\|\| b_{\nu}\|_{L_{x_1}^\infty}\|\p_{\nu}b\|_{L_{x_1}^2}\right\|_{L_{x_2,x_3}^1} \\
&\leq C\|b_1\|_{L^2}^{\frac12}\|\p_1b_1\|_{L^2}^{\frac12}\|\p_1b\|_{L^2}+C\|b_{\nu}\|_{L^2}^{\frac12}\|\p_1b_{\nu}\|_{L^2}^{\frac12}\|\p_{\nu}b\|_{L^2} \\
&\leq CC_0^2\varepsilon^2 \left(1+t\right)^{-\frac{11}{8}}+ CC_0^2\varepsilon^2\left(1+t\right)^{-\frac32}\\
& \leq CC_0^2\varepsilon^2 \left(1+t\right)^{-\frac{11}{8}},
\end{align*}
where we have used the fact that $\nabla\cdot b=0$ to infer from \eqref{u10} that
\beq\label{p1b1-1}
\|\p_1 b_1\|_{L^2}\leq \|\p_2b_2\|_{L^2}+\|\p_3b_3\|_{L^2}\leq CC_0\varepsilon (1+t)^{-1}.
\eeq
In a similar manner,
\begin{align}
&\left\|\|\p_1\left( b\cdot \na b\right)\|_{L_{x_2,x_3}^1}\right\|_{L_{x_1}^2}\notag\\
&\quad\leq C\|\p_{\nu} b \|_{L^2}\|\p_1b_{\nu}\|_{L^2}^{\frac12}\|\p_1^2b_{\nu}\|_{L^2}^{\frac12}+C\|\p_{\nu}\p_1 b \|_{L^2}\|b_{\nu}\|_{L^2}^{\frac12}\|\p_1b_{\nu}\|_{L^2}^{\frac12}\nonumber\\
&\qquad+C\|\p_1 b_1 \|_{L^2}\|\p_1b\|_{L^2}^{\frac12}\|\p_1^2b\|_{L^2}^{\frac12}+C\|\p_1^2b \|_{L^2}\|b_1\|_{L^2}^{\frac12}\|\p_1b_1\|_{L^2}^{\frac12}\nonumber\\
&\quad\leq CC_0^2\varepsilon^2\left(1+t\right)^{-\frac54}.  \label{1b1}
\end{align}

To proceed,  we first estimate $\| b \|_{L^\infty} $. It is easy to derive from the anisotropic Sobolev embedding inequality and  \eqref{u10} that
\begin{align}
\| b \|_{L^\infty}&
\leq C \| b \|_{L^2}^{\frac 18}\|\p_1b \|_{L^2}^{\frac 18}\|\p_2b \|_{L^2}^{\frac 18}\|\p_3b \|_{L^2}^{\frac 18}\nonumber\\
&\quad\times\|\p_{1}\p_2b \|_{L^2}^{\frac 18}\|\p_{1}\p_3b \|_{L^2}^{\frac 18}\|\p_{2}\p_3b \|_{L^2}^{\frac 18}\|\p_{1}\p_2\p_3b \|_{L^2}^{\frac 18}\nonumber\\
&\leq CC_0 \varepsilon \left(1+t\right)^{-\frac{11}{12}}. \label{binfty}
\end{align}
Using the similar inequalities as that in (\ref{p1b1-1}), we have by (\ref{u10}) that
\begin{align}
\| b_1 \|_{L^\infty}
&\leq C \| b_1 \|_{L^2}^{\frac 18}\|\p_1b_1 \|_{L^2}^{\frac 18}\|\p_2b_1 \|_{L^2}^{\frac 18}\|\p_3b_1 \|_{L^2}^{\frac 18}\notag\\
&\quad\times\|\p_{1}\p_2b_1 \|_{L^2}^{\frac 18}\|\p_{1}\p_3b_1 \|_{L^2}^{\frac 18}\|\p_{2}\p_3b_1 \|_{L^2}^{\frac 18}\|\p_{1}\p_2\p_3b_1 \|_{L^2}^{\frac 18}\nonumber\\
&\leq CC_0 \varepsilon \left(1+t\right)^{-\frac{7}{6}}. \label{b1infty}
\end{align}

With the help of  \eqref{u10}, \eqref{binfty}  and \eqref{b1infty}, we obtain
\begin{align}
\| b\cdot \na b \|_{L^2}
&\leq C \| b_1 \|_{L^\infty} \|\p_1 b \|_{L^2} +C \| b_\nu\|_{L^\infty} \|\p_\nu b \|_{L^2}\nonumber\\
&\leq CC_0^{2} \varepsilon^2 \left(1+t\right)^{-\frac{5}{3}}, \label{3bb2}
\end{align}
and
\begin{align}
\|\p_1\left(b\cdot \na b\right)\|_{L^2}
&\leq  \| b_1 \|_{L^\infty}\| \p_1^2 b \|_{L^2}+\| b_\nu \|_{L^\infty}\| \p_1\p_\nu b \|_{L^2}  +\| \p_1  b\cdot \na b\|_{L^2}\nonumber\\
& \leq C C_0^{2} \varepsilon^2 \left(1+t\right)^{-\frac{37}{24}}+\|\p_1b_1\|_{L_{x_1,x_3}^2L_{x_2}^\infty}\| \p_1 b \|_{L_{x_1,x_3 }^\infty L_{x_2}^2} \nonumber\\
&\quad+\|\p_1  b_\nu\|_{L_{x_1,x_3}^\infty L_{x_2}^2}\| \p_\nu b \|_{L_{x_1,x_3}^2L_{x_2}^\infty}
\nonumber\\
&\leq C C_0^{2} \varepsilon^2 \left(1+t\right)^{-\frac{37}{24}},\label{1b2}
\end{align}
where we have used the following anisotropic Sobolev inequalities for $i,j,k\in \{1,2,3\}$ and $i\neq j\neq k$:
\begin{align}
\|f\|_{L_{x_i,x_j}^2L_{x_k}^\infty}&\leq C\|f\|_{L^2}^{\frac12}\|\p_kf\|_{L^2}^{\frac12},\notag
\\
\| f \|_{L_{x_i,x_j }^\infty L_{x_k}^2}&\leq C\|f\|_{L^2}^{\frac14}\|\p_if\|_{L^2}^{\frac{1}{4}}\|\p_jf\|_{L^2}^{\frac{1}{4}}\|\p_i\p_jf\|_{L^2}^{\frac{1}{4}}.\label{anso}
\end{align}

Similarly to the derivation of (\ref{1b2}), we have
\begin{align}
\|\p_2\left( b\cdot \na b\right)\|_{L^2}
&\leq \|b_1\|_{ L^\infty}\|\p_1\p_2 b\|_{L^2}+\|b_\nu\|_{ L^\infty}\|\p_2\p_\nu b\|_{L^2} \nonumber\\
&\quad+\|\p_2  b_1\|_{L_{x_1,x_3}^\infty L_{x_2}^2}\| \p_1 b\|_{L_{x_1,x_3}^2 L_{x_2}^\infty} \notag\\
&\quad+\|\p_2  b_\nu\|_{L_{x_1,x_3}^2L_{x_2}^\infty}\| \p_\nu b \|_{L_{x_1,x_3}^\infty L_{x_2}^2} \nonumber\\
&\leq C  C_0^{2} \varepsilon^2\left(1+t\right)^{-\frac{11}{6}},\nonumber
\end{align}
and
\begin{align}
\|\p_3(b\cdot \na b)\|_{L^2}
&\leq \|b_1\|_{ L^\infty}\|\p_1\p_3 b\|_{L^2}+\|b_\nu\|_{ L^\infty}\|\p_\nu\p_3 b\|_{L^2}\nonumber\\
&\quad+\|\p_3  b_1\|_{L_{x_1,x_3}^\infty L_{x_2}^2}\| \p_1 b\|_{L_{x_1,x_3}^2 L_{x_2}^\infty}\nonumber\\
&\quad+\|\p_3  b_\nu\|_{L_{x_2,x_3}^\infty L_{x_1}^2}\| \p_\nu b \|_{L_{x_2,x_3}^2 L_{x_1}^\infty }\nonumber\\
&\leq C C_0^2 \varepsilon^2\left(1+t\right)^{-\frac{13}{6}}. \nonumber
\end{align}

Using the divergence-free condition $\nabla\cdot b=0$ and (\ref{p1b1-1}) again, we deduce from H\"{o}lder inequality  and (\ref{u10}) that
\begin{align}
&\left\|\|\p_1^2 ( b\cdot \na b )\|_{L_{x_2,x_3}^1}\right\|_{L_{x_1}^2}\notag\\
&\quad\leq C\|\p_{\nu} b\|_{L^2}\|\p_1^2b\|_{L^2}^{\frac12}\|\p_1^3 b\|_{L^2}^{\frac12}+C\|\p_{1}b\|_{L^2}\|\p_\nu\p_1 b\|_{L^2}^{\frac12}\|\p_\nu\p_1^2 b \|_{L^2}^{\frac12}\nonumber\\
&\qquad+C\|\p_{\nu}\p_1^2 b \|_{L^2}\|b_{\nu}\|_{L^2}^{\frac12}\|\p_1b_{\nu}\|_{L^2}^{\frac12}+C\|\p_1^3b \|_{L^2}\|b_1\|_{L^2}^{\frac12}\|\p_1b_1\|_{L^2}^{\frac12}\nonumber\\
&\quad\leq  CC_0\varepsilon\left(1+t\right)^{-\frac12}\|\p_{\nu}b\|_{H^2}+ CC_0\varepsilon^2\left(1+t\right)^{-\frac78}.\label{p11bb}
\end{align}

For $l,h\in\{1,2\}$, by (\ref{p1b1-1}) one infers from (\ref{u10}) and (\ref{binfty}) in a similar manner as the derivation of  (\ref{1b2}) that
\begin{align}
\|\p_l\p_h ( b\cdot \na b )\|_{L^2}&\leq \|b\|_{L^\infty}\|b\|_{H^3}+ \|\p_l\p_h b_1\|_{L^2_{x_1,x_2}L^\infty_{x_3}}\|\p_1 b\|_{L^\infty_{x_1,x_2}L^2_{x_3}}\notag\\
&\quad+ \|\p_l\p_h b_\nu\|_{L^2_{x_2,x_3}L^\infty_{x_1}}\|\p_\nu b\|_{L^\infty_{x_2,x_3}L^2_{x_1}}\notag\\
&\quad+\|\p_l b_1\|_{L^\infty_{x_2,x_3}L^2_{x_1}}\|\p_h\p_1 b\|_{L^2_{x_2,x_3}L^\infty_{x_1}}\notag\\
&\quad+\|\p_l b_\nu\|_{L^\infty_{x_1,x_2}L^2_{x_3}}\|\p_h\p_\nu b\|_{L^2_{x_1,x_2}L^\infty_{x_3}}\notag\\
&\quad+\|\p_h b_1\|_{L^\infty_{x_2,x_3}L^2_{x_1}}\|\p_l\p_\nu b\|_{L^2_{x_2,x_3}L^\infty_{x_1}}\notag\\
&\quad+\|\p_h b_\nu\|_{L^\infty_{x_1,x_2}L^2_{x_3}}\|\p_l\p_\nu b\|_{L^2_{x_1,x_2}L^\infty_{x_3}}\notag\\
&\leq CC_0\varepsilon^2\left(1+t\right)^{-\frac{11}{12}}. \nonumber
\end{align}

Using (\ref{binfty}), (\ref{b1infty}), the anisotropic inequalities (\ref{anso}) and the analogous inequality as that in (\ref{p1b1-1}), we infer from (\ref{u10}) that
\begin{align*}
&\|\p_1\p_3( b\cdot \na b )\|_{L^2}\notag\\
&\quad\leq \| b_1\|_{L^\infty}\|\p_1^2\p_3 b\|_{L^2}+\|b_\nu\|_{L^\infty}\|\p_1\p_3\p_\nu b\|_{L^2}\nonumber\\
&\qquad+\|\p_1\p_3b_1\|_{L^2_{x_2,x_3}L^\infty_{x_1}}\|\p_1 b \|_{L^\infty_{x_2,x_3}L^2_{x_1}}+\|\p_1\p_3b_\nu\|_{L^2_{x_2,x_3}L^\infty_{x_1}}\|\p_\nu b \|_{L^\infty_{x_2,x_3}L^2_{x_1}}\nonumber\\
&\qquad+\|\p_1 b_1\|_{L^\infty_{x_1,x_3}L^2_{x_2}}\|\p_1\p_3 b \|_{L^2_{x_1,x_3}L^\infty_{x_2}}+\|\p_1b_\nu\|_{L^\infty_{x_2,x_3}L^2_{x_1}}\|\p_3\p_\nu b \|_{L^2_{x_2,x_3}L^\infty_{x_1}}\nonumber\\
&\qquad+\|\p_3 b_1\|_{L^\infty_{x_1,x_2}L^2_{x_3}}\|\p_1^2 b \|_{L^2_{x_1,x_2}L^\infty_{x_3}}+\|\p_3 b_\nu\|_{L^\infty_{x_1,x_2}L^2_{x_3}}\|\p_1\p_\nu b \|_{L^2_{x_1,x_2}L^\infty_{x_3}}\nonumber\\
&\quad\leq  CC_0^2\varepsilon^2\left(1+t\right)^{-\frac53},
\notag\\[2mm]
&\|\p_2\p_3\left( b\cdot \na b\right)\|_{L^2}\notag\\
&\quad\leq \| b_1\|_{L^\infty}\|\p_1\p_2\p_3 b\|_{L^2}+\|b_\nu\|_{L^\infty}\|\p_2\p_3\p_\nu b\|_{L^2}\nonumber\\
&\qquad+\|\p_2\p_3b\|_{L^2_{x_2,x_3}L^\infty_{x_1}}\|\nabla b \|_{L^\infty_{x_2,x_3}L^2_{x_1}}+\|\p_2b\|_{L^\infty_{x_1,x_3}L^2_{x_2}}\|\p_3\nabla b \|_{L^2_{x_1,x_3}L^\infty_{x_2}}\nonumber\\
&\qquad+\|\p_3b\|_{L^\infty_{x_1,x_2}L^2_{x_3}}\|\p_2\nabla b \|_{L^2_{x_1,x_2}L^\infty_{x_3}}\nonumber\\
&\quad\leq  CC_0^2\varepsilon^2\left(1+t\right)^{-\frac{11}{6}} ,
\end{align*}
and
\begin{align*}
&\|\p_3^2 ( b\cdot \na b )\|_{L^2}\notag\\
&\quad\leq \| b_1\|_{L^\infty}\|\p_1\p_3^2 b\|_{L^2}+\|b_\nu\|_{L^\infty}\|\p_3^2\p_\nu b\|_{L^2}\nonumber\\
&\qquad+\|\p_3^2b_1\|_{L^2_{x_2,x_3}L^\infty_{x_1}}\|\p_1 b \|_{L^\infty_{x_2,x_3}L^2_{x_1}}+\|\p_3^2b_\nu\|_{L^2_{x_2,x_3}L^\infty_{x_1}}\|\p_\nu b \|_{L^\infty_{x_2,x_3}L^2_{x_1}}\nonumber\\
&\qquad+2\|\p_3b_1\|_{L^\infty_{x_1,x_2}L^2_{x_3}}\|\p_3\p_1 b \|_{L^2_{x_1,x_2}L^\infty_{x_3}}+2\|\p_3b_\nu\|_{L^\infty_{x_1,x_2}L^2_{x_3}}\|\p_3\p_\nu b \|_{L^2_{x_1,x_2}L^\infty_{x_3}}\nonumber\\
&\quad\leq  CC_0^2\varepsilon^2\left(1+t\right)^{-\frac{29}{12}} .
\end{align*}

Thus, collecting the above estimates together ends the proof of Lemma \ref{lemma5.1}.
\end{proof}

\subsection{Decay rates of $(u,b)$.}

With Lemma \ref{lemma5.1} at hand,  we are now ready to justify the decay estimate of the solution $(u,b)$. For brevity, we will focus on the decay rates of $u$, since the same decay rates of $b$ can be achieved in a similar
manner  due to  the resemblance between \eqref{u} and \eqref{b}. The proof will be divided into several steps by considering the different-order derivatives of the solutions. We only focus on the large-time behavior on the interval $t\geq1$, since the short-time boundedness  of the solutions on $0\leq t\leq 1$ is trivial, due to Theorem \ref{thm1.1}. To begin, we first observe from Proposition \ref{lem2.1}   that
\begin{align}
& |\widehat{K_1}|, |\widehat{K_2}|\leq Ce^{-c\xi_{\nu}^2t},\quad{\rm if}\quad\xi \in \Om_1\cup \Om_{21}\cup \Om_{22},\label{S1}
\\& |\widehat{K_1}|, |\widehat{K_2}| \leq C\left(e^{-c t}+e^{-c\xi_{\nu}^2t}\right),\quad{\rm if}\quad\xi \in \Om_{23}.\label{S2}
\end{align}

\textbf{Step  I.   The decay estimates  of $\|u\|_{L^2}$ }

\vskip .1in

It follows from Plancherel's Theorem and \eqref{u} that
\begin{align}
\|u\|_{L^2}=\|\widehat{u}\|_{L^2}&\leq \left\| \widehat{K_1}(t)\widehat{u_{0}}\right\|_{L^2}+\left\|\widehat{K_2}(t)\widehat{b_{0}}\right\|_{L^2}\nonumber\\
&\quad+\int_0^t\left\|\widehat{K_1}(t-\tau)\widehat{N_{1}}(\tau)\right\|_{L^2}d\tau\nonumber\\
&\quad+\int_0^t\left\|\widehat{K_2}(t-\tau)\widehat{N_{2}}(\tau)\right\|_{L^2}d\tau \triangleq \sum_{m=1}^4\text{I}_m. \label{uL2}
\end{align}

First, using  Corollary \ref{PI} and  Plancherel's theorem, we see that for any $t\geq1$,
\begin{align}
\text{I}_1+\text{I}_2&\leq C\|e^{-c\xi_{\nu}^2t}(\widehat{u_0},\widehat{b_0})\|_{L^2}+ Ce^{-ct} \|(\widehat{u_0},\widehat{b_0})\|_{L^2}\nonumber\\
&\leq C(1+t)^{-\frac12}\left(\|(u_{0},b_{0}) \|_{L^2_{x_1}L^1_{x_2 ,x_3}}+ \| (u_{0},b_{0} )\|_{L^2 }\right).\label{ui12}
\end{align}

Keeping the definitions of $N_1$ and $N_2$ in mind, by \eqref{S1} and \eqref{S2} we have
\begin{align}
\text{I}_3+\text{I}_4
&\leq C\int_0^t\left\|e^{-c\xi_{\nu}^2\left(t-\tau \right)}\widehat{b\cdot \na b} \right\|_{L^2}d\tau+C\int_0^t\left\|e^{-c\xi_{\nu}^2\left(t-\tau \right)}\widehat{u\cdot \na u} \right\|_{L^2}d\tau\nonumber\\
&\quad+C\int_0^t\left\|e^{-c\left(t-\tau \right)}\widehat{b\cdot \na b} \right\|_{L^2}d\tau+C\int_0^t\left\|e^{-c\left(t-\tau \right)}\widehat{u\cdot \na u} \right\|_{L^2}d\tau\nonumber\\
&\quad+ C\int_0^t\left\|e^{-c\xi_{\nu}^2\left(t-\tau \right)}\widehat{b\cdot \na u} \right\|_{L^2}d\tau+C\int_0^t\left\|e^{-c\xi_{\nu}^2\left(t-\tau \right)}\widehat{u\cdot \na b} \right\|_{L^2}d\tau\nonumber\\
&\quad+C\int_0^t\left\|e^{-c\left(t-\tau \right)}\widehat{b\cdot \na u} \right\|_{L^2}d\tau+C\int_0^t\left\|e^{-c\left(t-\tau \right)}\widehat{u\cdot \na b} \right\|_{L^2}d\tau\nonumber\\
&\triangleq  \text{I}_{34}^1+\ldots+\text{I}_{34}^8,\label{ij3}
\end{align}
where we have used the fact that the Leray projection operator $\mathbb P$ is bounded in $L^{2}$ (see \cite{Stein}). In the next, we only deal with $\text{I}_{34}^1$ and $\text{I}_{34}^3$, since the other terms can be treated in a similar manner, due to their analogous  structures.


By   Corollary \ref{PI} with $\alpha=0$ and Lemmas \ref{ID}--\ref{ED}, we infer from \eqref{pbl1}  that
\begin{align}
\text{I}_{34}^1
&\leq C\int_0^{t}\left(1+t-\tau \right)^{-\frac12}\left\|\| b\cdot \na b \|_{L_{x_2,x_3}^1}\right\|_{L_{x_1}^2}d\tau  \nonumber\\
&\leq CC_0^2\varepsilon^2\int_0^t\left(1+t-\tau \right)^{-\frac12} \left(1+\tau\right)^{-\frac{11}{8}}d\tau\nonumber\\
&\leq C C_0^2 \varepsilon^2\left(1+t\right)^{-\frac12},\quad \forall\ t\geq1.  \label{r31g}
\end{align}
It is easily derived from Lemma \ref{ED} and \eqref{3bb2} that
\begin{align}
\text{I}_{34}^3
&\leq C\int_{0}^{t} e^{-c\left( t-\tau \right)}\|  b\cdot \na b  \|_{L^2}d\tau  \nonumber\\
&\leq  CC_0^2\varepsilon^2\int_0^t e^{-c\left( t-\tau \right)} \left(1+\tau\right)^{-\frac{5}{3}}d\tau \leq CC_0^2\varepsilon^2\left(1+t\right)^{-\frac12}, \label{r33}
\end{align}
where we have also used Plancherel's theorem. Analogously, the other terms in (\ref{ij3}) admit the same decay rates. Thus,
in view of  \eqref{r31g} and \eqref{r33}, we deduce from \eqref{ij3} that
\begin{align*}
\text{I}_3+\text{I}_4
\leq   C C_0^2\varepsilon^2\left(1+t\right)^{-\frac12}, \quad\forall~ t\geq 1.  
\end{align*}
which, combined with \eqref{uL2} and \eqref{ui12}, yields
\begin{align}\label{ul2}
\|u(t)\|_{L^2}
&\leq C_1\varepsilon(1+t)^{-\frac{1}{2}}+C_2 C_0^2 \varepsilon^2(1+t)^{-\frac{1}{2}}.
\end{align}

Now, if   $C_0$ and $\varepsilon$ are chosen to be such that
\begin{align*} 
C_0\geq 4C_1  \quad\text{and}\quad  0<\varepsilon\leq (4C_2C_0 )^{-1},
\end{align*}
then it readily follows from (\ref{ul2}) that
\begin{align*}
\|u(t)\|_{L^2}\leq \frac{C_0}{2}\varepsilon(1+t)^{-\frac{1}{2}},\quad\forall\ t\geq1.
\end{align*}

\vskip .1in

\textbf{Step  II.   The decay estimates  of $\|\p_iu\|_{L^2}$ with $i\in\{1,2,3\}$ }

\vskip .1in
For $i\in\{1,2,3\}$, we have by  Plancherel's Theorem and \eqref{u} that
\begin{align}
\|\p_iu\|_{L^2}&=\|\widehat{\p_iu}\|_{L^2}=\||\xi_i|\widehat u\|_{L^2}\nonumber\\
&\leq \left\| |\xi_i|\widehat{K_1}(t)\widehat{u_{0}}\right\|_{L^2}+\left\||\xi_i|\widehat{K_2}(t)\widehat{b_{0}}\right\|_{L^2}\nonumber\\
&\quad+\int_0^t\left\||\xi_i|\widehat{K_1}(t-\tau)\widehat{N_{1}}(\tau)\right\|_{L^2}d\tau\nonumber\\
&\quad+\int_0^t\left\||\xi_i|\widehat{K_2}(t-\tau)\widehat{N_{2}}(\tau)\right\|_{L^2}d\tau\triangleq\sum_{m=1}^3\text{II}_{i,m}.
\label{iuL2}
\end{align}

\textbf{Step  II-1.   The decay estimate of $\|\p_1u\|_{L^2}$ }
\vskip .1in
Similarly to the derivation of (\ref{ui12}) for $\text{I}_1$ and $\text{I}_2$ in Step I-1, by Plancherel's Theorem and Corollary \ref{PI} we infer from (\ref{iuL2}) that for $i=1$,
\begin{align}
\text{II}_{1,1}+\text{II}_{1,2} &\leq \left\| \widehat{K_1}(t)\widehat{\p_1 u_{0}}\right\|_{L^2}+\left\| \widehat{K_2}(t)\widehat{\p_1 b_{0}}\right\|_{L^2}\nonumber\\
&\leq C\| e^{-c\xi_{\nu}^2t} (\widehat{\p_1 u_0},\widehat{\p_1 b_0} )\|_{L^2}+ Ce^{-ct} \left\| (\widehat{\p_1 u_0},\widehat{\p_1b_0})\right\|_{L^2}\nonumber\\
&  \leq  C(1+t)^{-\frac12}\left(\|(\p_1u_0, \p_1b_0)\|_{L^2_{x_1}L^1_{x_2, x_3}}+\|(\p_1u_0, \p_1b_0)\|_{L^2}\right)\nonumber\\
& \leq C \varepsilon\left(1+t\right)^{-\frac12}. \label{p1ug1}
\end{align}

From now on, analogously to treatment of (\ref{ij3}) in Step I-1, for simplicity we only consider the terms associated with $b\cdot\nabla b$ in $N_1$ and $N_2$. So, to deal with $\text{II}_{1,3}$ and $\text{II}_{1,4}$,   by (\ref{S1}) and (\ref{S2}) we deduce from   Plancherel's Theorem,  Corollary \ref{PI},  (\ref{1bb11}) and (\ref{11bb2}) that
\begin{align}
&\text{II}_{1,3}+\text{II}_{1,4}\nonumber\\
&\quad\triangleq\int_0^t\left\|\xi_1\widehat{K_1}(t-\tau)\widehat{b\cdot\nabla b} \right\|_{L^2}d\tau+\int_0^t\left\|\xi_1\widehat{K_2}(t-\tau)\widehat{b\cdot\nabla b} \right\|_{L^2}d\tau\nonumber\\
& \quad\leq \int_0^t\left\| \widehat{K_1}(t-\tau)\widehat{
\p_1(b\cdot\nabla b)} \right\|_{L^2}d\tau+\int_0^t\left\| \widehat{K_2}(t-\tau)\widehat{\p_1(b\cdot\nabla b)} \right\|_{L^2}d\tau\nonumber\\
& \quad\leq C\int_0^t\left\| e^{-c\xi_\nu^2(t-\tau)}\widehat{
\p_1(b\cdot\nabla b)} \right\|_{L^2}d\tau+C\int_0^te^{-c(t-\tau)}\left\| \widehat{\p_1(b\cdot\nabla b)} \right\|_{L^2}d\tau\nonumber\\
& \quad\leq C\int_0^{t}\left[(1+t-\tau )^{-\frac12}\|\p_1 (b\cdot \na b ) \|_{L^2_{x_1}L_{x_2,x_3}^1}+e^{-c (t-\tau )}\| \p_1 (b\cdot \na b ) \|_{L^2}\right]d\tau \nonumber\\
&\quad \leq CC_0^2\varepsilon^2\int_0^t\left[\left(1+t-\tau \right)^{-\frac12} \left(1+\tau\right)^{-\frac54}+ e^{-c\left(t-\tau\right)}\left(1+\tau\right)^{-\frac{37}{24}}\right]d\tau \nonumber\\
&\quad\leq  C C_0^2  \varepsilon^2 (1+t )^{-\frac12}. \label{p1ug2}
\end{align}

In view of (\ref{p1ug1}) and (\ref{p1ug2}), by choosing $C_0$ large enough and $\varepsilon$ suitably small, we know
$$\|\p_1u(t)\|_{L^2}\leq \left(C \varepsilon +C C_0^2  \varepsilon^2 \right)(1+t )^{-\frac12}\leq \frac {C_0}{2}  \varepsilon (1+t )^{-\frac12}.$$

\vskip .1in

\textbf{Step  II-2.   The decay estimates of $\|\p_i u\|_{L^2}$ with $i\in\{2,3\}$ }
\vskip .1in

Analogously  to the treatment of  (\ref{ij3}) in Step I-1,   we focus our attention on the estimates  of the terms involving $b\cdot\nabla b$.
Using  Corollary \ref{PI} with $\alpha=1$, Lemmas \ref{ID} and \ref{ED}, we deduce from \eqref{iuL2}, (\ref{pbl1}),  (\ref{22bb2}) and (\ref{23bb2}) that for $i\in\{2,3\}$,
\begin{align*}
\|\p_i u\|_{L^2}
&\leq  C\|\xi_ie^{-c\xi_{\nu}^2t} (\widehat{u_0},\widehat{b_0} )\|_{L^2}+ Ce^{-ct} \left\||(\widehat{\p_i u_0},\widehat{\p_i b_0})\right\|_{L^2}\nonumber\\
&\quad +C\int_0^t\left\| \xi_i  e^{-c\xi_\nu^2(t-\tau)}\widehat{
b\cdot\nabla b} \right\|_{L^2}d\tau+C\int_0^te^{-c(t-\tau)}\left\| \widehat{\p_i(b\cdot\nabla b)} \right\|_{L^2}d\tau\notag\\
&\leq  C(1+t)^{-1}\left(\| (u_0, b_0 )\|_{L^2_{x_1}L^1_{x_2, x_3}}+\|(\p_i u_0, \p_i b_0)\|_{L^2}\right)\nonumber\\
&\quad+C\int_0^{t}\left[\left(1+t-\tau \right)^{-1} \| b\cdot \na b  \|_{L^2_{x_1}L^1_{x_2, x_3}} +e^{-c\left(t-\tau\right)}\| \p_i (b\cdot \na b ) \|_{L^2}\right] d\tau \nonumber\\
&\leq  C\varepsilon(1+t)^{-1} +CC_0^2\varepsilon^2\int_0^t \left(1+t-\tau \right)^{-1} \left(1+\tau\right)^{-\frac{11}{8}} d\tau \nonumber\\
&\quad+CC_0^2\varepsilon^2\int_0^t e^{-c\left(t-\tau\right)}\left(1+\tau\right)^{-\frac{11}{6}} d\tau \nonumber\\
&\leq C \varepsilon\left(1+t\right)^{-1}+C C_0^2  \varepsilon^2\left(1+t\right)^{-1},\nonumber
\end{align*}
from which it follows that
\begin{equation*}\|\p_i u(t)\|_{L^2}\leq \frac {C_0}{2}  \varepsilon (1+t )^{-1},\quad\forall\ i\in\{2,3\},
\end{equation*}
provided $C_0$ is chosen to be large enough and $\varepsilon$ is chosen to be suitably small.

\vskip .1in

\textbf{Step  III. The decay estimates  of $\|\p_{i}\p_ju\|_{L^2}$ with  $i,j\in\{1,2,3\}$}
\vskip .1in

In terms of (\ref{u}), we see that for $i,j\in\{1,2,3\}$,
\begin{align}
\|\p_{i}\p_ju\|_{L^2}&=\|\widehat{\p_{ij}u}\|_{L^2}=\|\xi_i\xi_j\widehat u\|_{L^2}\nonumber\\
&\leq \left\| \xi_{i}\xi_{j}\widehat{K_1}(t)\widehat{u_{0}}\right\|_{L^2}+\left\|\xi_{i}\xi_{j}\widehat{K_2}(t)\widehat{b_{0}}\right\|_{L^2}\nonumber\\
&\quad+\int_0^t\left\|\xi_{i}\xi_{j}\widehat{K_1}(t-\tau)\widehat{N_{1}}(\tau)\right\|_{L^2}d\tau\nonumber\\
&\quad+\int_0^t\left\|\xi_{i}\xi_{j}\widehat{K_2}(t-\tau)\widehat{N_{2}}(\tau)\right\|_{L^2}d\tau\triangleq \sum_{m=1}^4\text{III}_{ij,m}.\label{iiiuL2}
\end{align}

\textbf{Step  III-1.  The decay estimates  of $\|\p_3\p_{j}u\|_{L^2}$ with $j\in\{2,3\}$}
\vskip .1in

Since $|\xi_3\xi_j|\leq \xi_\nu^2$ for $j\in\{2,3\}$, using (\ref{S1}), (\ref{S2}) and  Corollary \ref{PI} with $\alpha=2$,  we deduce from (\ref{iiiuL2}) with $i=3$ and $j\in\{2,3\}$ that
\begin{align}
\text{III}_{3j,1}+\text{III}_{3j,2}
&\leq C\|\xi_3\xi_je^{-c\xi_{\nu}^2t} (\widehat{u_0},\widehat{b_0} )\|_{L^2}+ Ce^{-ct} \left\|\xi_3\xi_j(\widehat{u_0},\widehat{b_0})\right\|_{L^2}\nonumber\\
&\leq C(1+t)^{-\frac32}\left(\| (u_0, b_0 )\|_{L^2_{x_1}L^1_{x_2, x_3}}+\|\p_3\p_j (u_0, b_0 )\|_{L^2}\right) \nonumber\\
&\leq C\varepsilon(1+t)^{-\frac32}. \label{p3p2ul2-1}
\end{align}

Thanks to (\ref{KS1}) and (\ref{KS21})--(\ref{KS23}), to bound $\text{III}_{3j,3}$ and $\text{III}_{3j,4}$, it suffices to consider the following two items with  $j\in\{2,3\}$:
\begin{align*}
\text{III}_{3j,3}&\triangleq\int_0^t\left\|\xi_3\xi_je^{-c\xi_{\nu}^2\left(t-\tau \right)}\widehat{b\cdot \na b}\right\|_{L^2}d\tau ,
\\
\text{III}_{3j,4}&\triangleq\int_0^t\left\|\xi_3\xi_je^{-c\left(1+\xi_3^2\right)\left(t-\tau \right)}\widehat{b\cdot \na b}\right\|_{L^2}d\tau.
\end{align*}
It is worth mentioning here that the factor $e^{-\xi_3^2 t}$ is taken into account (see (\ref{KS23})), which is  different from the previous one in (\ref{S2}). However, this is of particular importance  for the decay estimates involving the $x_3$-direction.  Using Corollary \ref{PI} and Lemmas \ref{ID}--\ref{ED}, we can make use of (\ref{pbl1}) and (\ref{23bb2}) to bound $\text{III}_{3j,3}$ as follows,
\begin{align}
\text{III}_{3j,3}
&\leq C\int_0^{\frac t2}\left(1+t-\tau \right)^{-\frac32}\| b\cdot \na b \|_{L_{x_1}^2L_{x_2,x_3}^1}d\tau\nonumber\\
 &\quad +C\int_{\frac t2}^{t}\left(t-\tau \right)^{-\frac12}\|\p_3(b\cdot \na b)\|_{L^2}d\tau \nonumber\\
&\leq CC_0^2\varepsilon^2\int_0^{\frac t2}\left(1+t-\tau \right)^{-\frac32} \left(1+\tau\right)^{-\frac{11}{8}}d\tau\nonumber\\
&\quad+CC_0^2\varepsilon^2\int_{\frac t2}^{t} \left(t-\tau \right)^{-\frac12}\left(1+\tau\right)^{-\frac{13}{6}}d\tau \nonumber\\
&\leq C C_0^2 \varepsilon^2\left(1+t\right)^{-\frac32},\quad \forall\ j\in\{2,3\}.\label{iii3231-2}
\end{align}

Due to  (\ref{22bb2}), (\ref{23bb2}) and the fact that $|\xi_3|e^{-\xi_3^2t}\leq Ct^{-1/2}$, we see that for $j\in\{2,3\}$,
\begin{align}
\text{III}_{3j,4}
&\leq C\int_0^{t}  |\xi_3| e^{-c\xi_3^2\left(t-\tau\right)} e^{-c\left(t-\tau\right)}\|\p_j(b\cdot \na b)\|_{L^2}d\tau\nonumber\\
&\leq CC_0^2\varepsilon^2\int_0^{t}e^{-c\left(t-\tau\right)}\left(t-\tau \right)^{-\frac12} \left(1+\tau\right)^{-\frac{11}{6}}d\tau  \nonumber\\
&\leq CC_0^2\varepsilon^2\left(1+t\right)^{-\frac{11}{6}},\label{iii3233}
\end{align}
since  Lemma \ref{ED}, together with H\"{o}lder inequality, shows that  for $m_1\in(1,2)$, $\frac{1}{m_1}+\frac{1}{m_2}=1$ and $q>0$,
\begin{align}
&\int_0^{t}\left(t-\tau \right)^{-\frac12}e^{-c\left(t-\tau\right)}\left(1+\tau\right)^{-q}d\tau  \nonumber\\
&\quad\leq \left(\int_0^t\left(t-\tau \right)^{-\frac{m_1}{2}}e^{-\frac {c m_1}{2}\left(t-\tau\right)}d\tau\right)^{
	\frac{1}{m_1}} \left(\int_0^te^{-\frac {cm_2}{2}\left(t-\tau\right)}\left(1+\tau\right)^{-q m_2}d\tau\right)^{\frac{1}{m_2}} \nonumber\\
&\quad\leq C\left(1+t\right)^{-q}.\label{ineq}
\end{align}

Thus, collecting (\ref{p3p2ul2-1}), \eqref{iii3231-2} and \eqref{iii3233} together implies that for $\forall$ $j\in\{2,3\}$,
\begin{align*}
\|\p_3\p_{j}u(t)\|_{L^2}
&\leq \left( C\varepsilon +C C_0^2 \varepsilon^2\right)\left(1+t\right)^{-\frac32}\leq \frac{C_0}{2} \varepsilon \left(1+t\right)^{-\frac32}, 
\end{align*}
provided $C_0$ is chosen to be large enough and $\varepsilon$ is chosen to be  suitably small.

\vskip .1in

\textbf{Step III-2. The decay estimates of $\|\p_2\p_{l}u\|_{L^2}$ with $l=1,2$}
\vskip .1in
In view of (\ref{S1}), (\ref{S2}) and Corollary \ref{PI} with $\alpha=1$, we infer from \eqref{iiiuL2} with $i=2$ and $j=1$ that
\begin{align}
\text{III}_{21,1}+\text{III}_{21,2}
&\leq  C\|\xi_2\xi_1e^{-c\xi_{\nu}^2t} (\widehat{u_0},\widehat{b_0} )\|_{L^2}+ Ce^{-ct} \left\|\xi_2\xi_1(\widehat{u_0},\widehat{b_0})\right\|_{L^2}\nonumber\\
&\leq  C(1+t)^{-1}\left(\|\p_1(u_0, b_0)\|_{L^2_{x_1}L^1_{x_2, x_3}}+\|\p_2\p_1 (u_0, b_0)\|_{L^2}\right)\nonumber\\
&\leq C \varepsilon\left(1+t\right)^{-1}.\label{p21ug-1}
\end{align}

For  $\text{III}_{21,3}$ and $\text{III}_{21,4}$, we can deal with the terms involving $b\cdot\nabla b$ as follows:
\begin{align}
\text{III}_{21,3}+\text{III}_{21,4}
&\leq C\int_0^{t}\left(1+t-\tau \right)^{-1}\|\p_1(b\cdot \na b)  \|_{L_{x_1}^2L_{x_2,x_3}^1}d\tau\nonumber\\
&\quad+C\int_{0}^{t}e^{-c\left(t-\tau\right)}\|\p_2\p_1(b\cdot \na b)\|_{L^2}d\tau\nonumber\\
&\leq
CC_0^2\varepsilon^2\int_0^t\left(1+t-\tau \right)^{-1} \left(1+\tau\right)^{-\frac{5}{4}}d\tau\nonumber\\
&\quad+CC_0^2 \varepsilon^2\int_{0}^{t} e^{-c\left(t-\tau\right)}\left(1+\tau\right)^{-\frac{11}{12}}d\tau \nonumber\\
&\leq  C C_0^2  \varepsilon^2\left(1+t\right)^{-\frac{11}{12}},
\label{piug-2}
\end{align}
where we have   used  \eqref{1bb11}, \eqref{lhbb2}  Lemmas \ref{ID} and \ref{ED}. So, combining (\ref{p21ug-1}) with (\ref{piug-2}) gives
\begin{align*}
\|\p_2\p_{1}u(t)\|_{L^2}
&\leq C\varepsilon(1+t)^{-1} +C C_0^2 \varepsilon^2\left(1+t\right)^{-\frac{11}{12}},\quad \forall\ t\geq1.\nonumber 
\end{align*}
The decay estimate of $\|\p_2^2u\|_{L^2}$ can be achieved in a similar manner by applying Corollary \ref{PI} with $\alpha=2$, and the decay rate is  of the same order (i.e.  $11/12$), since it is completely determined by the estimate of $\|\p_2^2(b\cdot\nabla b)\|_{L^2}$ in \eqref{lhbb2}. As a result, by choosing $C_0$ large enough and $\varepsilon$ suitably small, we arrive at
$$
\|\p_2\p_{l}u(t)\|_{L^2}\leq
\frac{ C_0}{2} \varepsilon\left(1+t\right)^{-\frac{11}{12}},\quad \forall\ l\in\{1,2\}.
$$

\vskip .1in

\textbf{Step III-3.  The decay estimate of $\|\p_3\p_1u\|_{L^2}$}
\vskip .1in
Similarly to that in Step III-2, by (\ref{S1}), (\ref{S2}) and (\ref{13bbl2}) we deduce from (\ref{iiiuL2}) with $i=3$ and $j=1$ that
\begin{align*}
\|\p_3\p_{1}u \|_{L^2}
&\leq   C(1+t)^{-1}\left(\|\p_1 (u_0, b_0 )\|_{L^2_{x_1}L^1_{x_2, x_3}}+\|\p_1\p_3 (u_0, b_0 )\|_{L^2}\right)\nonumber\\
&\quad+C\int_0^{t}\left(1+t-\tau \right)^{-1}\|\p_1(b\cdot \na b)  \|_{L_{x_1}^2L_{x_2,x_3}^1}d\tau\nonumber\\
&\quad+C\int_{0}^{t}e^{-c\left(t-\tau\right)}\|\p_3\p_1(b\cdot \na b)\|_{L^2}d\tau\nonumber\\
&\leq  C\varepsilon (1+t)^{-1}+CC_0^2\varepsilon^2\int_0^t\left(1+t-\tau \right)^{-1} \left(1+\tau\right)^{-\frac{5}{4}}d\tau\nonumber\\
&\quad+CC_0^2 \varepsilon^2\int_{0}^{t} e^{-c\left(t-\tau\right)}\left(1+\tau\right)^{-\frac{5}{3}}d\tau \nonumber\\
&\leq C\varepsilon (1+t)^{-1}+ C C_0^2  \varepsilon^2\left(1+t\right)^{-1}\leq \frac{ C_0}{2} \varepsilon\left(1+t\right)^{-1},
\end{align*}
provided $C_0$ is chosen to be large enough and $\varepsilon$ is chosen to be  suitably small.

\vskip .1in

\textbf{Step III-4.  The decay estimate  of $\|\p_1^2u\|_{L^2}$}
\vskip .1in
In order to estimate $\|\p_1^2u\|_{L^2}$, by (\ref{S1}), (\ref{S2}) and Corollary \ref{PI} with $\alpha=0$ we first infer from (\ref{iiiuL2}) with $i=j=1$ that
\begin{align}
\text{III}_{11,1}+\text{III}_{11,2}
&\leq  C\|\xi_1^2e^{-c\xi_{\nu}^2t} (\widehat{u_0},\widehat{b_0} )\|_{L^2}+ Ce^{-ct} \|\xi_1^2(\widehat{u_0},\widehat{b_0}) \|_{L^2}\nonumber\\
&\leq  C(1+t)^{-\frac{1}{2}}\left(\|\p_1^2 (u_0, b_0 )\|_{L^2_{x_1}L^1_{x_2, x_3}}+\|\p_1^2 (u_0, b_0 )\|_{L^2}\right)\nonumber\\
&\leq C\varepsilon\left(1+t\right)^{-\frac12}.\label{12u-1}
\end{align}

Next, it follows from (\ref{S1}), (\ref{S2}), \eqref{11bb11},  \eqref{lhbb2}, Corollary  \ref{PI} and  Lemma \ref{ID} that for any $\delta\in(0,1)$,
\begin{align}
\text{III}_{11,3}
&\leq   C\int_0^{t} (1+t-\tau  )^{-\frac12}\|\p_1^2( b\cdot \na b )\|_{L_{x_1}^2L_{x_2,x_3}^1}  d\tau \nonumber\\
&\leq  CC_0\varepsilon^2\int_0^{t} (1+t-\tau  )^{-\frac12} (1+\tau )^{-\frac78}  d\tau \nonumber\\
&\quad+CC_0\varepsilon\int_0^{t} (1+t-\tau  )^{-\frac12} (1+\tau )^{-\frac12}\|\p_{\nu}b\|_{H^2}d\tau\nonumber\\
&\leq  C C_0 \varepsilon^2\left(1+t\right)^{-\frac38}+ CC_0\varepsilon^2\left(\int_0^{t}\left(1+t-\tau \right)^{-1}\left(1+\tau\right)^{-1}d\tau\right)^{\frac12}\nonumber\\
&\leq  C C_0 \varepsilon^2\left(1+t\right)^{-\frac38}+CC_0\varepsilon^2\left(1+t\right)^{-\frac{1-\delta}{2}}\label{piug}
\end{align}
and
\begin{align}
 \text{III}_{11,4}
&\leq   C\int_0^{t}  e^{-c (t-\tau )}\|\p_1^2(b\cdot \na b)\|_{L^2}  d\tau \nonumber\\
&\leq  CC_0\varepsilon^2\int_0^{t} e^{-c (t-\tau )} (1+\tau )^{-\frac{11}{12}} d\tau  \leq  C C_0 \varepsilon^2\left(1+t\right)^{-\frac{11}{12}}.\label{piug-1}
\end{align}

Thus, by choosing $\delta\in(0,1/4]$,  we conclude from (\ref{12u-1})--(\ref{piug-1}) that
$$
\|\p_1^2u(t)\|_{L^2}\leq C_3\varepsilon\left(1+t\right)^{-\frac12}+C_4C_0 \varepsilon^2\left(1+t\right)^{-\frac38},\quad\forall\ t\geq1.
$$
So, if $C_0$ and $\varepsilon$ are chosen to be such that
$$
C_0\geq 4C_3\quad\text{and}\quad 0<\varepsilon\leq (4C_4)^{-1},
$$
then it holds that
$$
\|\p_1^2u(t)\|_{L^2}\leq \frac{ C_0}{2} \varepsilon\left(1+t\right)^{-\frac38},\quad\forall\ t\geq1.
$$
\vskip .1in

\textbf{Step  IV.  The decay estimates of $\|\p_i\p_j\p_k u\|_{L^2}$ with  $i,j,k\in\{1,2,3\}$}
\vskip .1in

For $i,j,k\in\{1,2,3\}$, it readily follows from  \eqref{u} that
\begin{align}
\|\p_i\p_j\p_k u\|_{L^2}&=\|\widehat{\p_{ijk}u}\|_{L^2}
=\|\xi_{i}\xi_{j}\xi_{k}\widehat{u}\|_{L^2}\nonumber\\
&\leq \left\| \xi_{i}\xi_{j}\xi_{k}\widehat{K_1}(t)\widehat{u_0}\right\|_{L^2}+\left\|\xi_{i}\xi_{j}\xi_{k}\widehat{K_2}(t)\widehat{b_0}\right\|_{L^2}\nonumber\\
&\quad+\int_0^t\left\|\xi_{i}\xi_{j}\xi_{k}\widehat{K_1}(t-\tau)\widehat{N_1}(\tau)\right\|_{L^2}d\tau\nonumber\\
&\quad+\int_0^t\left\|\xi_{i}\xi_{j}\xi_{k}\widehat{K_2}(t-\tau)\widehat{N_2}(\tau)\right\|_{L^2}d\tau\triangleq\sum_{m=1}^{4}\text{IV}_{ijk,m}.\label{p3uL2}
\end{align}

\textbf{Step IV-1.  The decay estimates of $\|\p_{l}\p_2\p_{3}u\|_{L^2}$ with $l\in\{1,2\}$}
\vskip .1in
Taking $i=1, j=2$ and $k=3$ in \eqref{p3uL2},   by (\ref{S1}), (\ref{S2}) and Corollary \ref{PI} with $\alpha=2$  we derive from \eqref{1bb11}  that
\begin{align}
\text{IV}_{123,1}+\text{IV}_{123,2}
&\leq  C\|\xi_1\xi_2\xi_3e^{-c\xi_{\nu}^2t} (\widehat{u_0},\widehat{b_0} )\|_{L^2}+ Ce^{-ct} \left\|\xi_1\xi_2\xi_3(\widehat{u_0},\widehat{b_0})\right\|_{L^2}\nonumber\\
&\leq  C(1+t)^{-\frac{3}{2}} \|\p_1 (u_0, b_0 )\|_{L^2_{x_1}L^1_{x_2 ,x_3}}+Ce^{-ct}\|\p_1\p_2\p_3 (u_0, b_0 )\|_{L^2} \nonumber\\
&\leq C \varepsilon\left(1+t\right)^{-\frac{3}{2}}.\label{iv12k1}
\end{align}

Using (\ref{KS23}), (\ref{S1}) and (\ref{ineq}),  we deduce from \eqref{1bb11}, \eqref{lhbb2} and Corollary \ref{PI} in a similar manner as the derivation of (\ref{iii3233}) that
\begin{align*}
\text{IV}_{123,3}+\text{IV}_{123,4}
&\leq C\int_0^t\left\|\xi_1\xi_2\xi_3e^{-c\xi_{\nu}^2\left(t-\tau \right)}\widehat{b\cdot \na b} \right\|_{L^2}d\tau\nonumber\\
&\quad+C\int_0^t\left\|\xi_1\xi_2\xi_3e^{-c\left(1+\xi_3^2\right)\left(t-\tau \right)}\widehat{b\cdot \na b} \right\|_{L^2}d\tau\nonumber\\
&\leq    C\int_0^{t}\left(1+t-\tau \right)^{-\frac{3}{2}}\|\p_1(b\cdot \na b)  \|_{L_{x_1}^2L_{x_2,x_3}^1}d\tau\nonumber\\
&\quad+C\int_{0}^{t}e^{-c\left(t-\tau\right)}\left(t-\tau\right)^{-\frac12}\|\p_1\p_2(b\cdot \na b)\|_{L^2}d\tau\nonumber\\
&\leq  CC_0^2\varepsilon^2\int_0^t \left(1+t-\tau \right)^{-\frac{3}{2}} \left(1+\tau\right)^{-\frac{5}{4}}d\tau\nonumber\\
&\quad+CC_0 \varepsilon^2\int_{0}^{t} e^{-c\left(t-\tau\right)}\left(t-\tau\right)^{-\frac12}\left(1+\tau\right)^{-\frac{11}{12}}d\tau \nonumber\\
&\leq C  C_0^2\varepsilon^2(1+t)^{-\frac54}+ C C_0  \varepsilon^2\left(1+t\right)^{-\frac{11}{12}},
\end{align*}
which, together with (\ref{iv12k1}), gives
\begin{align}
\|\p_{1}\p_2\p_{3}u(t)\|_{L^2}&\leq  C \varepsilon\left(1+t\right)^{-\frac{3}{2}}+ C  C_0^2\varepsilon^2(1+t)^{-\frac54}+ C C_0  \varepsilon^2\left(1+t\right)^{-\frac{11}{12}}\nonumber\\
&\leq
C \varepsilon\left(1+t\right)^{-\frac{3}{2}}+ C  C_0^2\varepsilon^2(1+t)^{-\frac{11}{12}}\leq  \frac{C_0}{2}\varepsilon (1+t)^{-\frac{11}{12}},\label{iv12k}
\end{align}
provided $C_0$ is chosen to be large enough and $\varepsilon $ is chosen to be suitably small.

Analogously, applying Corollary \ref{PI}
 with $\alpha=3$ to the terms $\| |\xi_{2}\xi_{2}\xi_{3}|e^{-c\xi_{\nu}^2 t }\widehat{u_0} \|_{L^2}$ and $\||\xi_2\xi_2\xi_3|e^{-c\xi_{\nu}^2 (t-\tau )}\widehat{b\cdot \na b}\|_{L^2}$, we can make use of \eqref{pbl1} and \eqref{lhbb2}  to obtain the same estimate as that in (\ref{iv12k}) for $\|\p_{2}\p_2\p_{3}u(t)\|_{L^2}$. Hence,
$$
\|\p_{l}\p_2\p_{3}u(t)\|_{L^2}\leq  \frac{C_0}{2}\varepsilon (1+t)^{-\frac{11}{12}},\quad\forall\ l\in\{1,2\}.
$$

\vskip .1in

\textbf{Step IV-2.  The decay estimate of $\|\p_1^2\p_{3}u\|_{L^2}$}
\vskip .1in
To derive the decay estimate of $\|\p_1^2\p_{3}u\|_{L^2}$,  we first infer from (\ref{KS23}) and (\ref{S1}) and Corollary \ref{PI} with $\alpha=1$ that
\begin{align}
\|\p_1^2\p_{3}u\|_{L^2}
&\leq   C(1+t)^{-1}\left(\|\p_1^2 (u_0, b_0 )\|_{L^2_{x_1}L^1_{x_2 ,x_3}}+\|\p_1^2\p_3 (u_0, b_0 )\|_{L^2}\right)\nonumber\\
&\quad+ C\int_0^t\left\|\xi_1^2\xi_3e^{-c\xi_{\nu}^2\left(t-\tau \right)}\widehat{b\cdot \na b}(\tau)\right\|_{L^2}d\tau\nonumber\\
&\quad+C\int_0^t\left\|\xi_1^2\xi_3e^{-c\left(1+\xi_3^2\right)\left(t-\tau \right)}\widehat{b\cdot \na b}(\tau)\right\|_{L^2}d\tau\nonumber\\
&\leq  C\varepsilon(1+t)^{-1} +\text{IV}_{113,3}+\text{IV}_{113,4}.\label{iv1131}
\end{align}

Noting that
$$
\int_0^t\|\p_{\nu}b\|_{H^2}^2d\tau\leq C\varepsilon^2
$$
due to Theorem \ref{thm1.1},
by (\ref{11bb11}), Corollary \ref{PI} with $\alpha=1$ and Lemma \ref{ID} we have
\begin{align}
\text{IV}_{113,3}
&\leq    C\int_0^{t}\left(1+t-\tau \right)^{-1}\|\p_1^2\left(b\cdot \na b\right)  \|_{L_{x_1}^2L_{x_2,x_3}^1}d\tau\nonumber\\
&\leq
CC_0 \varepsilon \int_0^{t}\left(1+t-\tau\right)^{-1}   \left(1+\tau\right)^{-\frac12}\|\p_{\nu}b\|_{H^2} d\tau\nonumber\\
&\quad
+CC_0 \varepsilon^2\int_0^{t}\left(1+t-\tau\right)^{-1}   \left(1+\tau\right)^{-\frac78}d\tau \nonumber\\
&\leq C C_0 \varepsilon^2\left(1+t\right)^{-\frac78+\delta}+CC_0 \varepsilon^2 \left(\int_0^{t}\left(1+t-\tau\right)^{-2}   \left(1+\tau\right)^{-1}d\tau\right)^{\frac12}  \nonumber\\
&\leq C C_0\varepsilon^2\left(1+t\right)^{-\frac12}, \label{iv1133}
\end{align}
provided $\delta$ is chosen to be such that  $0<\delta\leq \frac38$. The term $\text{IV}_{113,4}$ can be bounded  by (\ref{ineq}) as follows,
\begin{align*}
\text{IV}_{113,4}
&\leq    C\int_{0}^{t}e^{-c\left(t-\tau\right)}\left(t-\tau\right)^{-\frac12}\|\p_1^2(b\cdot \na b)\|_{L^2}d\tau\nonumber\\
&\leq CC_0 \varepsilon^2\int_{0}^{t} e^{-c\left(t-\tau\right)}\left(t-\tau\right)^{-\frac12}\left(1+\tau\right)^{-\frac{11}{12}}d\tau \nonumber\\
&\leq   C C_0  \varepsilon^2\left(1+t\right)^{-\frac{11}{12}},
\end{align*}
which, combined with (\ref{iv1131}) and (\ref{iv1133}), gives rise to
$$
\|\p_1^2\p_{3}u(t)\|_{L^2}
 \leq  C\varepsilon(1+t)^{-1} +C C_0\varepsilon^2\left(1+t\right)^{-\frac12}\leq \frac{C_0}{2}\varepsilon\left(1+t\right)^{-\frac12},
 $$
provided $C_0$ is chosen to be large enough and $\varepsilon $ is chosen to be suitably small.

\vskip .1in

\textbf{Step IV-3.  The decay estimate of $\|\p_1\p_{3}^2u\|_{L^2}$}
\vskip .1in

Similarly to the arguments in Step IV-2, choosing $C_0$ large enough and $\varepsilon$ suitably small,  we   deduce from  \eqref{1bb11} and \eqref{13bbl2} that
\begin{align*}
\|\p_{1}\p_3^2u\|_{L^2}
&\leq  C(1+t)^{-\frac32}\left(\|\p_{1} (u_0, b_0 )\|_{L^2_{x_1}L^1_{x_2, x_3}}+\|\p_{1}\p_3^2 (u_0, b_0 )\|_{L^2}\right)\nonumber\\
&\quad+ C\int_0^{t}\left(1+t-\tau \right)^{-\frac32}\|\p_{1} (b\cdot \na b   )\|_{L_{x_1}^2L_{x_2x_3}^1}d\tau\nonumber\\
&\quad+C\int_{0}^{t}e^{-c\left(t-\tau\right)}\left(t-\tau\right)^{-\frac12}\|\p_{1}\p_3(b\cdot \na b)\|_{L^2}d\tau\nonumber\\
&\leq C \varepsilon\left(1+t\right)^{-\frac32}
+CC_0^2\varepsilon^2\int_0^t\left(1+t-\tau \right)^{-\frac32} \left(1+\tau\right)^{-\frac{5}{4}}d\tau\nonumber\\
&\quad+CC_0^2 \varepsilon^2\int_{0}^{t} e^{-c\left(t-\tau\right)}\left(t-\tau\right)^{-\frac12}\left(1+\tau\right)^{-\frac{5}{3}}d\tau \nonumber\\
&\leq  \left(C\varepsilon + C C_0^2 \varepsilon^2\right)\left(1+t\right)^{-\frac{5}{4}}\leq \frac{ C_0}{2} \varepsilon\left(1+t\right)^{-\frac{5}{4}}.
\end{align*}

\vskip .1in

\textbf{Step IV-4. The decay estimate of $\|\p_2\p_{3}^2u\|_{L^2}$}
\vskip .1in
In view of (\ref{pbl1}), (\ref{23bb2}), (\ref{33bbl2}), (\ref{ineq}) and Corollary \ref{PI} with $\alpha=3$,  we have from \eqref{p3uL2}, (\ref{S1}) and (\ref{S2}) that
\begin{align*}
\|\p_{2}\p_3^2u\|_{L^2}
&\leq C\varepsilon(1+t)^{-2}+C\int_0^{\frac t2}\left(1+t-\tau \right)^{-2}\|b\cdot \na b\|_{L_{x_1}^2L_{x_2,x_3}^1}d\tau\nonumber\\
&\quad+C\int_{\frac t2}^{t}\left(t-\tau \right)^{-\frac12} \|\p_3^2(b\cdot \na b)\|_{L^2}d\tau \nonumber\\
&\quad+ C\int_{0}^{t}e^{-c\left(t-\tau\right)}\left(t-\tau\right)^{-\frac12}\|\p_{2}\p_3(b\cdot \na b)\|_{L^2}d\tau\nonumber\\
&\leq C\varepsilon(1+t)^{-2}+CC_0^2\varepsilon^2\int_0^{\frac t2}\left(1+t-\tau \right)^{-2} \left(1+\tau\right)^{-\frac{11}{8}}d\tau\nonumber\\
&\quad+CC_0^2\varepsilon^2\int_{\frac t2}^{t} \left(t-\tau\right)^{-\frac12}\left(1+\tau \right)^{-\frac{29}{12}} d\tau\nonumber\\
&\quad +CC_0^2 \varepsilon^2\int_{0}^{t} e^{-c\left(t-\tau\right)}\left(t-\tau\right)^{-\frac12}\left(1+\tau\right)^{-\frac{11}{6}}d\tau \nonumber\\
&\leq   C\varepsilon(1+t)^{-2}+ C C_0^2 \varepsilon^2 \left(1+t\right)^{-\frac{11}{6}},
\end{align*}
from which we obtain after choosing $C_0$ large  and $\varepsilon$ suitably   small that
\begin{align*}
\|\p_{2}\p_3^2u(t)\|_{L^2}
&\leq    C\varepsilon(1+t)^{-2}+C  C_0^2 \varepsilon^2\left(1+t\right)^{-\frac{11}{6}}\leq \frac{C_0}{2} \varepsilon\left(1+t\right)^{-\frac{11}{6}}.
\end{align*}

\vskip .1in

\textbf{Step IV-5.  The decay estimate of $\|\p_3^3u\|_{L^2}$}
\vskip .1in
Analogously  to that in Step IV-4, by \eqref{pbl1} and  \eqref{33bbl2}  we have
\begin{align*}
\|\p_3^3 u\|_{L^2}
&\leq C\varepsilon (1+t)^{-2} + C\int_0^{\frac{t}{2}} ( 1+ t-\tau  )^{-2}\|b\cdot \na b\|_{L_{x_1}^2L_{x_2,x_3}^1}d\tau\nonumber\\
&\quad +C\int_{\frac{t}{2}}^{t}  (t-\tau )^{-\frac12}\|\p_3^2(b\cdot \na b)\|_{L^2}d\tau\nonumber\\
 & \quad+C\int_{0}^{t} e^{-c (t-\tau )} (t-\tau )^{-\frac12}\|\p_3^2(b\cdot \na b)\|_{L^2}d\tau \nonumber\\
&\leq  C \varepsilon (1+t )^{-2}+CC_0^2\varepsilon^2\int_0^{\frac t2} (1+t-\tau )^{-2}  (1+\tau )^{-\frac{11}{8}}d\tau\nonumber\\
&\quad+CC_0^2\varepsilon^2\int_{\frac t2}^{t} (t-\tau)^{-\frac12} (1+\tau  )^{-\frac{29}{12}} d\tau\nonumber\\
&\quad +CC_0^2 \varepsilon^2\int_{0}^{t} e^{-c (t-\tau )}(t-\tau )^{-\frac12} (1+\tau )^{-\frac{29}{12}}d\tau \nonumber\\
&\leq C \varepsilon (1+t )^{-\frac{23}{12}}+C C_0^2  \varepsilon^2 (1+t )^{-\frac{23}{12}}+C C_0^2  \varepsilon^2 (1+t )^{-\frac{29}{12}},
\end{align*}
so that, choosing $C_0$ large enough and $\varepsilon$ suitably small, we find
$$
\|\p_3^3 u(t)\|_{L^2}\leq \frac{C_0}{2}  \varepsilon  (1+t )^{-\frac{23}{12}},\quad\forall\ t\geq1.
$$

\vskip .1in
\subsection{Enhanced decay rates of $(u_1,b_1)$.}

In this subsection, we aim to improve the decay rates of $(u_1,b_1)$ by applying Corollaries \ref{PI} and   \ref{divfree}.

\vskip .1in

\textbf{Part I. Improved decay rate of  $\|u_1\|_{L^2}$}
\vskip .1in

We start with the improved decay estimate of $\|u_1\|_{L^2}$. First, it follows from   Plancherel's Theorem  and  \eqref{u} that
\begin{align}
\|u_1\|_{L^2}=\|\widehat{u_1}\|_{L^2}&\leq \left\| \widehat{K_1}(t)\widehat{u_{10}}\right\|_{L^2}+\left\|\widehat{K_2}(t)\widehat{b_{10}}\right\|_{L^2}\nonumber\\
&\quad+\int_0^t\left\|\widehat{K_1}(t-\tau)\widehat{N_{11}}(\tau)\right\|_{L^2}d\tau\nonumber\\
&\quad+\int_0^t\left\|\widehat{K_2}(t-\tau)\widehat{N_{21}}(\tau)\right\|_{L^2}d\tau\triangleq \sum_{m=1}^4\Lambda_{1,m}\label{u1L2}
\end{align}

 Since $\nabla\cdot u_0=\nabla\cdot b_0=0$, by \eqref{S1} and \eqref{S2}  we can make use of Corollary \ref{divfree} with $\beta=0$ and   Plancherel's Theorem to bound the first two terms on the right-hand side of (\ref{u1L2}) as follows,
\begin{align}
\Lambda_{1,1}+\Lambda_{1,2} &\leq C\|e^{-c\xi_{\nu}^2t}(\widehat{u_{10}},\widehat{b_{10}})\|_{L^2}+ Ce^{-ct}  \|(\widehat{u_{10}},\widehat{b_{10}}) \|_{L^2}\nonumber\\
&\leq C(1+t)^{-\frac{3}{4}}\left(\|(u_0,b_{0})\|_{L^1} +\|(u_0,b_0)\|_{L^2 }\right)\nonumber\\
&\leq C\varepsilon (1+t)^{-\frac{3}{4}}.\label{i12}
\end{align}

Since $N_{11}=\mathbb P_1(b\cdot\nabla b)+\mathbb P_1(u\cdot\nabla u)$, it holds that
\begin{align}
\Lambda_{1,3}
&\leq \int_0^t\left\|\widehat{K_1}(t-\tau)\widehat{\mathbb P_1\left(b\cdot\nabla b\right)}\right\|_{L^2}d\tau\nonumber\\
&\quad+\int_0^t\left\|\widehat{K_1}(t-\tau)\widehat{\mathbb P_1\left(u\cdot\nabla u\right)}\right\|_{L^2}d\tau \triangleq \Lambda_{1,31}+\Lambda_{1,32}.\label{i3}
\end{align}

Clearly, it suffices to consider the first  term $\Lambda_{1,31}$, since the second one can be treated in the same manner due to their  mathematical resemblance. To do this, we first infer from \eqref{P1} that
\begin{align}
\Lambda_{1,31}
&\leq C\int_0^t\left\|e^{-c\xi_{\nu}^2\left(t-\tau \right)}\sum_{k=1}^{3}\xi_{\nu}^2|\xi|^{-2}\xi_k\widehat{b_kb_1}\right\|_{L^2}d\tau\nonumber\\
&\quad+C\int_0^t\left\|e^{-c\xi_{\nu}^2\left(t-\tau \right)}\sum_{k=1}^{3}\sum_{l=2}^{3}\xi_1|\xi|^{-2}\xi_k\xi_l\widehat{b_kb_l}\right\|_{L^2}d\tau\nonumber\\
&\quad+C\int_0^t\left\|e^{-c\left(t-\tau \right)}\sum_{k=1}^{3}\xi_{\nu}^2|\xi|^{-2}\xi_k\widehat{b_kb_1}\right\|_{L^2}d\tau\nonumber\\
&\quad+C\int_0^t\left\|e^{-c\left(t-\tau \right)}\sum_{k=1}^{3}\sum_{l=2}^{3}\xi_1|\xi|^{-2}\xi_k\xi_l\widehat{b_kb_l}\right\|_{L^2}d\tau \triangleq \sum_{m=1}^4 \Lambda_{1,31m} .\label{i3b}
\end{align}

Using Corollary \ref{PI} with $\alpha=1$ and Lemma \ref{ID}, we see that for any $t\geq1$,
\begin{align}
\Lambda_{1,311}
&\leq C\int_0^t\left\||\xi_\nu| e^{-c\xi_{\nu}^2\left(t-\tau \right)}\sum_{k=1}^{3}|\widehat{b_kb_1}|\right\|_{L^2}d\tau \nonumber\\
&\leq C\int_0^{t}\left(1+t-\tau \right)^{-1}\left\|\| b b_1 \|_{L_{x_2,x_3}^1}\right\|_{L_{x_1}^2}d\tau \nonumber\\
&\leq CC_0^2\varepsilon^2\int_0^t\left(1+t-\tau \right)^{-1} \left(1+\tau\right)^{-\frac{11}{8}}d\tau  \nonumber\\
&\leq CC_0^2\varepsilon^2\left(1+t\right)^{-1}\leq C C_0^2 \varepsilon^2\left(1+t\right)^{-\frac34},  \label{i31g}
\end{align}
where we have used  (\ref{1IE}), (\ref{MIE}) and the divergence-free condition $\nabla\cdot b=0$ to get from (\ref{u10}) that
\begin{align}
&\left\|\| bb_1  \|_{L_{x_2,x_3}^1}\right\|_{L_{x_1}^2}\leq \left\|\| bb_1  \|_{L_{x_1}^2} \right\|_{L_{x_2,x_3}^1}\leq \left\|\| b\|_{L^2_{x_1}}\|b_1  \|_{L_{x_1}^\infty} \right\|_{L_{x_2,x_3}^1}\notag\\
&\quad\leq  C\left\|\| b\|_{L^2_{x_1}}\|b_1  \|_{L_{x_1}^2}^{\frac12} \|\p_1 b_1  \|_{L_{x_1}^2}^{\frac12} \right\|_{L_{x_2,x_3}^1} \leq  C\| b\|_{L^2}\|b_1  \|_{L^2}^{\frac12} \|\p_1 b_1  \|_{L^2}^{\frac12}   \notag\\
&\quad\leq  C\| b\|_{L^2}\|b_1  \|_{L^2}^{\frac12} \|\p_\nu b   \|_{L^2}^{\frac12}  \leq CC_0^2\varepsilon^2\left(1+t\right)^{-\frac{11}{8}}.\label{bb11}
\end{align}

Similarly, by choosing $0<\delta\leq1/4$ we can bound $\Lambda_{1,312}$ as follows,
\begin{align}
\Lambda_{1,312}
&\leq C\int_0^{t}\left(1+t-\tau \right)^{-1}\left\|\| (b b_2, bb_3) \|_{L_{x_2,x_3}^1}\right\|_{L_{x_1}^2}d\tau \nonumber\\
&\leq CC_0^2\varepsilon^2\int_0^t\left(1+t-\tau \right)^{-1} \left(1+\tau\right)^{-1}d\tau\notag\\
& \leq  C C_0^2 \varepsilon^2\left(1+t\right)^{-1+\delta}\leq C C_0^2 \varepsilon^2\left(1+t\right)^{-\frac34}, \label{i32g}
\end{align}
since it follows from (\ref{u10}) in a manner similar to the derivation of (\ref{bb11}) that
\begin{align}\label{1bbv2}
\left\|\| bb_{\nu} \|_{L_{x_2,x_3}^1}\right\|_{L_{x_1}^2}
\leq \| b \|_{L^2}\| b_{\nu}\|_{L^2}^{\frac12}\|\p_1 b_{\nu}\|_{L^2}^{\frac12} \leq CC_0^2\varepsilon^2\left(1+t\right)^{-1}.
\end{align}

Due to (\ref{u10}), (\ref{binfty}), (\ref{b1infty}) and the fact that $\nabla\cdot b=0$, one has
\begin{align}
\| b\cdot \na b_1 \|_{L^2}
&\leq \| b_1 \|_{L^\infty}\| \p_1 b_1\|_{L^2}+\| b_\nu \|_{L^\infty}\| \p_\nu b_1\|_{L^2} \nonumber\\
&\leq \| b_1 \|_{L^\infty}\| \p_\nu b\|_{L^2}+\| b_\nu \|_{L^\infty}\| \p_\nu b_1\|_{L^2} \nonumber\\
& \leq CC_0^2\varepsilon^2\left(1+t\right)^{-\frac{13}{6}} \label{bb2}
\end{align}
and
\begin{align}
\|(\p_2(bb_2),\p_3(bb_3)) \|_{L^2}
& \leq C \| b  \|_{L^\infty}\| \p_\nu b \|_{L^2}  \leq   CC_0^2\varepsilon^2\left(1+t\right)^{-\frac{23}{12}},\label{2bbl2}
\end{align}
so that, by Lemma \ref{ED} we have
\begin{align}
\Lambda_{1,313}+ \Lambda_{1,314}
&\leq C\int_0^te^{-c (t-\tau  )}\left(\| b\cdot\na b_1 \|_{L^2}+\|\p_2 (b b_2 )\|_{L^2}+\|\p_3 (b b_3)\|_{L^2}\right)d\tau  \nonumber\\
&\leq  CC_0^2\varepsilon^2\int_0^te^{-c\left(t-\tau \right)} \left(1+\tau\right)^{-\frac{23}{12}}d\tau\leq C C_0^2\varepsilon^2\left(1+t\right)^{-\frac34}. \label{i33}
\end{align}

Thus, based upon \eqref{i31g}, \eqref{i32g}  and \eqref{i33}, we conclude from (\ref{i3b}) that
\begin{align*}
\Lambda_{1,31}
\leq C C_0^2 \varepsilon^2\left(1+t\right)^{-\frac34},\quad \forall\ t\geq1.
\end{align*}
The same estimate also holds for $\Lambda_{1,32}$. Thus,
it follows from \eqref{i3} that
\begin{align}
\Lambda_{1,3}
\leq C C_0^2 \varepsilon^2\left(1+t\right)^{-\frac34},\quad \forall\ t\geq1.\label{i13g}
\end{align}

We proceed to estimate $\Lambda_{1,4}$. Noting that
\begin{align}
N_{21}&=b\cdot \nabla u_1-u\cdot \nabla b_1=\nabla \cdot (b u_1-ub_1)\nonumber\\
&=\p_2(b_2 u_1-u_2b_1)+\p_3(b_3 u_1-u_3b_1), \label{N21}
\end{align}
we have from \eqref{S1} and \eqref{S2} that
\begin{align}
\Lambda_{1,4}
&\leq \int_0^t\left\|e^{-c\xi_{\nu}^2\left(t-\tau \right)}\xi_{2}\widehat{b_2 u_1}(\tau) \right\|_{L^2}d\tau+\int_0^t\left\|e^{-c\xi_{\nu}^2\left(t-\tau \right)}\xi_{2}\widehat{u_2 b_1}(\tau) \right\|_{L^2}d\tau\nonumber\\
&\quad+\int_0^t\left\|e^{-c\xi_{\nu}^2\left(t-\tau \right)}\xi_{3}\widehat{b_3 u_1}(\tau) \right\|_{L^2}d\tau+\int_0^t\left\|e^{-c\xi_{\nu}^2\left(t-\tau \right)}\xi_{3}\widehat{u_3 b_1}(\tau) \right\|_{L^2}d\tau\nonumber\\
&\quad+\int_0^t\left\|e^{-c\left(t-\tau \right)}\xi_{2}\widehat{b_2 u_1}(\tau) \right\|_{L^2}d\tau+\int_0^t\left\|e^{-c\left(t-\tau \right)}\xi_{2}\widehat{u_2 b_1}(\tau) \right\|_{L^2}d\tau\nonumber\\
&\quad+\int_0^t\left\|e^{-c\left(t-\tau \right)}\xi_{3}\widehat{b_3 u_1}(\tau)\right\|_{L^2}d\tau+\int_0^t\left\|e^{-c\left(t-\tau \right)}\xi_{3}\widehat{u_3 b_1}(\tau) \right\|_{L^2}d\tau\nonumber\\
&\triangleq \Lambda_{1,41}+\ldots+\Lambda_{1,48}\label{i4}.
\end{align}

It is easily checked that  the estimates (\ref{bb11}) and (\ref{bb2}) hold for $bu_1$ and $\p_\nu(b_\nu u_1)$, respectively. So, analogously to the derivations of \eqref{i31g} and (\ref{i33}),   we have
\begin{align}
&\Lambda_{1,41}+\Lambda_{1,43}+\Lambda_{1,45}+\Lambda_{1,47}\nonumber \\
&\quad \leq C\int_0^{t}\left[(1+t-\tau )^{-1 } \| b_\nu u_1 \|_{L^2_{x_1}L_{x_2,x_3}^1}+e^{-c (t-\tau )} \|\p_\nu ( b_\nu u_1 )\|_{L^2}\right] d\tau \nonumber\\
&\quad\leq CC_0^2\varepsilon^2\int_0^t \left[(1+t-\tau )^{-1} (1+\tau )^{-\frac{11}{8}}+e^{-c (t-\tau )} (1+\tau )^{-\frac{13}{6}}\right] d\tau \nonumber\\
&\quad\leq C C_0^2 \varepsilon^2\left(1+t\right)^{-1}. \label{i4b}
\end{align}

The terms involving $u_2b_1$ and $u_3b_1$ can be treated similarly and admit the same decay estimate as that in (\ref{i4b}). Consequently,
$$
\Lambda_{1,4}\leq CC_0^2\varepsilon^2(1+t)^{-1},\quad\forall\ t\geq1,
$$
which, together with  \eqref{u1L2},  \eqref{i12} and \eqref{i13g}, yields
\begin{align}
\|u_1(t)\|_{L^2}\leq C_5\varepsilon(1+t)^{-\frac34}+ C_6C_0^2  \varepsilon^2\left(1+t\right)^{-\frac34},\quad \forall \ t\geq1.\label{u1L2g}
\end{align}

Thus, if $C_0$ and $\varepsilon$ are chosen to be such that
\begin{align*}
C_0\geq 4C_5\quad \text{and}\quad  0<\varepsilon\leq (4C_6C_0 )^{-1},
\end{align*}
then one infers from (\ref{u1L2g}) that
$$
\| u_1(t)\|_{L^2}\leq \frac{C_0}{2}\varepsilon(1+t)^{-\frac34},\quad \forall\ t\geq1.
$$

\vskip .1in

\textbf{Part II.   Improved decay rates of $\|\p_\nu u_{1}\|_{L^2}$ with $\nu\in\{2,3\}$}
\vskip .1in

Let $\xi_\nu=(\xi_2,\xi_3)$ and $\xi_\nu^2=\xi_2^2+\xi_3^2$. This step aims  to improve the decay estimate of $\|\p_\nu u_1\|_{L^2}$ by applying Corollaries \ref{PI}, \ref{divfree} and (\ref{P1}). First, it is easy to see that
\begin{align}
\|\p_\nu u_1\|_{L^2}&=\|\widehat{\p_\nu u_1}\|_{L^2}=\|\xi_\nu\widehat{u_1}\|_{L^2}\nonumber\\
&\leq \left\| \xi_\nu\widehat{K_1}(t)\widehat{u_{10}}\right\|_{L^2}+\left\|\xi_\nu\widehat{K_2}(t)\widehat{b_{10}}\right\|_{L^2}\nonumber\\
&\quad+\int_0^t\left\|\xi_\nu\widehat{K_1}(t-\tau)\widehat{N_{11}}(\tau)\right\|_{L^2}d\tau\nonumber\\
&\quad+\int_0^t\left\|\xi_\nu\widehat{K_2}(t-\tau)\widehat{N_{21}}(\tau)\right\|_{L^2}d\tau\triangleq \sum_{m=1}^4\Lambda_{2,m} . \label{iu1L2}
\end{align}

Since $\nabla\cdot u_0=\nabla\cdot b_0=0$, it follows from Corollary \ref{divfree} with $\beta=1$ that
\begin{align}
\Lambda_{2,1}+\Lambda_{2,2}&\leq  C\||\xi_\nu|e^{-c\xi_{\nu}^2t}( \widehat{u_{10}},\widehat{b_{10}})\|_{L^2}+ Ce^{-ct} \|\p_\nu(\widehat{u_{10}},\widehat{b_{10}}) \|_{L^2}\nonumber\\
&\leq C(1+t)^{-\frac{5}{4}}\left(\|(u_{0},b_0)\|_{L^1}+\|\p_\nu (u_0,b_0)\|_{L^2 }\right)\notag
\\
&\leq C\varepsilon(1+t)^{-\frac{5}{4}}. \label{ii1}
\end{align}

To estimate $\Lambda_{2,3}$, similarly to the treatment of (\ref{i3}) in Part I, we only deal with  the term associated with $\widehat{\mathbb{P}_1(b\cdot\nabla b)}$ in $\widehat{N_{11}}$, which is still denoted by $\Lambda_{2,3}$ for simplicity. By (\ref{S1}) and (\ref{S2}), we infer from \eqref{P1} that
\begin{align}
\Lambda_{2,3}
&\leq C\int_0^t\left\|\xi_\nu e^{-c\xi_{\nu}^2\left(t-\tau \right)} \sum_{k=1}^{3}\xi_{\nu}^2|\xi|^{-2}\xi_k\widehat{b_kb_1}\right\|_{L^2}d\tau\notag\\
&\quad+C\int_0^t\left\|\xi_\nu e^{-c\xi_{\nu}^2\left(t-\tau \right)} \sum_{k=1}^{3}\sum_{l=2}^{3}\xi_1|\xi|^{-2}\xi_k\xi_l\widehat{b_kb_l}\right\|_{L^2}d\tau\nonumber\\
&\quad+C\int_0^t\left\|\xi_\nu e^{-c\left(t-\tau \right)} \sum_{k=1}^{3}\xi_{\nu}^2|\xi|^{-2}\xi_k\widehat{b_kb_1}\right\|_{L^2}d\tau\notag\\
&\quad+C\int_0^t\left\|\xi_\nu e^{-c\left(t-\tau \right)} \sum_{k=1}^{3}\sum_{l=2}^{3}\xi_1|\xi|^{-2}\xi_k\xi_l\widehat{b_kb_l}\right\|_{L^2}d\tau
\triangleq \sum_{m=1}^4\Lambda_{2,3m} .\label{i13b}
\end{align}

Using Corollary \ref{PI} with $\alpha=2$, Lemma \ref{ID} and (\ref{bb11}), we obtain
\begin{align}
\Lambda_{2,31}
&\leq C\int_0^{t}\left(1+t-\tau \right)^{-\frac32}\left\|\| b b_1 \|_{L_{x_2,x_3}^1}\right\|_{L_{x_1}^2}d\tau \nonumber\\
&\leq CC_0^2\varepsilon^2 \int_0^t\left(1+t-\tau \right)^{-\frac32} \left(1+\tau\right)^{-\frac{11}{8}}d\tau  \nonumber\\
&\leq C C_0^2 \varepsilon^2\left(1+t\right)^{-\frac54},\quad \forall\ t\geq1.  \label{ii2131-2}
\end{align}

The treatment of $\Lambda_{2,32}$ is slightly different. Indeed, noting that
\begin{align*}
\Lambda_{2,32}
&\leq  C\left(\int_0^{\frac t2}+\int_{\frac t2}^{t}\right)\left\|\xi_\nu e^{-c\xi_{\nu}^2\left(t-\tau \right)}\sum_{k=1}^{3}\sum_{l=2}^{3}\xi_1|\xi|^{-2}\xi_k\xi_l\widehat{b_kb_l}\right\|_{L^2}d\tau \nonumber\\
&\leq C\int_0^{\frac t2}\left\|\xi_\nu^2e^{-c\xi_{\nu}^2\left(t-\tau \right)}  \widehat{b b_\nu}\right\|_{L^2}d\tau+C\int_{\frac t2}^t\left\|\xi_\nu e^{-c\xi_{\nu}^2\left(t-\tau \right)}  \widehat{\p_\nu (b b_\nu)}\right\|_{L^2}d\tau ,
\end{align*}
we infer from Corollary \ref{PI}, (\ref{1bbv2}) and \eqref{2bbl2} that for any $0<\delta<1$,
\begin{align*}
\Lambda_{2,32}& \leq C\int_0^{\frac t2}\left(1+t-\tau \right)^{-\frac32}\left\|\| b b_\nu \|_{L_{x_2,x_3}^1}\right\|_{L_{x_1}^2}d\tau\notag\\
&\quad+C\int_{\frac t2}^{t}\left\|\xi_{\nu} e^{-c\xi_{\nu}^2\left(t-\tau \right)}  \right\|_{L^\infty}  \| \p_\nu (b b_\nu) \|_{L^2} d\tau \nonumber\\
&\leq CC_0^2\varepsilon^2\int_0^{\frac t2}\left(1+t-\tau \right)^{-\frac 32}\left(1+\tau \right)^{-1}d\tau\notag\\
&\quad+CC_0^2\varepsilon^2\int_{\frac t2}^{t} \left(t-\tau \right)^{-\frac12}\left(1+\tau \right)^{-\frac{23}{12}}d\tau \nonumber\\
&\leq CC_0^2\varepsilon^2 \left(1+t \right)^{-\frac 32+\delta}  +CC_0^2\varepsilon^2  \left(1+\tau\right)^{-\frac{17}{12}} ,
\end{align*}
and hence, by choosing $0<\delta\leq 1/4$ we easily deduce
\begin{align}
\Lambda_{2,32}\leq C C_0^{2} \varepsilon^2\left(1+t\right)^{-\frac54}. \label{ii2132-2}
\end{align}

Next, by virtue of (\ref{22bb2}), (\ref{23bb2}) and Lemma \ref{ED}, we see that
\begin{align}
\Lambda_{2,33}+ \Lambda_{2,34}
&\leq C\int_0^te^{-c\left(t-\tau \right)}\left\|\widehat{\p_\nu\left(b\cdot\na b\right)}\right\|_{L^2}d\tau\nonumber\\
&\leq C C_0^{2} \varepsilon^2\int_0^te^{-c\left(t-\tau \right)} \left(1+\tau\right)^{-\frac{11}{6}}  d\tau \nonumber\\
&\leq C C_0^{2} \varepsilon^2\left(1+t\right)^{-\frac54}. \label{ii2133}
\end{align}

Thus, substituting \eqref{ii2131-2}, \eqref{ii2132-2}  and \eqref{ii2133} into (\ref{i13b}) gives
\begin{align}
\Lambda_{2,3}
\leq C  C_0^2  \varepsilon^2\left(1+t\right)^{-\frac54}.\label{ii33}
\end{align}

For $\Lambda_{2,4}$, analogously to treatment of (\ref{i4}) in Part I, we only deal with the terms associated with $\widehat{b_2 u_1}$ and $\widehat{b_3 u_1}$ in $\widehat{N_{21}}$. By (\ref{N21}) we have
\begin{align*}
\Lambda_{2,4}
&\leq \int_0^t\left\|\xi_\nu e^{-c\xi_{\nu}^2\left(t-\tau \right)}\xi_{2}\widehat{b_2 u_1}  \right\|_{L^2}d\tau +\int_0^t\left\|\xi_\nu e^{-c\xi_{\nu}^2\left(t-\tau \right)}\xi_{3}\widehat{b_3 u_1}  \right\|_{L^2}d\tau \nonumber\\
&\quad+\int_0^t\left\|\xi_\nu e^{-c\left(t-\tau \right)}\xi_{2}\widehat{b_2 u_1}  \right\|_{L^2}d\tau +\int_0^t\left\|\xi_\nu e^{-c\left(t-\tau \right)}\xi_{3}\widehat{b_3 u_1} \right\|_{L^2}d\tau d\tau.
\end{align*}

The estimate of $\Lambda_{2,4}$ is based on the  estimates of $\|\p_i\p_2 (b_2u_1 )\|_{L^2}$ and $\|\p_i\p_3 (b_3u_1 )\|_{L^2}$ with $ i=2,3$.  Noting that (\ref{u10}), together with (\ref{anso}), implies
\begin{align}
\|\p_\nu b \|_{L_{x_1,x_3}^2L_{x_2}^\infty}
& \leq C\|\p_\nu b \|_{L^2}^{\frac12} \|\p_2\p_\nu b \|_{L^2}^{\frac12}
 \leq CC_0\varepsilon (1+t)^{-\frac{23}{24}}\label{23b-1}
  \end{align}
and
\begin{align}
\|\p_\nu u_1 \|_{L_{x_1,x_3}^{\infty}L_{x_2}^2} &\leq    C
\|\p_\nu u_1 \|_{L^2}^{\frac14} \|\p_1\p_\nu u_1 \|_{L^2}^{\frac14}\|\p_3\p_\nu u_1\|_{L^2}^{\frac14}\|\p_1\p_3\p_\nu u_1\|_{L^2}^{\frac14}\notag\\
&  \leq CC_0\varepsilon (1+t)^{-\frac{29}{24}}.\label{23u1-2}
\end{align}
Thus, it follows from  (\ref{23b-1}), (\ref{23u1-2}), (\ref{binfty}), (\ref{b1infty}) and (\ref{u10}) that for $\nu\in\{2,3\}$,
\begin{align}
&\|\p_2\p_\nu\left(b_2 u_1\right)\|_{L^2} +\|\p_3\p_\nu\left(b_3 u_1\right)\|_{L^2}  \nonumber\\
&\quad \leq C \|b \|_{L^\infty} \left(\|\p_2\p_\nu u_1 \|_{L^2}+\|\p_3\p_\nu u_1 \|_{L^2}\right) \nonumber\\
&\qquad +C\|u_1\|_{L^\infty}\left( \|\p_2\p_\nu b  \|_{L^2} +\|\p_3\p_\nu b\|_{L^2}\right)\nonumber\\
&\qquad+ C\|\p_\nu u_1 \|_{L_{x_1,x_3}^{\infty}L_{x_2}^2} \|\p_\nu b\|_{L_{x_1,x_3}^{2}L_{x_2}^\infty} \nonumber\\
&\quad\leq   CC_0^2\varepsilon^2\left(1+t\right)^{-\frac{11}{6}}. \label{i22N2N}
\end{align}

With the help of  \eqref{bb11} and \eqref{i22N2N}, similarly to the derivation of (\ref{i4b}), we deduce
\begin{align}
\Lambda_{2,4}&\leq C\int_0^{t}\left(1+t-\tau \right)^{-\frac32}\left(\left\|\|  b_2u_1 \|_{L_{x_2,x_3}^1}\right\|_{L_{x_1}^2}+\left\|\|  b_3u_1 \|_{L_{x_2,x_3}^1}\right\|_{L_{x_1}^2}\right)d\tau \nonumber\\
&\quad\quad+C\int_{0}^{t} e^{-c\left(t-\tau\right)}\left(\|\p_\nu\p_2\left( b_2u_1\right)  \|_{L^2}+\|\p_\nu\p_3\left(b_3u_1\right)  \|_{L^2}\right)d\tau \nonumber\\
&\quad\leq CC_0^2\varepsilon^2\int_0^t\left[\left(1+t-\tau \right)^{-\frac32} \left(1+\tau\right)^{-\frac{11}{8}}+e^{-c\left(t-\tau\right)}\left(1+\tau\right)^{-\frac{11}{6}}\right]d\tau \nonumber\\
&\quad\leq C C_0^2 \varepsilon^2\left(1+t\right)^{-\frac{11}{8}}. \label{ii14b}
\end{align}

Plugging (\ref{ii1}), (\ref{ii33}) and (\ref{ii14b}) into (\ref{iu1L2}), choosing $C_0$ large enough and $\varepsilon$ suitably small, we get that for $ \nu\in\{2,3\}$,
\begin{align*}
\|\p_\nu u_1(t)\|_{L^2}\leq \left(C\varepsilon + C C_0^2 \varepsilon^2\right)\left(1+t\right)^{-\frac54}\leq \frac{ C_0}{2} \varepsilon\left(1+t\right)^{-\frac54}.
\end{align*}

\vskip .1in

\textbf{Part III.  Improved decay rates  of $\|\p_3\p_{i}u_1 \|_{L^2} $ with $i\in\{2,3\}$}
\vskip .1in

This step aims to improve the decay estimates of $\|\p_3\p_i u_1\|_{L^2}$ with $i=2,3$. First of all, we have by
  \eqref{u}  that
\begin{align}
\|\p_{3}\p_{i}u_1\|_{L^2}&=\|\widehat{\p_{3}\p_{i}u_1}\|_{L^2}=\|\xi_{3}\xi_{i}\widehat{u_1}\|_{L^2}\nonumber\\
&\leq \left\| \xi_{3}\xi_{i}\widehat{K_1}(t)\widehat{u_{10}}\right\|_{L^2}+\left\|\xi_{3}\xi_{i}\widehat{K_2}(t)\widehat{b_{10}}\right\|_{L^2}\nonumber\\
&\quad+\int_0^t\left\|\xi_{3}\xi_{i}\widehat{K_1}(t-\tau)\widehat{N_{11}}(\tau)\right\|_{L^2}d\tau\nonumber\\
&\quad+\int_0^t\left\|\xi_{3}\xi_{i}\widehat{K_2}(t-\tau)\widehat{N_{21}}(\tau)\right\|_{L^2}d\tau
\triangleq\sum_{m=1}^4 \Lambda_{3,m}.\label{p3p2u1L2}
\end{align}

For $i\in\{2,3\}$, we can apply (\ref{S1}), (\ref{S2}) and Corollary \ref{divfree} with $\beta=2$ to get
\begin{align}
\Lambda_{3,1}+\Lambda_{3,2}
&\leq C\||\xi_\nu|^2e^{-c\xi_{\nu}^2t}\widehat{u_{10}}\|_{L^2}+ Ce^{-ct}  \|\xi_3\xi_i \widehat{u_{10}}  \|_{L^2}\nonumber\\
&\leq C(1+t)^{-\frac{7}{4}}\left(\| u_{0}\|_{L^1}+\|\p_3\p_iu_0\|_{L^2 }\right)\leq C\varepsilon(1+t)^{-\frac{7}{4}}.\label{M321}
\end{align}

Analogously to that in Part II, to deal with $\Lambda_{3,3}$ and $\Lambda_{3,4}$, it suffices to consider the terms associated with $\widehat{\mathbb{P}_1(b\cdot\nabla b)}$ in $\widehat{N_{11}}$ and $\widehat{b_2u_1}, \widehat{ b_3u_1}$ in $\widehat{N_{21}}$ (still denoted by $\Lambda_{3,3}$ and $\Lambda_{3,4}$).
First, it follows from \eqref{P1} and Proposition \ref{lem2.1} that
\begin{align}
\Lambda_{3,3}
&\triangleq\int_0^t\left\|\xi_3\xi_i\widehat{K_1}(t-\tau)\widehat{\mathbb P_1 (b\cdot\nabla b )}\right\|_{L^2}d\tau\nonumber\\
&\leq C\int_0^t\left\|\xi_3\xi_i e^{-c\xi_{\nu}^2 (t-\tau  )}\sum_{k=1}^{3}\xi_{\nu}^2|\xi|^{-2}\xi_k\widehat{b_kb_1}\right\|_{L^2}d\tau\nonumber\\
&\quad+C\int_0^t\left\|\xi_3\xi_i e^{-c\xi_{\nu}^2 (t-\tau  )} \sum_{k=1}^{3}\sum_{l=2}^{3}\xi_1|\xi|^{-2}\xi_k\xi_l\widehat{b_kb_l}\right\|_{L^2}d\tau\nonumber\\
&\quad+C\int_0^t\left\|\xi_3\xi_i e^{-c (1+\xi_3^2 ) (t-\tau  )} \sum_{k=1}^{3}\xi_{\nu}^2|\xi|^{-2}\xi_k\widehat{b_kb_1}\right\|_{L^2}d\tau\nonumber\\
&\quad+C\int_0^t\left\| \xi_3\xi_i e^{-c (1+\xi_3^2 ) (t-\tau )} \sum_{k=1}^{3}\sum_{l=2}^{3}\xi_1|\xi|^{-2}\xi_k\xi_l\widehat{b_kb_l}\right\|_{L^2}d\tau\nonumber\\
&\triangleq \Lambda_{3,31} +\Lambda_{3,32} +\Lambda_{3,33} +\Lambda_{3,34}.\label{M323b}
\end{align}

To bound $\Lambda_{3,31}$, we first observe from (\ref{anso}) and (\ref{u10})  that
\begin{align}
\|\p_3 ( b\cdot \na b_1 )\|_{L^2}
&\leq \|b\|_{ L^\infty}\|\p_3 \na b_1 \|_{L^2} +C\|\p_3 b \|_{L^2}^{\frac12}\|\p_2\p_3b \|_{L^2}^{\frac12}\nonumber\\
&\qquad\times\|\na b_1 \|_{L^2}^{\frac14}\|\p_1\na b_1\|_{L^2}^{\frac14}\|\p_3\na b_1\|_{L^2}^{\frac14}
 \|\p_1\p_3\na b_1\|_{L^2}^{\frac14} \nonumber\\
&\leq CC_0^2 \varepsilon^2\left(1+t\right)^{-\frac{7}{3}}, \label{p3bb}
\end{align}
which,  together with  \eqref{bb11} and Corollary \ref{PI}, shows that
\begin{align}
\Lambda_{3,31}&\leq \left(\int_0^{\frac{t}{2}}+ \int_{\frac t2}^t \right) \left\|\xi_3\xi_i e^{-c\xi_{\nu}^2 (t-\tau  )}\sum_{k=1}^{3}\xi_{\nu}^2|\xi|^{-2}\xi_k\widehat{b_kb_1}\right\|_{L^2}d\tau\nonumber\\
&  \leq  C\int_0^{\frac{t}{2}}\left\||\xi_\nu|^3 e^{-c\xi_{\nu}^2 (t-\tau  )}\widehat{b b_1 }\right\|_{L^2}d\tau \nonumber\\
&\quad + C\int^t_{\frac{t}{2}}\left\|\xi_\nu e^{-c\xi_{\nu}^2 (t-\tau  )}\widehat{\p_3(b \cdot \nabla b_1) }\right\|_{L^2}d\tau\nonumber\\
& \leq C\int_0^{\frac{t}{2}}\left(1+t-\tau \right)^{-2} \| bb_1 \|_{L^2_{x_1}L_{x_2,x_3}^1} d\tau \nonumber\\
 &\quad+C\int_{{\frac{t}{2}}}^{t}\left\|\xi_\nu e^{-c\xi_{\nu}^2 (t-\tau  )}\right\|_{L^\infty} \| \p_3(b\cdot \na b_1)  \|_{L^2}d\tau  \nonumber\\
& \leq CC_0^2\varepsilon^2 \int_0^{\frac{t}{2}}\left(1+t-\tau \right)^{-2} \left(1+\tau\right)^{-\frac{11}{8}}d\tau\nonumber\\
 &\quad+ CC_0^2 \varepsilon^2 \int_{{\frac{t}{2}}}^t\left(t-\tau \right)^{-\frac12} \left(1+\tau\right)^{-\frac{7}{3}}d\tau  \nonumber\\
& \leq   C C_0^2 \varepsilon^2\left(1+t\right)^{-\frac74}. \label{M3231}
\end{align}

In a similar manner, by choosing $ 0<\delta\leq1/4$ we deduce that for $i\in\{2,3\}$,
\begin{align}
\Lambda_{3,32}
&\leq C\int_0^{\frac t2} (1+t-\tau  )^{-2}  \| bb_\nu \|_{L^2_{x_1}L_{x_2,x_3}^1}  d\tau \nonumber\\
&\quad+C\int_{\frac t2}^{t}  (t-\tau  )^{-\frac12} \|\p_3\p_\nu (b b_\nu )  \|_{L^2} d\tau \nonumber\\
& \leq CC_0^2\varepsilon^2\int_0^{\frac t2}\left(1+t-\tau \right)^{-2}\left(1+\tau \right)^{-1}d\tau\nonumber\\
&\quad+CC_0^2\varepsilon^2\int_{\frac t2}^{t} \left(t-\tau \right)^{-\frac12}\left(1+\tau\right)^{-\frac{7}{3}}d\tau \nonumber\\
& \leq   CC_0^2\varepsilon^2 \left(1+t\right)^{-\frac74}, \label{M3232}
\end{align}
where we have used    \eqref{1bbv2} and the following inequality (due to (\ref{u10})):
\begin{align}
\|\p_3\p_\nu ( b  b_\nu )\|_{L^2} &\leq C\|b\|_{ L^\infty}\|\p_3\p_
\nu b \|_{L^2}+C\||\p_3b|| \p_\nu b| \|_{L^2}\nonumber\\
&\leq C\|b\|_{ L^\infty}\|\p_3\p_
\nu b \|_{L^2}+C\|\p_3 b\|_{L^2}^{\frac12}\|\p_2\p_3b \|_{L^2}^{\frac12}\nonumber\\
&\qquad\times\|\p_\nu b \|_{L^2}^{\frac14}\|\p_1\p_\nu b \|_{L^2}^{\frac14}\|\p_3\p_\nu b\|_{L^2}^{\frac14}\|\p_1\p_3\p_\nu b \|_{L^2}^{\frac14}\nonumber\\
& \leq CC_0^2 \varepsilon^2\left(1+t\right)^{-\frac{7}{3}}.\nonumber 
\end{align}

Similarly to that in (\ref{iii3233}),
 by \eqref{22bb2}, \eqref{23bb2} and (\ref{ineq}) we  obtain
\begin{align}
\Lambda_{3,33}+\Lambda_{3,34}
&\leq C\int_0^te^{-c (t-\tau  )} (t-\tau )^{-\frac{1}{2}} \|\p_i (b\cdot\na b ) \|_{L^2}d\tau\nonumber\\
&\leq  CC_0^2\varepsilon^2\int_0^te^{-c (t-\tau  )} (t-\tau )^{-\frac{1}{2}} (1+\tau\ )^{-\frac{11}{6}} d\tau \nonumber\\
&\leq CC_0^2\varepsilon^2\left(1+t\right)^{-\frac{7}{4}},\quad\forall\ i\in\{2,3\}. \label{M3233}
\end{align}

Dragging \eqref{M3231}, \eqref{M3232}  and \eqref{M3233} into \eqref{M323b}, we see that for $i\in\{2,3\}$,
\begin{align}
\Lambda_{3,3}
\leq C C_0^2 \varepsilon^2\left(1+t\right)^{-\frac74} .\label{M323bg}
\end{align}

Next, we deal with $\Lambda_{3,4}$. First, we make use of (\ref{u10}), (\ref{binfty}), (\ref{b1infty}) and the divergence-free condition $\nabla \cdot u=0$ in a manner as that in (\ref{p1b1-1}) to get
\begin{align}
&\|\p_2\p_3 (b_2 u_1 )\|_{L^2} +\|\p_3^2 (b_3 u_1 )\|_{L^2}   \nonumber\\
&\quad \leq  \|b \|_{L^\infty} \|\p_\nu \p_3 u_1 \|_{L^2} +\|u_1\|_{L^\infty} \|\p_\nu\p_3 b  \|_{L^2} \nonumber\\
&\qquad + \|\p_2 u_1 \|_{L_{x_1,x_2}^{2}L_{x_3}^\infty} \|\p_3 b_2\|_{L_{x_1,x_2}^{\infty}L_{x_3}^2}+\|\p_3 u_1 \|_{L_{x_1,x_2}^{\infty}L_{x_3}^2} \|\p_2 b_2\|_{L_{x_1,x_2}^{2}L_{x_3}^\infty} \nonumber\\
&\qquad+ 2\|\p_3 u_1 \|_{L_{x_1,x_3}^{\infty}L_{x_2}^2} \|\p_3 b_3\|_{L_{x_1,x_3}^{2}L_{x_2}^\infty}  \nonumber\\
&\quad\leq   CC_0^2\varepsilon^2\left(1+t\right)^{-\frac{125}{48}}.\label{23b23u1}
\end{align}

Analogously to the derivations of (\ref{iii3231-2}) and (\ref{iii3233}), using  \eqref{bb11}, \eqref{i22N2N}, (\ref{23b23u1}), Proposition \ref{lem2.1} and  Corollary \ref{PI}, we infer from the definition of $\widehat{N_{21}}$ in (\ref{N21}) that
\begin{align}
\Lambda_{3,4}& \leq C\int_0^{\frac{t}{2}}\left(1+t-\tau \right)^{-2}\left\|\| ( b_2u_1, b_3u_1)\|_{L_{x_2x_3}^1}\right\|_{L_{x_1}^2} d\tau \nonumber\\
&\quad+ C\int_{\frac{t}{2}}^{t}\left(t-\tau \right)^{-\frac12}\|\left (\p_2\p_3 (b_2u_1 ),\p^2_3 (b_3u_1 )\right)\|_{L^2}d\tau \nonumber\\
&\quad+C\int_{0}^{t} e^{-c\left(t-\tau\right)}\left(t-\tau \right)^{-\frac12}\|\left(\p_i\p_2 ( b_2u_1 ) ,\p_i\p_3 (b_3u_1)\right)  \|_{L^2} d\tau \nonumber\\
&\leq CC_0^2\varepsilon^2\int_0^{\frac{t}{2}}\left(1+t-\tau \right)^{-2} \left(1+\tau\right)^{-\frac{11}{8}}d\tau\nonumber\\
&\quad+CC_0^2\varepsilon^2\int_{\frac{t}{2}}^t\left(t-\tau \right)^{-\frac{1}{2}} \left(1+\tau\right)^{-\frac{125}{48}}d\tau\nonumber\\
&\quad
+CC_0^2\varepsilon^2\int_0^te^{-c\left(t-\tau\right)}\left(t-\tau \right)^{-\frac{1}{2}}\left(1+\tau\right)^{-\frac{11}{6}}d\tau \nonumber\\
&\leq C C_0^2 \varepsilon^2\left(1+t\right)^{-\frac{7}{4}}. \label{M324b}
\end{align}

Thus, plugging   \eqref{M321}, (\ref{M323bg})  and \eqref{M324b} into (\ref{p3p2u1L2}), we see that for $i\in\{2,3\}$,
\begin{align*}
\|\p_3\p_iu_1(t)\|_{L^2}\leq \left( C\varepsilon +  C C_0^2 \varepsilon^2\right)\left(1+t\right)^{-\frac74}\leq \frac{ C_0}{2} \varepsilon\left(1+t\right)^{-\frac74}, \nonumber
\end{align*}
provided $C_0$ is chosen to be large enough and $\varepsilon$ is chosen to be suitably small.

\vskip .1in

\textbf{Part IV. Improved decay rate of $\|\p_3^3u_1\|_{L^2}$}
\vskip .1in

The enhanced decay rate of $\|\p_3^3u_1\|_{L^2}$ needs more subtle analysis. To do so, we first infer from \eqref{u} that
\begin{align}
\|\p_3^3u_1\|_{L^2}&=\|\widehat{\p_3^3u_1}\|_{L^2}=\|\xi_{3}^3\widehat{u_1}\|_{L^2}\nonumber\\
&\leq \left\| \xi_{3}^3\widehat{K_1}(t)\widehat{u_{10}}\right\|_{L^2}+\left\|\xi_{3}^3\widehat{K_2}(t)\widehat{b_{10}}\right\|_{L^2}\nonumber\\
&\quad+\int_0^t\left\|\xi_{3}^3\widehat{K_1}(t-\tau)\widehat{N_{11}}(\tau)\right\|_{L^2}d\tau\nonumber\\
&\quad+\int_0^t\left\|\xi_{3}^3\widehat{K_2}(t-\tau)\widehat{N_{21}}(\tau)\right\|_{L^2}d\tau \triangleq \sum_{m=1}^4\Lambda_{4,m},\label{p33u1L2}
\end{align}
where the first two terms can be easily bounded by Corollary \ref{divfree} with $\beta=3$,
\begin{align}
\Lambda_{4,1}+\Lambda_{4,2}&\leq  C\|\xi_{3}^3e^{-c\xi_{\nu}^2t}\widehat{u_{10}}\|_{L^2}+ Ce^{-ct} \left\|\xi_{3}^3\widehat{u_{10}}\right\|_{L^2}\nonumber\\
&\leq C(1+t)^{-\frac{9}{4}}\left(\| u_{0}\|_{L^1}+\|\p_3^3u_0\|_{L^2 }\right)\leq C\varepsilon(1+t)^{-\frac{9}{4}}.\label{Q1}
\end{align}

Next, we consider the term  $\widehat{\mathbb P_1 (b\cdot\nabla b )}$ in $\widehat{N_{11}}$, which is still denoted by $\Lambda_{4,3}$ for notational convenience. It follows from (\ref{P1}), (\ref{S1}) and (\ref{KS23}) that
\begin{align}
\Lambda_{4,3}
&\leq C\int_0^t\left\|\xi_3^3e^{-c\xi_{\nu}^2\left(t-\tau \right)} \sum_{k=1}^{3}\xi_{\nu}^2|\xi|^{-2}\xi_k\widehat{b_kb_1}\right\|_{L^2}d\tau\nonumber\\
&\quad+C\int_0^t\left\|\xi_3^3e^{-c\xi_{\nu}^2\left(t-\tau \right)} \sum_{k=1}^{3}\sum_{l=2}^{3}\xi_1|\xi|^{-2}\xi_k\xi_l\widehat{b_kb_l}\right\|_{L^2}d\tau\nonumber\\
&\quad+C\int_0^t\left\|\xi_3^3e^{-c\left(1+\xi_3^2\right)\left(t-\tau \right)} \sum_{k=1}^{3} \xi_{\nu}^2|\xi|^{-2}\xi_k\widehat{b_kb_1}\right\|_{L^2}d\tau\nonumber\\
&\quad+C\int_0^t\left\|\xi_3^3e^{-c\left(1+\xi_3^2\right)\left(t-\tau \right)} \sum_{k=1}^{3}\sum_{l=2}^{3}\xi_1|\xi|^{-2}\xi_k\xi_l\widehat{b_kb_l}\right\|_{L^2}d\tau\nonumber\\
&\triangleq \Lambda_{4,31}+\Lambda_{4,32}+\Lambda_{4,33}+\Lambda_{4,34}.\label{Q3b}
\end{align}

Similarly to the derivation of \eqref{ii2132-2}, by (\ref{bb11})  we have
\begin{align}
\Lambda_{4,31}
&\leq C\int_0^{\frac{t}{2}} (1+t-\tau  )^{-\frac52}\|  bb_1 \|_{L_{x_1}^2 L_{x_2,x_3}^1} d\tau \nonumber\\
 &\quad+C\int_{{\frac{t}{2}}}^{t} (t-\tau  )^{-\frac12} \| \p_3^2(b\cdot \na b_1)(\tau)\|_{L^2}d\tau \nonumber\\
&\leq CC_0^2\varepsilon^2 \int_0^{\frac{t}{2}}\left(1+t-\tau \right)^{-\frac52} \left(1+\tau\right)^{-\frac{11}{8}}d\tau\nonumber\\
&\quad+ CC_0^2 \varepsilon^2 \int_{{\frac{t}{2}}}^t\left(t-\tau \right)^{-\frac12} \left(1+\tau\right)^{-\frac{11}{4}}d\tau  \nonumber\\
&\leq C C_0^2 \varepsilon^2\left(1+t\right)^{-\frac94}, \label{Q31}
\end{align}
since it  follows from (\ref{u10}) and the divergence-free condition $\nabla\cdot b=0$ that
\begin{align*}
\|\p_3^2\left( b\cdot \na b_1\right)\|_{L^2}
&\leq \|b\|_{L^\infty}\|\na \p_3^2b_1\|_{L^2} +\|\p_3^2  b\|_{L_{x_1,x_3}^2L_{x_2}^\infty}\| \na b_1 \|_{L_{x_1,x_3}^\infty L_{x_2}^2}\nonumber\\
&\quad+\|\na\p_3 b_1\|_{L_{x_1,x_2}^2L_{x_3}^\infty}\| \p_3 b \|_{L_{x_1,x_2}^\infty L_{x_3}^2}\nonumber\\
& \leq CC_0^2 \varepsilon^2\left(1+t\right)^{-\frac{11}{4}}. 
\end{align*}

In a similar manner, applying Corollary \ref{PI}  with $\alpha=4$, by \eqref{1bbv2} we have
\begin{align}
\Lambda_{4,32}
&\leq C\int_0^{\frac t2}\left(1+t-\tau \right)^{-\frac52}  \| (bb_2 , bb_3  )\|_{L_{x_1}^2 L_{x_2,x_3}^1} d\tau\nonumber\\
&\quad+C\int_{\frac t2}^{t}\left\|\xi_{\nu} e^{-c\xi_{\nu}^2\left(t-\tau \right)} \right\|_{L^\infty}  \|(\p_3^3 (b b_2), \p_3^3(b b_3 )) \|_{L^2} d\tau \nonumber\\
&\leq CC_0^2\varepsilon^2\int_0^{\frac t2}\left(1+t-\tau \right)^{-\frac52}\left(1+\tau \right)^{-1}d\tau\nonumber\\
&\quad+CC_0^2\varepsilon^2\int_{\frac t2}^{t} \left(t-\tau \right)^{-\frac12}\left(1+\tau\right)^{-\frac{17}{6}}d\tau \nonumber\\
&\leq CC_0^2\varepsilon^2 \left(1+t \right)^{-\frac52+\delta}  +CC_0^2\varepsilon^2  \left(1+\tau\right)^{-\frac{7}{3}}\nonumber\\
&\leq CC_0^2\varepsilon^2 \left(1+t\right)^{-\frac{9}{4}}, \label{Q32}
\end{align}
provided  $0<\delta\leq1/4$. Here, we have used (\ref{binfty}) and (\ref{u10}) to get that
\begin{align}
\|&\p_3^3\left( b  b_2\right)\|_{L^2}+\|\p_3^3\left(b b_3\right) \|_{L^2}   \nonumber\\
&\leq C\|b\|_{ L^\infty}\|\p_3^3b \|_{L^2}+C\|\p_3^2 b\|_{L^2_{x_2,x_3}L^\infty_{x_1}}\|\p_3b\|_{L^\infty_{x_2,x_3}L^2_{x_1}}\nonumber\\
&\leq C\|b\|_{ L^\infty}\|\p_3^3b \|_{L^2}+ C\|\p_3^2 b\|_{L^2}^{\frac12}\|\p_1\p_3^2b \|_{L^2}^{\frac12}\nonumber\\
&\qquad\qquad \times\|\p_3 b \|_{L^2}^{\frac14}\|\p_2\p_3 b \|_{L^2}^{\frac14}\|\p_3^2 b\|_{L^2}^{\frac14}\|\p_2\p_3^2b \|_{L^2}^{\frac14}\nonumber\\
& \leq CC_0^2 \varepsilon^2\left(1+t\right)^{-\frac{17}{6}}.\nonumber 
\end{align}

Due to \eqref{33bbl2} and (\ref{ineq}), it is easily seen that
\begin{align*}
\Lambda_{4,33}+\Lambda_{4,34}
&\leq C\int_0^te^{-c\left(t-\tau \right)}\left(t-\tau\right)^{-\frac{1}{2}}\left\|\p_3^2\left(b\cdot\na b\right) \right\|_{L^2}d\tau\nonumber\\
&\leq  CC_0^2\varepsilon^2\int_0^te^{-c\left(t-\tau \right)}\left(t-\tau\right)^{-\frac{1}{2}}\left(1+\tau\right)^{-\frac{29}{12}} d\tau \nonumber\\
&\leq CC_0^2\varepsilon^2\left(1+t\right)^{-\frac{29}{12}}, 
\end{align*}
which,  together with (\ref{Q31}), (\ref{Q32}) and (\ref{Q3b}), yields
\begin{equation}
\Lambda_{4,3}\leq CC_0^2\varepsilon^2\left(1+t\right)^{-\frac{9}{4}},\quad\forall \ t\geq1.\label{Q3}
\end{equation}

Analogously to that in Part II, we only deal with the terms $\widehat{b_2u_1}, \widehat{ b_3u_1}$ in $\Lambda_{4,4}$, which are still marked by $\Lambda_{4,4}$. Due to (\ref{u10}), \eqref{binfty} and (\ref{b1infty}), one has
\begin{align}
&\| \p_3^3\left(b_2u_1\right)\|_{L^2}+\|\p^3_3\left(b_3u_1\right)\|_{L^2}  \nonumber\\
&\quad \leq C\|b\|_{L^\infty}\|\p_3^3 u_1 \|_{L^2}
+C\|\p_3u_1\|_{L_{x_1,x_3}^{\infty}L_{x_2}^2}\|\p_3^2b\|_{L_{x_1,x_3}^{2}L_{x_2}^\infty} \nonumber\\
&\qquad+C\|u_1\|_{L^\infty}\|\p_3^3b \|_{L^2}+C\|\p_3^2u_1\|_{L_{x_1,x_2}^{2}L_{x_3}^\infty}\|\p_3b\|_{L_{x_1,x_2}^{\infty}L_{x_3}^2}\nonumber\\
&\quad\leq CC_0^2\varepsilon^2\left(1+t\right)^{-\frac{37}{12}},\nonumber
\end{align}
and
\begin{align*}
&\|\p_2\p_3^2\left( b_2u_1\right)\|_{L^2}+\|\p_3^3\left(b_3u_1\right)\|_{L^2}  \nonumber\\
&\quad \leq \|b\|_{L^\infty}\left(\|\p_2\p_3^2u_1\|_{L^2}+\|\p_3^3 u_1 \|_{L^2}\right) +\|u_1\|_{L^\infty}\left(\|\p_2\p_3^2b_2\|_{L^2}+\|\p_3^3b_3 \|_{L^2}\right) \nonumber\\
&\qquad+C\|\p_3^2b_2\|_{L_{x_1,x_3}^2L_{x_2}^\infty}\|\p_2u_1\|_{L_{x_1,x_3}^{\infty}L_{x_2}^2}+C\|\p_3u_1 \|_{L_{x_1,x_2}^{\infty}L_{x_3}^2}\|\p_2\p_3b_2\|_{L_{x_1,x_2}^{2}L_{x_3}^\infty}\nonumber\\
&\qquad
+C\|\p_2\p_3u_1\|_{L_{x_1,x_2}^{2}L_{x_3}^\infty}\|\p_3b_2\|_{L_{x_1,x_2}^{\infty}L_{x_3}^2}+C\|\p_3^2u_1\|_{L_{x_1,x_3}^{2}L_{x_2}^\infty}\|\p_2b_2\|_{L_{x_1,x_3}^{\infty}L_{x_2}^2}\nonumber\\
&\qquad
+C\|\p_3u_1\|_{L_{x_2,x_3}^\infty L_{x_1}^{2}}\|\p_3^2b_3\|_{L_{x_2,x_3}^2 L_{x_1}^{\infty}}+C\|\p_3^2u_1\|_{L_{x_1,x_2}^{2}L_{x_3}^\infty}\|\p_3b_3\|_{L_{x_1,x_2}^{\infty}L_{x_3}^2}\nonumber\\
&\quad\leq CC_0^2\varepsilon^2\left(1+t\right)^{-\frac{11}{4}},
\end{align*}
 which, together with \eqref{bb11}, \eqref{ineq} and \eqref{i22N2N}, yields
\begin{align}
\Lambda_{4,4}&\leq C\int_0^{\frac{t}{2}}\left(1+t-\tau \right)^{-\frac52}\left\|\| ( b_2u_1, b_3u_1)\|_{L_{x_2,x_3}^1}\right\|_{L_{x_1}^2}d\tau \nonumber\\
&\quad\quad+ C\int_{\frac{t}{2}}^{t}\left(t-\tau \right)^{-\frac12} \left\| \left(\p_3^3(b_2u_1),\p^3_3 (b_3u_1)\right)\right\|_{L^2} d\tau \nonumber\\
&\quad\quad+C\int_{0}^{t} e^{-c\left(t-\tau\right)}\left(t-\tau \right)^{-\frac12}\left\|\left(\p_2\p_3^2 ( b_2u_1),  \p_3^3 (b_3u_1)\right)\right \|_{L^2} d\tau \nonumber\\
&\quad\leq CC_0^2\varepsilon^2\int_0^{\frac{t}{2}}\left(1+t-\tau \right)^{-\frac52}\left(1+\tau\right)^{-\frac{11}{8}}d\tau\nonumber\\
&\qquad+CC_0^2\varepsilon^2\int_{\frac{t}{2}}^t\left(t-\tau \right)^{-\frac{1}{2}} \left(1+\tau\right)^{-\frac{11}{4}}d\tau\nonumber\\
&\quad\quad
+CC_0^2\varepsilon^2\int_0^te^{-c\left(t-\tau\right)}\left(t-\tau \right)^{-\frac{1}{2}}\left(1+\tau\right)^{-\frac{37}{12}}d\tau \nonumber\\
&\quad\leq C C_0^2 \varepsilon^2\left(1+t\right)^{-\frac{9}{4}}. \label{Q4b}
\end{align}

Thus, inserting
  \eqref{Q1}, \eqref{Q3}  and \eqref{Q4b} into \eqref{p33u1L2}, we obtain by choosing $C_0$ large enough and $\varepsilon$ suitably small that
$$
\|\p_3^3u_1(t)\|_{L^2}\leq \left(C\varepsilon +  C C_0^2 \varepsilon^2\right)\left(1+t\right)^{-\frac94}\leq \frac{C_0}{2}\varepsilon (1+t)^{-\frac94}.
$$

\vskip .1in

\begin{proof}[Proof of Theorem \ref{thm1.2}]Based on (\ref{u10}) and all the estimates established in Subsections 4.2 and 4.3, we
immediately arrive at the desired decay rates of the solutions stated in Theorem \ref{thm1.2} by applying the bootstrapping arguments.
\end{proof}

\vskip .1in

\section{Improved Decay Rates and Proof of Theorem \ref{thm1.3} }
\subsection{Estimates of nonlinear terms.}

This section is devoted to the proof of Theorem \ref{thm1.3}. It is worth pointing out that the following decay rates of $(u,b)$ stated in Theorem \ref{thm1.2}  are optimal and cannot be improved any more,
\begin{align}
& \|\p_1^k(u,b)(t)\|_{L^2}\leq C_0\varepsilon(1+t)^{-\frac{1}{2}},  \quad\quad\,  \|\p_{i}(u, b)(t)\|_{L^2}\leq C_0\varepsilon(1+t)^{-1},  \nonumber\\
& \|(u_1,b_1)(t)\|_{L^2}\leq C_0\varepsilon(1+t)^{-\frac{3}{4}},  \quad\quad\;\,   \|\p_{i}(u_1, b_1)(t)\|_{L^2}\leq C_0\varepsilon(1+t)^{-\frac{5}{4}},  \nonumber\\
& \|\p_{i}\p_3(u_1, b_1)(t)\|_{L^2}\leq C_0\varepsilon(1+t)^{-\frac{7}{4}},\;\;\,  \| \p_3^3(u_1, b_1)(t)\|_{L^2}\leq C_0\varepsilon(1+t)^{-\frac{9}{4}},\label{u101}
\end{align}
where $k\in\{0,1\}$ and $i\in\{2,3\}$. Moreover, under the conditions of Theorem \ref{thm1.3} it has been shown in \cite{LinWZ1} that for any $t\geq0$,
\begin{equation}
	\|(u,b)(t)\|_{H^4}^2+\int_0^T\left(\|(\p_3u,\p_2b,\p_3 b)\|_{H^4}^2+\|\p_2u\|_{H^3}^2\right)d t\leq \ep^2,	\label{u102}
	\end{equation}
provided $\varepsilon>0$ is sufficiently small.
In addition to (\ref{u101}),  to prove Theorem \ref{thm1.3},  we make the following further ansatz:
	\begin{align}
	& \|\p_1\p_{i}(u,b)(t)\|_{L^2}\leq C_0\varepsilon(1+t)^{-1},\;\; \;\;\;
	\|\p_i\p_j(u, b)(t)\|_{L^2}\leq C_0\varepsilon(1+t)^{-\frac32},  \nonumber\\
	& \|\p_1^2(u, b)\|_{L^2}\leq C_0\varepsilon(1+t)^{-\frac{1}{2}}, \;\;\;\;\;\;\;\;\;\; \; \|\p_{1}\p_{i}\p_3(u,b)(t)\|_{L^2}\leq C_0\varepsilon(1+t)^{-\frac32}, \nonumber\\
	& \|\p_l\p_2^2(u,b)(t)\|_{L^2}\leq C_0\varepsilon(1+t)^{-1},\;\;\; \;\; \|\p_1^2\p_i(u,b)(t)\|_{L^2}\leq C_0\varepsilon(1+t)^{-1},\nonumber\\
	& \|\p_{i}\p_{j}\p_3(u,b)(t)\|_{L^2}\leq C_0\varepsilon(1+t)^{-2}, \;\;
	\|\p_1^3(u,b)(t)\|_{L^2}\leq C_0\varepsilon(1+t)^{-\frac38},\nonumber\\
	&\|\p_1^3\p_3(u,b)(t)\|_{L^2}\leq C_0\varepsilon(1+t)^{-\frac12},\;\;\;\; \|\p_1^2\p_3^2(u,b)\|_{L^2}\leq C_0\varepsilon(1+t)^{-\frac{11}{8}},\nonumber\\
	& \|\p_l\p_h\p_2\p_3(u,b)\|_{L^2}\leq C_0\varepsilon(1+t)^{-1},\;\; \;\|\p_1\p_3^3(u,b)\|_{L^2}\leq C_0\varepsilon(1+t)^{-2},\nonumber\\
	&\|\p_l\p_2\p_3^2(u,b)\|_{L^2}\leq C_0\varepsilon(1+t)^{-2},
	\;\;\;\;\;\;
	\|\p_i\p_3^3(u,b)\|_{L^2}\leq C_0\varepsilon(1+t)^{-\frac52} \label{u4}
\end{align}
and
\begin{align}
&\|\p_i\p_j(u_1,b_1)\|_{L^2}\leq C_0\varepsilon(1+t)^{-\frac74},
		\quad \,\,
		\|\p_i\p_3^2(u_1,b_1)\|_{L^2}\leq C_0\varepsilon(1+t)^{-\frac94},\nonumber\\
		&\|\p_i\p_3^3(u_1,b_1)\|_{L^2}\leq C_0\varepsilon(1+t)^{-\frac{11}{4}},\label{u5}
\end{align}
where $k\in\{0,1\}$, $i,j\in\{2,3\}$ and $l,h\in\{1,2\}$ and $C_0$ will be specified later.  By means of the method of spectral analysis, we can show that the bounds in the ansatz are indeed halved, provided $C_0$ is chosen to be large enough and $\varepsilon$ is chosen to be suitably small. Then the bootstrapping argument asserts that the  half upper bounds  in the ansatz indeed hold for all time.

We begin with the following estimates of nonlinear terms, which are mathematically crucial for the entire analysis.
\begin{lem}\label{lemma4.5}
	Assume that  $(u,b)$	 is a smooth solution of the problem \eqref{PMHD}, satisfying \eqref{u101} and \eqref{u4}. Then, there exists  an absolute positive constant $C$, such that
	\begin{align}
	&\left\|\| b\cdot \na b \|_{L_{x_2,x_3}^1}\right\|_{L_{x_1}^2}\leq CC_0^2\varepsilon^2 \left(1+t\right)^{-\frac{11}{8}},\label{b12}\\
	&\left\|\|\p_1\left( b\cdot \na b\right)\|_{L_{x_2,x_3}^1}\right\|_{L_{x_1}^2}\leq CC_0^2\varepsilon^2 \left(1+t\right)^{-\frac{11}{8}},\label{p1b1}\\
	&\|\p_1\left(b\cdot \na b\right)\|_{L^2}\leq  C C_0^{2} \varepsilon^2 \left(1+t\right)^{-\frac{15}{8}}, \label{p1b2}\\
	&\|\p_i\left( b\cdot \na b\right)\|_{L^2}\leq CC_0^2\varepsilon^2 \left(1+t\right)^{-\frac{19}{8}}, \quad\forall\ i\in\{2,3\},  \label{pib2}\\
	&\left\|\|\p_1^2\left( b\cdot \na b\right)\|_{L_{x_2,x_3}^1}\right\|_{L_{x_1}^2}\leq   CC_0^2\varepsilon^2\left(1+t\right)^{-\frac{5}{4}},\label{p11b1} \\
	&\|\p_1^2(b\cdot \na b)\|_{L^2}\leq CC_0^2\varepsilon^2\left(1+t\right)^{-\frac{7}{4}},   \label{p11b2}\\
	&\|\p_l\p_2(b\cdot \na b)\|_{L^2}\leq C C_0^2\varepsilon^2\left(1+t\right)^{-2}, \quad\forall \ l\in\{1,2\},   \label{pl2b2}\\
	&\|\p_1\p_3(b\cdot \na b)\|_{L^2}\leq CC_0^2\varepsilon^2 \left(1+t\right)^{-\frac{19}{8}}, \label{p13b2} \\
	&\|\p_i\p_3(b\cdot \na b)\|_{L^2}\leq CC_0^2\varepsilon^2 \left(1+t\right)^{-\frac{23}{8}},\quad\forall\ i\in\{2,3\},   \label{pi3b2} \\
	&\left\|\|\p_1^3\left( b\cdot \na b\right)\|_{L_{x_2,x_3}^1}\right\|_{L_{x_1}^2}\leq  CC_0\varepsilon\left(1+t\right)^{-\frac12}\|\p_{\nu}b\|_{H^3}+   CC_0\varepsilon^2\left(1+t\right)^{-\frac78},\label{p111b1} \\
	&\|\p_l\p_g\p_h(b\cdot \na b)\|_{L^2}\leq CC_0\varepsilon^2\left(1+t\right)^{-1}, \quad \forall\ l,g,h\in\{1,2\},  \label{plghb2}\\
	&\|\p_1^2\p_3(b\cdot \na b)\|_{L^2}\leq CC_0^2\varepsilon^2\left(1+t\right)^{-
   \frac{15}{8}},   \label{p113b2}\\
	&\|\p_l\p_2\p_3(b\cdot \na b)\|_{L^2}\leq CC_0^2\varepsilon^2\left(1+t\right)^{-2},\quad\forall\   l\in\{1,2\}, \label{pl23b2}\\
	&\|\p_1\p_3^2(b\cdot \na b)\|_{L^2}\leq CC_0^2\varepsilon^2\left(1+t\right)^{-\frac{21}{8}},   \label{p133b2}\\
	&\|\p_2\p_3^2(b\cdot \na b)\|_{L^2}\leq CC_0^2\varepsilon^2\left(1+t\right)^{-3},   \label{p233b2}\\
	&\|\p_3^3(b\cdot \na b)\|_{L^2}\leq CC_0^2\varepsilon^2\left(1+t\right)^{-\frac{27}{8}}.   \label{p333b2}
	\end{align}
\end{lem}
\begin{proof}
The estimates of (\ref{b12}) and (\ref{p1b1}) can be shown in the same way as that used for (\ref{pbl1}) and (\ref{1bb11}), using the decay rate of $\|\p_1^2(u,b)\|_{L^2}\leq C(1+t)^{-1/2}$ stated in (\ref{u4}). To proceed, we need to improve the estimates of $\|b\|_{L^\infty}$ and $\|b_1\|_{L^\infty}$.
Similarly to that in Lemma \ref{lemma5.1}, by (\ref{u101}) and \eqref{u4}  we have
\begin{align}
\| b \|_{L^\infty}
&\leq C \| b \|_{L^2}^{\frac 18}\|\p_1b \|_{L^2}^{\frac 18}\|\p_2b \|_{L^2}^{\frac 18}\|\p_3b \|_{L^2}^{\frac 18}\nonumber\\
&\qquad\times\|\p_{1}\p_2 b \|_{L^2}^{\frac 18}\|\p_{1}\p_3b \|_{L^2}^{\frac 18}\|\p_{2}\p_3b \|_{L^2}^{\frac 18}\|\p_{1}\p_2\p_3b \|_{L^2}^{\frac 18}\nonumber\\
&\leq CC_0 \varepsilon \left(1+t\right)^{-1}, \label{bbinfty}
\end{align}
and
\begin{align}
\| b_1 \|_{L^\infty}
&\leq C \| b_1 \|_{L^2}^{\frac 18}\|\p_1b_1 \|_{L^2}^{\frac 18}\|\p_2b_1 \|_{L^2}^{\frac 18}\|\p_3b_1 \|_{L^2}^{\frac 18}\nonumber\\
&\qquad\times\|\p_{1}\p_2b_1 \|_{L^2}^{\frac 18}\|\p_{1}\p_3 b_1 \|_{L^2}^{\frac 18}\|\p_{2}\p_3 b_1 \|_{L^2}^{\frac 18}\|\p_{1}\p_2\p_3b_1 \|_{L^2}^{\frac 18}\nonumber\\
&\leq CC_0 \varepsilon \left(1+t\right)^{-\frac{11}{8}}. \label{b1binfty}
\end{align}

With the help of  \eqref{bbinfty}, \eqref{b1binfty} and the anisotropic Sobolev inequalities \eqref{anso}, we infer from (\ref{u101})  and  \eqref{u4}  that
\begin{align}
\|\p_1 (b\cdot \na b )\|_{L^2}
&\leq  \| b_1 \|_{L^\infty}\| \p_1^2 b \|_{L^2}+\| b_\nu \|_{L^\infty}\| \p_1\p_\nu b \|_{L^2} +\| \p_1  b\cdot \na b\|_{L^2}\nonumber\\
& \leq C C_0^{2} \varepsilon^2 \left(1+t\right)^{-\frac{15}{8}}+\|\p_1b_1\|_{L_{x_1}^\infty L_{x_2,x_3}^2}\| \p_1 b \|_{ L_{x_1}^2L_{x_2,x_3 }^\infty}\nonumber\\
&\quad+\|\p_1  b_\nu\|_{L_{x_1}^2L_{x_2,x_3}^\infty }\|\p_\nu b \|_{L_{x_1}^\infty L_{x_2,x_3}^2} \nonumber\\
&\leq C C_0^{2} \varepsilon^2 \left(1+t\right)^{-\frac{15}{8}},\label{p1u1b2}
\end{align}
due to the divergence-free condition $\nabla\cdot b=0$ and the fact that $\p_1 b_1=-\p_\nu b_\nu$ with $\p_\nu=(\p_2,\p_3)$ and $b_\nu=(b_2,b_3)$.
Analogously,
\begin{align}
&
\|\p_2 ( b\cdot \na b )\|_{L^2}+\|\p_3(b\cdot \na b)\|_{L^2}\nonumber\\
&\quad \leq \|b_1\|_{ L^\infty}\|(\p_2\p_1 b,\p_3\p_1 b)\|_{L^2}+\|b_\nu\|_{ L^\infty}\|(\p_2\p_\nu b,\p_3\p_\nu b)\|_{L^2}\nonumber\\
&\qquad +\|(\p_2  b_1,\p_3b_1)\|_{L_{x_1}^\infty L_{x_2,x_3}^2}\|\p_1 b\|_{L_{x_1}^2L_{x_2,x_3}^\infty }\nonumber\\
&\qquad+\|(\p_2  b_\nu, \p_3b_\nu)\|_{L_{x_1}^\infty L_{x_2,x_3}^2}\|\p_\nu b \|_{ L_{x_1}^2L_{x_2,x_3}^\infty}\nonumber\\
&\quad\leq C  C_0^{2} \varepsilon^2\left(1+t\right)^{-\frac{19}{8}}. \label{p3bl2}
\end{align}

Based on the H\"{o}lder inequality and (\ref{anso}) that, by  (\ref{u101})  and  \eqref{u4} we have
\begin{align}
&\left\|\|\p_1^2\left( b\cdot \na b\right)\|_{L_{x_2,x_3}^1}\right\|_{L_{x_1}^2}\nonumber\\
&\quad\leq \left\|\|\p_1^2\left( b_{\nu} \p_{\nu} b\right)\|_{L_{x_2,x_3}^1}\right\|_{L_{x_1}^2}+\left\|\|\p_1^2\left( b_1\p_1 b\right)\|_{L_{x_2,x_3}^1}\right\|_{L_{x_1}^2}\nonumber\\
&\quad\leq \left\|\|\p_1^2b_{\nu}\|_{L_{x_2,x_3}^2}\|\p_{\nu} b \|_{L_{x_2,x_3}^2}\right\|_{L_{x_1}^2}+2\left\|\|\p_1b_{\nu}\|_{L_{x_2,x_3}^2}\|\p_{\nu}\p_1 b \|_{L_{x_2,x_3}^2}\right\|_{L_{x_1}^2}\nonumber\\
&\qquad+\left\|\|b_{\nu}\|_{L_{x_2,x_3}^2}\|\p_{\nu}\p_1^2b\|_{L_{x_2,x_3}^2}\right\|_{L_{x_1}^2}+\left\|\|\p_1^2b_1\|_{L_{x_2,x_3}^2}\|\p_1 b \|_{L_{x_2,x_3}^2}\right\|_{L_{x_1}^2}\nonumber\\
&\qquad+2\left\|\|\p_1b_1\|_{L_{x_2,x_3}^2}\|\p_1^2b \|_{L_{x_2,x_3}^2}\right\|_{L_{x_1}^2}+\left\|\|b_1\|_{L_{x_2,x_3}^2}\|\p_1^3b \|_{L_{x_2,x_3}^2}\right\|_{L_{x_1}^2}\nonumber\\
&\quad\leq C\|\p_1^2b_{\nu}\|_{L^2}\|\p_{\nu} b\|_{L^2}^{\frac12}\|\p_1\p_{\nu} b\|_{L^2}^{\frac12}+C\|\p_{\nu}\p_1b \|_{L^2}\|\p_1 b_{\nu}\|_{L^2}^{\frac12}\|\p_1^2 b_{\nu}\|_{L^2}^{\frac12}\nonumber\\
&\qquad+C\|\p_{\nu}\p_1^2 b \|_{L^2}\|b_{\nu}\|_{L^2}^{\frac12}\|\p_1b_{\nu}\|_{L^2}^{\frac12}+C\|\p_1 b  \|_{L^2}\|\p_1^2b_1\|_{L^2}^{\frac12}\| \p_1^3b_{1}\|_{L^2}^{\frac12}\nonumber\\
&\qquad+C\|\p_1^2b \|_{L^2}\|\p_1 b_1\|_{L^2}^{\frac12}\|\p_1^2 b_1\|_{L^2}^{\frac12}+C\|\p_1^3b \|_{L^2}\|b_1\|_{L^2}^{\frac12}\|\p_1b_1\|_{L^2}^{\frac12}\nonumber\\
&\quad\leq CC_0^2\varepsilon^2\left(1+t\right)^{-\frac{5}{4}}.\label{p11bl1}
\end{align}

Using the divergence-free condition $\nabla\cdot b=0$, (\ref{u101}) and (\ref{u4}), we deduce
\begin{align}
\|\p_1^2\left( b\cdot \na b\right)\|_{L^2}
&\leq C\|b_1\|_{L^\infty}\|\p_1^3 b\|_{L^2}+C\|b_{\nu}\|_{L^\infty}\|\p_{\nu}\p_1^2 b\|_{L^2}\nonumber\\
&\quad+C\|\p_1^2b_1\|_{L_{x_1,x_2}^2L_{x_3}^\infty}\|\p_1 b  \|_{L_{x_1,x_2}^\infty L_{x_3}^2}\nonumber\\
&\quad+C\|\p_1^2 b_\nu \|_{L_{x_1}^\infty L_{x_2,x_3}^2} \|\p_{\nu} b \|_{L_{x_1}^2 L_{x_2,x_3}^\infty}\nonumber\\
&\quad+C\|\p_1^2b\|_{L_{x_1,x_2}^2L_{x_3}^\infty}\|\p_1 b_1 \|_{L_{x_1,x_2}^\infty L_{x_3}^2}\nonumber\\
&\quad+C\|\p_1\p_\nu b \|_{L_{x_1}^\infty L_{x_2,x_3}^2} \|\p_1 b_\nu \|_{L_{x_1}^2 L_{x_2,x_3}^\infty}\nonumber\\
&\leq CC_0^2\varepsilon^2\left(1+t\right)^{-\frac74}. \label{p11bl2}
\end{align}
By the same token, we have
\begin{align*}
\|\p_l\p_2\left( b\cdot \na b\right)\|_{L^2}
&\leq C\|b_1\|_{L^\infty}\|\p_l\p_2\p_1 b\|_{L^2}+C\|b_{\nu}\|_{L^\infty}\|\p_l\p_2\p_{\nu} b\|_{L^2}\nonumber\\
&\quad+C\|\p_l\p_2 b_1\|_{L_{x_1,x_2}^2L_{x_3}^\infty}\|\p_1 b  \|_{L_{x_1,x_2}^\infty L_{x_3}^2}\nonumber\\
&\quad+C\|\p_l\p_2 b_\nu \|_{L_{x_1}^\infty L_{x_2,x_3}^2} \|\p_{\nu} b \|_{L_{x_1}^2 L_{x_2,x_3}^\infty}\nonumber\\
&\quad+C\|(\p_1\p_l b,\p_1\p_2 b)\|_{L_{x_1,x_2}^2L_{x_3}^\infty}\|(\p_2 b_1,\p_lb_1) \|_{L_{x_1,x_2}^\infty L_{x_3}^2}\nonumber\\
&\quad+C\|(\p_2\p_\nu b,\p_l\p_\nu b) \|_{L_{x_1}^\infty L_{x_2,x_3}^2} \|(\p_l b_\nu,\p_2 b_\nu) \|_{L_{x_1}^2 L_{x_2,x_3}^\infty}\nonumber\\
&\leq CC_0^2\varepsilon^2\left(1+t\right)^{-2},\quad \forall\ l\in\{1,2\},\nonumber\\[2mm]
\|\p_1\p_3\left( b\cdot \na b\right)\|_{L^2}
&\leq C\|b_1\|_{L^\infty}\|\p_1^2\p_3 b\|_{L^2}+C\|b_{\nu}\|_{L^\infty}\|\p_1\p_3\p_{\nu} b\|_{L^2}\nonumber\\
&\quad+C\|\p_1\p_3 b_1\|_{L_{x_1,x_2}^2L_{x_3}^\infty}\|\p_1 b  \|_{L_{x_1,x_2}^\infty L_{x_3}^2}\nonumber\\
&\quad+C\|\p_1\p_3 b_\nu \|_{L_{x_1}^\infty L_{x_2,x_3}^2} \|\p_{\nu} b \|_{L_{x_1}^2 L_{x_2,x_3}^\infty}\nonumber\\
&\quad+C\|\p_1\p_3 b \|_{L_{x_1,x_2}^2L_{x_3}^\infty}\|\p_1 b_1 \|_{L_{x_1,x_2}^\infty L_{x_3}^2}\nonumber\\
&\quad+C\|\p_3\p_\nu b\|_{L_{x_1}^\infty L_{x_2,x_3}^2} \|\p_1 b_\nu \|_{L_{x_1}^2 L_{x_2,x_3}^\infty}\nonumber\\
&\quad+C\| \p_1^2 b\|_{L_{x_1,x_2}^2L_{x_3}^\infty}\| \p_3 b_1\|_{L_{x_1,x_2}^\infty L_{x_3}^2}\nonumber\\
&\quad+C\|\p_1\p_\nu b \|_{L_{x_1}^\infty L_{x_2,x_3}^2} \|\p_3 b_\nu \|_{L_{x_1}^2 L_{x_2,x_3}^\infty}\nonumber\\
&\leq CC_0^2\varepsilon^2\left(1+t\right)^{-\frac{19}{8}},\nonumber
\end{align*}
and
\begin{align*}
\|\p_i\p_3\left( b\cdot \na b\right)\|_{L^2}
&\leq C\|b_1\|_{L^\infty}\|\p_1\p_i\p_3 b\|_{L^2}+C\|b_{\nu}\|_{L^\infty}\|\p_i\p_3\p_{\nu} b\|_{L^2}\nonumber\\
&\quad+C\|\p_i\p_3 b_1\|_{L_{x_2,x_3}^2L_{x_1}^\infty}\|\p_1 b  \|_{L_{x_2,x_3}^\infty L_{x_1}^2}\nonumber\\
&\quad+C\|\p_i\p_3 b_\nu \|_{ L_{x_2,x_3}^2L_{x_1}^\infty} \|\p_{\nu} b \|_{ L_{x_2,x_3}^\infty L_{x_1}^2}\nonumber\\
&\quad+C\|(\p_1\p_i b, \p_1\p_3 b) \|_{L_{x_1,x_3}^2L_{x_2}^\infty}\|(\p_3 b_1,\p_i b_1) \|_{L_{x_1,x_3}^\infty L_{x_2}^2}\nonumber\\
&\quad+C\|(\p_\nu\p_i b, \p_\nu\p_3 b) \|_{L_{x_1,x_3}^2L_{x_2}^\infty}\|(\p_3 b_\nu,\p_i b_\nu) \|_{L_{x_1,x_3}^\infty L_{x_2}^2}\nonumber\\
&\leq CC_0^2\varepsilon^2\left(1+t\right)^{-\frac{23}{8}},\quad \forall\ i\in\{2,3\}.
\end{align*}

Analogously to the derivations of  \eqref{p11bl1} and \eqref{p11bl2}, using the fact that $\nabla\cdot b=0$, we obtain
\begin{align}
&\left\|\|\p_1^3\left( b\cdot \na b\right)\|_{L_{x_2,x_3}^1}\right\|_{L_{x_1}^2}\nonumber\\
&\quad\leq \left\|\|\p_1^3\left( b_{\nu} \p_{\nu} b\right)\|_{L_{x_2,x_3}^1}\right\|_{L_{x_1}^2}+\left\|\|\p_1^3\left( b_1\p_1 b\right)\|_{L_{x_2,x_3}^1}\right\|_{L_{x_1}^2}\nonumber\\
&\quad\leq C\|\p_{\nu}b\|_{L^2}\|\p_{1}^3 b\|_{L^2}^{\frac12}\|\p_{1}^4 b   \|_{L^2}^{\frac12}+C\|\p_{\nu}\p_1 b \|_{L^2}\|\p_1^2b \|_{L^2}^{\frac12}\|\p_1^3b \|_{L^2}^{\frac12}\nonumber\\
&\qquad+C\|\p_{\nu}\p_1^2 b \|_{L^2}\|\p_1b \|_{L^2}^{\frac12}\|\p_1^2b \|_{L^2}^{\frac12}+C\|\p_{\nu}\p_1^3 b \|_{L^2}\|b_{\nu}\|_{L^2}^{\frac12}\|\p_1b_{\nu}\|_{L^2}^{\frac12}\nonumber\\
&\qquad +C\|\p_1^4b \|_{L^2}\|b_1\|_{L^2}^{\frac12}\|\p_1b_1\|_{L^2}^{\frac12}\nonumber\\
&\quad\leq CC_0\varepsilon^2\left(1+t\right)^{-\frac12}\|\p_{\nu}b\|_{H^3}+ CC_0\varepsilon^2\left(1+t\right)^{-\frac78} \nonumber
\end{align}
and
\begin{align}
\|\p_1^3\left( b\cdot \na b\right)\|_{L^2}
&\leq C\|b\|_{L^\infty}\|\nabla\p_1^3 b\|_{L^2} +C\| \p_1b\|_{L_{x_2,x_3}^\infty L_{x_1}^2}\|\na\p_1^2 b\|_{ L_{x_2,x_3}^2 L_{x_1}^\infty}\nonumber\\
&\quad+C\|\p_1^2 b\|_{ L_{x_2,x_3}^\infty L^2_{x_1}} \|\nabla\p_1 b \|_{ L_{x_2,x_3}^2L_{x_1}^\infty}\nonumber\\
&\quad +C\|\p_1^3 b \|_{ L_{x_2,x_3}^2 L_{x_1}^\infty}\|\nabla b\|_{L_{x_2,x_3}^\infty L_{x_1}^2}\nonumber\\
&\leq CC_0 \varepsilon^2\left(1+t\right)^{-1}. \label{bb112}
\end{align}

In a similar manner as that used for \eqref{bb112}, we conclude from (\ref{u101}) and (\ref{u4}) that
\begin{align*}
\|\p_l\p_g \p_h \left( b\cdot \na b\right)\|_{L^2}
 \leq  CC_0 \varepsilon^2\left(1+t\right)^{-1},\quad\forall\ l,g,h\in\{1,2\}.
\end{align*}

For (\ref{p113b2}) and (\ref{pl23b2}), by (\ref{u101}), (\ref{u4}) and direct calculations we have
\begin{align*}
&\|\p_1^2\p_3\left( b\cdot \na b\right)\|_{L^2}\nonumber\\
&\quad\leq C\|b_1\|_{L^\infty}\| \p_1^3\p_3 b\|_{L^2} +C\|b_\nu\|_{L^\infty}\| \p_1^2\p_3\p_\nu b\|_{L^2}\nonumber\\
&\qquad+C\|\p_3 b_1\|_{L^\infty_{x_1,x_2} L^2_{x_3}}\|\p_1^3 b\|_{L^2_{x_1,x_2} L^\infty_{x_3}} +C\|\p_3 b_\nu\|_{L^\infty_{x_1,x_2} L^2_{x_3}}\|\p_\nu\p_1^2 b\|_{L^2_{x_1,x_2} L^\infty_{x_3}}\nonumber\\
&\qquad+C\| \p_\nu b\|_{L_{x_2,x_3}^\infty L_{x_1}^2}\|\p_1^2\p_3 b\|_{ L_{x_2,x_3}^2 L_{x_1}^\infty} +C\| \p_1b \|_{L_{x_2,x_3}^\infty L_{x_1}^2}\|\p_\nu\p_1\p_3 b\|_{ L_{x_2,x_3}^2 L_{x_1}^\infty}\nonumber\\
&\qquad+C\|\p_\nu\p_1 b\|_{L^\infty_{x_2,x_3} L^2_{x_1}}\|\p_1\p_3 b\|_{L^2_{x_2,x_3} L^\infty_{x_1}} +C\|\p_1^2 b\|_{L^\infty_{x_2,x_3} L^2_{x_1}}\|\p_\nu\p_3 b\|_{L^2_{x_2,x_3} L^\infty_{x_1}}\nonumber\\
&\quad\leq CC_0^2 \varepsilon^2\left(1+t\right)^{-\frac{15}{8}},
\end{align*}
and analogously,
\begin{align*}
&\|\p_l\p_2\p_3\left( b\cdot \na b\right)\|_{L^2}\nonumber\\
&\quad\leq C\|b_1\|_{L^\infty}\|\p_l\p_2\p_3\p_1 b\|_{L^2} +C\|b_\nu\|_{L^\infty}\| \p_l\p_2\p_3\p_\nu b\|_{L^2}\nonumber\\
&\qquad+C\|\p_3 b\|_{L^\infty_{x_1,x_2} L^2_{x_3}}\|\nabla\p_l\p_2 b\|_{L^2_{x_1,x_2} L^\infty_{x_3}}+C\|\p_2 b\|_{L^\infty_{x_1,x_2} L^2_{x_3}}\|\nabla\p_l\p_3 b\|_{L^2_{x_1,x_2} L^\infty_{x_3}}\nonumber\\
&\qquad  +C\| \p_l b\|_{L_{x_1,x_2}^\infty L_{x_3}^2}\|\nabla\p_2\p_3 b\|_{ L_{x_1,x_2}^2 L_{x_3}^\infty} +C\|\p_l \p_2b \|_{L_{x_2,x_3}^\infty L_{x_1}^2}\|\nabla\p_3 b\|_{ L_{x_2,x_3}^2 L_{x_1}^\infty}\nonumber\\
&\qquad+C\|\p_l\p_3 b\|_{L^\infty_{x_2,x_3} L^2_{x_1}}\|\nabla\p_2 b\|_{L^2_{x_2,x_3} L^\infty_{x_1}} +C\|\p_2\p_3 b\|_{L^\infty_{x_1,x_2} L^2_{x_3}}\|\nabla\p_l b\|_{L^2_{x_1,x_2} L^\infty_{x_3}}\nonumber\\
&\qquad+C\|\p_l\p_2\p_3 b\|_{L^2_{x_2,x_3} L^\infty_{x_1}}\|\nabla b\|_{L^\infty_{x_2,x_3} L^2_{x_1}} \nonumber\\
&\quad\leq CC_0^2 \varepsilon^2\left(1+t\right)^{-2},\quad\forall\ l\in\{1,2\}.
\end{align*}

Finally, it is easily deduced from (\ref{anso}), (\ref{u101}), (\ref{u4}) and the divergence-free condition $\nabla\cdot b=0$ that
\begin{align*}
&\|\p_1\p_3^2\left( b\cdot \na b\right)\|_{L^2}\nonumber\\
&\quad\leq C\|b_1\|_{L^\infty}\| \p_1^2\p_3^2 b\|_{L^2} +C\|b_\nu\|_{L^\infty}\| \p_1\p_3^2\p_\nu b\|_{L^2}\nonumber\\
&\qquad+C\|\p_1 b_1\|_{L^\infty_{x_2,x_3} L^2_{x_1}}\|\p_1\p_3^2 b\|_{L^2_{x_2,x_3} L^\infty_{x_1}}+C\|\p_1 b_\nu\|_{L^\infty_{x_2,x_3} L^2_{x_1}}\|\p_\nu\p_3^2 b\|_{L^2_{x_2,x_3} L^\infty_{x_1}}\nonumber\\
&\qquad+C\|\p_3 b_1\|_{L^\infty_{x_1,x_2} L^2_{x_3}}\|\p_1^2\p_3 b\|_{L^2_{x_1,x_2} L^\infty_{x_3}}+C\|\p_3 b_\nu\|_{L^\infty_{x_1,x_2} L^2_{x_3}}\|\p_\nu\p_1\p_3 b\|_{L^2_{x_1,x_2} L^\infty_{x_3}}\nonumber\\
&\qquad  +C\| \p_1\p_3 b_1\|_{L_{x_2,x_3}^\infty L_{x_1}^2}\|\p_1\p_3 b\|_{ L_{x_2,x_3}^2 L_{x_1}^\infty} +C\|\p_1 \p_3b_\nu \|_{L_{x_2,x_3}^\infty L_{x_1}^2}\|\p_\nu\p_3 b\|_{ L_{x_2,x_3}^2 L_{x_1}^\infty}\nonumber\\
&\qquad+C\|\p_3^2 b_1\|_{L^\infty_{x_1,x_2} L^2_{x_3}}\|\p_1^2 b\|_{L^2_{x_1,x_2} L^\infty_{x_3}} +C\|\p_3^2 b_\nu\|_{L^\infty_{x_1,x_2} L^2_{x_3}}\|\p_\nu\p_1 b\|_{L^2_{x_1,x_2} L^\infty_{x_3}}\nonumber\\
&\qquad+C\|\p_1\p_3^2 b_1\|_{L^2_{x_2,x_3} L^\infty_{x_1}}\|\p_1b\|_{L^\infty_{x_2,x_3} L^2_{x_1}} +C\|\p_1\p_3^2 b_\nu\|_{L^2_{x_2,x_3} L^\infty_{x_1}}\|\p_\nu b\|_{L^\infty_{x_2,x_3} L^2_{x_1}}\nonumber\\
&\quad\leq CC_0^2 \varepsilon^2\left(1+t\right)^{-\frac{21}{8}},
\end{align*}
\begin{align*}
&\|\p_2\p_3^2\left( b\cdot \na b\right)\|_{L^2}\nonumber\\
&\quad\leq C\|b\|_{L^\infty}\|\nabla \p_2\p_3^2 b\|_{L^2}  +C\|\p_\nu b_1\|_{L^\infty_{x_1,x_2} L^2_{x_3}}\|\p_1\p_\nu\p_3 b\|_{L^2_{x_1,x_2} L^\infty_{x_3}}\nonumber\\
&\qquad+C\|\p_\nu b\|_{L^\infty_{x_1,x_2} L^2_{x_3}}\|\p_\nu^2\p_3 b\|_{L^2_{x_1,x_2} L^\infty_{x_3}} +C\| \p_2\p_3 b_1\|_{L_{x_1,x_3}^\infty L_{x_2}^2}\|\p_1\p_3 b\|_{ L_{x_1,x_3}^2 L_{x_2}^\infty}  \nonumber\\
&\qquad +C\|\p_2 \p_3b_\nu \|_{L_{x_2,x_3}^\infty L_{x_1}^2}\|\p_\nu\p_3 b\|_{ L_{x_2,x_3}^2 L_{x_1}^\infty} +C\|\p_3^2 b\|_{L^\infty_{x_1,x_2} L^2_{x_3}}\|\nabla\p_2 b\|_{L^2_{x_1,x_2} L^\infty_{x_3}} \nonumber\\
&\qquad+C\|\p_2\p_3^2 b\|_{L^2_{x_2,x_3} L^\infty_{x_1}}\|\nabla b\|_{L^\infty_{x_2,x_3} L^2_{x_1}}  \nonumber\\
&\quad\leq CC_0^2 \varepsilon^2\left(1+t\right)^{-3},
\end{align*}
and analogously,
\begin{align*}
&\|\p_3^3\left( b\cdot \na b\right)\|_{L^2}\nonumber\\
&\quad\leq C\|b_1\|_{L^\infty}\|\p_1 \p_3^3 b\|_{L^2} + C\|b_\nu\|_{L^\infty}\|\p_\nu \p_3^3 b\|_{L^2} \nonumber\\
&\qquad+C\|\p_3 b_1\|_{L^\infty_{x_1,x_2} L^2_{x_3}}\|\p_1\p_3^2 b\|_{L^2_{x_1,x_2} L^\infty_{x_3}}+C\|\p_3 b_\nu\|_{L^\infty_{x_1,x_2} L^2_{x_3}}\|\p_\nu \p_3^2 b\|_{L^2_{x_1,x_2} L^\infty_{x_3}} \nonumber\\
&\qquad+C\| \p_3^2 b_1\|_{L_{x_2,x_3}^2 L_{x_1}^\infty}\|\p_1\p_3 b\|_{ L_{x_2,x_3}^\infty L_{x_1}^2} +C\| \p_3^2b_\nu \|_{L_{x_2,x_3}^\infty L_{x_1}^2}\|\p_\nu\p_3 b\|_{ L_{x_2,x_3}^2 L_{x_1}^\infty}  \nonumber\\
&\qquad +C\|\p_3^3 b_1\|_{L^2_{x_2,x_3} L^\infty_{x_1}}\|\p_1 b\|_{L^\infty_{x_2,x_3} L^2_{x_1} }+C\|\p_3^3 b_\nu\|_{L^2_{x_2,x_3} L^\infty_{x_1}}\|\p_\nu b\|_{L^\infty_{x_2,x_3} L^2_{x_1} }  \nonumber\\
&\quad\leq CC_0^2 \varepsilon^2\left(1+t\right)^{-\frac{27}{8}}.
\end{align*}

Now, collecting the above estimates together finishes the proof of Lemma \ref{lemma4.5}.
\end{proof}

\vskip .1in
\subsection{The decay rates of $(u,b)$.}

In this subsection, we prove the decay estimates of $(u,b)$ stated in Theorem \ref{thm1.3}. Analogously to the proof of Theorem \ref{thm1.2}, we only consider the large-time behavior for $t\geq1$. Moreover, it suffices to deal with the decay estimates of $u$, since those of $b$ can be achieved similarly.

 \vskip .1in
\textbf{Step  I.  The decay rates  of $\|\p_{i}\p_j u\|_{L^2}$ with $ i,j\in\{1,2,3\}$ }

\vskip .1in

We start with the decay rates of the second-order derivatives. It follows from \eqref{u} that for $\forall\ i,j\in\{1,2,3\}$,
\begin{align}
\|\p_{i}\p_ju\|_{L^2}&=\|\widehat{\p_{ij}u}\|_{L^2}=\|\xi_{i}\xi_{j}\widehat{u}\|_{L^2}\nonumber\\
&\leq \left\| \xi_{i}\xi_{j}\widehat{K_1}(t)\widehat{u_0}\right\|_{L^2}+\left\||\xi_{i}\xi_{j}|\widehat{K_2}(t)\widehat{b_0}\right\|_{L^2}\nonumber\\
&\quad+\int_0^t\left\|\xi_{i}\xi_{j}\widehat{K_1}(t-\tau)\widehat{N_1}(\tau)\right\|_{L^2}d\tau\nonumber\\
&\quad+\int_0^t\left\|\xi_{i}\xi_{j}\widehat{K_2}(t-\tau)\widehat{N_2}(\tau)\right\|_{L^2}d\tau. \label{2uL2}
\end{align}

\textbf{Step I-1.  The decay rates  of $\|\p_i\p_{j}u\|_{L^2}$ with $i,j\in \{2,3\}$.}
\vskip .1in

Similarly the treatment of Step III-1 in Section 4, we only deal with the terms associated with $b\cdot\nabla b$. by (\ref{S1}), (\ref{S2}) and Corollary \ref{PI} we infer from  \eqref{b12}, \eqref{pib2}, \eqref{pl2b2}, \eqref{pi3b2} and \eqref{2uL2} that for $i,j\in\{2,3\}$,
\begin{align*}
\|\p_i\p_{j}u\|_{L^2}
&\leq   C(1+t)^{-\frac32}\left(\| ({u_0}, {b_0})\|_{L^2_{x_1}L^1_{x_2, x_3}}+\|\p_i\p_j ({u_0}, {b_0})\|_{L^2}\right)\nonumber\\
&\quad+ C\int_0^{\frac t2}\left(1+t-\tau \right)^{-\frac32}\|b\cdot \na b\|_{L_{x_1}^2L_{x_2,x_3}^1}d\tau\nonumber\\
&\quad+C\int_{\frac t2}^{t}\left(t-\tau \right)^{-\frac12} \|\p_i(b\cdot \na b)\|_{L^2}d\tau \nonumber\\
&\quad+C\int_{0}^{t}e^{-c\left(t-\tau\right)}\|\p_{i}\p_j(b\cdot \na b)\|_{L^2}d\tau\nonumber\\
&\leq C \varepsilon\left(1+t\right)^{-\frac32}+ CC_0^2\varepsilon^2\int_0^{\frac t2}\left(1+t-\tau \right)^{-\frac32}\left(1+\tau\right)^{-\frac{11}{8}}d\tau\nonumber\\
&\quad+CC_0^2\varepsilon^2\int_{\frac t2}^{t} \left(t-\tau\right)^{-\frac12}\left(1+\tau \right)^{-\frac{19}{8}} d\tau \nonumber\\
&\quad+CC_0^2 \varepsilon^2\int_{0}^{t} e^{-c\left(t-\tau\right)}\left(1+\tau\right)^{-2}d\tau \nonumber\\
&\leq  C\varepsilon(1+t)^{-\frac32}+ C C_0^2 \varepsilon^2\left(1+t\right)^{-\frac32}\leq \frac{C_0}{2} \varepsilon\left(1+t\right)^{-\frac32},
\end{align*}
provided $C_0$ is chosen to be large enough and $\varepsilon$ is chosen to be suitably small.

\vskip .1in
\textbf{Step I-2.  The decay rates   of  $\|\p_{1}\p_iu\|_{L^2}$ with $i\in\{2,3\}$}
\vskip .1in

For $i\in\{2,3\}$, using   \eqref{p1b1}, \eqref{pl2b2}, \eqref{p13b2} and  Corollary \ref{PI} with $\alpha=1$,   we deduce from (\ref{2uL2})  by choosing $C_0$ large enough and $\varepsilon$ suitably small that
\begin{align*}
\|\p_1\p_{i}u\|_{L^2}
&\leq   C\varepsilon(1+t)^{-1} +C\int_0^{t}\left(1+t-\tau \right)^{-1}\|\p_1(b\cdot \na b)  \|_{L_{x_1}^2L_{x_2,x_3}^1}d\tau\nonumber\\ &\quad+C\int_{0}^{t}e^{-c\left(t-\tau\right)}\|\p_i\p_1(b\cdot \na b)\|_{L^2}d\tau \nonumber\\
&\leq C \varepsilon\left(1+t\right)^{-1}+ CC_0^2\varepsilon^2\int_0^{t}\left(1+t-\tau \right)^{-1}\left(1+\tau\right)^{-\frac{11}{8}}d\tau\nonumber\\
&\quad+CC_0^2 \varepsilon^2\int_{0}^{t} e^{-c\left(t-\tau\right)}\left(1+\tau\right)^{-2}d\tau \nonumber\\
&\leq C\varepsilon(1+t)^{-1}+C C_0^2 \varepsilon^2(1+t)^{-1}
\leq \frac{C_0}{2} \varepsilon\left(1+t\right)^{-1}.
\end{align*}

\vskip .1in

\textbf{Step I-3.  The decay rate of $\|\p_1^2 u \|_{L^2}$}
 \vskip .1in

In view of \eqref{p11b1}, \eqref{p11b2} and Corollary \ref{PI} with $\alpha=0$, we obtain by taking $i=j=1$ in \eqref{2uL2} that
\begin{align*}
\|\p_1^2u\|_{L^2}
&\leq   C\varepsilon(1+t)^{-\frac12} +C\int_0^{t}\left(1+t-\tau \right)^{-\frac12}\|\p_1^2(b\cdot \na b)  \|_{L_{x_1}^2L_{x_2,x_3}^1}d\tau\nonumber\\ &\quad+C\int_{0}^{t}e^{-c\left(t-\tau\right)}\|\p_1^2(b\cdot \na b)\|_{L^2}d\tau \nonumber\\
&\leq C \varepsilon\left(1+t\right)^{-\frac12}+ CC_0^2\varepsilon^2\int_0^{t}\left(1+t-\tau \right)^{-1}\left(1+\tau\right)^{-\frac{5}{4}}d\tau\nonumber\\
&\quad+CC_0^2 \varepsilon^2\int_{0}^{t} e^{-c\left(t-\tau\right)}\left(1+\tau\right)^{-\frac{5}{4}}d\tau \nonumber\\
&\leq C\varepsilon(1+t)^{-\frac12}+C C_0^2 \varepsilon^2(1+t)^{-\frac12} \leq \frac{C_0}{2} \varepsilon\left(1+t\right)^{-\frac12},
\end{align*}
provided $C_0$ is chosen to be large enough and $\varepsilon$ is chosen to be suitably small.

\vskip .1in

\textbf{Step  II.  The decay rates  of $\|\p_i\p_j\p_ku\|_{L^2}$ with $ i,j,k\in\{1,2,3\}$}
\vskip .1in

It follows from  \eqref{u} that for $\forall\ i,j,k\in\{1,2,3\}$,
\begin{align}
\|\p_{ijk}u\|_{L^2}&=\|\widehat{\p_{ijk}u}\|_{L^2}
=\|\xi_{i}\xi_{j}\xi_{k}\widehat{u}\|_{L^2}\nonumber\\
&\leq \left\| \xi_{i}\xi_{j}\xi_{k}\widehat{K_1}(t)\widehat{u_0}\right\|_{L^2}+\left\|\xi_{i}\xi_{j}\xi_{k}\widehat{K_2}(t)\widehat{b_0}\right\|_{L^2}\nonumber\\
&\quad+\int_0^t\left\|\xi_{i}\xi_{j}\xi_{k}\widehat{K_1}(t-\tau)\widehat{N_1}(\tau)\right\|_{L^2}d\tau\nonumber\\
&\quad+\int_0^t\left\|\xi_{i}\xi_{j}\xi_{k}\widehat{K_2}(t-\tau)\widehat{N_2}(\tau)\right\|_{L^2}d\tau .\label{3uL2}
\end{align}

\vskip .1in

\textbf{Step II-1.  The decay rates  of $\|\p_1\p_{i}\p_{3}u\|_{L^2}$ with $i\in\{2,3\}$}
\vskip .1in

Using \eqref{p1b1},   \eqref{pl2b2}, \eqref{p13b2} and Corollary \ref{PI} with $\alpha=2$, we deduce from (\ref{S1}), (\ref{S2}), (\ref{ineq}) and  \eqref{3uL2} that for $i\in\{2,3\}$,
\begin{align*}
\|\p_1\p_{i}\p_3 u\|_{L^2}
&\leq   C\varepsilon(1+t)^{-\frac32} +C\int_0^{\frac t2}\left(1+t-\tau \right)^{-\frac32}\| \p_1(b\cdot \na b)\|_{L_{x_1}^2L_{x_2,x_3}^1}d\tau\nonumber\\
&\quad+C\int_{\frac t2}^{t}\left(t-\tau \right)^{-\frac12} \|\p_1\p_3(b\cdot \na b)\|_{L^2}d\tau \nonumber\\
&\quad+C\int_{0}^{t}e^{-c\left(t-\tau\right)}\left(t-\tau \right)^{-\frac12}\|\p_i\p_1(b\cdot \na b)\|_{L^2}d\tau \nonumber\\
&\leq C \varepsilon\left(1+t\right)^{-\frac32}+  CC_0^2\varepsilon^2\int_0^{\frac t2}\left(1+t-\tau \right)^{-\frac32} \left(1+\tau\right)^{-\frac{11}{8}}d\tau \nonumber\\ &\quad+CC_0^2\varepsilon^2\int_{\frac t2}^{t}\left(t-\tau \right)^{-\frac12} \left(1+\tau\right)^{-\frac{19}{8}}d\tau \nonumber\\
&\quad+CC_0^2 \varepsilon^2\int_{0}^{t} e^{-c\left(t-\tau\right)}\left(t-\tau \right)^{-\frac12}\left(1+\tau\right)^{-2}d\tau \nonumber\\
&\leq C\varepsilon(1+t)^{-\frac32}+C C_0^2 \varepsilon^2(1+t)^{-\frac32} \leq \frac{C_0}{2} \varepsilon(1+t)^{-\frac32},
\end{align*}
provided $C_0$ is chosen to be large enough and $\varepsilon$ is chosen to be suitably small.

\vskip .1in

\textbf{Step II-2. The decay rates  of $\|\p_1^2\p_{i}u\|_{L^2}$ with $i\in\{2,3\}$}
 \vskip .1in

Taking  $j=k=1$ in \eqref{3uL2}, using  (\ref{p11b1}), (\ref{plghb2}), (\ref{p113b2}) and    Corollary \ref{PI}, we obtain by  choosing $C_0$  large enough and $\varepsilon$   suitably small that for $i\in\{2,3\}$,
 \begin{align*}
 \|\p_1^2\p_iu\|_{L^2}
 &\leq   C\varepsilon(1+t)^{-1} +C\int_0^{t}\left(1+t-\tau \right)^{-1}\| \p_1^2(b\cdot \na b)\|_{L_{x_1}^2L_{x_2,x_3}^1}d\tau\nonumber\\
 &\quad +C\int_{0}^{t} e^{-c\left(t-\tau\right)}\|\p_1^2\p_i(b\cdot \na b)\|_{L^2}d\tau \nonumber\\
 &\leq C \varepsilon\left(1+t\right)^{-1}+CC_0^2\varepsilon^2\int_0^t\left(1+t-\tau \right)^{-1} \left(1+\tau\right)^{-\frac{5}{4}}d\tau\nonumber\\
 &\quad +CC_0^2\varepsilon^2\int_{0}^{t} e^{-c\left(t-\tau\right)}\left(1+\tau\right)^{-1}d\tau \nonumber\\
 &\leq C\varepsilon(1+t)^{-1}+C C_0^2 \varepsilon^2(1+t)^{-1} \leq \frac{C_0}{2} \varepsilon(1+t)^{-1}.
\end{align*}

\vskip .1in

\textbf{Step II-3.  The decay rates of $\|\p_l \p_{2}^2u\|_{L^2}$ with $l\in\{1,2\}$}
 \vskip .1in

For $i=1$ and $j=k=2$ in \eqref{3uL2}, by \eqref{p1b1} and \eqref{plghb2},  we have
\begin{align}
\|\p_1\p_2^2 u\|_{L^2}
&\leq   C\varepsilon(1+t)^{-\frac32} +C\int_0^{t}\left(1+t-\tau \right)^{-\frac32}\| \p_1(b\cdot \na b)\|_{L_{x_1}^2L_{x_2,x_3}^1}d\tau\nonumber\\
&\quad +C\int_{0}^{t} e^{-c\left(t-\tau\right)}\|\p_1\p_2^2(b\cdot \na b)\|_{L^2}d\tau \nonumber\\
&\leq C \varepsilon\left(1+t\right)^{-\frac32}+CC_0^2\varepsilon^2\int_0^t\left(1+t-\tau \right)^{-\frac32} \left(1+\tau\right)^{-\frac{11}{8}}d\tau\nonumber\\
&\quad+CC_0\varepsilon^2\int_{0}^{t} e^{-c\left(t-\tau\right)}\left(1+\tau\right)^{-1}d\tau \nonumber\\
&\leq C \varepsilon\left(1+t\right)^{-1}+C\left(C_0+C_0^2\right) \varepsilon^2\left(1+t\right)^{-1 }, \label{p122ug}
\end{align}
where we have used Corollary \ref{PI} with $\alpha=2$. Similarly, using \eqref{b12}, \eqref{plghb2} and  Corollary \ref{PI} with $\alpha=3$, we get
\begin{align*}
\|\p_2^3 u\|_{L^2}
&\leq C\varepsilon(1+t)^{-2} + C\int_0^{t}\left(1+t-\tau \right)^{-2}\|b\cdot \na b\|_{L_{x_1}^2L_{x_2,x_3}^1}d\tau\nonumber\\
&\quad +C\int_{0}^{t} e^{-c\left(t-\tau\right)}\|\p_2^3(b\cdot \na b)\|_{L^2}d\tau \nonumber\\
&\leq  C \varepsilon\left(1+t\right)^{-1}+CC_0^2\varepsilon^2\int_0^t\left(1+t-\tau \right)^{-2} \left(1+\tau\right)^{-\frac{11}{8}}d\tau\nonumber\\
&\quad +CC_0 \varepsilon^2\int_{0}^{t} e^{-c\left(t-\tau\right)}\left(1+\tau\right)^{-1}d\tau \nonumber\\
&\leq C \varepsilon\left(1+t\right)^{-1}+C\left(C_0+C_0^2\right) \varepsilon^2\left(1+t\right)^{-1 }, 
\end{align*}
which, together with  \eqref{p122ug}, shows
$$
\|\p_l\p_2^2 u(t)\|_{L^2}
 \leq \frac{C_0}{2} \varepsilon\left(1+t\right)^{-1 },\quad\forall~l\in\{1,2\},
 $$
provided $C_0$ is chosen to be large enough and $\varepsilon$ is chosen to be suitably small.

\vskip .1in

\textbf{Step II-4.  The decay rates of $\|\p_{i}\p_j\p_3u\|_{L^2}$ with $i,j\in \{2,3\}$}
 \vskip .1in

 Similarly to that in Step I-1, using \eqref{b12}, \eqref{pl2b2}, \eqref{pi3b2},  Corollary \ref{PI} with $\alpha=3$ and (\ref{ineq}),  we infer from  \eqref{3uL2} that for $i,j\in \{2,3\}$,
\begin{align*}
\|\p_i\p_{j}\p_3u\|_{L^2}
&\leq   C\varepsilon(1+t)^{-2} + C\int_0^{\frac t2}\left(1+t-\tau \right)^{-2}\|b\cdot \na b\|_{L_{x_1}^2L_{x_2,x_3}^1}d\tau\nonumber\\
&\quad+C\int_{\frac t2}^{t}\left(t-\tau \right)^{-\frac12} \|\p_i\p_3(b\cdot \na b)\|_{L^2}d\tau \nonumber\\
&\quad+C\int_{0}^{t}e^{-c\left(t-\tau\right)}\left(t-\tau\right)^{-\frac12}\|\p_{i}\p_j(b\cdot \na b)\|_{L^2}d\tau\nonumber\\
&\leq C \varepsilon\left(1+t\right)^{-2}+ CC_0^2\varepsilon^2\int_0^{\frac t2}\left(1+t-\tau \right)^{-2}\left(1+\tau\right)^{-\frac{11}{8}}d\tau\nonumber\\
&\quad+CC_0^2\varepsilon^2\int_{\frac t2}^{t} \left(t-\tau\right)^{-\frac12}\left(1+\tau \right)^{-\frac{23}{8}} d\tau \nonumber\\
&\quad+CC_0^2 \varepsilon^2\int_{0}^{t} e^{-c\left(t-\tau\right)}\left(t-\tau\right)^{-\frac12}\left(1+\tau\right)^{-2}d\tau \nonumber\\
&\leq  C\varepsilon(1+t)^{-2}+ C C_0^2 \varepsilon^2\left(1+t\right)^{-2}
 \leq \frac{C_0}{2} \varepsilon\left(1+t\right)^{-2 },
\end{align*}
provided $C_0$ is chosen to be large enough and $\varepsilon$ is chosen to be suitably small.

\vskip .1in

 \textbf{Step II-5. The decay rate of $\|\p_{1}^3u\|_{L^2}$}
 \vskip .1in
Taking $i=j=k=1$ in \eqref{3uL2} and applying Corollary \ref{PI} with $\alpha=0$,  by \eqref{p111b1} and \eqref{plghb2}  we deduce that for any $0<\delta<1$,
\begin{align*}
\|\p_1^3u\|_{L^2}
&\leq  C\varepsilon(1+t)^{-\frac{1}{2}} +C\int_0^{t}\left(1+t-\tau \right)^{-\frac12}\|\p_1^3( b\cdot \na b )\|_{L_{x_1}^2L_{x_2,x_3}^1}d\tau \nonumber\\
&\quad +C\int_{0}^{t}e^{-c\left(t-\tau\right)}\|\p_1^3(b\cdot \na b)\|_{L^2}d\tau \nonumber\\
&\leq C\varepsilon\left(1+t\right)^{-\frac12}+  CC_0\varepsilon^2\int_0^{t}\left(1+t-\tau \right)^{-\frac12}\left(1+\tau\right)^{-\frac78} d\tau \nonumber\\
&\quad+CC_0\varepsilon\int_0^{t}\left(1+t-\tau \right)^{-\frac12}\left(1+\tau\right)^{-\frac12}\|\p_{\nu}b\|_{H^3}d\tau\nonumber\\
&\quad +CC_0\varepsilon^2\int_{0}^{t}e^{-c\left(t-\tau\right)}\left(1+\tau\right)^{-1}d\tau \nonumber\\
&\leq C\varepsilon\left(1+t\right)^{-\frac12}+C C_0 \varepsilon^2\left(1+t\right)^{-\frac38}+C C_0 \varepsilon^2\left(1+t\right)^{-1}\nonumber\\
&\quad + CC_0\varepsilon\left(\int_0^{t}\left(1+t-\tau \right)^{-1}\left(1+\tau\right)^{-1}d\tau\right)^{\frac12}\left(\int_0^{t} \|\p_{\nu}b\|_{H^3}^2d\tau\right)^{\frac12}\nonumber\\
&\leq C\varepsilon\left(1+t\right)^{-\frac12}+C C_0 \varepsilon^2\left(1+t\right)^{-\frac38}+CC_0\varepsilon^2\left(1+t\right)^{-\frac{1-\delta}{2}},
\end{align*}
where we have also used (\ref{u102}).
So, if we choose $C_0$ large enough and $\varepsilon$ suitably small, then we have by taking $\delta\in(0,1/4]$  that
$$
\|\p_1^3u(t)\|_{L^2}\leq \frac{C_0}{2}\varepsilon\left(1+t\right)^{-\frac38}, \quad\forall\ t\geq1.
$$

\vskip .1in

\textbf{Step III. The decay rates of $\|\p_i\p_j\p_k\p_l u\|_{L^2}$ with  $i,j,k,l\in\{1,2,3\}$ }
\vskip .1in

For any $i,j,k,l\in\{1,2,3\}$,
we infer from  \eqref{u} that
\begin{align}
\|\p_{ijkl}u\|_{L^2}&=\|\widehat{\p_{ijkl}u}\|_{L^2}
=\|\xi_{i}\xi_{j}\xi_{k}\xi_{l}\widehat{u}\|_{L^2}\nonumber\\
&\leq \left\| \xi_{i}\xi_{j}\xi_{k}\xi_{l}\widehat{K_1}(t)\widehat{u_0}\right\|_{L^2}+\left\|\xi_{i}\xi_{j}\xi_{k}\xi_{l}\widehat{K_2}(t)\widehat{b_0}\right\|_{L^2}\nonumber\\
&\quad+\int_0^t\left\|\xi_{i}\xi_{j}\xi_{k}\xi_{l}\widehat{K_1}(t-\tau)\widehat{N_1}(\tau)\right\|_{L^2}d\tau\nonumber\\
&\quad+\int_0^t\left\|\xi_{i}\xi_{j}\xi_{k}\xi_{l}\widehat{K_2}(t-\tau)\widehat{N_2}(\tau)\right\|_{L^2}d\tau.\label{4uL2}
\end{align}

\vskip .1in

\textbf{Step III-1.  The decay rate of $\|\p_1^3\p_3u\|_{L^2}$}
\vskip .1in

Taking $i=j=k=1$ and  $l=3$ in \eqref{4uL2}, using (\ref{ineq}) and Corollary \ref{PI} with $\alpha=1$, by \eqref{p111b1} and \eqref{plghb2} we see that for any $0<\delta<1$,
\begin{align*}
\|\p_1^3\p_{3}u\|_{L^2}
&\leq  C\varepsilon(1+t)^{-1} + C\int_0^{t}\left(1+t-\tau \right)^{-1}\|\p_1^3\left(b\cdot \na b\right)  \|_{L_{x_1}^2L_{x_2,x_3}^1}d\tau\nonumber\\
&\quad+C\int_{0}^{t}e^{-c\left(t-\tau\right)}\left(t-\tau\right)^{-\frac12}\|\p_1^3(b\cdot \na b)\|_{L^2}d\tau\nonumber\\
&\leq C \varepsilon\left(1+t\right)^{-1}
+CC_0 \varepsilon \int_0^{t}\left(1+t-\tau\right)^{-1}   \left(1+\tau\right)^{-\frac12}\|\p_{\nu}b\|_{H^3} d\tau\nonumber\\
&\quad
+CC_0 \varepsilon^2\int_0^{t}\left(1+t-\tau\right)^{-1}   \left(1+\tau\right)^{-\frac78}d\tau \nonumber\\
&\quad+CC_0 \varepsilon^2\int_{0}^{t} e^{-c\left(t-\tau\right)}\left(t-\tau\right)^{-\frac12}\left(1+\tau\right)^{-1}d\tau \nonumber\\
&\leq C\varepsilon(1+t)^{-1}+C C_0 \varepsilon^2\left(1+t\right)^{-\frac78+\delta}+C C_0 \varepsilon^2\left(1+t\right)^{-1}\nonumber\\
&\quad
+ CC_0 \varepsilon \left(\int_0^{t}\left(1+t-\tau\right)^{-2}   \left(1+\tau\right)^{-1}d\tau\right)^{\frac12}\left(\int_0^{t}\|\p_{\nu}b\|_{H^3}^2 d\tau\right)^{\frac12} \nonumber\\
&\leq  C\varepsilon(1+t)^{-\frac12}+C C_0 \varepsilon^2\left(1+t\right)^{-\frac12} \leq \frac{C_0}{2} \varepsilon\left(1+t\right)^{-\frac12},
\end{align*}
provided $C_0$ is chosen to be large enough and $\varepsilon$ is chosen to be suitably small.

\vskip .1in

\textbf{Step III-2. The decay rates of $\|\p_l\p_h\p_2\p_3u\|_{L^2}$ with $l,h\in\{1,2\}$}
 \vskip .1in

Based upon \eqref{p1b1}, \eqref{plghb2}, \eqref{ineq} and Corollary \ref{PI}, it follows from  \eqref{4uL2} that
\begin{align}
\|\p_1\p_{2}^2\p_{3}u\|_{L^2}
&\leq  C\varepsilon (1+t)^{-2} + C\int_0^{t}\left(1+t-\tau \right)^{-2}\|\p_1\left(b\cdot \na b\right) \|_{L_{x_1}^2L_{x_2,x_3}^1}d\tau\nonumber\\
&\quad+C\int_{0}^{t}e^{-c\left(t-\tau\right)}\left(t-\tau\right)^{-\frac12}\|\p_1\p_2^2(b\cdot \na b)\|_{L^2}d\tau\nonumber\\
&\leq C \varepsilon\left(1+t\right)^{-2} +CC_0^2\varepsilon^2\int_0^{\frac t2}\left(1+t-\tau \right)^{-2} \left(1+\tau\right)^{-\frac{11}{8}}d\tau\nonumber\\
&\quad+CC_0 \varepsilon^2\int_{0}^{t} e^{-c\left(t-\tau\right)}\left(t-\tau\right)^{-\frac12}\left(1+\tau\right)^{-1}d\tau \nonumber\\
&\leq C\varepsilon(1+t)^{-2}+C\left(C_0+C_0^2\right)\varepsilon^2\left(1+t\right)^{-1}.\label{1223u2}
\end{align}

In a similar manner, we have by \eqref{p1b1}, \eqref{p11b1} and \eqref{plghb2}  that
\begin{align*}
\|\p_1^2\p_{2}\p_3u\|_{L^2}
&\leq  C\varepsilon(1+t)^{-\frac32}+ C\int_0^{t}\left(1+t-\tau \right)^{-\frac32}\|\p_1^2\left(b\cdot \na b\right) \|_{L_{x_1}^2L_{x_2,x_3}^1}d\tau\nonumber\\
&\quad+C\int_{0}^{t}e^{-c\left(t-\tau\right)}\left(t-\tau\right)^{-\frac12}\|\p_1\p_2^2(b\cdot \na b)\|_{L^2}d\tau\nonumber\\
&\leq
C\varepsilon(1+t)^{-\frac32} +CC_0^2\varepsilon^2\int_0^{t}\left(1+t-\tau \right)^{-\frac{3}{2}} \left(1+\tau\right)^{-\frac{5}{4}}d\tau\nonumber\\
&\quad+CC_0^2 \varepsilon^2\int_{0}^{t} \left(1+t-\tau\right)^{-\frac52}\left(1+\tau\right)^{-\frac{11}{8}}d\tau \nonumber\\
&\quad+CC_0 \varepsilon^2\int_{0}^{t} e^{-c\left(t-\tau\right)}\left(t-\tau\right)^{-\frac12}\left(1+\tau\right)^{-1}d\tau \nonumber\\
&\leq C\varepsilon(1+t)^{-\frac32}
  +C\left(C_0+C_0^2\right)\varepsilon^2\left(1+t\right)^{-1}
\end{align*}
and
\begin{align*}
\|\p_2^3\p_3u\|_{L^2}
&\leq C\varepsilon(1+t)^{-\frac52} + C\int_0^{t}\left(1+t-\tau \right)^{-\frac52}\|b\cdot \na b \|_{L_{x_1}^2L_{x_2,x_3}^1}d\tau\nonumber\\
&\quad+C\int_{0}^{t}e^{-c\left(t-\tau\right)}\left(t-\tau\right)^{-\frac12}\|\p_2^3(b\cdot \na b)\|_{L^2}d\tau\nonumber\\
&\leq
C\varepsilon(1+t)^{-\frac52} +CC_0^2\varepsilon^2\int_0^{t}\left(1+t-\tau \right)^{-\frac{3}{2}} \left(1+\tau\right)^{-\frac{5}{4}}d\tau\nonumber\\
&\quad+CC_0^2 \varepsilon^2\int_{0}^{t} \left(1+t-\tau\right)^{-\frac52}\left(1+\tau\right)^{-\frac{11}{8}}d\tau \nonumber\\
&\quad+CC_0 \varepsilon^2\int_{0}^{t} e^{-c\left(t-\tau\right)}\left(t-\tau\right)^{-\frac12}\left(1+\tau\right)^{-1}d\tau \nonumber\\
&\leq C\varepsilon(1+t)^{-\frac52}
  +C\left(C_0+C_0^2\right)\varepsilon^2\left(1+t\right)^{-1},
\end{align*}
which, together with \eqref{1223u2}, give rise to
$$
\|\p_l\p_h\p_{2}\p_3u(t)\|_{L^2}  \leq
\frac{C_0}{2}\varepsilon(1+t)^{-1},\quad \forall\ l,h\in\{1,2\},
$$
provided $C_0$ is chosen to be large enough and $\varepsilon$ is chosen to be suitably small.

\vskip .1in

\textbf{Step III-3.  The decay rate of $\|\p_1^2\p_3^2u\|_{L^2}$}
\vskip .1in
Using \eqref{p11b1}, \eqref{p113b2} and (\ref{ineq}), from \eqref{4uL2} we obtain after choosing $C_0$ large enough and $\varepsilon$ suitably small that
\begin{align*}
\|\p_1^2\p_{3}^2u\|_{L^2}
&\leq  C\varepsilon(1+t)^{-\frac32} + C\int_0^{\frac t2}\left(1+t-\tau \right)^{-\frac32}\|\p_1^2\left(b\cdot \na b\right) \|_{L_{x_1}^2L_{x_2,x_3}^1}d\tau\nonumber\\
&\quad +C\int_{\frac t2}^{t}\left(t-\tau \right)^{-\frac12} \|\p_1^2\p_3(b\cdot \na b)\|_{L^2}d\tau \nonumber\\
&\quad+C\int_{0}^{t}e^{-c\left(t-\tau\right)}\left(t-\tau\right)^{-\frac12}\|\p_1^2\p_3(b\cdot \na b)\|_{L^2}d\tau\nonumber\\
&\leq C \varepsilon\left(1+t\right)^{-\frac32} +CC_0^2\varepsilon^2\int_0^{\frac t2}\left(1+t-\tau \right)^{-\frac32} \left(1+\tau\right)^{-\frac54}d\tau\nonumber\\
&\quad+CC_0^2\varepsilon^2\int_{\frac t2}^{t} \left(t-\tau\right)^{-\frac12}\left(1+\tau \right)^{-\frac{15}{8}} d\tau \nonumber\\
&\quad+CC_0^2 \varepsilon^2\int_{0}^{t} e^{-c\left(t-\tau\right)}\left(t-\tau\right)^{-\frac12}\left(1+\tau\right)^{-\frac{15}{8}}d\tau \nonumber\\
&\leq C\varepsilon(1+t)^{-\frac32}+C C_0^2 \varepsilon^2\left(1+t\right)^{-\frac{11}{8}}\leq \frac{C_0}{2} \varepsilon\left(1+t\right)^{-\frac{11}{8}} .
\end{align*}

\vskip .1in

\textbf{Step III-4. The decay rates of $\|\p_2^2\p_3^2u\|_{L^2}, \|\p_1\p_i\p_3^2u\|_{L^2}$ with $ i\in\{2,3\}$}
 \vskip .1in
Using (\ref{KS1}), (\ref{KS21})--(\ref{KS23}) of Proposition \ref{lem2.1} and Corollary \ref{PI} with $\alpha=3$, we infer from \eqref{int2}, \eqref{p1b1}, \eqref{pl23b2}  and \eqref{p133b2} that for $i\in\{2,3\}$,
\begin{align}
\|\p_1\p_i\p_{3}^2u\|_{L^2}
&\leq  C\varepsilon(1+t)^{-2} + C\int_0^{\frac t2}\left(1+t-\tau \right)^{-2}\|\p_1\left(b\cdot \na b\right) \|_{L_{x_1}^2L_{x_2,x_3}^1}d\tau\nonumber\\
&\quad +C\int_{\frac t2}^{t}\left(t-\tau \right)^{-\frac12} \|\p_1\p_3^2(b\cdot \na b)\|_{L^2}d\tau \nonumber\\
&\quad+C\int_{0}^{t}e^{-c\left(t-\tau\right)}\left(t-\tau\right)^{-\frac12}\|\p_1\p_i\p_3(b\cdot \na b)\|_{L^2}d\tau\nonumber\\
&\leq C \varepsilon\left(1+t\right)^{-2} +CC_0^2\varepsilon^2\int_0^{\frac t2}\left(1+t-\tau \right)^{-2} \left(1+\tau\right)^{-\frac{11}{8}}d\tau\nonumber\\
&\quad+CC_0^2\varepsilon^2\int_{\frac t2}^{t} \left(t-\tau\right)^{-\frac12}\left(1+\tau \right)^{-\frac{21}{8}} d\tau \nonumber\\
&\quad+CC_0^2 \varepsilon^2\int_{0}^{t} e^{-c\left(t-\tau\right)}\left(t-\tau\right)^{-\frac12}\left(1+\tau\right)^{-2}d\tau \nonumber\\
&\leq C\varepsilon(1+t)^{-2}+C C_0^2 \varepsilon^2\left(1+t\right)^{-2}, \label{p1i33u}
\end{align}
where we have also used (\ref{ineq}). Analogously,
\begin{align*}
\|\p_{2}^2\p_3^2u\|_{L^2}
&\leq  C\varepsilon(1+t)^{-\frac52} + C\int_0^{\frac t2}\left(1+t-\tau \right)^{-\frac{5}{2}}\|b\cdot \na b\|_{L_{x_1}^2L_{x_2,x_3}^1}d\tau\nonumber\\
&\quad +C\int_{\frac t2}^{t}\left(t-\tau \right)^{-\frac12} \|\p_2\p_3^2(b\cdot \na b)\|_{L^2}d\tau \nonumber\\
&\quad+C\int_{0}^{t}e^{-c\left(t-\tau\right)}\left(t-\tau\right)^{-\frac12}\|\p_2^2\p_3(b\cdot \na b)\|_{L^2}d\tau\nonumber\\
&\leq C \varepsilon\left(1+t\right)^{-\frac{5}{2}} +CC_0^2\varepsilon^2\int_0^{\frac t2}\left(1+t-\tau \right)^{-\frac{5}{2}} \left(1+\tau\right)^{-\frac{11}{8}}d\tau\nonumber\\
&\quad+CC_0^2\varepsilon^2\int_{\frac t2}^{t} \left(1+t-\tau\right)^{-\frac12}\left(1+\tau \right)^{-3} d\tau \nonumber\\
&\quad+CC_0^2 \varepsilon^2\int_{0}^{t} e^{-c\left(t-\tau\right)}\left(t-\tau\right)^{-\frac12}\left(1+\tau\right)^{-2}d\tau \nonumber\\
&\leq C\varepsilon(1+t)^{-\frac52}+C C_0^2 \varepsilon^2\left(1+t\right)^{-2},
\end{align*}
which, combined with (\ref{p1i33u}), yields
$$
\|\p_2^2\p_3^2u(t)\|_{L^2}+\|\p_1\p_i\p_3^2u(t)\|_{L^2}\leq \frac{C_0}{2}\varepsilon(1+t)^{-2},\quad\forall\ i\in\{2,3\},
$$
provided $C_0$ is chosen to be large enough and $\varepsilon$ is chosen to be suitably small.

\vskip .1in

\textbf{Step III-5. The decay rates of $\|\p_i\p_3^3u\|_{L^2}$ with $i\in\{2,3\}$}
 \vskip .1in

In terms of Proposition \ref{lem2.1} and Corollary \ref{PI}  with $\alpha=4$, by choosing $C_0$ large enough and $\varepsilon$ suitably small we infer from \eqref{p233b2}  and \eqref{p333b2}  that for $i\in\{2,3\}$,
\begin{align*}
\|\p_i\p_{3}^3u\|_{L^2}
&\leq  C\varepsilon(1+t)^{-\frac52} + C\int_0^{\frac t2}\left(1+t-\tau \right)^{-\frac52}\|b\cdot \na b \|_{L_{x_1}^2L_{x_2,x_3}^1}d\tau\nonumber\\
&\quad +C\int_{\frac t2}^{t}\left(t-\tau \right)^{-\frac12} \|\p_3^3(b\cdot \na b)\|_{L^2}d\tau \nonumber\\
&\quad+C\int_{0}^{t}e^{-c\left(t-\tau\right)}\left(t-\tau\right)^{-\frac12}\|\p_i\p_3^2(b\cdot \na b)\|_{L^2}d\tau\nonumber\\
&\leq C \varepsilon\left(1+t\right)^{-\frac52} +CC_0^2\varepsilon^2\int_0^{\frac t2}\left(1+t-\tau \right)^{-\frac52} \left(1+\tau\right)^{-\frac{11}{8}}d\tau\nonumber\\
&\quad+CC_0^2\varepsilon^2\int_{\frac t2}^{t} \left(t-\tau\right)^{-\frac12}\left(1+\tau \right)^{-\frac{27}{8}} d\tau \nonumber\\
&\quad+CC_0^2 \varepsilon^2\int_{0}^{t} e^{-c\left(t-\tau\right)}\left(t-\tau\right)^{-\frac12}\left(1+\tau\right)^{-3}d\tau \nonumber\\
&\leq C\varepsilon(1+t)^{-\frac52}+C C_0^2 \varepsilon^2\left(1+t\right)^{-\frac52}\leq \frac{C_0}{2} \varepsilon\left(1+t\right)^{-\frac52}.
\end{align*}

\vskip .1in

\subsection{Enhanced decay rates of $(u_1, b_1)$}

In this subsection, we aim to improve the decay rates of $(u_1, b_1)$, based on (\ref{P1}), (\ref{N21}) and Corollary \ref{divfree}. Note that the decay estimates of $(u_1, b_1)$, $\p_i(u_1, b_1)$, $\p_i\p_3(u_1, b_1)$ with $i\in\{2,3\}$ and $\p_3^3(u_1, b_1)$ stated in (\ref{u101}) have been  achieved in Theorem \ref{thm1.2}.

\vskip .1in
\textbf{Part I. Improved decay rate of $\|\p_2^2u_1\|_{L^2}$}
\vskip .1in
It follows from \eqref{u} and Plancherel's theorem that
\begin{align}
\|\p_2^2u_1\|_{L^2}&=\|\widehat{\p_2^2u_1}\|_{L^2}=\|\xi_{2}^2\widehat{u_1}\|_{L^2}\nonumber\\
&\leq \left\| \xi_{2}^2\widehat{K_1}(t)\widehat{u_{10}}\right\|_{L^2}+\left\|\xi_{2}^2\widehat{K_2}(t)\widehat{b_{10}}\right\|_{L^2}\nonumber\\
&\quad+\int_0^t\left\|\xi_{2}^2\widehat{K_1}(t-\tau)\widehat{N_{11}}(\tau)\right\|_{L^2}d\tau\nonumber\\
&\quad+\int_0^t\left\|\xi_{2}^2\widehat{K_2}(t-\tau)\widehat{N_{21}}(\tau)\right\|_{L^2}d\tau\triangleq \sum_{m=1}^4 \Omega_{1,m}.\label{22u1}
\end{align}

It is easily derived from (\ref{S1}), (\ref{S2}) and Corollary \ref{divfree} with $\beta=2$ that
\begin{align}
\Omega_{1,1}+\Omega_{1,2}&\leq   C(1+t)^{-\frac{7}{4}}\|(u_0,b_0)\|_{L^1} + Ce^{-ct} \|\p_2^2(u_{10},b_{10}) \|_{L^2}\nonumber\\
&\leq C\varepsilon(1+t)^{-\frac{7}{4}}.\label{w1}
\end{align}

Analogously to the treatment in Subsection 4.3, for simplicity we only deal with the terms associated with $\widehat{\mathbb{P}_1(b\cdot\nabla b)}$ in $\widehat{N_{11}}$ and $\widehat{b_2u_1}, \widehat{ b_3u_1}$ in $\widehat{N_{21}}$, which are still denoted by $\Omega_{1,3}$ and $\Omega_{1,4}$. First, by (\ref{P1}) and Proposition \ref{lem2.1} we have
\begin{align}
\Omega_{1,3}
&\leq C\int_0^t\left\|\xi_2^2e^{-c\xi_{\nu}^2\left(t-\tau \right)}\sum_{k=1}^{3}\xi_{\nu}^2|\xi|^{-2}\xi_k\widehat{b_kb_1}\right\|_{L^2}d\tau\nonumber\\
&\quad+C\int_0^t\left\|\xi_2^2 e^{-c\xi_{\nu}^2\left(t-\tau \right)} \sum_{k=1}^{3}\sum_{l=2}^{3}\xi_1|\xi|^{-2}\xi_k\xi_l\widehat{b_kb_l}\right\|_{L^2}d\tau\nonumber\\
&\quad+C\int_0^t\left\|\xi_2^2 e^{-c\left(1+\xi_3^2\right)\left(t-\tau \right)} \sum_{k=1}^{3}\xi_{\nu}^2|\xi|^{-2}\xi_k\widehat{b_kb_1}\right\|_{L^2}d\tau\nonumber\\
&\quad+C\int_0^t\left\|\xi_2^2 e^{-c\left(1+\xi_3^2\right)\left(t-\tau \right)} \sum_{k=1}^{3}\sum_{l=2}^{3}\xi_1|\xi|^{-2}\xi_k\xi_l\widehat{b_kb_l}\right\|_{L^2}d\tau\nonumber\\
&\triangleq \Omega_{1,31}+ \Omega_{1,32}+ \Omega_{1,33}+ \Omega_{1,34}.\label{Q3b}
\end{align}

In terms of \eqref{bb11} and Corollary \ref{PI} with $\alpha=3$, we have
\begin{align}
\Omega_{1,31}
&\leq C\int_0^{\frac{t}{2}}\left(1+t-\tau \right)^{-2}\| bb_1 \|_{L_{x_1}^2L_{x_2,x_3}^1}d\tau\nonumber\\
 &\quad+C\int_{{\frac{t}{2}}}^{t}\left(t-\tau \right)^{-\frac12} \|\p_2(b\cdot \na b_1)  \|_{L^2}d\tau \nonumber\\
&\leq CC_0^2\varepsilon^2 \int_0^{\frac{t}{2}}\left(1+t-\tau \right)^{-2} \left(1+\tau\right)^{-\frac{11}{8}}d\tau\nonumber\\
&\quad+ CC_0^2 \varepsilon^2 \int_{{\frac{t}{2}}}^t\left(t-\tau \right)^{-\frac12} \left(1+\tau\right)^{-\frac{11}{4}}d\tau  \nonumber\\
&\leq C C_0^2 \varepsilon^2\left(1+t\right)^{-2}, \label{w11}
\end{align}
where we have used (\ref{u101}), (\ref{u4}), (\ref{u5}), \eqref{bbinfty}, \eqref{b1binfty} and the divergence-free condition $\nabla\cdot b=0$ to get that ($\p_\nu=(\p_2,\p_3)$ and $b_\nu=(b_2,b_3)$)
\begin{align}
&\|\p_2 ( b\cdot \na b_1 )\|_{L^2}\nonumber\\
&\quad\leq C\|\p_2 b_\nu \|_{L^2}^{\frac12}\|\p_2^2b_\nu \|_{L^2}^{\frac12}\|\p_\nu b_1 \|_{L^2}^{\frac14}\|\p_1\p_\nu b_1\|_{L^2}^{\frac14}\|\p_3\p_\nu  b_1\|_{L^2}^{\frac14}\|\p_1\p_3 \p_\nu b_1\|_{L^2}^{\frac14}\nonumber\\
&\qquad +C\|\p_2 b_1 \|_{L^2}^{\frac12}\|\p_2^2b_1 \|_{L^2}^{\frac12}\|\p_1 b_1 \|_{L^2}^{\frac14}\|\p_1^2 b_1\|_{L^2}^{\frac14}\|\p_1\p_3  b_1\|_{L^2}^{\frac14}\|\p_1^2\p_3 b_1\|_{L^2}^{\frac14}\nonumber\\
&\qquad+C\|b_1\|_{ L^\infty}\|\p_2\p_1 b_1 \|_{L^2} +C\|b_\nu\|_{ L^\infty}\|\p_2\p_\nu b_1 \|_{L^2} \nonumber\\
&\quad\leq CC_0^2 \varepsilon^2\left(1+t\right)^{-\frac{11}{4}}. \nonumber 
\end{align}

Similarly, by choosing $\delta\in(0,1/4]$  we deduce from \eqref{1bbv2} that
\begin{align}
\Omega_{1,32}
&\leq C\int_0^{\frac t2}\left( 1+t-\tau \right)^{-2}\|(bb_2 , bb_3)  \|_{L_{x_1}^2 L_{x_2,x_3}^1} d\tau\nonumber\\
&\quad+C\int_{\frac t2}^{t}\left\|\xi_{\nu} e^{-c\xi_{\nu}^2\left(t-\tau \right)} \right\|_{L^\infty} \|(\p_2^2 (b b_2 ),\p_2\p_3 (b b_3 ) ) \|_{L^2} d\tau \nonumber\\
&\leq CC_0^2\varepsilon^2\int_0^{\frac t2}\left(1+t-\tau \right)^{-2}\left(1+\tau \right)^{-1}d\tau\nonumber\\
&\quad+CC_0^2\varepsilon^2\int_{\frac t2}^{t} \left(t-\tau \right)^{-\frac12}\left(1+\tau\right)^{-\frac{5}{2}}d\tau \nonumber\\
&\leq CC_0^2\varepsilon^2 \left((1+t  )^{- 2+\delta}  + (1+t)^{-2}\right)\leq CC_0^2\varepsilon^2 \left(1+t\right)^{-\frac74}, \label{w12}
\end{align}
since it follows from  (\ref{u4}), \eqref{bbinfty} and \eqref{b1binfty} that
\begin{align}
&\|\p_2^2\left(b b_2\right)\|_{L^2}+\|\p_2\p_3\left(b b_3\right)\|_{L^2} \nonumber\\
 &\quad\leq C\|b\|_{ L^\infty}\|\p_2\p_\nu b\|_{L^2} +C\|\p_2 b\|_{L^2}^{\frac12}\|\p_1\p_2b \|_{L^2}^{\frac12} \nonumber\\
&\qquad\qquad\times\|\p_\nu b \|_{L^2}^{\frac14}\|\p_2\p_\nu b \|_{L^2}^{\frac14}\|\p_3\p_\nu b\|_{L^2}^{\frac14}\|\p_2\p_3\p_\nu b \|_{L^2}^{\frac14}\nonumber\\
& \quad\leq CC_0^2 \varepsilon^2\left(1+t\right)^{-\frac{5}{2}}.\nonumber
\end{align}

Thanks to \eqref{pl2b2}, it is easily seen from Lemma \ref{ED} that
\begin{align}
\Omega_{1,33}+ \Omega_{1,34}
&\leq C\int_0^te^{-c\left(t-\tau \right)}\left\|\p_2^2\left(b\cdot\na b\right) \right\|_{L^2}d\tau\nonumber\\
&\leq  CC_0^2\varepsilon^2\int_0^te^{-c\left(t-\tau \right)}\left(1+\tau\right)^{-2} d\tau \leq CC_0^2\varepsilon^2\left(1+t\right)^{-2}. \label{w134}
\end{align}

Plugging \eqref{w11}, \eqref{w12}  and \eqref{w134} into \eqref{Q3b}, we obtain
\begin{align}
\Omega_{1,3}\leq C C_0^2 \varepsilon^2\left(1+t\right)^{-\frac74}. \label{w3g}
\end{align}

Recalling the definition of $\widehat{N_{21}}$ in (\ref{N21}), by (\ref{KS1}) and (\ref{KS21})--(\ref{KS23}) of  Proposition \ref{lem2.1} we have
\begin{align*}
\Omega_{1,4}&=\int_0^t\left\|\xi^3_{2} e^{-c\xi_{\nu}^2\left(t-\tau \right)}\widehat{b_2 u_1} \right\|_{L^2}d\tau +\int_0^t\left\|\xi_{2}^2\xi_{3} e^{-c\xi_{\nu}^2\left(t-\tau \right)}\widehat{b_3 u_1} \right\|_{L^2}d\tau\nonumber \\
&\quad+\int_0^t\left\|\xi^3_{2} e^{-c\left(1+\xi_3^2\right)\left(t-\tau \right)}\widehat{b_2 u_1} \right\|_{L^2}d\tau +\int_0^t\left\|\xi_{2}^2\xi_{3} e^{-c\left(1+\xi_3^2\right)\left(t-\tau \right)}\widehat{b_3 u_1} \right\|_{L^2}d\tau,
\end{align*}
and hence,  similarly the derivation of (\ref{M324b}),  we deduce from  (\ref{ineq}), \eqref{bb11} and Corollary \ref{PI} with $\alpha=3$ that
\begin{align}
\Omega_{1,4}& \leq C\int_0^{\frac{t}{2}}\left(1+t-\tau \right)^{-2} \| ( b_2u_1 ,b_3u_1 )\|_{L_{x_1}^2 L_{x_2,x_3}^1} d\tau \nonumber\\
&\quad + C\int_{\frac{t}{2}}^{t}\left(t-\tau \right)^{-\frac12} \|( \p_2^2 (b_2u_1 ),\p_2\p_3 (b_3u_1 ) )\|_{L^2} d\tau \nonumber\\
&\quad +C\int_{0}^{t} e^{-c\left(t-\tau\right)}\left(\|\p_2^3( b_2u_1 ) \|_{L^2}+ (t-\tau )^{-\frac12}\|\p_2^2 (b_3u_1 )\|_{L^2}\right)d\tau \nonumber\\
& \leq CC_0^2\varepsilon^2\int_0^{\frac{t}{2}}\left(1+t-\tau \right)^{-2} \left(1+\tau\right)^{-\frac{11}{8}}d\tau\nonumber\\
&\quad+CC_0^2\varepsilon^2\int_{\frac{t}{2}}^t\left(t-\tau \right)^{-\frac{1}{2}} \left(1+\tau\right)^{-\frac{11}{4}}d\tau\nonumber\\
&\quad+CC_0^2\varepsilon^2\int_0^te^{-c\left(t-\tau\right)}\left[(1+\tau )^{-2}+(t-\tau)^{-\frac{1}{2}} (1+\tau )^{-\frac{11}{4}}\right]d\tau \nonumber\\
&\leq C C_0^2 \varepsilon^2\left(1+t\right)^{-2}. \label{w4b}
\end{align}
Here, we have also used  the following estimates due to (\ref{u4}), \eqref{bbinfty} and \eqref{b1binfty},
\begin{align}
\|\p_2\p_\nu\left(bu_1\right) \|_{L^2}
& \leq C\|b\|_{L^\infty} \|\p_2\p_\nu u_1\|_{L^2} +C\|u_1\|_{L^\infty} \|\p_2\p_\nu b\|_{L^2}  \nonumber\\
&\quad+C\|\p_\nu b\|_{L_{x_1,x_3}^\infty L_{x_3}^2}\|\p_\nu u_1\|_{L_{x_1,x_3}^2 L_{x_2}^{\infty}} \nonumber\\
&\leq CC_0^2\varepsilon^2\left(1+t\right)^{-\frac{11}{4}},\nonumber 
 \end{align}
and
\begin{align}
\|\p_2^3\left(bu_1\right) \|_{L^2}
& \leq C\|b\|_{L^\infty}\|\p_2^3 u_1 \|_{L^2}+C\|\p_2b\|_{L_{x_1,x_2}^\infty L_{x_3}^2}\|\p_2^2u_1\|_{L_{x_1,x_2}^{2}L_{x_3}^\infty} \nonumber\\
&\quad+C\|u_1\|_{L^\infty} \|\p_2^3 b \|_{L^2} +C\|\p_2u_1 \|_{L_{x_1,x_2}^{\infty}L_{x_3}^2}\|\p_2^2b\|_{L_{x_1,x_2}^{2}L_{x_3}^\infty}\nonumber\\
&\leq CC_0^2\varepsilon^2\left(1+t\right)^{-2}.\nonumber 
\end{align}

Thus, combining  \eqref{w1}, \eqref{w3g} and \eqref{w4b} with  \eqref{22u1}, we have
\begin{align}
\|\p_2^2u_1(t)\|_{L^2}\leq C\varepsilon(1+t)^{-\frac74}+  C C_0^2 \varepsilon^2\left(1+t\right)^{-\frac74}\leq \frac{C_0}{2} \varepsilon\left(1+t\right)^{-\frac74},
\end{align}
provided $C_0$ is chosen to be large enough and $\varepsilon$ is chosen to be suitably small.

\vskip .1in
\textbf{Part II. Improved decay rate of $\|\p_2\p_3^2u_1\|_{L^2}$}
\vskip .1in
Clearly, from  \eqref{u} and Plancherel's theorem it follows that
\begin{align}
\|\p_2\p_3^2u_1\|_{L^2}&=\|\widehat{\p_2\p_3^2u_1}\|_{L^2}=\|\xi_{2}\xi_{3}^2\widehat{u_1}\|_{L^2}\nonumber\\
&\leq \left\| \xi_{2}\xi_{3}^2\widehat{K_1}(t)\widehat{u_{10}}\right\|_{L^2}+\left\|\xi_{2}\xi_{3}^2\widehat{K_2}(t)\widehat{b_{10}}\right\|_{L^2}\nonumber\\
&\quad+\int_0^t\left\|\xi_{2}\xi_{3}^2\widehat{K_1}(t-\tau)\widehat{N_{11}}(\tau)\right\|_{L^2}d\tau\nonumber\\
&\quad+\int_0^t\left\|\xi_{2}\xi_{3}^2\widehat{K_2}(t-\tau)\widehat{N_{21}}(\tau)\right\|_{L^2}d\tau\triangleq \sum_{m=1}^4\Omega_{2,m}.\label{233u1}
\end{align}

Using (\ref{S1}), \eqref{S2} and Corollary \ref{divfree} with $\beta=3$, we see that
\begin{align}
\Omega_{2,1}+\Omega_{2,2} \leq C(1+t)^{-\frac{9}{4}}\left(\| u_{0}\|_{L^1}+\|\p_2\p_3^2u_0\|_{L^2 }\right)\leq  C\varepsilon(1+t)^{-\frac{9}{4}}.\label{x1}
\end{align}

Analogously to that in \eqref{Q3b}, by (\ref{P1}), (\ref{KS1}) and \eqref{KS21}--\eqref{KS23} one has
\begin{align}
\Omega_{2,3}
&\leq C\int_0^t\left\|\xi_2\xi_3^2 e^{-c\xi_{\nu}^2\left(t-\tau \right)}\sum_{k=1}^{3}\xi_{\nu}^2|\xi|^{-2}\xi_k\widehat{b_kb_1}\right\|_{L^2}d\tau\nonumber\\
&\quad+C\int_0^t\left\|\xi_2\xi_3^2 e^{-c\xi_{\nu}^2\left(t-\tau \right)} \sum_{k=1}^{3}\sum_{l=2}^{3}\xi_1|\xi|^{-2}\xi_k\xi_l\widehat{b_kb_l}\right\|_{L^2}d\tau\nonumber\\
&\quad+C\int_0^t\left\|\xi_2\xi_3^2 e^{-c\left(1+\xi_3^2\right)\left(t-\tau \right)} \sum_{k=1}^{3}\xi_{\nu}^2|\xi|^{-2}\xi_k\widehat{b_kb_1}\right\|_{L^2}d\tau\nonumber\\
&\quad+C\int_0^t\left\|\xi_2\xi_3^2 e^{-c\left(1+\xi_3^2\right)\left(t-\tau \right)} \sum_{k=1}^{3}\sum_{l=2}^{3}\xi_1|\xi|^{-2}\xi_k\xi_l\widehat{b_kb_l}\right\|_{L^2}d\tau\nonumber\\
&\triangleq \Omega_{2,31}+\Omega_{2,32}+\Omega_{2,33}+\Omega_{2,34}. \label{Q233b}
\end{align}

To deal with $\Omega_{2,31}$, we first utilize (\ref{u101}), (\ref{u4}), (\ref{u5}), \eqref{bbinfty}, \eqref{b1binfty} and the divergence-free condition $\nabla\cdot b=0$ to deduce
\begin{align}
&\|\p_3^2\left( b\cdot \na b_1\right)\|_{L^2}\nonumber\\
&\quad\leq C\|\p_3^2 b_1 \|_{L^2}^{\frac12}\|\p_2\p_3^2 b_1 \|_{L^2}^{\frac12}\|\p_1 b_1 \|_{L^2}^{\frac14}\|\p_1^2 b_1\|_{L^2}^{\frac14}\|\p_3 \p_1 b_1\|_{L^2}^{\frac14}\|\p_1^2\p_3 b_1\|_{L^2}^{\frac14}\nonumber\\
&\qquad+C\|\p_3^2 b_\nu \|_{L^2}^{\frac12}\|\p_2\p_3^2 b_\nu\|_{L^2}^{\frac12}\|\p_\nu b_1 \|_{L^2}^{\frac14}\|\p_1\p_\nu b_1\|_{L^2}^{\frac14}\|\p_3 \p_\nu b_1\|_{L^2}^{\frac14}\|\p_1\p_3\p_\nu b_1\|_{L^2}^{\frac14}\nonumber\\
&\qquad+C\|\p_3b_1 \|_{L^2}^{\frac14}\|\p_1\p_3 b_1 \|_{L^2}^{\frac14}\|\p_3^2 b_1 \|_{L^2}^{\frac14}\|\p_1\p_3^2 b_1\|_{L^2}^{\frac14}\|\p_1\p_3b_1\|_{L^2}^{\frac12}\|\p_1\p_2\p_3 b_1\|_{L^2}^{\frac12}\nonumber\\
&\qquad+C\|\p_3 b_\nu \|_{L^2}^{\frac14}\|\p_1\p_3b_\nu\|_{L^2}^{\frac14}\|\p_2\p_3b_\nu \|_{L^2}^{\frac14}\|\p_1\p_2\p_3b_\nu\|_{L^2}^{\frac14}\|\p_3\p_\nu b_1\|_{L^2}^{\frac12}\|\p_3^2\p_\nu b_1\|_{L^2}^{\frac12}\nonumber\\
&\qquad+C\|b_1\|_{ L^\infty}\|\p_1\p_3^2 b_1 \|_{L^2} +C\|b_\nu\|_{ L^\infty}\|\p_\nu\p_3^2 b_1 \|_{L^2} \nonumber\\
&\quad\leq CC_0^2 \varepsilon^2\left(1+t\right)^{-\frac{13}{4}},\nonumber 
\end{align}
which, combined with \eqref{bb11} and Corollary \ref{PI} with $\alpha=4$, yields
\begin{align}
\Omega_{2,31}
&\leq C\int_0^{\frac{t}{2}}\left(1+t-\tau \right)^{-\frac52}\|\| bb_1 \|_{L_{x_2,x_3}^1}\|_{L_{x_1}^2}d\tau \nonumber\\
&\quad+C\int_{{\frac{t}{2}}}^{t}\left(t-\tau \right)^{-\frac12} \|\p_3^2(b\cdot \na b_1)  \|_{L^2}d\tau \nonumber\\
&\leq CC_0^2\varepsilon^2 \int_0^{\frac{t}{2}}\left(1+t-\tau \right)^{-\frac52} \left(1+\tau\right)^{-\frac{11}{8}}d\tau\nonumber\\
&\quad+ CC_0^2 \varepsilon^2 \int_{{\frac{t}{2}}}^t\left(t-\tau \right)^{-\frac12} \left(1+\tau\right)^{-\frac{13}{4}}d\tau  \nonumber\\
&\leq C C_0^2 \varepsilon^2\left(1+t\right)^{-\frac{5}{2}}. \label{x31}
\end{align}

Analogously, since
\begin{align*}
\|\p_\nu\p_3^2(bb_\nu)\|_{L^2}
&\leq C\|b\|_{ L^\infty}\|\p_\nu\p_3^2b\|_{L^2}+C\|\p_\nu\p_3 b\|_{L^2}^{\frac12}\|\p_2\p_\nu\p_3b \|_{L^2}^{\frac12}
\nonumber\\
&\quad\qquad\times\|\p_\nu b \|_{L^2}^{\frac14}\|\p_1\p_\nu b \|_{L^2}^{\frac14}\|\p_3\p_\nu b\|_{L^2}^{\frac14}\|\p_1\p_3\p_\nu b \|_{L^2}^{\frac14}\nonumber\\
& \leq CC_0^2 \varepsilon^2\left(1+t\right)^{-3},
\end{align*}
due to (\ref{u4}) and \eqref{bbinfty}, by choosing $0<\delta\leq 1/4$ we infer from \eqref{1bbv2} that
\begin{align}
\Omega_{2,32}
&\leq C\int_0^{\frac t2}\left(1+t-\tau \right)^{-\frac52} \|( bb_2, bb_3 )\|_{L_{x_1}^2 L_{x_2,x_3}^1} d\tau\nonumber\\
&\quad+C\int_{\frac t2}^{t}\left\|\xi_2 e^{-c\xi_{\nu}^2\left(t-\tau \right)} \right\|_{L^\infty}  \|(\p_2\p_3^2(b b_2 ),  \p_3^3 (b b_3 )) \|_{L^2}d\tau \nonumber\\
&\leq CC_0^2\varepsilon^2\int_0^{\frac t2}\left(1+t-\tau \right)^{-\frac52}\left(1+\tau \right)^{-1}d\tau\nonumber\\
&\quad+CC_0^2\varepsilon^2\int_{\frac t2}^{t} \left(t-\tau \right)^{-\frac12}\left(1+\tau\right)^{-3}d\tau \nonumber\\
&\leq CC_0^2\varepsilon^2 \left( (1+t )^{-\frac52+\delta}  +  (1+t )^{-\frac52} \right)\leq CC_0^2\varepsilon^2 (1+t )^{-\frac94}. \label{x32}
\end{align}

Similarly to the treatment of $\Omega_{1,4}$, by \eqref{pi3b2} and (\ref{ineq}) we  obtain
\begin{align}
\Omega_{2,33}+ \Omega_{2,34}
&\leq C\int_0^te^{-c\left(t-\tau \right)}\left(t-\tau \right)^{-\frac12}\left\|\p_2\p_3\left(b\cdot\na b\right) \right\|_{L^2}d\tau\nonumber\\
&\leq  CC_0^2\varepsilon^2\int_0^te^{-c\left(t-\tau \right)}\left(t-\tau \right)^{-\frac12}\left(1+\tau\right)^{-\frac{23}{8}} d\tau\nonumber\\
&\leq CC_0^2\varepsilon^2\left(1+t\right)^{-\frac{23}{8}}. \label{x334}
\end{align}

Thus, inserting \eqref{x31}, \eqref{x32}  and \eqref{x334} into \eqref{Q233b} gives
\begin{align}
\Omega_{2,3}\leq C C_0^2 \varepsilon^2\left(1+t\right)^{-\frac94}. \label{x3g}
\end{align}

In view of (\ref{N21}), we observe that
\begin{align*}
\Omega_{2,4}&=\int_0^t\left\|\xi_2^2\xi^2_{3} e^{-c\xi_{\nu}^2\left(t-\tau \right)}\widehat{b_2 u_1} \right\|_{L^2}d\tau +\int_0^t\left\|\xi_{2}\xi^3_{3} e^{-c\xi_{\nu}^2\left(t-\tau \right)}\widehat{b_3 u_1}  \right\|_{L^2}d\tau\nonumber \\
&\quad+\int_0^t\left\|\xi_2^2\xi^2_{3} e^{-c\left(1+\xi_3^2\right)\left(t-\tau \right)}\widehat{b_2 u_1} \right\|_{L^2}d\tau +\int_0^t\left\|\xi_{2}\xi^3_{3} e^{-c\left(1+\xi_3^2\right)\left(t-\tau \right)}\widehat{b_3 u_1}  \right\|_{L^2}d\tau,
\end{align*}
so that, similarly to the treatment of $\Omega_{1,4}$ in (\ref{w4b}),   we find
\begin{align}
\Omega_{2,4}& \leq C\int_0^{\frac{t}{2}}\left(1+t-\tau \right)^{-\frac52} \| ( b_2u_1, b_3u_1 )\|_{L_{x_1}^2 L_{x_2,x_3}^1} d\tau \nonumber\\
&\quad + C\int_{\frac{t}{2}}^{t}\left(t-\tau \right)^{-\frac12} \| (\p_2\p_3^2 (b_2u_1 ), \p_3^3 (b_3u_1) )\|_{L^2} d\tau \nonumber\\
&\quad +C\int_{0}^{t} e^{-c\left(t-\tau\right)}\left(t-\tau\right)^{-\frac12}\|(\p_2^2\p_3 ( b_2u_1 ),  \p_2\p_3^2 (b_3u_1) )\|_{L^2} d\tau \nonumber\\
& \leq CC_0^2\varepsilon^2\int_0^{\frac{t}{2}}\left(1+t-\tau \right)^{-\frac52} \left(1+\tau\right)^{-\frac{11}{8}}d\tau\nonumber\\
&\quad+CC_0^2\varepsilon^2\int_{\frac{t}{2}}^t\left(t-\tau \right)^{-\frac{1}{2}} \left(1+\tau\right)^{-\frac{13}{4}}d\tau\nonumber\\
&\quad+CC_0^2\varepsilon^2\int_0^te^{-c\left(t-\tau\right)}\left(t-\tau\right)^{-\frac12}\left((1+\tau )^{-3}+(1+\tau )^{-\frac{13}{4}}\right)d\tau \nonumber\\
&\leq C C_0^2 \varepsilon^2\left(1+t\right)^{-\frac{5}{2}}, \label{x4b}
\end{align}
where we have used  \eqref{bb11}, (\ref{ineq}) and the following estimates due to
(\ref{u101}), (\ref{u4}), (\ref{u5}), \eqref{bbinfty}, \eqref{b1binfty} and the divergence-free condition $\nabla\cdot u=0$,
\begin{align}
\| \p_\nu\p_3^2(bu_1)\|_{L^2}
& \leq  C\|u_1\|_{L^\infty} \|\p_\nu\p_3^2b\|_{L^2} +C\|b\|_{L^\infty} \|\p_\nu\p_3^2u_1\|_{L^2} \nonumber\\
 &\quad+ C\|\p_\nu b\|_{L_{x_1,x_3}^{\infty}L_{x_2}^2}\| \p_3^2u_1\|_{L_{x_1,x_3}^2 L_{x_2}^\infty}\nonumber\\
&\quad
+C\|\p_3b \|_{L_{x_1,x_2}^{\infty}L_{x_3}^2}\|\p_\nu\p_3u_1\|_{L_{x_1,x_2}^2L_{x_3}^{\infty}}\nonumber\\
&\quad
 +C\|\p_\nu\p_3 b\|_{L_{x_1,x_3}^2 L_{x_2}^{\infty}}\|\p_\nu u_1\|_{L_{x_1,x_3}^\infty L_{x_2}^2}\nonumber\\
&\leq CC_0^2\varepsilon^2\left(1+t\right)^{-\frac{13}{4}},\nonumber 
\end{align}
and
\begin{align}
\|\p_2^2\p_3 (bu_1 )\|_{L^2}
& \leq C\|b\|_{L^\infty}\|\p_2^2\p_3 u_1 \|_{L^2}+C\|u_1\|_{L^\infty} \|\p_2^2\p_3 b_2 \|_{L^2} \nonumber\\
&\quad
+C\|\p_\nu b\|_{L_{x_1,x_2}^{\infty}L_{x_3}^2}\|\p_2\p_\nu u_1\|_{L_{x_1,x_2}^2 L_{x_3}^\infty} \nonumber\\
&\quad+C\|\p_3\p_2b\|_{L_{x_1,x_3}^2L_{x_2}^{\infty}}\|\p_2u_1\|_{L_{x_1,x_3}^\infty  L_{x_2}^2}\nonumber\\
&\quad+C\|\p_2^2b \|_{L_{x_1,x_2}^{2}L_{x_3}^\infty}\|\p_3u_1\|_{L_{x_1,x_2}^{\infty}L_{x_3}^2}\nonumber\\
&\leq CC_0^2\varepsilon^2\left(1+t\right)^{-3}.\nonumber 
\end{align}

Thus, inserting \eqref{x1}, \eqref{x3g}  and \eqref{x4b} into \eqref{233u1}, we obtain after choosing $C_0$ large enough and $\varepsilon$ suitably small that
$$
\|\p_2\p_3^2u_1(t)\|_{L^2}\leq C\varepsilon(1+t)^{-\frac94}+  C C_0^2 \varepsilon^2\left(1+t\right)^{-\frac94}\leq \frac{ C_0}{2} \varepsilon\left(1+t\right)^{-\frac94}. $$

\vskip .1in
\textbf{Part III. Improved decay rate of $\|\p_2\p_3^3u_1\|_{L^2}$}
\vskip .1in

Based on \eqref{u} and Plancherel's theorem, one has
\begin{align}
\|\p_2\p_3^3u_1\|_{L^2}&=\|\widehat{\p_2\p_3^3u_1}\|_{L^2}=\||\xi_{2}\xi_{3}^3|\widehat{u_1}\|_{L^2}\nonumber\\
&\leq \left\| \xi_{2}\xi_{3}^3\widehat{K_1}(t)\widehat{u_{10}}\right\|_{L^2}+\left\|\xi_{2}\xi_{3}^3\widehat{K_2}(t)\widehat{b_{10}}\right\|_{L^2}\nonumber\\
&\quad+\int_0^t\left\|\xi_{2}\xi_{3}^3\widehat{K_1}(t-\tau)\widehat{N_{11}}(\tau)\right\|_{L^2}d\tau\nonumber\\
&\quad+\int_0^t\left\|\xi_{2}\xi_{3}^3\widehat{K_2}(t-\tau)\widehat{N_{21}}(\tau)\right\|_{L^2}d\tau
\triangleq\sum_{m=1}^4\Omega_{3,m} .\label{2333u1}
\end{align}

It follows from  \eqref{S1}, \eqref{S2} and Corollary \ref{divfree} with $\beta=4$ that
\begin{align}
\Omega_{3,1}+\Omega_{3,2}
\leq C(1+t)^{-\frac{11}{4}}\left(\| u_{0}\|_{L^1}+\|\p_2\p_3^3u_0\|_{L^2 }\right)\leq C\varepsilon(1+t)^{-\frac{11}{4}}.\label{Re1}
\end{align}

Next, we deal with the term $\Omega_{3,3}$, which can be written in terms of (\ref{KS1}), (\ref{KS21})--(\ref{KS23}) and (\ref{P1}) as follows,
\begin{align}
\Omega_{3,3}
&\leq C\int_0^t\left\|\xi_2\xi_3^3 e^{-c\xi_{\nu}^2\left(t-\tau \right)}\sum_{k=1}^{3}\xi_{\nu}^2|\xi|^{-2}\xi_k\widehat{b_kb_1}\right\|_{L^2}d\tau\nonumber\\
&\quad+C\int_0^t\left\|\xi_2\xi_3^3 e^{-c\xi_{\nu}^2\left(t-\tau \right)} \sum_{k=1}^{3}\sum_{l=2}^{3}\xi_1|\xi|^{-2}\xi_k\xi_l\widehat{b_kb_l}\right\|_{L^2}d\tau\nonumber\\
&\quad+C\int_0^t\left\|\xi_2\xi_3^3 e^{-c\left(1+\xi_3^2\right)\left(t-\tau \right)} \sum_{k=1}^{3}\xi_{\nu}^2|\xi|^{-2}\xi_k\widehat{b_kb_1}\right\|_{L^2}d\tau\nonumber\\
&\quad+C\int_0^t\left\|\xi_2\xi_3^3 e^{-c\left(1+\xi_3^2\right)\left(t-\tau \right)} \sum_{k=1}^{3}\sum_{l=2}^{3}\xi_1|\xi|^{-2}\xi_k\xi_l\widehat{b_kb_l}\right\|_{L^2}d\tau\nonumber\\
&\triangleq \Omega_{3,31}+\Omega_{3,32}+\Omega_{3,33}+\Omega_{3,34}.\label{Q2333b}
\end{align}

Using (\ref{u101}), (\ref{u4}), (\ref{u5}), \eqref{bbinfty}, \eqref{b1binfty} and the divergence-free condition $\nabla\cdot b=0$, we have
\begin{align}
\|\p_3^3\left( b\cdot \na b_1\right)\|_{L^2}
&\leq C\|b_1\|_{ L^\infty}\|\p_1\p_3^3 b_1 \|_{L^2} +C\|b_\nu\|_{ L^\infty}\|\p_\nu\p_3^3 b_1 \|_{L^2} \nonumber\\
&\quad+ C\|\p_3^3b_1\|_{L_{x_1,x_3}^2L_{x_2}^{\infty}}\|\p_1b_1\|_{L_{x_1,x_3}^\infty L_{x_2}^2}\nonumber\\
&\quad+C\|\p_3^3b_\nu \|_{L_{x_1,x_3}^{2}L_{x_2}^\infty}\|\p_\nu b_1\|_{L_{x_1,x_3}^{\infty}L_{x_2}^2}\nonumber\\
&\quad+ C \|\p_3^2b_1 \|_{L_{x_1,x_3}^{2}L_{x_2}^\infty}\|\p_1\p_3b_1\|_{L_{x_1,x_3}^{\infty}L_{x_2}^2}\nonumber\\
&\quad+ C\|\p_3^2b_\nu\|_{L_{x_1,x_2}^{\infty}L_{x_3}^2}\|\p_\nu\p_3b_1\|_{L_{x_1,x_2}^2 L_{x_3}^\infty } \nonumber\\
&\quad+ C\|\p_3b_1\|_{L_{x_1x_2}^\infty L_{x_3}^2}\|\p_1\p_3^2b_1\|_{L_{x_1,x_2}^{2}L_{x_3}^\infty}\nonumber\\
&\quad+C\|\p_3b_\nu  \|_{L_{x_1,x_2}^{\infty}L_{x_3}^2}\|\p_\nu\p_3^2b_1\|_{L_{x_1,x_2}^2L_{x_3}^{\infty}}\nonumber\\
&\leq CC_0^2 \varepsilon^2\left(1+t\right)^{-\frac{15}{4}},  \label{p3bb}
\end{align}
which, together with \eqref{bb11} and Corollary \ref{PI} with $\alpha=5$, yields
\begin{align}
\Omega_{3,31}
&\leq C\int_0^{\frac{t}{2}}\left(1+t-\tau \right)^{-3}\|\| bb_1 \|_{L_{x_2x_3}^1}\|_{L_{x_1}^2}d\tau\nonumber\\
&\quad +C\int_{{\frac{t}{2}}}^{t}\left(t-\tau \right)^{-\frac12} \|\p_3^3(b\cdot \na b_1) \|_{L^2}d\tau \nonumber\\
&\leq CC_0^2\varepsilon^2 \int_0^{\frac{t}{2}}\left(1+t-\tau \right)^{-3} \left(1+\tau\right)^{-\frac{11}{8}}d\tau\nonumber\\
&\quad+ CC_0^2 \varepsilon^2 \int_{{\frac{t}{2}}}^t\left(t-\tau \right)^{-\frac12} \left(1+\tau\right)^{-\frac{15}{4}}d\tau  \nonumber\\
&\leq C C_0^2 \varepsilon^2\left(1+t\right)^{-3}. \label{R11}
\end{align}

Analogously, noting that
\begin{align}
\|\p_\nu\p_3^3(b b_\nu)\|_{L^2}
&\leq C\|b\|_{ L^\infty} \|\p_\nu\p_3^3b\|_{L^2} +C\|\p_3^3b\|_{L_{x_1,x_3}^2 L_{x_2}^\infty}\|\p_\nu b\|_{L_{x_1,x_3}^{\infty}L_{x_2}^2}\nonumber\\
&\quad+C\|\p_3^2\p_\nu b \|_{L_{x_1,x_2}^{2}L_{x_3}^\infty}\|\p_3b\|_{L_{x_1,x_2}^\infty L_{x_3}^{2}}\nonumber\\
&\quad+ C\|\p_3^2b\|_{ L_{x_2,x_3}^\infty L_{x_1}^2}\|\p_\nu\p_3b\|_{L_{x_2,x_3}^2L_{x_1}^{\infty}}\nonumber\\
& \leq CC_0^2 \varepsilon^2\left(1+t\right)^{-\frac72},  \label{334b2}
\end{align}
from which and \eqref{1bbv2}, we obtain by choosing $\delta\in(0,1/4]$ that
\begin{align}
\Omega_{3,32}
&\leq C\int_0^{\frac t2}\left(1+t-\tau \right)^{-3} \|( bb_2, bb_3)  \|_{L_{x_1}^2 L_{x_2,x_3}^1}d\tau\nonumber\\
&\quad+C\int_{\frac t2}^{t}\left\|\xi_2e^{-c\xi_{\nu}^2\left(t-\tau \right)}  \right\|_{L^\infty}  \|(\p_2\p_3^3 (b b_2 ), \p_3^4 (b b_3)) \|_{L^2} d\tau \nonumber\\
&\leq CC_0^2\varepsilon^2\int_0^{\frac t2}\left(1+t-\tau \right)^{-3}\left(1+\tau \right)^{-1}d\tau\nonumber\\
&\quad+CC_0^2\varepsilon^2\int_{\frac t2}^{t} \left(t-\tau \right)^{-\frac12}\left(1+\tau\right)^{-\frac72}d\tau \nonumber\\
&\leq CC_0^2\varepsilon^2 \left(1+t \right)^{-3+\delta}+ CC_0^2\varepsilon^2\left(1+\tau\right)^{-3}  \nonumber\\
  &\leq CC_0^2\varepsilon^2 \left(1+t\right)^{-\frac{11}{4}}. \label{R12}
\end{align}

It is easily deduced from \eqref{p233b2} and (\ref{ineq}) that
\begin{align}
\Omega_{3,33}+ \Omega_{3,34}
&\leq C\int_0^te^{-c\left(t-\tau \right)}\left(t-\tau \right)^{-\frac12}\left\|\p_2\p_3^2\left(b\cdot\na b\right) \right\|_{L^2}d\tau\nonumber\\
&\leq  CC_0^2\varepsilon^2\int_0^te^{-c\left(t-\tau \right)}\left(t-\tau \right)^{-\frac12}\left(1+\tau\right)^{-3} d\tau\nonumber\\
&\leq CC_0^2\varepsilon^2\left(1+t\right)^{-3}. \label{R34}
\end{align}

Thus, by virtue of \eqref{R11}, \eqref{R12}  and \eqref{R34}, we infer from \eqref{Q2333b} that
\begin{align}
\Omega_{3,3}\leq C C_0^2 \varepsilon^2\left(1+t\right)^{-\frac{11}{4}}. \label{R123g}
\end{align}

Similarly to that for $\Omega_{2,4}$,  we have from (\ref{KS1}), (\ref{KS21})--(\ref{KS23}) and  (\ref{N21}) that
\begin{align*}
\Omega_{3,4}&\leq \int_0^t\left\|\xi_2^2\xi^3_{3} e^{-c\xi_{\nu}^2\left(t-\tau \right)}\widehat{b_2 u_1} \right\|_{L^2}d\tau +\int_0^t\left\|\xi_{2}\xi^4_{3} e^{-c\xi_{\nu}^2\left(t-\tau \right)}\widehat{b_3 u_1}  \right\|_{L^2}d\tau\nonumber \\
&\quad+\int_0^t\left\|\xi_2^2\xi^3_{3}e^{-c\left(1+\xi_3^2\right)\left(t-\tau \right)}\widehat{b_2 u_1}  \right\|_{L^2}d\tau +\int_0^t\left\|\xi_{2}\xi^4_{3}e^{-c\left(1+\xi_3^2\right)\left(t-\tau \right)}\widehat{b_3 u_1}  \right\|_{L^2}d\tau
\end{align*}
and thus,
\begin{align}
\Omega_{3,4}
& \leq C\int_0^{\frac{t}{2}}\left(1+t-\tau \right)^{-3} \|  (b_2u_1, b_3u_1)\|_{L_{x_1}^2 L_{x_2,x_3}^1} d\tau \nonumber\\
&\quad + C\int_{\frac{t}{2}}^{t}\left(t-\tau \right)^{-\frac12} \| (\p_2\p_3^3 (b_2u_1 ),\p_3^4 (b_3u_1 ))\|_{L^2} d\tau \nonumber\\
&\quad +C\int_{0}^{t} e^{-c\left(t-\tau\right)}\left(t-\tau\right)^{-\frac12} \|(\p_2^2\p_3^2 ( b_2u_1 ),\p_2\p_3^3 (b_3u_1 ) )\|_{L^2}d\tau \nonumber\\
& \leq CC_0^2\varepsilon^2\int_0^{\frac{t}{2}}\left(1+t-\tau \right)^{-3} \left(1+\tau\right)^{-\frac{11}{8}}d\tau\nonumber\\
&\quad+CC_0^2\varepsilon^2\int_{\frac{t}{2}}^t\left(t-\tau \right)^{-\frac{1}{2}} \left(1+\tau\right)^{-\frac{15}{4}}d\tau\nonumber\\
&\quad +CC_0^2\varepsilon^2\int_0^te^{-c\left(t-\tau\right)}\left(t-\tau\right)^{-\frac12}\left((1+\tau )^{-3}+(1+\tau )^{-\frac{15}{4}}\right)d\tau \nonumber\\
&\leq C C_0^2 \varepsilon^2\left(1+t\right)^{-3}, \label{R4b}
\end{align}
where we have used (\ref{u101}), (\ref{u4}), (\ref{u5}), \eqref{bbinfty}, \eqref{b1binfty} and the divergence-free condition $\nabla\cdot u=0$ to get that
\begin{align}
&\|\p_2\p_3^3\left(bu_1\right)\|_{L^2}+\|\p_3^4\left(bu_1\right) \|_{L^2}+\| \p_2\p_3^3\left(bu_1\right) \|_{L^2} \nonumber\\
&\quad \leq C\|b\|_{L^\infty} \|\p_\nu\p_3^3u_1\|_{L^2} +C\|u_1\|_{L^\infty} \|\p_\nu\p_3^3b\|_{L^2} \nonumber\\
&\qquad +C\|\p_\nu\p_3^2b\|_{L_{x_1,x_2}^2L_{x_3}^{\infty}}\|\p_3u_1\|_{L_{x_1,x_2}^\infty L_{x_3}^2}+C\|\p_3^3b \|_{L_{x_1,x_3}^{2}L_{x_2}^\infty}\|\p_\nu u_1\|_{L_{x_1,x_3}^{\infty}L_{x_2}^2}\nonumber\\
&\qquad
+C\|\p_3^2b\|_{L_{x_1,x_2}^{\infty}L_{x_3}^2}\|\p_\nu \p_3u_1\|_{L_{x_1,x_2}^2 L_{x_3}^\infty}+\|\p_\nu\p_3b \|_{L_{x_1,x_3}^{\infty}L_{x_2}^2}\|\p_3^2u_1\|_{L_{x_1,x_3}^2L_{x_2}^{\infty}}
\nonumber\\
&\qquad
+C\|\p_3b\|_{L_{x_1,x_2}^{\infty}L_{x_3}^2}\|\p_2\p_3^2u_1\|_{L_{x_1,x_2}^2 L_{x_3}^\infty}+C\|\p_2b\|_{L_{x_1,x_3}^\infty L_{x_2}^{2}}\|\p_3^3u_1 \|_{L_{x_1,x_3}^{2}L_{x_2}^\infty}\nonumber\\
&\quad\leq CC_0^2\varepsilon^2\left(1+t\right)^{-\frac{15}{4}},\label{p2p3bb}
\end{align}
and
\begin{align}
\|\p_2^2\p_3^2\left(b_2u_1\right) \|_{L^2}
& \leq C\|b\|_{L^\infty}\|\p_2^2\p_3^2 u_1 \|_{L^2}+C\|u_1\|_{L^\infty} \|\p_2^2\p_3^2 b \|_{L^2} \nonumber\\
&\quad+C\|\p_i\p_j\p_3b\|_{L_{x_1,x_2}^2L_{x_3}^{\infty}}\|\p_\nu u_1\|_{L_{x_1,x_2}^\infty L_{x_3}^2} \nonumber\\
&\quad
+C\|\p_\nu \p_3b\|_{L_{x_1,x_2}^{\infty}L_{x_3}^2}\|\p_2\p_\nu u_1\|_{L_{x_1,x_2}^2 L_{x_3}^\infty}\nonumber\\
&\quad+C\|\p_2^2b \|_{L_{x_1,x_3}^{\infty}L_{x_2}^2}\|\p_3^2u_1\|_{L_{x_1,x_3}^2L_{x_2}^{\infty}}\nonumber\\
&\quad
+\|\p_\nu b \|_{L_{x_1,x_2}^{\infty}L_{x_3}^2}\|\p_i\p_j\p_3 u_1\|_{L_{x_1,x_2}^2L_{x_3}^{\infty}}\nonumber\\
&\leq CC_0^2\varepsilon^2\left(1+t\right)^{-3}\nonumber 
\end{align}
with  $i,j\in\{2,3\}$. Thus, inserting \eqref{Re1}, \eqref{R123g}  and \eqref{R4b} into \eqref{2333u1} gives rise to
\begin{align}
\|\p_2\p_3^2u_1(t)\|_{L^2}\leq \left( C\varepsilon+ C C_0^2 \varepsilon^2\right)\left(1+t\right)^{-\frac{11}{4}}
\leq  \frac{ C_0}{2} \varepsilon\left(1+t\right)^{-\frac{11}{4}},
\end{align}
provided $C_0$ is chosen to be large enough and $\varepsilon$ is chosen to be suitably small.

\vskip .1in
\textbf{Part IV. Improved decay rate of $\|\p_3^4u_1 \|_{L^2}$}
\vskip .1in
In terms of \eqref{u},  we have
\begin{align}
\|\p_3^4u_1\|_{L^2}&=\|\widehat{\p_3^4u_1}\|_{L^2}=\||\xi_{3}^4|\widehat{u_1}\|_{L^2}\nonumber\\
&\leq \left\| \xi_{3}^4\widehat{K_1}(t)\widehat{u_{10}}\right\|_{L^2}+\left\|\xi_{3}^4\widehat{K_2}(t)\widehat{b_{10}}\right\|_{L^2}\nonumber\\
&\quad+\int_0^t\left\|\xi_{3}^4\widehat{K_1}(t-\tau)\widehat{N_{11}}(\tau)\right\|_{L^2}d\tau\nonumber\\
&\quad+\int_0^t\left\|\xi_{3}^4\widehat{K_2}(t-\tau)\widehat{N_{21}}(\tau)\right\|_{L^2}d\tau \triangleq \sum_{m=1}^4\Omega_{4,m}.\label{34u1}
\end{align}

It follows from Corollary \ref{divfree} with $\beta=4$ that
\begin{align}
\Omega_{4,1}+\Omega_{4,2}
 \leq C(1+t)^{-\frac{11}{4}}\left(\| u_{0}\|_{L^1}+\|\p_3^4u_0\|_{L^2 }\right)\leq  C\varepsilon(1+t)^{-\frac{11}{4}}.\label{L1}
\end{align}

Due to (\ref{KS1}), (\ref{KS21})--(\ref{KS23}) and (\ref{P1}), we have
\begin{align}
\Omega_{4,3}&\leq C\int_0^t\left\|\xi_3^4 e^{-c\xi_{\nu}^2\left(t-\tau \right)}\sum_{k=1}^{3}\xi_{\nu}^2|\xi|^{-2}\xi_k\widehat{b_kb_1}\right\|_{L^2}d\tau\nonumber\\
&\quad+C\int_0^t\left\||\xi_3|^4 e^{-c\xi_{\nu}^2\left(t-\tau \right)} \sum_{k=1}^{3}\sum_{l=2}^{3}\xi_1|\xi|^{-2}\xi_k\xi_l\widehat{b_kb_l}\right\|_{L^2}d\tau\nonumber\\
&\quad+C\int_0^t\left\||\xi_3|^4 e^{-c\left(1+\xi_3^2\right)\left(t-\tau \right)} \sum_{k=1}^{3}\xi_{\nu}^2|\xi|^{-2}\xi_k\widehat{b_kb_1}\right\|_{L^2}d\tau\nonumber\\
&\quad+C\int_0^t\left\||\xi_3|^4 e^{-c\left(1+\xi_3^2\right)\left(t-\tau \right)} \sum_{k=1}^{3}\sum_{l=2}^{3}\xi_1|\xi|^{-2}\xi_k\xi_l\widehat{b_kb_l}\right\|_{L^2}d\tau\nonumber\\
&\triangleq \Omega_{4,31}+\Omega_{4,32}+\Omega_{4,33}+\Omega_{4,34}\label{Q3333b}
\end{align}

Similarly to the derivations of (\ref{R11}) and (\ref{R12}), by  \eqref{bb11} and \eqref{p3bb} we find
\begin{align}
\Omega_{4,31}
&\leq C\int_0^{\frac{t}{2}}\left(1+t-\tau \right)^{-3}\|  bb_1 \|_{L_{x_1}^2 L_{x_2,x_3}^1} d\tau\nonumber\\
 &\quad +C\int_{{\frac{t}{2}}}^{t}\left(t-\tau \right)^{-\frac12} \|\p_3^3(b\cdot \na b_1) \|_{L^2}d\tau \nonumber\\
&\leq C C_0^2 \varepsilon^2\left(1+t\right)^{-3}, \label{L11}
\end{align}
and by \eqref{1bbv2} and \eqref{334b2} we obtain
\begin{align}
\Omega_{4,32}&\leq C\int_0^{\frac t2}\left(1+t-\tau \right)^{-3} \| (bb_2, bb_3 ) \|_{L_{x_1}^2L_{x_2,x_3}^1} d\tau\nonumber\\
&\quad+C\int_{\frac t2}^{t}\left\||\xi_\nu|e^{-c\xi_{\nu}^2\left(t-\tau \right)} \right\|_{L^\infty}  \|\p_3^4 (b b_\nu) \|_{L^2} d\tau \nonumber\\
&\leq CC_0^2\varepsilon^2 \left(1+t\right)^{-\frac{11}{4}}. \label{L12}
\end{align}

Due to  \eqref{p333b2} and (\ref{ineq}), we deduce
\begin{align}
\Omega_{4,33}+ \Omega_{4,34}
&\leq C\int_0^te^{-c\left(t-\tau \right)}\left(t-\tau \right)^{-\frac12}\left\|\p_3^3\left(b\cdot\na b\right) \right\|_{L^2}d\tau\nonumber\\
&\leq  CC_0^2\varepsilon^2\int_0^te^{-c\left(t-\tau \right)}\left(t-\tau \right)^{-\frac12}\left(1+\tau\right)^{-\frac{27}{8}} d\tau\nonumber\\
&\leq CC_0^2\varepsilon^2\left(1+t\right)^{-\frac{27}{8}}. \label{L134}
\end{align}

By \eqref{L11}, \eqref{L12}  and \eqref{L134}, we conclude from (\ref{Q3333b}) that
\begin{align}
\Omega_{4,3}\leq C C_0^2 \varepsilon^2\left(1+t\right)^{-\frac{11}{4}}. \label{L123g}
\end{align}

Finally, analogously to the treatment of $\Omega_{3,4}$ in \eqref{R4b}, using \eqref{bb11}, \eqref{334b2} and  \eqref{p2p3bb}, we obtain
\begin{align}
\Omega_{4,4}&\leq C\int_0^{\frac{t}{2}}\left(1+t-\tau \right)^{-3} \|  (b_2u_1,  b_3u_1 )\|_{L_{x_1}^2L_{x_2,x_3}^1} d\tau \nonumber\\
&\quad+ C\int_{\frac{t}{2}}^{t}\left(t-\tau \right)^{-\frac12} \| (\p_3^4 ( b_2u_1 ),\p_3^4 (b_3u_1 ) \|_{L^2} d\tau \nonumber\\
&\quad +C\int_{0}^{t} e^{-c\left(t-\tau\right)}\left(t-\tau\right)^{-\frac12}\|(\p_2\p_3^3 ( b_2u_1 ),\p_3^4(b_3u_1 ) \|_{L^2}d\tau \nonumber\\
& \leq  C C_0^2 \varepsilon^2\left(1+t\right)^{-3}. \label{L4}
\end{align}

Now, plugging \eqref{L1}, \eqref{L123g}  and \eqref{L4} into \eqref{34u1} gives rise to
\begin{align*}
\|\p_3^4u_1\|_{L^2}\leq\left( C\varepsilon+ C C_0^2 \varepsilon^2\right) (1+t )^{-\frac{11}{4}}\leq \frac{C_0}{2}\varepsilon  (1+t )^{-\frac{11}{4}},
\end{align*}
provided $C_0$ is chosen to be large enough and $\varepsilon$ is chosen to be suitably small.

\vskip .1in

\begin{proof}[Proof of Theorem \ref{thm1.3}] With all the estimates established in Subsections 5.2 and 5.3 at hand,  we
immediately obtain the desired decay rates of the solutions stated in Theorem \ref{thm1.3},
based on  (\ref{u101}), (\ref{u4}), (\ref{u5}) and the bootstrapping arguments.
\end{proof}

\vskip .2in
\noindent {\bf Acknowledgments}.
S. Lai was partially supported by the National Natural Science Foundation of China (Grant Nos. 12201262, 12071390), the Natural Science Foundation of Jiangxi Province (Grant No. 20232BAB211001) and Jiangxi Province Key Subject Academic and Technical Leader Funding Project (Grant No. 20232BCJ23037). J. Wu was partially supported by the National Science Foundation of the United States
under the grant DMS 2104682 and DMS 2309748.  J. Zhang was partially supported by  the National Natural Science Foundation of China (Grant Nos. 12071390, 12131007, 12226344). X. Zhao was partially supported by  the National Natural Science Foundation of China  (Grant No. 12101200) and China Postdoctoral Science Foundation (Grant No. 2022M721035).

\end{document}